\documentclass[11pt,a4paper]{report}

\usepackage{fullpage}
\usepackage[ngerman]{babel}
\usepackage[utf8]{inputenc}
\usepackage{bibgerm}
\usepackage{amsmath}
\usepackage{amssymb}
\usepackage{amsthm}
\usepackage{dsfont}
\usepackage{enumitem}
\usepackage{scalerel}
\usepackage{mathtools}

\theoremstyle{plain}
\newtheorem{satz}{Satz}[section]
\newtheorem{korollar}[satz]{Korollar}
\newtheorem{lemma}[satz]{Lemma}

\theoremstyle{definition}
\newtheorem{beispiel}[satz]{Beispiel}
\newtheorem{definition}[satz]{Definition}
\newtheorem{notation}[satz]{Notation}

\theoremstyle{remark}
\newtheorem{bemerkung}[satz]{Bemerkung}

\numberwithin{equation}{section} 

\newcommand{\N}{\mathbb{N}}

\newcommand{\R}{\mathbb{R}}
\newcommand{\C}{\mathbb{C}}

\newcommand{\T}{\mathbb{T}}

\newcommand{\BB}{\mathcal{B}}

\newcommand{\EE}{\mathcal{E}}
\newcommand{\FF}{\mathcal{F}}

\newcommand{\KK}{\mathcal{K}}
\newcommand{\MM}{\mathcal{M}}

\newcommand{\OO}{\mathcal{O}}

\newcommand{\TT}{\mathcal{T}}

\newcommand{\FrU}{\mathfrak{U}}

\newcommand{\set}[2]{\left\{#1:#2\right\}}
\newcommand{\compl}[1]{#1^{\textnormal{c}}}
\newcommand{\intr}[1]{#1^{o}}
\newcommand{\cl}[1]{\overline{#1}}

\newcommand{\abs}[2][]{\left|#2\right|_{#1}}
\newcommand{\norm}[2][]{\left\|#2\right\|_{#1}}
\newcommand{\chevr}[3][]{\left\langle #2,\,#3\right\rangle_{#1}}
\newcommand{\pred}[2][]{PD_{#1}(#2)}
\newcommand{\lanh}[1]{{}^{\bot}#1}
\newcommand{\ranh}[1]{#1^{\bot}}
\newcommand{\orth}[1]{#1^{\bot}} 

\DeclareMathOperator{\spn}{span}
\DeclareMathOperator{\ran}{ran}
\DeclareMathOperator{\id}{id}

\DeclareMathOperator{\re}{Re} 
\DeclareMathOperator{\im}{Im} 
\DeclareMathOperator{\tr}{tr}
\DeclareMathOperator{\supp}{supp}

\DeclareMathOperator{\Sub}{Sub}

\usepackage{wrapfig}
\usepackage{graphicx}

\begin{document}

\begin{titlepage}
  \begin{center}
    \includegraphics[width=0.45\textwidth]{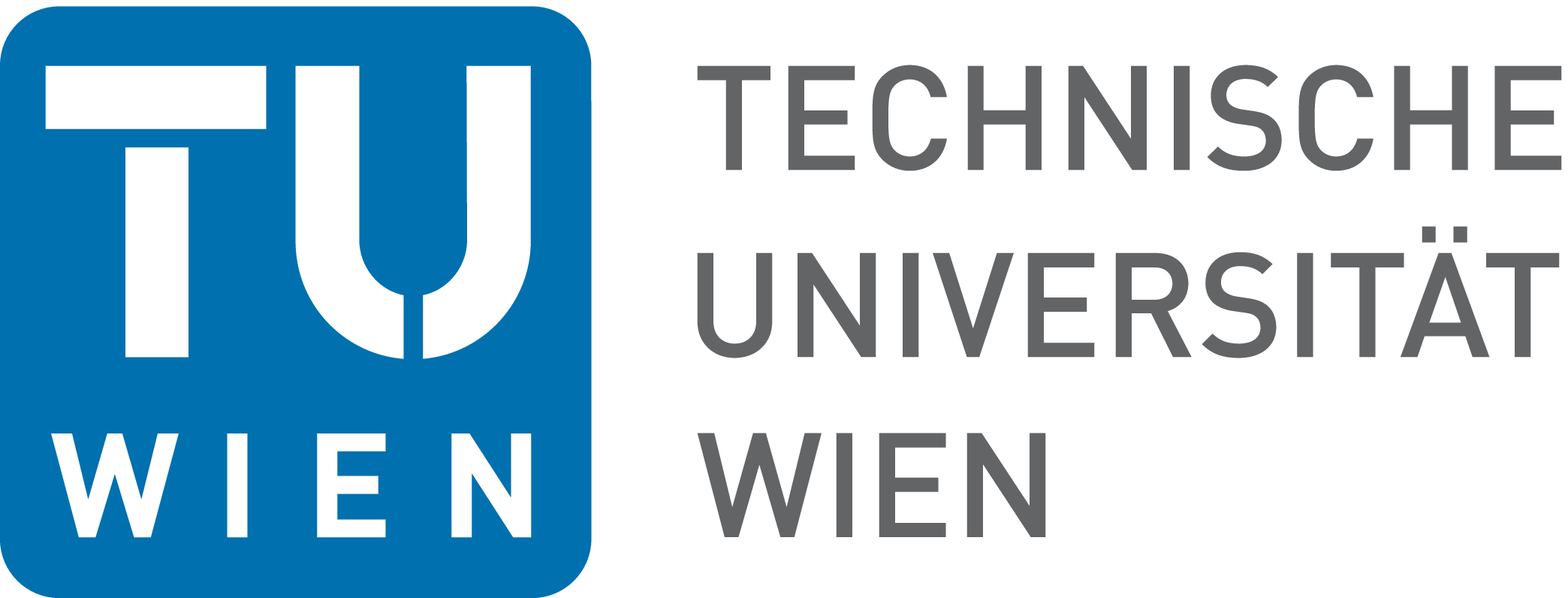}
    \vskip 1.5cm
    {\LARGE \textsc{Diplomarbeit}}
    \vskip 1cm
    {\huge\bfseries $W^{*}$-Algebren als vollständige\\[0.4ex]
      Axiomatisierung der\\[1ex]
      Von-Neumann-Algebren}
    \vskip 1.5cm
    \large 
    ausgeführt am    
    \vskip 0.75cm
    {\Large Institut für\\[1ex] Analysis und Scientific Computing}\\[1ex]
    {\Large TU Wien}
    \vskip 1.25cm
    unter der Anleitung von
    \vskip 1cm
    {\Large\bfseries Ao.Univ.Prof. Dipl.-Ing. Dr.techn.\\[1ex]
      Michael Kaltenbäck}\\[1ex]
    \vskip 1.25cm
    durch
    \vskip 1cm
    {\Large\bfseries Clemens Schindler, BSc.}\\[1ex]
  \end{center}
  
  \vfill
  
  \small
  Wien, am 28. Dezember 2019
  \vspace*{-15mm}
\end{titlepage}

\tableofcontents
\newpage
\addcontentsline{toc}{chapter}{Einleitung}
\chapter*{Einleitung}
\label{cha:einleitung}
Die Theorie der $C^{*}$-Algebren ist entstanden als Verallgemeinerung
der Theorie beschränkter Operatoren auf einem Hilbertraum,
insbesondere der Spektraltheorie. Anstelle von Räumen von Operatoren
werden abstrakte $*$-Algebren untersucht, die eine derartige
Normstruktur tragen, dass die algebraische und die
analytisch-topologische Struktur kompatibel sind. Die Definition einer
$C^{*}$-Algebra besteht also nur aus einer Liste von Axiomen, ohne auf
Eigenschaften einer konkreten zugrundeliegenden Struktur,
beispielsweise eines Hilbertraums, Bezug zu nehmen. Diese allgemeinere
Sichtweise bietet -- vom unbestrittenen ästhetischen Wert einer
axiomatischen Behandlung abgesehen -- einige "`handfeste"' Vorteile;
zu nennen ist vor allem die gesamte Theorie kommutativer
$C^{*}$-Algebren, die auch bei der Behandlung nichtkommutativer
$C^{*}$-Algebren sowie konkreter normaler Operatoren einfließt.

Es würde jedoch den $C^{*}$-Algebren nicht gerecht werden, sie als
echte Verallgemeinerung des Raums $L_{b}(H)$ der beschränkten
Operatoren auf einem Hilbertraum $H$ und der abgeschlossenen
$*$-Unteralgebren von $L_{b}(H)$ zu sehen. Der Satz von
Gelfand-Naimark besagt nämlich, dass eine beliebige $C^{*}$-Algebra
isometrisch isomorph als $*$-Algebra zu einer solchen abgeschlossenen
$*$-Unteralgebra ist. Dies bedeutet, dass durch $C^{*}$-Algebren die
Struktur abgeschlossener $*$-Unteralgebren des Raums der beschränkten
Operatoren auf einem beliebigen Hilbertraum vollständig axiomatisiert
wird, obwohl die Abgeschlossenheit über die Abbildungsnorm vom
zugrundeliegenden Hilbertraum in ad hoc nicht zu trennender Weise
abhängig ist. Dieser Standpunkt der abstrakten Strukturanalyse soll in
der folgenden Arbeit eingenommen werden, um einerseits den Satz von
Gelfand-Naimark zu beweisen und andererseits eine weitere Klasse von
Operatoralgebren zu untersuchen, die Von-Neumann-Algebren. Diese
ergeben sich auch als abgeschlossene $*$-Unteralgebren, wobei
allerdings eine andere Topologie auf dem Raum der beschränkten
Operatoren betrachtet wird -- eine Möglichkeit ist die schwache
Operatortopologie. Klarerweise hängt diese ebenfalls, intuitiv noch
enger als die Abbildungsnorm, mit der Hilbert\-raumstruktur zusammen,
sodass die Frage nach einer passenden Axiomatisierung aufgeworfen
wird. \emph{S. Sakai} konnte dieses Problem mit der Einführung der
sogenannten \emph{$W^{*}$-Algebren} lösen. Dabei handelt es sich um
$C^{*}$-Algebren, die gleichzeitig, als Banachraum betrachtet, bis auf
isometrische Isomorphie der Dualraum eines Banachraums sind. Dass
dieses Axiomensystem, das keinerlei Hilbertraum erwähnt, die
gewünschten Eigenschaften hat, wird durch den Satz von Sakai
gezeigt. Das Hauptziel der vorliegenden Arbeit soll es sein, aufbauend
auf den Inhalten eines Funktionalanalysis-Vorlesungszyklus den Beweis
dieses Satzes und die dafür notwendigen Grundlagen in moderner Weise
darzustellen. Dabei stützen wir uns auf die Argumentation in Sakais
Buch \cite{sakai:cstar-wstar}, die wir an vielen Stellen ergänzt und
erweitert haben, um diverse Ungenauigkeiten zu klären und wenn
notwendig zu korrigieren. Im Zuge dessen war es möglich, den Satz von
Sakai noch auszubauen, wodurch wir den Beweis eines weiteren Resultats
vereinfachen konnten.

In Kapitel~\ref{cha:resultate-fana} werden notwendige Grundlagen aus
verschiedenen Teilgebieten der Funktionalanalysis behandelt,
Kapitel~\ref{cha:eins-proj-cstar} greift die $C^{*}$-Algebren auf und
konzentriert sich auf gewisse Elemente, nämlich Einselemente und
Projektionen. Kapitel~\ref{cha:abstrakte-c-algebren} schließt die
Behandlung abstrakter $C^{*}$-Algebren mit dem Satz von
Gelfand-Naimark ab. Mit Operatoralgebren beschäftigt sich
Kapitel~\ref{cha:optop-von-neumann-algebren}; insbesondere werden
verschiedene Operatortopologien einführt, deren Zusammenhänge
untersucht sowie die Von-Neumann-Algebren definiert. In den
verbleibenden Kapiteln~\ref{cha:w-algebren} bis
\ref{cha:eindeutigkeit} stehen $W^{*}$-Algebren im
Mittelpunkt. Kapitel~\ref{cha:w-algebren} behandelt allgemeine
Eigenschaften wie eine kanonische Topologie und diverse
Stetigkeitsaussagen. In Kapitel~\ref{cha:satz_sakai} werden die
Verbindungen zu Operatoren, vor allem zu Von-Neumann-Algebren,
untersucht sowie der Satz von Sakai bewiesen. Das abschließende
Kapitel~\ref{cha:eindeutigkeit} beschäftigt sich mit
Eindeutigkeitsaussagen, die der Definition der $W^{*}$-Algebren zu
hoher Ästhetik verhelfen.
\vspace{1cm}

\noindent An dieser Stelle möchte ich mich bei meinem Betreuer Michael
Kaltenbäck dafür bedanken, dass er mich auf ein interessantes Stück
Mathematik aufmerksam gemacht hat, und natürlich für seine
Verbesserungsvorschläge dieser Arbeit. Großer Dank gebührt außerdem
meinen Eltern für ihre moralische und finanzielle Unterstützung.

\clearpage

\chapter{Einige Resultate aus der Funktionalanalysis}
\label{cha:resultate-fana}
Wir starten mit der Auflistung diverser in dieser Arbeit benötigter
Ergebnisse der Funktionalanalysis. Der vollständige Beweis sämtlicher
Aussagen würde allerdings den Rahmen dieser Arbeit sprengen.

\section{Banachräume}
\label{sec:banachraume}
\begin{notation}
  \hspace{0mm}
  \begin{enumerate}[label=(\roman*),ref=\roman*]
  \item Im Folgenden bezeichne $S$ oder $S_{X}$ die abgeschlossene
    Einheitskugel $K_{1}(0)$ eines Banachraums $(X,\norm{\cdot})$ über
    $\C$. Für $r>0$ ist $rS$ dann die Kugel $K_{r}(0)$ mit Radius $r$.
  \item Für die Funktionsauswertung werden wir im Kontext von
    schwachen und schwach-*-Topologien die Dualitätsklammer
    $\chevr{x}{f}:=f(x)$ verwenden.
  \end{enumerate}
\end{notation}
Zunächst wollen wir konvexe, bezüglich verschiedener Topologien
abgeschlossene Teilmengen eines Banachraums betrachten.
\begin{satz}\label{satz:krein-sm-motiv}
  Sei $(X,\norm{\cdot})$ ein Banachraum.
  \begin{enumerate}[label=(\roman*),ref=\roman*]
  \item Eine Teilmenge $C\subseteq X$ ist genau dann bezüglich
    $\norm{\cdot}$ abgeschlossen, wenn für jedes $r>0$ die Menge
    $C\cap rS$ bezüglich $\norm{\cdot}$ abgeschlossen ist.
  \item Ist die Teilmenge $C\subseteq X$ zusätzlich konvex, so ist sie
    genau dann bezüglich der schwachen Topologie $\sigma(X,X')$
    abgeschlossen, wenn für jedes $r>0$ die Menge $C\cap rS$ bezüglich
    $\sigma(X,X')$ ab\-ge\-schlos\-sen ist.
  \end{enumerate}
\end{satz}
\begin{proof}
  \hspace{0mm}\begin{enumerate}[label=(\roman*),ref=\roman*]
  \item Die Beweisrichtung "`$\Rightarrow$"' ist klar, da $rS$
    abgeschlossen ist. Sei also $C\cap rS$ für jedes $r>0$
    abgeschlossen und $(x_{n})_{n\in\N}$ eine Folge in $C$, die gegen
    $x$ konvergiert. Da konvergente Folgen beschränkt sind, gilt
    $\norm{x_{n}}\leq r$ für ein $r>0$ und alle $n\in\N$. Die Folge
    $(x_{n})_{n\in\N}$ liegt folglich in $C\cap rS$ und wegen der
    Abgeschlossenheit dieser Menge damit auch $x$. Insbesondere ist
    $x$ in $C$ enthalten.
  \item Als Folgerung des Satzes von Hahn-Banach ist eine konvexe
    Menge genau dann $\sigma(X,X')$-abgeschlossen, wenn sie
    $\norm{\cdot}$-abgeschlossen ist, da $X$ versehen mit beiden
    Topologien dieselben stetigen linearen Funktionale aufweist. Die
    Aussage folgt daher sofort aus dem letzten Punkt.
  \end{enumerate}
\end{proof}
Ersetzt man $X$ im obigen Satz durch seinen Dualraum $X'$, so stellt
sich die Frage, ob eine analoge Aussage nicht nur für die schwache
Topologie $\sigma(X',X'')$ sondern auch für die schwach-*-Topologie
$\sigma(X',X)$ gilt. Der gerade gegebene Beweis ist nicht einfach
adaptierbar, da die schwache Topologie und die schwach-*-Topologie im
Allgemeinen nicht dieselben stetigen Funktionale
induzieren\footnote{Präzise formuliert passiert dies genau dann, wenn
  $X$ nicht reflexiv ist.}. Es ist nun eine wichtige Tatsache, dass
man dennoch aus den Schnitten mit Vielfachen der Einheitskugel
$S=S_{X'}$ die Abgeschlossenheit einer Menge extrahieren kann.
\begin{satz}[Krein-Smulian]\label{satz:krein-sm}
  Sei $(X,\norm{\cdot})$ ein Banachraum. Eine konvexe Teilmenge
  $C\subseteq X'$ ist genau dann bezüglich der schwach-*-Topologie
  $\sigma(X',X)$ abgeschlossen, wenn für jedes $r>0$ die Menge
  $C\cap rS$ bezüglich $\sigma(X',X)$ abgeschlossen ist.
\end{satz}
\begin{proof}
  Die Richtung "`$\Rightarrow$"' ist wieder klar, da $rS$ auch
  bezüglich $\sigma(X',X)$ abgeschlossen ist; nach dem Satz von
  Banach-Alaoglu ist sie ja sogar kompakt. Für die Umkehrung sei
  auf~\cite[Theorem~V.12.1]{conway:fana} verwiesen.
\end{proof}
Daraus folgt auf einfache Weise:
\begin{satz}[Banach-Dieudonné]\label{satz:banach-dieud}
  Sei $(X,\norm{\cdot})$ ein Banachraum. Eine konvexe, bezüglich der
  skalaren Multiplikation mit positiven Zahlen abgeschlossene
  Teilmenge $C\subseteq X'$ -- beispielsweise ein Unterraum $C\leq X'$
  -- ist genau dann bezüglich der schwach-*-Topologie $\sigma(X',X)$
  abgeschlossen, wenn $C\cap S$ bezüglich $\sigma(X',X)$ abgeschlossen
  ist.
\end{satz}
\begin{proof}
  Wegen der vorausgesetzten Abgeschlossenheit bezüglich der skalaren
  Multiplikation gilt $C\cap rS=r(C\cap S)$ für $r>0$. Da die
  Streckung $x\mapsto rx$ einen Homöomorphismus darstellt, ist somit
  $C\cap rS$ genau dann für jedes $r>0$ abgeschlossen, wenn $C\cap S$
  abgeschlossen ist. Aus dem Satz von Krein-Smulian,
  Satz~\ref{satz:krein-sm}, folgt die Aussage.
\end{proof}
Zum Satz von Banach-Dieudonné existiert ein praktisches Korollar.
\begin{korollar}\label{kor:banach-dieud}
  Sei $(X,\norm{\cdot})$ ein Banachraum und $f$ irgendein lineares
  Funktional auf $X'$, d.~h. $f\in (X')^{*}$. Die Abbildung $f$ ist
  genau dann bezüglich der schwach-*-Topologie $\sigma(X',X)$ stetig,
  wenn die Einschränkung $f|_{S}$ auf die Einheitskugel $S$ stetig
  ist, wobei $S$ mit der Spurtopologie $\sigma(X',X)|_{S}$ versehen
  wird.
\end{korollar}
\begin{proof}
  Ist $f$ schwach-*-stetig, so ist klarerweise auch die Einschränkung
  auf $S$ stetig.

  Sei umgekehrt $f|_{S}$ schwach-*-stetig. Als lineares Funktional ist
  $f$ genau dann schwach-*-stetig, wenn $\ker f$
  schwach-*-abgeschlossen ist. Nach dem Satz von Banach-Dieudonné ist
  das wiederum genau dann der Fall, wenn $(\ker f)\cap S$
  abgeschlossen ist. Wegen $(\ker f)\cap S=(f|_{S})^{-1}(\{0\})$ ist
  dies eine unmittelbare Konsequenz der angenommenen Stetigkeit.
\end{proof}

Das nächste Lemma behandelt die Stetigkeit von Projektionen, also
idempotenten und linearen Abbildungen, auf einem Banachraum. Es sei
daran erinnert, dass diese im Gegensatz zu Orthogonalprojektionen auf
einem Hilbertraum nicht zwingend beschränkt sein müssen.
\begin{lemma}\label{lem:proj-beschr}
  Sei $(X,\norm{\cdot})$ ein Banachraum. Eine Projektion $P:X\to X$
  ist genau dann beschränkt, wenn $\ker P$ und $\ran P$ beide
  abgeschlossen sind.
\end{lemma}
\begin{proof}
  Für beschränktes $P$ ist $\ker P$ als Kern eines stetigen Operators
  abgeschlossen. Mit $P$ ist auch $I-P$ beschränkt, sodass
  $\ran P=\ker(I-P)$ ebenfalls abgeschlossen ist.

  Sind umgekehrt $\ker P$ und $\ran P$ abgeschlossen, so ist mit
  $\ker P$ und $\ran P$ auch der Produktraum $\ker P\times\ran P$,
  versehen beispielsweise mit der Summennorm, ein Banachraum. Wir
  betrachten die offensichtlich beschränkte lineare Abbildung
  \begin{displaymath}
    \varphi: \ker P \times\ran P \to X,\quad \varphi(x_{1},x_{2}):=x_{1}+x_{2}.
  \end{displaymath}
  Da $X$ mit der direkten Summe $\ker P\dotplus\ran P$ von Unterräumen
  übereinstimmt, ist $\varphi$ bijektiv. Nach einem Korollar des
  Satzes von der offenen Abbildung ist $\varphi^{-1}$ und damit auch
  $P=\pi_{2}\circ\varphi^{-1}$ beschränkt, wobei
  $\pi_{2}:\ker P\dotplus\ran P\to\ran P$ die Projektion auf die
  Komponente aus $\ran P$ bezeichnet.
\end{proof}
Kombiniert man die letzten beiden Resultate, so erhält man das
folgende Analogon für die schwach-*-Topologie.
\begin{lemma}\label{lem:proj-wstar-stetig}
  Sei $X$ ein Banachraum und $P:X'\to X'$ eine Projektion. Werden
  sowohl Definitions- als auch Bildbereich mit der schwach-*-Topologie
  $\sigma(X',X)$ versehen, so ist $P$ genau dann stetig\footnote{Der
    Einfachheit halber nennen wir $P$ in diesem Fall auch
    schwach-*-stetig.}, wenn $\ker P$ und $\ran P$ beide
  schwach-*-abgeschlossen sind.
\end{lemma}
\begin{proof}
  Ist $P$ stetig, so folgt die Abgeschlossenheit von $\ker P$ und
  $\ran P$ wie im Beweis von Lemma~\ref{lem:proj-beschr}.

  Sind umgekehrt $\ker P$ und $\ran P$ schwach-*-abgeschlossen, so
  sind diese Unterräume klarerweise auch bezüglich der Normtopologie
  abgeschlossen. Lemma~\ref{lem:proj-beschr} zeigt, dass $P$
  beschränkt ist. Als Nächstes beweisen wir, dass die Einschränkung
  $P|_{S}$ auf die Einheitskugel $S=S_{X'}$ schwach-*-stetig ist. Dazu
  sei $(x_{i})_{i\in I}$ ein Netz in $S$, das bezüglich $\sigma(X',X)$
  gegen $x$ konvergiert. Alle Bilder $Px_{i}$ sind in der kompakten
  Menge $\norm{P}S$ enthalten, sodass es für die Konvergenz
  $Px_{i}\to Px$ genügt zu zeigen, dass $Px$ der einzige Häufungspunkt
  des Netzes $(Px_{i})_{i\in I}$ ist. Sei also $y$ ein Häufungspunkt
  und $(Px_{i(j)})_{j\in J}$ ein gegen $y$ konvergentes Teilnetz. Da
  $\ran P$ bezüglich $\sigma(X',X)$ abgeschlossen ist, gilt
  $y\in\ran P$. Wegen der Abgeschlossenheit von $\ker P$ erhalten wir
  auf ähnliche Weise $x-y=\lim_{j\in J}x_{i(j)}-Px_{i(j)}\in\ker P$.
  Die Summe $x=y+(x-y)$ ist somit die eindeutige Zerlegung in ein
  Element aus $\ran P$ und eines aus $\ker P$, woraus $y=Px$ folgt.

  Der Raum $X'$ trägt die schwach-*-Topologie, die als initiale
  Topologie bezüglich der Auswertungsfunktionale $\iota(x):X'\to\C$,
  $f\mapsto\chevr{x}{f}$ definiert ist. Somit ist $P$ genau dann
  stetig, wenn alle Kompositionen $\iota(x)\circ P:X'\to\C$
  schwach-*-stetig sind. Diese Abbildungen sind lineare Funktionale,
  sodass es nach Korollar~\ref{kor:banach-dieud} genügt, die
  Einschränkungen auf $S$ zu betrachten. Nach dem gerade Gezeigten
  handelt es sich dabei tatsächlich um schwach*-stetige Funktionen,
  womit die Aussage bewiesen ist.
\end{proof}
\begin{bemerkung}\label{bem:reeller-br}
  Alle bisherigen Resultate, insbesondere
  Lemma~\ref{lem:proj-wstar-stetig}, gelten auch für einen reellen
  Banachraum.
\end{bemerkung}

Die schwache Topologie $\sigma(X,Y)$ für einen Banachraum\footnote{Ein
  lokalkonvexer Raum würde ausreichen.} $X$ und einen punktetrennenden
Raum $Y\leq X^{*}$ von Funktionalen auf $X$, aus $\chevr{x}{f}=0$ für
alle $f\in Y$ folgt also schon $x=0$, ist bekanntlich die
\emph{gröbste} Topologie auf $X$, für die $X$ ein lokalkonvexer
topologischer Vektorraum\footnote{Zwecks kompakterer Formulierung
  nennen wir die Topologie in diesem Fall eine
  \emph{Vektorraum-Topologie.}}  und $Y$ der topologische Dualraum von
$X$ ist. Auf natürliche Weise ergibt sich nun die Frage, ob es auch
eine \emph{feinste} derartige Topologie gibt. Es stellt sich heraus,
dass diese Topologie tatsächlich existiert.
\begin{definition}\label{def:tech-mackey}
  Sei $\epsilon>0$ und $D\subseteq X^{*}$ eine Menge von Funktionalen
  auf $X$. Wir definieren
  \begin{displaymath}
    U(D,\epsilon):=\set{x\in
      X}{\abs{\chevr{x}{f}}<\epsilon\,\,\text{für alle}\,f\in D}.
  \end{displaymath}
\end{definition}
\begin{lemma}\label{lem:tech-mackey}
  Sei $Y\leq X^{*}$ ein punktetrennender Raum von Funktionalen auf
  $X$. Dann gibt es eine eindeutige lokalkonvexe Vektorraum-Topologie
  $\tau(X,Y)$ auf $X$, für die eine Nullumgebungsbasis gegeben ist
  durch
  \begin{displaymath}
    \set{U(D,\epsilon)}{\epsilon>0, D\subseteq Y\,\text{ist
        kreisförmig, konvex und $\sigma(Y,X)$-kompakt}}.
  \end{displaymath}
\end{lemma}
\begin{proof}
  Spezialfall von~\cite[III.3.2]{schaefer-wolff:topvctspcs} und
  Umformulierung mithilfe der kanonischen Einbettung $X\to X''$. Siehe
  auch~\cite[III.3.2 Example 4.b]{schaefer-wolff:topvctspcs}.
\end{proof}
\begin{definition}\label{def:mackey-top}
  Die Topologie $\tau(X,Y)$ aus dem letzten Lemma nennt man die
  \emph{Mackey-Topologie} auf $X$ zum Raum $Y$.
\end{definition}
\begin{bemerkung}\label{bem:mackey-top-konv}
  \hspace{0mm}
  \begin{enumerate}[label=(\roman*),ref=\roman*]
  \item\label{item:mackey-top-konv-i} Ein Netz $(x_{i})_{i\in I}$ in
    $X$ konvergiert bezüglich der Mackey-Topologie $\tau(X,Y)$ genau
    dann gegen $x\in X$, wenn es für alle $\epsilon>0$ und alle
    kreisförmigen, konvexen und $\sigma(Y,X)$-kompakten $D\subseteq Y$
    ein $i_{0}\in I$ gibt mit $\abs{\chevr{x_{i}-x}{f}}<\epsilon$ für
    alle $f\in D$ und $i\succcurlyeq i_{0}$. Es liegt also eine
    Verschärfung der schwachen Konvergenz vor, wobei zusätzlich die
    Konvergenz für Funktionale aus einer zulässigen Menge $D$ mit der
    gleichen Geschwindigkeit verläuft. Wir werden diesen
    Konvergenztypus im Folgenden als \emph{auf den kreisförmigen,
      konvexen und $\sigma(Y,X)$-kompakten Mengen gleichgradig
      schwache Konvergenz} bezeichnen.
  \item\label{item:mackey-top-konv-ii} Analog zur schwachen Topologie
    können wir auch die Mackey-Topologie $\tau(X',X)$ durch
    $\tau(X',\iota(X))$ mit der kanonischen Einbettung
    $\iota:X\to X''$ definieren. In dieser Situation erhalten wir als
    Konvergenzbedingung für ein Netz $(f_{i})_{i\in I}$ gegen
    $f\in X'$, dass es zu jedem $\epsilon>0$ und jeder kreisförmigen,
    konvexen und $\sigma(X,X')$-kompakten\footnote{Hier geht ein, dass
      die kanonische Einbettung $\iota:X\to X''$ bezüglich der
      schwachen Topologie auf $X$ und der schwach-*-Topologie auf
      $X''$ ein Homöomorphismus auf das Bild $\iota(X)$ ist.} Menge
    $C\subseteq X$ ein $i_{0}\in I$ geben muss mit
    $\abs{\chevr{x}{f_{i}-f}}<\epsilon$ für alle $x\in C$ und
    $i\succcurlyeq i_{0}$. Mit anderen Worten bedeutet Konvergenz in
    der Mackey-Topologie $\tau(X',X)$ genau gleichmäßige Konvergenz
    auf allen kreisförmigen, konvexen und $\sigma(X,X')$-kompakten
    Mengen $C\subseteq X$.
  \end{enumerate}
\end{bemerkung}
Der folgende Satz beantwortet die oben gestellte Frage:
\begin{satz}[Mackey-Arens]\label{satz:mackey-arens}
  Sei $Y\leq X^{*}$ ein punktetrennender Raum von Funktionalen auf
  $X$. Für die Mackey-Topologie $\tau(X,Y)$ gilt
  \begin{equation}\label{eq:mackey-arens}
    (X,\tau(X,Y))'=Y.
  \end{equation}
  Außerdem ist sie die feinste Vektorraum-Topologie auf $X$, für die
  \eqref{eq:mackey-arens} sinngemäß gilt.
\end{satz}
\begin{proof}
  Siehe~\cite[IV.3.2 -- Corollary 1]{schaefer-wolff:topvctspcs}.
\end{proof}
Im weiteren Verlauf dieser Arbeit wird die Bedeutung der
Mackey-Topologie darin liegen, dass ein Funktional $f:X\to\C$ genau
dann stetig bezüglich der schwachen Topologie $\sigma(X,Y)$ ist, wenn
es bezüglich der Mackey-Topologie $\tau(X,Y)$ stetig ist. Diese
Stetigkeit kann einfacher zu beweisen sein, da $\tau(X,Y)$ feiner ist
als $\sigma(X,Y)$. Um das volle Potenzial dieser Überlegung
auszunützen, beweisen wir bereits an dieser Stelle ein Analogon von
Korollar~\ref{kor:banach-dieud} für die Mackey-Topologie $\tau(X',X)$.
\begin{korollar}\label{kor:mackey-arens}
  Sei $(X,\norm{\cdot})$ ein Banachraum und $f$ irgendein lineares
  Funktional auf $X'$, d.~h. $f\in (X')^{*}$. Die Abbildung $f$ ist
  genau dann bezüglich der Mackey-Topologie $\tau(X',X)$ stetig, wenn
  die Einschränkung $f|_{S}$ auf die Einheitskugel $S$ stetig ist,
  wobei $S$ mit der Spurtopologie $\tau(X',X)|_{S}$ versehen wird.
\end{korollar}
\begin{proof}
  Aus der Stetigkeit von $f$ bezüglich der Mackey-Topologie folgt
  klarerweise auch die Stetigkeit der Einschränkung auf $S$.

  Sei umgekehrt $f|_{S}$ stetig bezüglich $\tau(X',X)|_{S}$. Die Menge
  $(\ker f)\cap S=(f|_{S})^{-1}(\{0\})$ ist
  $\tau(X',X)|_{S}$-abgeschlossen und konvex. Außerdem ist die
  Einheitskugel $S$ bezüglich $\sigma(X',X)$ abgeschlossen, also
  insbesondere bezüglich der feineren Topologie $\tau(X',X)$. Somit
  ist $(\ker f)\cap S$ sogar $\tau(X',X)$-abgeschlossen. Nach einer
  Folgerung des Satzes von Hahn-Banach hängt der Abschluss einer
  konvexen Menge nur vom topologischen Dualraum und nicht von der
  genauen Topologie ab. Da $X'$ versehen mit der Mackey-Topologie
  $\tau(X',X)$ und der schwach-*-Topologie $\sigma(X',X)$ denselben
  topologischen Dualraum hat, ist eine konvexe Teilmenge von $X'$
  genau dann $\tau(X',X)$-abgeschlossen, wenn sie
  $\sigma(X',X)$-abgeschlossen ist. Folglich ist $(\ker f)\cap S$
  abgeschlossen bezüglich $\sigma(X',X)$. Wie im Beweis von
  Korollar~\ref{kor:banach-dieud} ergibt sich daraus die
  $\sigma(X',X)$-Abgeschlossenheit von $\ker f$, also die
  schwach-*-Stetigkeit von $f$. Dies ist äquivalent dazu, dass $f$
  bezüglich der Mackey-Topologie stetig ist.
\end{proof}

\section{$C^{*}$-Algebren}
\label{sec:c-algebren}
Nun gehen wir von Banachräumen zu $C^{*}$-Algebren über. Es ist eine
wichtige Tatsache, dass man zu $C^{*}$-Algebren ohne Eins ein
Einselement adjungieren kann und wieder eine $C^{*}$-Algebra
erhält. Für Banach(-$*$-)algebren ist die Konstruktion relativ einfach
und durch die Betrachtung der Elemente $a+\lambda e$ mit einem
fiktiven Einselement $e$ motiviert.
\begin{bemerkung}\label{bem:b-alg-eins}
  Sei $(A,\norm{\cdot})$ eine Banachalgebra. Definiert man auf
  $\tilde{A}:=A\times\C$ die Addition und Skalarmultiplikation
  komponentenweise, so wird $\tilde{A}$ mit der Norm
  $\norm[\tilde{A}]{(a,\lambda)}:=\norm{a}+\abs{\lambda}$ zu einem
  Banachraum. Wir setzen zusätzlich
  $(a,\lambda)\cdot(b,\mu):=(ab+\lambda b+\mu a,\lambda\mu)$ und
  erhalten, wie man leicht nachrechnet, wieder die Struktur einer
  Banachalgebra. Ist $A$ sogar eine Banach-$*$-Algebra, dann ist
  $\tilde{A}$ mit der Operation
  $(a,\lambda)^{*}:=(a^{*},\overline{\lambda})$ ebenfalls eine
  Banach-$*$-Algebra. Zusätzlich enthält $\tilde{A}$ ein Einselement,
  nämlich $(0,1)$. Die Algebra $\tilde{A}$ enthält $A$ vermöge
  $a\mapsto(a,0)$ isometrisch als Ideal\footnote{Wir meinen hier und
    in der restlichen Arbeit ein Ideal im Sinne der Algebrentheorie
    (nicht der Ringtheorie), also einen \emph{Unterraum}
    $I\leq\tilde{A}$, der die Idealeigenschaft erfüllt: Aus $a\in I$
    und $b\in\tilde{A}$ folgt $ab,ba\in I$.}.  Aufgrund dieser
  isomorphen Einbettung können wir durch Umbenennen der Elemente
  $(a,0)$ zu $a$ zusätzlich erreichen, dass $A$ eine Teilmenge von
  $\tilde{A}$ ist. An dieser Stelle sei noch explizit bemerkt, dass
  sich ein etwaiges Einselement $e$ in $A$ \emph{nicht} auf
  $\tilde{A}$ vererbt, d.~h. $(e,0)$ ist kein Einselement in
  $\tilde{A}$.

  Eine wesentliche Eigenschaft von $\tilde{A}$ ist, dass sich
  multiplikative Funktionale auf $A$ eindeutig zu solchen auf
  $\tilde{A}$ fortsetzen lassen. Das liegt daran, dass jedes
  multiplikative Funktional $\tilde{m}$ auf $\tilde{A}$ die Gleichung
  $\tilde{m}((0,1))=1$ erfüllt: Es gilt
  $\tilde{m}((0,1))=\tilde{m}((0,1)^{2})=\tilde{m}((0,1))^{2}$, also
  $\tilde{m}((0,1))\in\{0,1\}$. Wäre $\tilde{m}((0,1))=0$, dann ergäbe
  sich
  $\tilde{m}((a,\lambda))=\tilde{m}((a,\lambda))\tilde{m}((0,1))=0$
  für jedes $(a,\lambda)\in\tilde{A}$, also $\tilde{m}=0$.  Dies ist
  aber nach Definition eines multiplikativen Funktionals
  ausgeschlossen. Ist $m$ ein multiplikatives Funktional auf $A$, dann
  ist die einzige Möglichkeit der Fortsetzung folglich
  $\tilde{m}((a,\lambda)):=m(a)+\lambda$.  Einfaches Nachrechnen
  zeigt, dass dies tatsächlich ein multiplikatives Funktional
  definiert. Da die Abbildungsnorm eines multiplikativen Funktionals
  auf einer Banach(-$*$-)algebra mit Eins stets gleich $1$ ist,
  siehe~\cite[Korollar~1.3.4]{kaltenb:fana2}, folgt aus dieser
  Konstruktion auch, dass die Abbildungsnorm eines multiplikativen
  Funktionals auf einer beliebigen Banach(-$*$-)algebra zumindest
  kleiner gleich $1$ ist.

  Es sei noch erwähnt, dass ein Element $a\in A$ wegen
  \begin{displaymath}
    (a,0)\cdot(b,\mu)=(ab+\mu a,0)\neq (0,1)
  \end{displaymath}
  niemals in $\tilde{A}$ invertierbar ist.
\end{bemerkung}
\begin{definition}\label{def:b-alg-eins-adjung}
  Ist $A$ eine Banach(-$*$-)algebra, so bezeichne $\tilde{A}$ die
  gemäß Bemerkung~\ref{bem:b-alg-eins} definierte Banach(-$*$-)algebra
  mit Einselement, wobei wir $A$ als Teilmenge von $\tilde{A}$
  auffassen.
\end{definition}

Für $C^{*}$-Algebren ist die Situation komplizierter, da man zur
Aufrechterhaltung der $C^{*}$-Eigen\-schaft eine andere Norm
definieren muss, wie das folgende Beispiel zeigt.
\begin{beispiel}\label{bsp:atilde-cstar-norm}
  Sei $A:=C_{0}(\R)$ die $C^{*}$-Algebra der stetigen, im Unendlichen
  verschwindenden Funktionen auf $\R$. Bezeichnet $f$ die Funktion
  $(t\mapsto\exp(-t^{2}))\in A$, so müsste ein Einselement $e$ die
  Gleichung $e\cdot f=f$ erfüllen. Da $f$ nie den Wert $0$ annimmt,
  kommt nur $e=\mathds{1}$ in Frage; $\R$ ist aber nicht kompakt,
  sodass die konstante Einsfunktion nicht in $A$ liegt. Folglich hat
  $A$ kein Einselement. Wir betrachten die Banach-$*$-Algebra
  $\tilde{A}$ und darin das Element $(-f,1)$. Es gilt
  \begin{displaymath}
    (-f,1)^{*}(-f,1)=(-f,1)(-f,1)=(f^{2}-2f,1),
  \end{displaymath}
  sodass wir\footnote{Ab dem zweiten Gleichheitszeichen rechnen wir im
    Raum $C_{b}(\R)$ der stetigen und beschränkten Funktionen.}
  \begin{align*}
    \norm[\tilde{A}]{(-f,1)^{*}(-f,1)}&=\norm[\infty]{f^{2}-2f}+1=\norm[\infty]{f(f-2)}+1\leq\underbrace{\norm[\infty]{f}}_{=1}\underbrace{\norm[\infty]{2-f}}_{=2}+1\\
                                      &=3<4=\norm[\tilde{A}]{(-f,1)}^{2}.
  \end{align*}
  erhalten. $\tilde{A}$ ist also mit dieser Norm keine
  $C^{*}$-Algebra.
\end{beispiel}
Dennoch gilt:
\begin{satz}\label{satz:cstar-alg-eins}
  Sei $(A,\norm{\cdot})$ eine $C^{*}$-Algebra. Dann gibt es eine Norm
  $\norm[C^{*}]{\cdot}$ auf $\tilde{A}$, sodass
  $(\tilde{A},\norm[C^{*}]{\cdot})$ eine $C^{*}$-Algebra mit Eins ist
  und $\norm[C^{*}]{(a,0)}=\norm{a}$ für alle $a\in A$ gilt.
\end{satz}
\begin{proof}
  Siehe~\cite[Theorem~2.1.6]{murphy:cstar}.
\end{proof}
Im weiteren Verlauf dieser Arbeit ist die von einer $C^{*}$-Algebra
ausgehend gebildete Algebra $\tilde{A}$ stets mit der Norm
$\norm{\cdot}:=\norm[C^{*}]{\cdot}$ versehen und daher eine
$C^{*}$-Algebra mit Eins. Außerdem identifizieren wir wieder $a\in A$
mit $(a,0)\in\tilde{A}$ und nehmen an, dass $A$ eine Teilmenge von
$\tilde{A}$ ist.

Die Adjunktion eines Einselements erlaubt es, die zentralen Begriffe
der Spektraltheorie in beliebigen Banachalgebren zu definieren. Wir
werden uns der Einfachheit halber auf $C^{*}$-Algebren beschränken.
\begin{definition}\label{def:spek-resmenge}
  Sei $A$ eine $C^{*}$-Algebra ohne Einselement. Dann definiert man
  für $a\in A$
  \begin{enumerate}[label=(\alph*),ref=\alph*]
  \item das \emph{Spektrum} von $a$ als
    $\sigma_{A}(a):=\sigma_{\tilde{A}}(a)$,
  \item die \emph{Resolventenmenge} von $a$ als
    $\rho_{A}(a):=\rho_{\tilde{A}}(a)$ und
  \item den \emph{Spektralradius} von $a$ als
    $r_{A}(a):=r_{\tilde{A}}(a)\,\,\big(=\sup_{\lambda\in\sigma_{\tilde{A}}(a)}\abs{\lambda}=\sup_{\lambda\in\sigma_{A}(a)}\abs{\lambda}\big)$,
  \end{enumerate}
  wobei wir die Indizes zur leichteren Lesbarkeit meist weglassen
  werden.
\end{definition}

Eines der mächtigsten Instrumente der Theorie von $C^{*}$-Algebren ist
die Gelfandtransformation. Oftmals wird nur der Fall von kommutativen
$C^{*}$-Algebren mit Einselement diskutiert, es gilt allerdings die
Verallgemeinerung in Satz~\ref{satz:gelfand-trafo}. Bevor wir dazu
kommen, führen wir eine Definition ein.
\begin{definition}\label{def:re-im-cstar}
  Sei $A$ eine $C^{*}$-Algebra. Für $a\in A$ sind der \emph{Real-} und
  \emph{Imaginärteil} von $a$ definiert durch $\re a:=(a+a^{*})/2$ und
  $\im a:=(a-a^{*})/2i$.
\end{definition}
Man überprüft unmittelbar folgende Eigenschaften von Real- und
Imaginärteil:
\begin{lemma}\label{lem:eig-re-im}
  Sei $A$ eine $C^{*}$-Algebra und $a\in A$. Dann sind $\re a$ und
  $\im a$ selbstadjungiert und es gilt $a=\re a+i\im a$, wobei
  $\norm{\re(a)},\norm{\im(a)}\leq\norm{a}$.
\end{lemma}
Der entscheidende Schritt im Beweis der Gelfandtransformation für
kommutative $C^{*}$-Alge\-bren ohne Einselement ist die Verbindung von
multiplikativen Funktionalen auf $A$ mit jenen auf der, wie man leicht
nachrechnet ebenfalls kommutativen, $C^{*}$-Algebra $\tilde{A}$. Dazu
legen wir folgende Notation fest.
\begin{notation}\label{not:mult-fkt}
  Sei $M$ die Menge der multiplikativen Funktionale auf $A$ und
  $\tilde{M}$ die entsprechende Menge auf $\tilde{A}$. Für ein
  Funktional $m\in M$ sei $\tilde{m}\in\tilde{M}$ die gemäß
  Bemerkung~\ref{bem:b-alg-eins} existierende eindeutige Fortsetzung
  zu einem multiplikativen Funktional auf $\tilde{A}$, nämlich
  $(a,\lambda)\mapsto m(a)+\lambda$. Außerdem bezeichne
  $m_{0}'\in\tilde{M}$ das multiplikative Funktional
  $(a,\lambda)\mapsto\lambda$ auf $\tilde{A}$.
\end{notation}

Mit diesen Bezeichnungen erhalten wir:
\begin{lemma}\label{lem:mult-fkt}
  Für eine kommutative $C^{*}$-Algebra $A$ ohne Einselement gilt
  \begin{displaymath}
    \tilde{M}=\set{\tilde{m}}{m\in M}\cup\{m_{0}'\}.
  \end{displaymath}
\end{lemma}
\begin{proof}
  Dass die rechte Seite in $\tilde{M}$ enthalten ist, ist klar.
  Umgekehrt ist für $m'\in\tilde{M}$ die Einschränkung\footnote{Es sei
    daran erinnert, dass wir $A$ als Teilmenge von $\tilde{A}$
    auffassen.}  $m'|_{A}$ auf $A$ ein Funktional $m\in M$ oder die
  Nullfunktion. Im ersten Fall gilt wegen der Eindeutigkeit der
  Fortsetzung $m'=\tilde{m}$, im zweiten Fall $m'=m_{0}'$ aufgrund von
  $m'((0,1))=1$.
\end{proof}

\begin{satz}[Gelfandtransformation]\label{satz:gelfand-trafo}
  Sei $A\neq\{0\}$ eine kommutative $C^{*}$-Algebra ohne
  Einselement. Es gelten folgende Aussagen:
  \begin{enumerate}[label=(\roman*),ref=\roman*]
  \item\label{item:gelfand-trafo-i} Der Gelfandraum $M$ aller
    multiplikativen Funktionale auf $A$ ist nicht leer und, versehen
    mit der Spurtopologie der schwach-*-Topologie, lokalkompakt.
  \item\label{item:gelfand-trafo-ii} Für $a\in A$ gilt
    $\sigma(a)=\set{m(a)}{m\in M}\cup\{0\}$.
  \item\label{item:gelfand-trafo-iii} Für $a\in A$ ist die Abbildung
    \begin{displaymath}
      \hat{a}:
      \begin{cases}
        \hfill M&\to\C \\
        \hfill m&\mapsto m(a)
      \end{cases}
    \end{displaymath}
    ein Element von $C_{0}(M)$.
  \item\label{item:gelfand-trafo-iv} Die \emph{Gelfandtransformation}
    $\hat{.}:A\to C_{0}(M)$ ist ein isometrischer
    $*$-Algebren\-isomorphis\-mus.
  \end{enumerate}
\end{satz}
\begin{proof}
  \hspace{0mm}
  \begin{enumerate}[label=(\roman*),ref=\roman*]
  \item Das bekannte Resultat für kommutative $C^{*}$-Algebren mit
    Einselement, siehe~\cite[Satz~1.4.4]{kaltenb:fana2}, liefert bei
    Anwendung in $\tilde{A}$ für beliebiges $a\in A$
    \begin{equation}\label{eq:gelfand-trafo}
      \sigma_{A}(a)=\sigma_{\tilde{A}}(a)=\set{\tilde{m}(a)}{\tilde{m}\in\tilde{M}}.
    \end{equation}
    Da $A\neq \{0\}$ ist, gibt es ein selbstadjungiertes Element
    $b\in A$ mit $b\neq 0$, nämlich entweder Real- oder Imaginärteil
    irgendeines von $0$ verschiedenen Elements. Wegen
    $r_{\tilde{A}}(b)=\norm{b}>0$ gibt es ein $\tilde{m}\in\tilde{M}$
    mit $\tilde{m}(b)\neq 0$. Die Einschränkung $\tilde{m}|_{A}$ ist
    dann mit den Operationen verträglich und von der Nullfunktion
    verschieden, also ein multiplikatives Funktional auf $A$. Wir
    erhalten $M\neq\emptyset$. Außerdem ist $M\cup\{0\}$ nach
    Bemerkung~\ref{bem:b-alg-eins} eine Teilmenge von $S_{A'}$, die
    sogar schwach-*-abgeschlossen ist: Definieren wir für Elemente
    $a,b\in A$ die Funktion $T_{a,b}:S_{A'}\to\C$ durch
    $T_{a,b}(f):=f(ab)-f(a)f(b)$, so gilt nämlich\footnote{An dieser
      Stelle sei auf die Analogie zum Beweis für $C^{*}$-Algebren mit
      Einselement hingewiesen; dort kann man das Nullfunktional
      mithilfe einer weiteren schwach-*-stetigen Funktion aussondern,
      da ja für das Einselement $e$ schon $m(e)=1$ gelten muss.}
    \begin{displaymath}
      M\cup\{0\}=\bigcap_{a,b\in A}T_{a,b}^{-1}(\{0\}).
    \end{displaymath}
    Da $T_{a,b}$ schwach-*-stetig ist, erhalten wir die behauptete
    Abgeschlossenheit. Nach dem Satz von Banach-Alaoglu ist
    $M\cup\{0\}$ somit schwach-*-kompakt. Folglich ist
    $M=(M\cup\{0\})\setminus\{0\}$ als offene Teilmenge eines
    kompakten Hausdorffraums lokalkompakt.
  \item Diese Aussage ergibt sich sofort aus Lemma~\ref{lem:mult-fkt}
    und \eqref{eq:gelfand-trafo}.
  \item Da $M$ mit der Spurtopologie der schwach-*-Topologie versehen
    ist, ist $\hat{a}$ für beliebiges $a\in A$ eine stetige
    Funktion. Für $\epsilon>0$ ist die Menge
    $\set{m\in M}{\abs{m(a)}\geq\epsilon}$ schwach-*-abgeschlossen in
    $S_{A'}$ und damit schwach-*-kompakt. Daraus folgt
    $\hat{a}\in C_{0}(M)$.
  \item Die Gelfandtransformation ist offenbar linear und
    multiplikativ. Die Verträglichkeit mit $.^{*}$ folgt daraus, dass
    für $m\in M$ und $a\in A$ die Beziehung $m(a^{*})=\overline{m(a)}$
    gilt: Wir schreiben $a=\re a+i\im a$. Da $\re a$ und $\im a$
    selbstadjungiert sind, gilt
    $\sigma(\re a),\sigma(\im a)\subseteq\R$,
    siehe~\cite[Lemma~1.5.5]{kaltenb:fana2}, woraus wegen
    (\ref{item:gelfand-trafo-ii}) sofort $m(\re a),m(\im a)\in\R$
    folgt. Wir erhalten
    \begin{displaymath}
      m(a^{*})=m(\re a-i\im a)=\overline{m(\re a)+i\cdot m(\im
        a)}=\overline{m(a)}.
    \end{displaymath}
    Für $a\in A$ ist $a^{*}a$ selbstadjungiert (sowohl in $A$ als auch
    in $\tilde{A}$). Die Rechnung
    \begin{displaymath}
      \norm[\infty]{\hat{a}}^{2}=\norm[\infty]{\abs{\hat{a}}^{2}}=\norm[\infty]{\overline{\hat{a}}\hat{a}}=\norm[\infty]{\widehat{a^{*}a}}=r_{\tilde{A}}(a^{*}a)=\norm[C^{*}]{(a^{*}a,0)}=\norm{a^{*}a}=\norm{a}^{2}
    \end{displaymath}
    zeigt, dass $\hat{.}$ eine isometrische Abbildung ist. Also ist
    $\hat{.}$ injektiv und $\ran \hat{.}$ abgeschlossen in
    $C_{0}(M)$. Andererseits stellt $\ran \hat{.}$ eine
    punktetrennende und nirgends identisch verschwindende
    $*$-Unteralgebra von $C_{0}(M)$ dar: Erstens gibt es für
    $m_{1}\neq m_{2}$ ein $a\in A$ mit
    $\hat{a}(m_{1})=m_{1}(a)\neq m_{2}(a)=\hat{a}(m_{2})$, und
    zweitens gibt es kein $m\in M$ mit $\hat{a}(m)=0$ für alle
    $a\in A$, da das Nullfunktional nicht im Gelfandraum enthalten
    ist. Der Satz von Stone-Weierstraß für lokalkompakte Räume
    impliziert daher die Dichtheit von $\ran \hat{.}$ in
    $C_{0}(M)$. Insgesamt erhalten wir, dass $\hat{.}$ auch surjektiv
    ist.
  \end{enumerate}
\end{proof}

\begin{bemerkung}\label{bem:alexandroff}
  Ist $A$ eine kommutative $C^{*}$-Algebra ohne Einselement, so ist
  $A$ bzw. $\tilde{A}$ isometrisch isomorph zu $C_{0}(M)$
  bzw. $C(\tilde{M})=C_{0}(\tilde{M})$. Über die (lokal-)kompakten
  Hausdorffräume $M$ und $\tilde{M}$ erhalten wir also einen
  topologischen Zugang zu kommutativen $C^{*}$-Algebren\footnote{Aus
    diesem Grund wird die Theorie von allgemeinen $C^{*}$-Algebren
    auch oft als \emph{nichtkommutative Topologie} bezeichnet.}. Dabei
  gilt folgende bemerkenswerte Tatsache: In der Sprache der
  Gelfandräume entspricht die Adjunktion eines Einselements zu einer
  kommutativen $C^{*}$-Algebra genau der
  Alexandroff-Kompakti\-fizierung\footnote{Ist $(X,\TT)$ ein
    lokalkompakter Hausdorffraum und $\infty$ kein Element von $X$, so
    bezeichnet man $Y:=X\cup\{\infty\}$ versehen mit der kompakten
    Topologie
    $\OO:=\TT\cup\set{(X\setminus K)\cup\{\infty\}}{K\subseteq X\
      \text{kompakt}}$ als \emph{Alexandroff-Kompaktifizierung}.}
  eines lokalkompakten Hausdorffraums. Mit anderen Worten ist
  $\tilde{M}$ homöomorph zur Alexandroff-Kompaktifizierung von $M$,
  die im Folgenden mit $Y$ bezeichnet sei. Die Abbildung
  \begin{displaymath}
    \iota:
    \begin{cases}
      M &\to\tilde{M} \\
      m &\mapsto\tilde{m}
    \end{cases}
  \end{displaymath}
  ist, wie man leicht nachprüft, ein Homöomorphismus $M\to\iota(M)$
  bezüglich der schwach-*-Topologien. Nach Lemma~\ref{lem:mult-fkt}
  gilt dabei $\iota(M)=\tilde{M}\setminus\{m_{0}'\}$, sodass
  $\iota(M)$ eine offene Teilmenge von $\tilde{M}$ ist. Die Funktion
  $f:\tilde{M}\to Y$ definiert durch
  \begin{displaymath}
    f(m'):=\begin{cases}
      m,& m'=\iota(m) \\
      \infty,& m'=m_{0}'
    \end{cases}
  \end{displaymath}
  ist wegen der Injektivität von $\iota$ und Lemma~\ref{lem:mult-fkt}
  wohldefiniert; die Bijektivität ist offensichtlich. Außerdem ist $f$
  stetig, denn ist $O$ eine offene Menge in $Y$, so ist entweder $O$
  eine offene Teilmenge von $M$ oder es gilt
  $O=(M\setminus K)\cup\{\infty\}$ für eine kompakte Menge
  $K\subseteq M$. Im ersten Fall ist $f^{-1}(O)=\iota(O)$ offen in
  $\iota(M)$, somit auch in $\tilde{M}$. Im zweiten Fall gilt
  \begin{displaymath}
    f^{-1}((M\setminus K)\cup\{\infty\})=\tilde{M}\setminus\iota(K),
  \end{displaymath}
  woraus wiederum die Offenheit in $\tilde{M}$ folgt. Als bijektive
  und stetige Abbildung von einem kompakten Raum in einen
  Hausdorffraum ist $f$ somit ein Homöomorphismus von $\tilde{M}$ nach
  $Y$.
\end{bemerkung}

Um die Aussagekraft der Gelfandtransformation auch in beliebigen
$C^{*}$-Algebren ausnutzen zu können, ist es oft hilfreich, zu einer
kommutativen $C^{*}$-Unteralgebra überzugehen, die die aktuell
untersuchten Elemente enthält\footnote{Diese Möglichkeit ist es, die
  die normalen Operatoren in $L_{b}(H)$ oder allgemeiner die normalen
  Elemente von $A$ auszeichnet und dafür sorgt, dass sie leichter
  zugänglich und besser verstanden sind.}.
\begin{definition}\label{def:erz-unteralg}
  Ist $A$ eine $C^{*}$-Algebra und $S\subseteq A$ eine Teilmenge, so
  definieren wir die \emph{von $S$ erzeugte $C^{*}$-Unteralgebra von
    $A$}, in Zeichen $C^{*}_{A}(S)$, als die kleinste
  $C^{*}$-Unteralgebra von $A$, die alle Elemente von $S$
  enthält. Dabei unterdrücken wir den Index $A$, wenn die
  zugrundeliegende $C^{*}$-Algebra aus dem Kontext klar ist. Im Falle
  $S=\{x_{1},\dots,x_{n}\}$ schreiben wir für $C^{*}_{A}(S)$ auch
  $C^{*}_{A}(x_{1},\dots,x_{n})$ bzw. $C^{*}(x_{1},\dots,x_{n})$.
\end{definition}
\begin{bemerkung}\label{bem:erz-unteralg}
  Implizit ist in Definition~\ref{def:erz-unteralg} die Behauptung
  enthalten, dass die Menge der $S$ umfassenden $C^{*}$-Unteralgebren
  von $A$ ein kleinstes Element hat. Bekanntermaßen existiert dieses
  kleinste Element tatsächlich und es gilt
  \begin{displaymath}
    C^{*}_{A}(S)=\cl{\bigcap\set{B}{S\subseteq B\leq A}},
  \end{displaymath}
  wobei $B\leq A$ bedeuten möge, dass $B$ eine $C^{*}$-Unteralgebra
  von $A$ ist.
  
  Ist $S=\{x_{1},\dots,x_{n}\}$ endlich und kommutieren alle Elemente
  von $\{x_{1},\dots,x_{n},x_{1}^{*},\dots,x_{n}^{*}\}$ paarweise
  miteinander, so lässt sich das Erzeugnis
  $C^{*}_{A}(x_{1},\dots,x_{n})$ auch von unten durch
  \begin{equation}\label{eq:erz-unteralg-i}
    C^{*}_{A}(S)=\cl{\set{p(x_{1},\dots,x_{n},x_{1}^{*},\dots
        x_{n}^{*})}{p\in\C[z_{1},\dots,z_{n},w_{1},\dots,w_{n}],\
        p(0,\dots, 0)=0}}
  \end{equation}
  konstruieren. Die Zusatzbedingung $p(0,\dots, 0)$ ist notwendig, da
  man ansonsten ein Einselement zur Verfügung haben müsste, um
  $p(x_{1},\dots,x_{n},x_{1}^{*},\dots,x_{n}^{*})$ berechnen zu
  können. Anhand dieser Darstellung kann man auf einfache Weise
  nachprüfen, dass $C^{*}_{A}(S)$ kommutativ ist.

  Wir wollen noch den Spezialfall betrachten, dass $A$ ein Einselement
  enthält und dass $S$ aus einem einzigen selbstadjungierten Element
  und dem Einselement besteht, $S=\{x,1\}$. Die $C^{*}$-Algebra
  $C_{A}^{*}(S)$ ist also die von $x=x^{*}$ erzeugte $C^{*}$-Algebra
  mit Eins. In diesem Fall können wir statt $p(x,1,x^{*},1^{*})$ für
  ein Polynom $p\in\C[z_{1},z_{2},w_{1},w_{2}]$ mit $p(0,0,0,0)=0$
  auch $r(x)$ für ein Polynom $r\in\C[z]$, das nicht notwendigerweise
  $r(0)=0$ erfüllt, verwenden. Die Formel \eqref{eq:erz-unteralg-i}
  vereinfacht sich somit zu
  \begin{equation}\label{eq:erz-unteralg-ii}
    C_{A}^{*}(x,1)=\cl{\set{r(x)}{r\in\C[z]}}.
  \end{equation}

  Betrachten wir nur $S=\{x\}$, so erhält man auf analoge Art
  \begin{equation}\label{eq:erz-unteralg-iii}
    C_{A}^{*}(x)=\cl{\set{r(x)}{r\in\C[z],\ r(0)=0}}.
  \end{equation}
\end{bemerkung}

Eine unscheinbar wirkende aber sehr nützliche Anwendung des Übergangs
zu einer kommutativen $C^{*}$-Unteralgebra ist das nächste Lemma.
\begin{lemma}\label{lem:norm-summe}
  Sei $A$ eine $C^{*}$-Algebra. Sind $a,b\in A$ selbstadjungiert mit
  $ab=ba=0$, dann gilt $\norm{a+b}=\max(\norm{a},\norm{b})$.
\end{lemma}
\begin{proof}
  Die von $a$ und $b$ erzeugte $C^{*}$-Unteralgebra $C^{*}_{A}(a,b)$
  von $A$ ist kommutativ, da die Elemente $a,a^{*}=a,b,b^{*}=b$
  miteinander kommutieren. Die Aussage betrifft nur die Norm in $A$
  und ist deswegen von der konkreten Unteralgebra unabhängig, solange
  sie nur die Elemente $a$ und $b$ enthält. Folglich können wir zu
  $C^{*}_{A}(a,b)$ übergehen und wegen der Gelfandtransformation sogar
  $A=C_{0}(M)$ mit einem lokalkompakten Hausdorffraum $M$
  annehmen. Die Voraussetzung $ab=ba=0$ besagt dann, dass für jedes
  $m\in M$ höchstens einer der Werte $a(m)$ und $b(m)$ von $0$
  verschieden ist. Daraus folgt
  \begin{displaymath}
    \abs{a(m)+b(m)}=\abs{a(m)}+\abs{b(m)}=\max(\abs{a(m)},\abs{b(m)})
  \end{displaymath}
  und wir erhalten unmittelbar die Aussage.
\end{proof}
Als Abschluss des Abschnitts wollen wir für den späteren Gebrauch das
Spektrum in einer bestimmten $C^{*}$-Algebra mit Einselement
bestimmen, nämlich im Raum $C(K)$.
\begin{beispiel}\label{bsp:spektrum-ck}
  Ist $K$ ein kompakter Hausdorffraum, so sind die invertierbaren
  Elemente von $C(K)$ jene stetigen Funktionen $g$, für die die
  Funktion $1/g$ wohldefiniert und stetig ist. Dies ist genau dann der
  Fall, wenn $g$ nullstellenfrei ist. Für eine Funktion $f\in C(K)$
  besteht somit $\sigma(f)$ aus den $\lambda\in\C$, für die
  $f-\lambda$ eine Nullstelle hat. Anders formuliert gilt
  $\sigma(f)=f(K)$.
\end{beispiel}

\section{Beschränkte Operatoren auf einem Hilbertraum}
\label{sec:beschr-op-auf-hr}
\begin{notation}\label{not:hr}
  \hspace{0mm}
  \begin{enumerate}[label=(\roman*),ref=\roman*]
  \item\label{item:not-hr-i} In der gesamten weiteren Arbeit bezeichne
    $H$ einen Hilbertraum mit Skalarprodukt
    $(\cdot,\cdot):=(\cdot,\cdot)_{H}$. Den Raum der beschränkten
    Operatoren auf $H$ notieren wir als $L_{b}(H)$.
  \item\label{item:not-hr-ii} Ist $T\in L_{b}(H)$, dann bezeichne
    $\abs{T}$ den Operator\footnote{Nach dem Funktionalkalkül für
      selbstadjungierte Operatoren existiert eine eindeutige positive
      Wurzel des positiven Operators $T^{*}T$.} $(T^{*}T)^{1/2}$.
  \end{enumerate}
\end{notation}

Die folgenden beiden Konstruktionen werden sich als nützlich
herausstellen.
\begin{bemerkung}\label{bem:kart-prod-matrix-hr}
  \hspace{0mm}
  \begin{enumerate}[label=(\roman*),ref=\roman*]
  \item\label{item:kart-prod-matrix-hr-i} Ist $I$ eine beliebige
    Indexmenge und $H_{i}$ für $i\in I$ ein Hilbertraum mit
    Skalarprodukt $(.,.)_{H_{i}}$, so ist die äußere direkte Summe
    \begin{displaymath}
      X_{I}:=\prod_{i\in I}H_{i}\,\,\big(=\set{(x_{i})_{i\in I}}{x_{i}\in H_{i}}\big),
    \end{displaymath}
    versehen mit den punktweisen Operationen klarerweise ein
    Vektorraum. Die Teilmenge
    \begin{displaymath}
      \ell^{2}(H_{i} : i\in I):=\set{(x_{i})_{i\in I}\in X_{I}}{\left(\norm[H_{i}]{x_{i}}\right)_{i\in I}\in\ell^{2}(I)}
    \end{displaymath}
    ist ein Unterraum von $X_{I}$, der das Skalarprodukt
    \begin{displaymath}
      \big((x_{i})_{i\in I},(y_{i})_{i\in I}\big):=\sum_{i\in I}(x_{i},y_{i})_{H_{i}}
    \end{displaymath}
    trägt, wobei wir die Summe als Integral bezüglich des Zählmaßes
    oder alternativ als unbedingt konvergente Reihe auffassen. Nach
    der Hölder'schen Ungleichung ist $(\cdot,\cdot)$ wohldefiniert,
    denn $\abs{(x_{i},y_{i})_{H_{i}}}$ können wir punktweise in $i$
    durch das Produkt $\norm[H_{i}]{x_{i}}\norm[H_{i}]{y_{i}}$ zweier
    quadratsummierbarer Funktionen abschätzen. Die durch
    $(\cdot,\cdot)$ induzierte Norm ist
    $\norm{(x_{i})_{i\in I}}=\sqrt{\sum_{i\in
        I}\norm[H_{i}]{x_{i}}^{2}}$.

    Als Nächstes zeigen wir, dass der Raum $\ell^{2}(H_{i}:i\in I)$
    damit zu einem Hilbertraum wird. Dazu sei
    $((x_{i}^{n})_{i\in I})_{n\in\N}$ eine Cauchy-Folge in
    $\ell^{2}(H_{i}:i\in I)$. Wegen
    \begin{displaymath}
      \norm[H_{i}]{x_{j}^{m}-x_{j}^{n}}\leq\norm{(x_{i}^{m}-x_{i}^{n})_{i\in
          I}}=\norm{(x_{i}^{m})_{i\in I}-(x_{i}^{n})_{i\in I}}
    \end{displaymath}
    ist für jedes feste $j\in I$ auch $(x_{j}^{n})_{n\in\N}$ eine
    Cauchy-Folge und infolge gegen einen Vektor $x_{j}\in H_{j}$
    konvergent. Für beliebiges $F$ aus der Menge $\EE(I)$ der
    endlichen Teilmengen von $I$ und für jedes $m\in\N$ gilt
    \begin{align}
      \begin{split}\label{eq:kart-prod-matrix-hr-i}
        \sqrt{\sum_{i\in
            F}\norm[H_{i}]{x_{i}^{n}-x_{i}}^{2}}&\leq\sqrt{\sum_{i\in
            F}\norm[H_{i}]{x_{i}^{n}-x_{i}^{m}}^{2}}+\sqrt{\sum_{i\in
            F}\norm[H_{i}]{x_{i}^{m}-x_{i}}^{2}}\\
        &\leq\norm{(x_{i}^{n})_{i\in I}-(x_{i}^{m})_{i\in
            I}}+\sqrt{\sum_{i\in F}\norm[H_{i}]{x_{i}^{m}-x_{i}}^{2}}.
      \end{split}
    \end{align}
    Wählen wir $N\in\N$ so groß, dass die Abschätzung
    $\norm{(x_{i}^{n})_{i\in I}-(x_{i}^{m})_{i\in I}}\leq\epsilon$ für
    $m,n\geq N$ erfüllt ist, setzen $n\geq N$ in
    \eqref{eq:kart-prod-matrix-hr-i} ein und bilden den Grenzwert
    $m\to\infty$, so folgt
    \begin{displaymath}
      \sqrt{\sum_{i\in F}\norm[H_{i}]{x_{i}^{n}-x_{i}}^{2}}\leq\epsilon,
    \end{displaymath}
    weil die linke Seite von \eqref{eq:kart-prod-matrix-hr-i} nicht
    von $m$ abhängt.  Da $F\in\EE(I)$ beliebig war, erhalten wir
    \begin{displaymath}
      \sqrt{\sum_{i\in I}\norm[H_{i}]{x_{i}^{n}-x_{i}}^{2}}\leq\epsilon
    \end{displaymath}
    für $n\geq N$, was einerseits
    $(x_{i})_{i\in I}\in\ell^{2}(H_{i}:i\in I)$ und andererseits
    $((x_{i}^{n})_{i\in I})_{n\in\N}\to (x_{i})_{i\in I}$
    impliziert. Somit liegt tatsächlich ein vollständiger Raum, also
    ein Hilbertraum, vor, den man die \emph{direkte Summe} der $H_{i}$
    nennt und als $\bigoplus_{i\in I}H_{i}$ anschreibt. An späterer
    Stelle werden wir benötigen, dass der Unterraum
    \begin{displaymath}
      Y:=\set{(x_{i})_{i\in I}\in\bigoplus_{i\in I}H_{i}}{x_{i}\neq 0\,\,\text{nur für endlich
          viele}\,\,i\in I},
    \end{displaymath}
    den wir nach geeigneter Identifikation als Raum aller
    Linearkombinationen von Elementen aus $\bigcup_{i\in I}H_{i}$
    auffassen können, dicht in $\bigoplus_{i\in I}H_{i}$ enthalten
    ist. Zum Beweis sei $(x_{i})_{i\in I}\in\bigoplus_{i\in I}H_{i}$
    gegeben. Wir definieren das Netz
    $\left((x_{i}^{F})_{i\in I}\right)_{F\in\EE(I)}$ durch
    \begin{displaymath}
      x_{i}^{F}:=
      \begin{cases}
        x_{i},& i\in F\\
        \hfill 0,& \text{sonst}
      \end{cases}
    \end{displaymath}
    und bemerken, dass die Elemente $x_{i}^{F}$ sicher in $Y$
    liegen. Weiters konvergiert dieses Netz gegen $(x_{i})_{i\in I}$,
    da
    \begin{displaymath}
      \norm{(x_{i})_{i\in I}-(x_{i}^{F})_{i\in
          I}}=\sqrt{\sum_{i\in
          I}\norm[H_{i}]{x_{i}-x_{i}^{F}}^{2}}=\sqrt{\sum_{i\in I\setminus
          F}\norm[H_{i}]{x_{i}}^{2}}
    \end{displaymath}
    wegen $(\norm[H_{i}]{x_{i}})_{i\in I}\in\ell^{2}(I)$ für
    $F\in\EE(I)$ gegen $0$ konvergiert.

    Wir wollen noch den Sonderfall hervorheben, dass $I=\{1,\dots,n\}$
    eine endliche Indexmenge ist und alle $H_{i}=H$ übereinstimmen. In
    diesem Fall erhalten wir für $\bigoplus_{i\in I}H_{i}$ das volle
    kartesische Produkt $H^{n}$, dessen Elemente wir aus formalen
    Gründen (siehe (\ref{item:kart-prod-matrix-hr-ii})) als Spalten
    $(x_{1},\dots,x_{n})^{T}$ schreiben. Das Skalarprodukt ist dann
    \begin{displaymath}
      \left((x_{1},\dots,x_{n})^{T},(y_{1},\dots,y_{n})^{T}\right):=\sum_{k=1}^{n}(x_{k},y_{k}).
    \end{displaymath}
  \item\label{item:kart-prod-matrix-hr-ii} Sei $A$ eine
    $C^{*}$-Unteralgebra von $L_{b}(H)$. Sind für $i,j\in\{1,2\}$
    Operatoren $T_{ij}\in A$ gegeben, so kann man auf $H^{2}$ den
    Operator
    \begin{displaymath}
      (x,y)^{T}\mapsto (T_{11}x+T_{12}y,T_{21}x+T_{22}y)^{T}
    \end{displaymath}
    definieren, den wir im Folgenden als Matrix
    $\left(\begin{smallmatrix}
        T_{11} & T_{12}\\
        T_{21} & T_{22}
      \end{smallmatrix}\right)$
    schreiben werden. Die Menge $M_{2}(A)$ aller derartigen Matrizen
    über $A$ bildet offensichtlich einen unter der Komposition von
    Operatoren abgeschlossenen Unterraum von $L_{b}(H^{2})$. Einfaches
    Nachrechnen zeigt außerdem, dass die Adjungierte in $L_{b}(H^{2})$
    von $\left(\begin{smallmatrix}
        T_{11} & T_{12}\\
        T_{21} & T_{22}
      \end{smallmatrix}\right)$ durch
    $\left(\begin{smallmatrix}
        T_{11}^{*} & T_{21}^{*}\\
        T_{12}^{*} & T_{22}^{*}
      \end{smallmatrix}\right)$
    gegeben ist, womit $M_{2}(A)$ auch unter $.^{*}$ abgeschlossen
    ist. Für die Einschränkung der Abbildungsnorm auf $M_{2}(A)$
    gelten die Ungleichungen
    \begin{equation}\label{eq:kart-prod-matrix-hr-ii}
      \max_{i,j=1,2}\norm{T_{ij}}\leq\norm{
        \begin{pmatrix}
          T_{11} & T_{12}\\
          T_{21} & T_{22}
        \end{pmatrix}
      }\leq\sum_{i,j=1}^{2}\norm{T_{ij}},
    \end{equation}
    wie man elementar überprüft. Es folgt, dass eine Folge aus
    $M_{2}(A)$ genau dann konvergiert bzw. eine Cauchy-Folge ist, wenn
    die Folgen der Einträge alle konvergieren bzw. Cauchy-Folgen
    sind. Da $A$ in $L_{b}(H)$ abgeschlossen ist, ergibt sich daraus,
    dass $M_{2}(A)$ in $L_{b}(H^{2})$ abgeschlossen und somit eine
    $C^{*}$-Unteralgebra ist.

    Es sei nicht verschwiegen, dass man dieselbe Konstruktion auch für
    $H^{n}$ statt $H^{2}$ durch\-führen kann, um den Raum $M_{n}(A)$
    der $n\times n$-Matrizen über $A$ zu erhalten. Sämtliche
    Überlegungen sind ident, der notationelle Aufwand steigt jedoch
    an. Da wir nur $M_{2}(A)$ verwenden werden, haben wir auf die
    volle Allgemeinheit verzichtet.
  \end{enumerate}
\end{bemerkung}

Schließlich benötigen wir, dass der Raum $L_{b}(H)$ bis auf
isometrische Isomorphie dem Dualraum eines Banachraums
entspricht. Dazu müssen wir etwas ausholen.

\begin{lemma}\label{lem:hs-norm-spnorm-wohldef}
  Sei $E$ eine Orthonormalbasis von $H$. Ist $S\in L_{b}(H)$, dann ist
  der Ausdruck
  \begin{equation}\label{eq:hs-norm}
    \sum_{e\in E}(Se,Se)=\sum_{e\in E}\norm{Se}^{2}\in[0,+\infty]
  \end{equation}
  unabhängig von der Wahl von $E$.
\end{lemma}
\begin{proof}
  Siehe~\cite[S.59f.]{murphy:cstar}.
\end{proof}
Diese Tatsache ermöglicht die Definition zweier Klassen von Operatoren.
\begin{definition}\label{def:hs-spklasse-op}
  \hspace{0mm}
  \begin{enumerate}[label=(\roman*),ref=\roman*]
  \item\label{item:hs-spklasse-op-i} Ein Operator $S\in L_{b}(H)$
    heißt \emph{Hilbert-Schmidt-Operator}, wenn der Ausdruck
    \eqref{eq:hs-norm} endlich ist. In diesem Fall definiert man
    \begin{displaymath}
      \norm[2]{S}:=\left(\sum_{e\in E}(Se,Se)\right)^{1/2}.
    \end{displaymath}
    Die Menge aller Hilbert-Schmidt-Operatoren wird mit $L^{2}(H)$
    bezeichnet.
  \item\label{item:hs-spklasse-op-ii} Ein Operator $S\in L_{b}(H)$
    heißt \emph{Spurklasseoperator}, wenn $\abs{S}^{1/2}$ ein
    Hilbert-Schmidt-Operator ist. In diesem Fall definiert man
    \begin{displaymath}
      \norm[1]{S}:=\norm[2]{\abs{S}^{1/2}}^{2}
    \end{displaymath}
    Die Menge aller Spurklasseoperatoren wird mit $L^{1}(H)$
    bezeichnet.
  \end{enumerate}
\end{definition}
\begin{bemerkung}
  Wegen $(|S|^{1/2}e,|S|^{1/2}e)=(|S|e,e)$ ist ein Operator
  $S\in L_{b}(H)$ genau dann ein Spurklasseoperator, wenn für eine
  Orthonormalbasis (äquivalent für alle Orthonormalbasen) $E$ von $H$
  der Ausdruck
  \begin{displaymath}
    \sum_{e\in E}(|S|e,e)
  \end{displaymath}
  endlich ist. Ist das der Fall, so stimmt diese Zahl mit
  $\norm[1]{S}$ überein.
\end{bemerkung}
Direkt aus der Definition erhalten wir einige äquivalente Bedingungen
dafür, dass $S$ ein Hilbert-Schmidt-Operator ist.
\begin{lemma}\label{lem:hs-equiv}
  Für einen Operator $S\in L_{b}(H)$ sind die folgenden Aussagen
  äquivalent:
  \begin{enumerate}[label=(\roman*),ref=\roman*]
  \item\label{item:hs-equiv-i} $S$ ist ein Hilbert-Schmidt-Operator.
  \item\label{item:hs-equiv-ii} $\abs{S}$ ist ein
    Hilbert-Schmidt-Operator.
  \item\label{item:hs-equiv-iii} $\abs{S}^{2}$ ist ein
    Spurklasseoperator.
  \end{enumerate}
  In diesem Fall gilt
  \begin{displaymath}
    \norm[2]{S}=\norm[2]{\abs{S}}=\norm[1]{\abs{S}^{2}}.
  \end{displaymath}
\end{lemma}
\begin{proof}
  Ist $e\in H$ ein Vektor, so gilt
  \begin{displaymath}
    (Se,Se)=(S^{*}Se,e)=(|S|^{2}e,e)=(|S|e,|S|e).
  \end{displaymath}
  Die Aussagen folgen nun sofort aus
  \begin{displaymath}
    \sum_{e\in E}(Se,Se)=\sum_{e\in E}(|S|e,|S|e)=\sum_{e\in
      E}(|S|^{2}e,e)
  \end{displaymath}
  für eine Orthonormalbasis $E$.
\end{proof}
\begin{bemerkung}\label{bem:motiv-notation-hs-spklasse}
  Die Äquivalenz von (\ref{item:hs-equiv-i}) und
  (\ref{item:hs-equiv-iii}) gilt bekanntlich auch für die
  Funktionenräume $L^{2}(\mu)$ und $L^{1}(\mu)$. Die Hilbert-Schmidt-
  und Spurklasseoperatoren haben noch weitere Eigenschaften mit den
  Elementen dieser Funktionenräume gemeinsam. Beispielsweise ist das
  Produkt zweier Hilbert-Schmidt-Operatoren stets ein
  Spurklasseoperator; vgl.~\cite[Theorem~2.4.13]{murphy:cstar}. Dies
  motiviert die Schreibweisen $L^{2}(H)$ und $L^{1}(H)$.
\end{bemerkung}
Die grundlegenden Eigenschaften von $L^{1}(H)$ haben stets ein
Analogon für $L^{2}(H)$. Im Folgenden werden wir uns auf die
Spurklasseoperatoren konzentrieren, da diese im Verlauf der Arbeit die
wesentlich prominentere Rolle spielen werden.
\begin{lemma}\label{lem:eig-spklasse}
  \hspace{0mm}
  \begin{enumerate}[label=(\roman*),ref=\roman*]
  \item\label{item:eig-spklasse-i} Die Funktion $\norm[1]{\cdot}$ ist
    eine Norm auf $L^{1}(H)$, die sogenannte \emph{Spurklassenorm}.
  \item\label{item:eig-spklasse-ii} $L^{1}(H)$ ist ein unter $.^{*}$
    abgeschlossenes Ideal (insbesondere ein Unterraum) in $L_{b}(H)$,
    wobei für $S\in L^{1}(H)$ und $T\in L_{b}(H)$ einerseits
    \begin{displaymath}
      \norm[1]{S^{*}}=\norm[1]{S}
    \end{displaymath}
    und andererseits die Abschätzungen
    \begin{equation}\label{eq:spklasse-ideal}
      \norm[1]{ST}\leq\norm[1]{S}\norm{T}\quad\textnormal{sowie}\quad\norm[1]{TS}\leq\norm{T}\norm[1]{S}
    \end{equation}
    gelten.
  \end{enumerate}
\end{lemma}
\begin{proof}
  Siehe~\cite[Theorem~2.4.15]{murphy:cstar}.
\end{proof}
\begin{definition}\label{def:spur}
  Die \emph{Spur} eines Spurklasseoperators $S\in L^{1}(H)$ ist
  definiert durch
  \begin{displaymath}
    \tr(S):=\sum_{e\in E}(Se,e),
  \end{displaymath}
  wobei $E$ eine Orthonormalbasis von $H$ bezeichnet.
\end{definition}
Diese Definition ist aus mehreren Gründen ad hoc problematisch, es
gilt aber folgendes Lemma:
\begin{lemma}\label{lem:spur-wohldef+eig}
  Die Spur ist ein wohldefiniertes lineares Funktional auf $L^{1}(H)$,
  das nicht von der konkreten Wahl der Orthonormalbasis
  abhängt. Weiters gelten für $S\in L^{1}(H)$ folgende Aussagen:
  \begin{enumerate}[label=(\roman*),ref=\roman*]
  \item\label{item:spur-wohldef+eig-i} $\abs{\tr(S)}\leq\norm[1]{S}$
  \item\label{item:spur-wohldef+eig-ii} $\tr(S^{*})=\overline{\tr(S)}$
  \end{enumerate}
\end{lemma}
\begin{proof}
  Für die Wohldefiniertheit sowie die Aussage
  (\ref{item:spur-wohldef+eig-i}) sei
  auf~\cite[Lemma~2.4.12 - Theorem~2.4.16]{murphy:cstar} verwiesen. Ist
  $E$ eine beliebige Orthonormalbasis von $H$, so folgt die zweite
  Aussage durch die Rechnung
  \begin{displaymath}
    \tr(S^{*})=\sum_{e\in E}(S^{*}e,e)=\sum_{e\in
      E}(e,Se)=\overline{\sum_{e\in E}(Se,e)}=\overline{\tr(S)}.
  \end{displaymath}
\end{proof}
Die nächste Aussage werden wir anders als die letzten Resultate nicht
nur für Spurklasseoperatoren, sondern auch für
Hilbert-Schmidt-Operatoren benötigen.
\begin{lemma}\label{lem:tr-komm}
  Seien $S$ und $T$ beschränkte Operatoren auf $H$, wobei
  \begin{enumerate}[label=(\roman*),ref=\roman*]
  \item\label{item:tr-komm-i} beide Operatoren
    Hilbert-Schmidt-Operatoren sind\quad oder
  \item\label{item:tr-komm-ii} zumindest einer der beiden Operatoren
    zur Spurklasse gehört.
  \end{enumerate}
  Dann gilt\footnote{Man beachte, dass die Operatoren $ST$ und $TS$
    wegen Bemerkung~\ref{bem:motiv-notation-hs-spklasse}
    bzw. Lemma~\ref{lem:eig-spklasse}(\ref{item:eig-spklasse-ii}) in
    jedem Fall Spurklasseoperatoren sind. Somit sind die Spurausdrücke
    wohldefiniert.} $\tr(ST)=\tr(TS)$.
\end{lemma}
\begin{proof}
  Siehe~\cite[Theorem~2.4.14]{murphy:cstar}.
\end{proof}

Zum Abschluss des Kapitels kommen wir nun zum vor
Lemma~\ref{lem:hs-norm-spnorm-wohldef} angekündigten Ergebnis.
\begin{satz}\label{satz:beschrop-spklasse}
  \hspace{0mm}
  \begin{enumerate}[label=(\roman*),ref=\roman*]
  \item\label{item:beschrop-spklasse-i} Versehen mit der
    Spurklassenorm $\norm[1]{\cdot}$ ist $L^{1}(H)$ ein Banachraum.
  \item\label{item:beschrop-spklasse-ii} Für $T\in L_{b}(H)$ ist das
    Funktional
    \begin{displaymath}
      \tr(.T):
      \begin{cases}
        L^{1}(H)&\to\C \\
        \hfill S&\mapsto\tr(ST)
      \end{cases}
    \end{displaymath}
    wohldefiniert, linear und beschränkt mit
    $\norm{\tr(.T)}\leq\norm{T}$.
  \item\label{item:beschrop-spklasse-iii} Die sogenannte
    \emph{kanonische Abbildung}
    \begin{displaymath}
      \theta:
      \begin{cases}
        L_{b}(H)&\to L^{1}(H)' \\
        \hfill T&\mapsto\tr(.T)
      \end{cases}
    \end{displaymath}
    ist eine isometrische, lineare Bijektion. Insbesondere gilt
    $\norm{\tr(.T)}=\norm{T}$.
  \end{enumerate}
\end{satz}
\begin{proof}
  \hspace{0mm}
  \begin{enumerate}[label=(\roman*),ref=\roman*]
  \item Siehe~\cite[Corollary~4.2.2]{murphy:cstar}.
  \item Da $L^{1}(H)$ ein Ideal bildet, ist $\tr(.T)$ wohldefiniert;
    die Linearität folgt aus der von $\tr$. Nach
    Lemma~\ref{lem:spur-wohldef+eig} gilt
    $\abs{\tr(ST)}\leq\norm[1]{ST}\stackrel{\eqref{eq:spklasse-ideal}}{\leq}\norm[1]{S}\norm{T}$,
    also $\norm{\tr(.T)}\leq\norm{T}$.
  \item Siehe~\cite[Theorem~4.2.3]{murphy:cstar}.
  \end{enumerate}
\end{proof}

\clearpage
\chapter{Einselement und Projektionen in $C^{*}$-Algebren}
\label{cha:eins-proj-cstar}
Das Ziel dieses Kapitels sind einige Aussagen, die im späteren Verlauf
der Arbeit erforderlich sein werden. Insbesondere sind das die
Charakterisierung von $C^{*}$-Algebren mit Einselement durch
Extremalpunkte der abgeschlossenen Einheitskugel und
Satz~\ref{satz:proj-dicht} über die Dichtheit von Projektionen.

\section{Positive Elemente und approximative Einselemente}
\label{sec:pos-approx-eins-elemente}
Bevor wir dazu kommen, benötigen wir eine Reihe an Tatsachen über
positive Elemente in $C^{*}$-Algebren.
\begin{notation}\label{not:cstar}
  \hspace{0mm}
  \begin{enumerate}[label=(\roman*),ref=\roman*]
  \item In diesem Kapitel sei $A$ stets eine $C^{*}$-Algebra und $S$
    ihre abgeschlossene Einheitskugel.
  \item Die Menge der selbstadjungierten Elemente von $A$, also jener
    Elemente mit $a^{*}=a$, wird mit $A_{sa}$ bezeichnet. Für die
    Menge der positiven Elemente von $A$, das sind die
    selbstadjungierten Elemente mit\footnote{Sollte $A$ kein
      Einselement enthalten, ist das Spektrum im Sinne von
      Definition~\ref{def:spek-resmenge} zu verstehen.}
    $\sigma(a)\subseteq[0,+\infty)$, schreiben wir $A^{+}$.
  \item In diesem und allen folgenden Kapiteln werden wir für
    Vielfache des Einselements $\lambda$ statt $\lambda e$ schreiben,
    wenn der Zusammenhang dies erlaubt. Insbesondere soll $1$ auch das
    Einselement einer $C^{*}$-Algebra bezeichnen.
  \item Die \emph{Projektionen} im Sinne der linearen Algebra
    bzw. Funktionalanalysis sind typischerweise die idempotenten
    linearen Abbildungen auf einem Vektorraum. In der Theorie von
    $C^{*}$-Algebren ist es üblich, nur jene idempotenten Elemente
    einer $C^{*}$-Algebra als Projektionen zu bezeichnen, die
    zusätzlich selbstadjungiert sind. Die \emph{Projektionen} in der
    $C^{*}$-Algebra $L_{b}(H)$ sind also genau die
    \emph{Orthogonalprojektionen} auf dem Hilbertraum $H$.
  \end{enumerate}
\end{notation}

Hat die $C^{*}$-Algebra $A$ ein Einselement, so existiert für jedes
positive Element $a$ ein eindeutiges positives Element $q\in A$ mit
$q^{2}=a$, wobei sogar $q\in C_{A}^{*}(a,1)$ gilt;
vgl.~\cite[Korollar~1.5.14]{kaltenb:fana2}. Wir verallgemeinern diese
wichtige Tatsache auf beliebige $C^{*}$-Algebren.
\begin{satz}\label{satz:quwurzel}
  Für ein Element $a\in A^{+}$ gibt es ein eindeutiges Element
  $q\in A^{+}$ mit $q^{2}=a$. Dabei gilt $q\in C_{A}^{*}(a)$.
\end{satz}
\begin{proof}
  Nach Definition ist $a$ auch in $\tilde{A}$ positiv, sodass wir aus
  dem bekannten Resultat ein eindeutiges positives
  $(q,\lambda)\in\tilde{A}$ mit\footnote{Man beachte, dass wir
    $a\in A$ mit $(a,0)\in\tilde{A}$ identifizieren.}
  $(q,\lambda)^{2}=(a,0)$ erhalten. Multiplizieren wir aus, so folgt
  \begin{displaymath}
    (q^{2}+2\lambda q,\lambda^{2})=(a,0),
  \end{displaymath}
  also $\lambda=0$ und $q^{2}=a$. Das Element $(q,\lambda)=(q,0)$ ist
  in $\tilde{A}$ positiv, folglich ist $q$ direkt nach Definition
  positiv in $A$. Ist $q'$ ein weiteres positives Element von $A$ mit
  $q'^{2}=a$, so ist $q'$ ein positives Element von $\tilde{A}$ mit
  $q'^{2}=a$. Aus der Eindeutigkeitsaussage für $\tilde{A}$ folgt
  $q=q'$.

  Es bleibt noch $q\in C_{A}^{*}(a)$ zu zeigen. Das Resultat in
  $\tilde{A}$ liefert $q\in C_{\tilde{A}}(a,1)$. Nach
  \eqref{eq:erz-unteralg-iii} aus Bemerkung~\ref{bem:erz-unteralg}
  gibt es Polynome $p_{n}\in\C[z]$ mit
  $(q,0)=\lim_{n\to\infty}p_{n}((a,0))$. Der konstante Term des
  Polynoms $r_{n}(z):=p_{n}(z)-p_{n}(0)$ verschwindet, also
  folgt\footnote{Es sei daran erinnert, dass eine komplexe Zahl in
    $\tilde{A}$ als entsprechende Vielfache des Einselements $(0,1)$
    zu interpretieren ist.}
  \begin{displaymath}
    p_{n}((a,0))=r_{n}((a,0))+p_{n}(0)\cdot (0,1)=(r_{n}(a),0)+(0,p_{n}(0))=(r_{n}(a),p_{n}(0)).
  \end{displaymath}
  Daraus erhalten wir, dass die komplexe Zahl $p_{n}(0)$ für
  $n\to\infty$ gegen $0$ konvergiert. Folglich gilt
  \begin{displaymath}
    (q,0)=\lim_{n\to\infty}p_{n}((a,0))=\lim_{n\to\infty}(r_{n}(a),0)
  \end{displaymath}
  oder anders formuliert $q=\lim_{n\to\infty}r_{n}(a)$. Wegen
  $r_{n}(0)=0$ ist dieses Element nach \eqref{eq:erz-unteralg-iii} aus
  Bemerkung~\ref{bem:erz-unteralg} in $C_{A}^{*}(a)$ enthalten.
\end{proof}
\begin{definition}\label{def:quwurzel}
  Für $a\in A^{+}$ ist die \emph{Quadratwurzel} von $A$ jenes
  eindeutige Element $q\in A^{+}$ mit $q^{2}=a$.
\end{definition}
\begin{lemma}\label{lem:char-pos-el}
  Habe $A$ ein Einselement und sei $a\in A$ selbstadjungiert. Gibt es
  eine reelle Zahl $t\geq 0$ mit
  \begin{equation}\label{eq:char-pos-el}
    \norm{a-t}\leq t,
  \end{equation}
  so ist $a$ positiv. Ist umgekehrt $a\geq 0$, dann gilt
  \eqref{eq:char-pos-el} für alle $t\geq\norm{a}$.
\end{lemma}
\begin{proof}
  Für $t\geq 0$ gilt nach dem Spektralabbildungssatz
  $\sigma(a-t)=\sigma(a)-t=\set{\lambda-t}{\lambda\in\sigma(a)}$;
  vgl.~\cite[Satz~1.1.7]{kaltenb:fana2}. Weiters ist $a-t$
  selbstadjungiert, insbesondere normal, sodass wir
  $r(a-t)=\norm{a-t}$ erhalten;
  siehe~\cite[Fakta~1.5.2.5]{kaltenb:fana2}.

  Sei zunächst \eqref{eq:char-pos-el} für ein $t\geq 0$
  angenommen. Wegen $a\in A_{sa}$ gilt $\sigma(a)\subseteq\R$. Gäbe es
  ein negatives $\lambda\in\sigma(a)$, dann erhielten wir den
  Widerspruch
  \begin{displaymath}
    \norm{a-t}=r(a-t)\geq\abs{\lambda-t}=t-\lambda>t.
  \end{displaymath}
  
  Sei umgekehrt $a$ positiv. Für $t\geq\norm{a}=r(a)$ gilt
  $t\geq\abs{\lambda}=\lambda$ für alle $\lambda\in\sigma(a)$. Es
  folgt
  \begin{displaymath}
    \norm{a-t}=r(a-t)=\sup_{\lambda\in\sigma(a)}\abs{\lambda-t}=\sup_{\lambda\in\sigma(a)}(t-\lambda)\leq t,
  \end{displaymath}
  also \eqref{eq:char-pos-el}.
\end{proof}

Als Nächstes wollen wir eine überaus praktische Konstruktion
einführen. Dazu sei $A$ eine beliebige $C^{*}$-Algebra und
$a\in A_{sa}$. Dann ist die von $a$ erzeugte $C^{*}$-Algebra
$C^{*}_{A}(a)$ kommutativ, wodurch die Gelfandtransformation
$\hat{.}:C^{*}_{A}(a)\to C_{0}(M)$ zur Verfügung steht. Dabei
bezeichnet $M$ den Gelfandraum von $C^{*}_{A}(a)$. Ist eine
reellwertige, also selbstadjungierte, Funktion $f\in C_{0}(M)$
gegeben, dann schreiben wir $f^{+}$ bzw. $f^{-}$ für den Positiv-
bzw. Negativteil von $f$, also $f^{+}=\max(f,0)$ bzw.
$f^{-}=-\min(f,0)$. Aus der Ungleichung
$\abs{f^{\pm}(m)}=f^{\pm}(m)\leq\abs{f(m)}$ folgt, dass $f^{+}$ und
$f^{-}$ ebenfalls im Unendlichen verschwinden. Infolge können wir
einen Positiv- und Negativteil von $a$ definieren:
\begin{definition}\label{def:pos-neg-cstar}
  Ist $a\in A$ selbstadjungiert, so heißen die Elemente
  \begin{displaymath}
    a^{+}:=(\widehat{\,.\,})^{-1}\left((\widehat{a})^{+}\right) \quad\text{und}\quad a^{-}:=(\widehat{\,.\,})^{-1}\left((\widehat{a})^{-}\right)
  \end{displaymath}
  der \emph{Positiv-} bzw. \emph{Negativteil} von $a$.
\end{definition}
\begin{bemerkung}\label{bem:pos-negteil-span}
  \hspace{0mm}
  \begin{enumerate}[label=(\roman*),ref=\roman*]
  \item\label{item:pos-negteil-span-i} Wegen
    $a^{+},a^{-}\in C^{*}_{A}(a)$ kommutieren $a^{+}$ und $a^{-}$ mit
    $a$. Aufgrund der entsprechenden Eigenschaften der Funktionen
    $(\widehat{a})^{+}$ und $(\widehat{a})^{-}$ gilt $a^{\pm}\geq 0$,
    $a=a^{+}-a^{-}$, $a^{+}a^{-}=0$ sowie
    $\norm{a^{\pm}}\leq\norm{a}$.

  \item\label{item:pos-negteil-span-ii} Kombiniert man die
    Definitionen~\ref{def:re-im-cstar} und \ref{def:pos-neg-cstar}, so
    kann man ein beliebiges Element $a\in A$ schreiben als
    \begin{equation}\label{eq:pos-negteil-span-i}
      a=\big((\re a)^{+}-(\re a)^{-}\big)+i\big((\im a)^{+}-(\im
      a)^{-}\big).
    \end{equation}
    Insbesondere ist $A=\spn A^{+}$. Ist zusätzlich $a\in S$, so gilt
    $\norm{(\re a)^{\pm}}\leq\norm{\re a}\leq\norm{a}\leq 1$ sowie die
    analoge Ungleichung für den Imaginärteil. Daraus folgt
    \begin{equation}\label{eq:pos-negteil-span-ii}
      S\subseteq (A^{+}\cap S-A^{+}\cap S)+i(A^{+}\cap S-A^{+}\cap S).
    \end{equation}
  \item\label{item:pos-negteil-span-iii} Enthält $A$ ein Einselement,
    so spannen auch die unitären Elemente ganz $A$ auf: Wegen
    Lemma~\ref{lem:eig-re-im} genügt es dafür, ein selbstadjungiertes
    Element $a$ aus der Einheitskugel als Linearkombination unitärer
    Elemente zu schreiben. Es gilt
    \begin{displaymath}
      \sigma(a^{2})\subseteq\left[-r(a^{2}),r(a^{2})\right]=\left[-\norm{a^{2}},\norm{a^{2}}\right]\subseteq[-1,1],
    \end{displaymath}
    sodass aus dem Spektralabbildungssatz
    $\sigma(1-a^{2})=1-\sigma(a^{2})\subseteq[0,2]$
    folgt. Insbesondere ist $1-a^{2}$ positiv und somit
    $u:=a+i(1-a^{2})^{1/2}$ wohldefiniert. Man rechnet unmittelbar
    nach, dass $uu^{*}=u^{*}u=1$ ist, sodass $u$ und $u^{*}$
    tatsächlich unitär sind. Die Gleichung $a=\frac{1}{2}(u+u^{*})$
    liefert die gewünschte Darstellung.

    Diese Konstruktion lässt sich sehr einfach motivieren, wenn man
    das Problem für die $C^{*}$-Algebra $\C$ betrachtet. In dieser
    Situation geht es darum, eine Zahl $a\in [-1,1]$ als
    Linearkombination von Elementen der Einheitskreislinie zu
    schreiben. Dazu wählt man jene beiden Punkte mit Betrag $1$, deren
    Realteil genau $a$ ist und bildet deren Mittelwert; diese Punkte
    sind genau $a\pm i\sqrt{1-a^{2}}$.
  \end{enumerate}
\end{bemerkung}
\begin{satz}\label{satz:eig-pos-el}
  Seien $a,b\in A^{+}$ und eine reelle Zahl $t\geq 0$ gegeben.
  \begin{enumerate}[label=(\roman*),ref=\roman*]
  \item\label{item:eig-pos-el-i} Es gilt $a+b\in A^{+}$ und
    $ta\in A^{+}$.
  \item\label{item:eig-pos-el-ii} $A^{+}$ ist konvex.
  \item\label{item:eig-pos-el-iii} Es gilt
    $A^{+}=\set{a^{*}a}{a\in A}$.
  \item\label{item:eig-pos-el-iv} $A^{+}$ ist abgeschlossen.
  \end{enumerate}
\end{satz}
\begin{proof}
  \hspace{0mm}
  \begin{enumerate}[label=(\roman*),ref=\roman*]
  \item Durch Übergang zu $\tilde{A}$ können wir annehmen, dass $A$
    ein Einselement hat. Offenbar sind $a+b$ und $ta$
    selbstadjungiert. Nach Lemma~\ref{lem:char-pos-el} gilt weiters
    $\norm{a-t}\leq t$ und $\norm{b-t}\leq t$ für
    $t=\max(\norm{a},\norm{b})$. Aus der Dreiecksungleichung folgt
    $\norm{(a+b)-2t}\leq 2t$, also -- wieder mit
    Lemma~\ref{lem:char-pos-el} -- der erste Teil der Aussage. Der
    zweite Teil ergibt sich unmittelbar aus
    $\sigma(ta)=t\sigma(a)=\set{t\lambda}{\lambda\in\sigma(a)}$.
  \item Folgt sofort aus (\ref{item:eig-pos-el-i}).
  \item Für $b\in A^{+}$ erfüllt die positive Quadratwurzel $q$
    offenbar $b=q^{*}q\in\set{a^{*}a}{a\in A}$.

    Umgekehrt ist zu zeigen, dass jedes Element der Form $a^{*}a$
    positiv ist, wobei die Selbstadjungiertheit klar ist. Auch hier
    können wir annehmen, dass $A$ ein Einselement hat.

    In einem ersten Schritt zeigen wir die Aussage, wenn $a$
    selbstadjungiert ist. Dazu betrachten wir $C^{*}_{A}(a,1)$, die
    von $a$ erzeugte $C^{*}$-Algebra mit Eins, und zeigen die
    Positivität in dieser $C^{*}$-Algebra\footnote{Man beachte, dass
      sich das Spektrum eines Elements beim Übergang zu einer
      $C^{*}$-Unteralgebra mit Eins nicht ändert;
      vgl.~\cite[Satz~1.5.6]{kaltenb:fana2}.}. Diese ist wegen der
    Normalität von $a$ kommutativ, sodass wir $A=C(K)$ mit einem
    kompakten Hausdorffraum $K$ annehmen können. Nach
    Beispiel~\ref{bsp:spektrum-ck} gilt dann $\sigma(a)=a(K)$ und, da
    $a$ selbstadjungiert ist, $a(K)\subseteq\R$. Klarerweise bildet
    $a^{2}=a^{*}a$ nur in die rechte Halbachse ab, ist also durch
    nochmalige Anwendung von Beispiel~\ref{bsp:spektrum-ck} positiv
    als Element von $C(K)$.

    Als zweiten Schritt zeigen wir, dass aus $-a^{*}a\geq 0$ schon
    $a=0$ folgt. Die bekannte Gleichung
    $\sigma(xy)\setminus\{0\}=\sigma(yx)\setminus\{0\}$ impliziert,
    dass mit $-a^{*}a$ auch $-aa^{*}$ positiv ist. Setzt man
    $b:=\re a$ und $c:=\im a$, dann gilt $a=b+ic$ sowie
    $a^{*}=b-ic$. Durch Ausmultiplizieren ergibt sich
    $a^{*}a+aa^{*}=2b^{2}+2c^{2}$, also
    $a^{*}a=2b^{*}b+2c^{*}c+(-aa^{*})\geq 0$ nach dem gerade
    Bewiesenen und (\ref{item:eig-pos-el-i}). Wegen $-a^{*}a\in A^{+}$
    folgt aus dem Spektralabbildungssatz
    $-\sigma(a^{*}a)=\sigma(-a^{*}a)\subseteq [0,+\infty)$, sodass
    $\sigma(a^{*}a)\subseteq [0,+\infty)\cap (-\infty,0]=\{0\}$ sein
    muss. Da das Spektrum jedes Elements nicht leer ist,
    siehe~\cite[Satz~1.1.14]{kaltenb:fana2}, erhalten wir
    $\sigma(a^{*}a)=\{0\}$. Aus
    $\norm{a}^{2}=\norm{a^{*}a}=r(a^{*}a)=0$ ergibt sich $a=0$.

    In einem letzten Schritt zeigen wir die allgemeine Aussage. Sei
    also $a\in A$ beliebig und $d:=a^{*}a$. Betrachten wir den
    Positiv- und Negativteil $d^{+}$ und $d^{-}$ sowie die
    Gelfandtransformation auf $C_{A}^{*}(d)$, so folgt
    \begin{equation}\label{eq:eig-pos-el}
      -(ad^{-})^{*}(ad^{-})=-d^{-}a^{*}ad^{-}=-d^{-}(d^{+}-d^{-})d^{-}=(d^{-})^{3}.
    \end{equation}
    Mit $\widehat{d^{-}}$ nimmt auch
    $\widehat{(d^{-})^{3}}=\left(\widehat{d^{-}}\right)^{3}$ nur
    nichtnegative Werte an, sodass $(d^{-})^{3}$ positiv ist. Aus
    \eqref{eq:eig-pos-el} und dem zweiten Beweisschritt folgt
    $ad^{-}=0$. Wir erhalten
    \begin{displaymath}
      0=ad^{-}=d^{+}d^{-}-(d^{-})^{2}=-(d^{-})^{*}d^{-},
    \end{displaymath}
    also $\norm{d^{-}}=\norm{-(d^{-})^{*}d^{-}}=0$. Dies liefert
    $a^{*}a=d=d^{+}\in A^{+}$.
  \item Sei $(a_{n})_{n\in\N}$ eine gegen $a$ konvergente Folge
    positiver Elemente. Zunächst ist $a$ wegen der Stetigkeit von
    $.^{*}$ selbstadjungiert. Weiters gibt es sicher ein $C>0$ mit
    $\norm{a_{n}}\leq C$ für alle $n\in\N$, sodass aus
    Lemma~\ref{lem:char-pos-el} die Ungleichung $\norm{a_{n}-C}\leq C$
    für alle $n\in\N$ folgt. Bilden wir hier den Grenzwert
    $n\to\infty$, so folgt $\norm{a-C}\leq C$ und wieder wegen
    Lemma~\ref{lem:char-pos-el} die Positivität von $a$.
  \end{enumerate}
\end{proof}
Mithilfe von Satz~\ref{satz:eig-pos-el}(\ref{item:eig-pos-el-iii})
können wir den Absolutbetrag eines Elements $a\in A$ definieren.
\begin{definition}\label{def:absbetrag}
  Für $a\in A$ ist der \emph{Absolutbetrag} die Quadratwurzel des
  positiven Elements $a^{*}a$, in Zeichen $\abs{a}:=(a^{*}a)^{1/2}\in
  C_{A}^{*}(a^{*}a)$.
\end{definition}
\begin{bemerkung}\label{bem:absbetrag}
  Für selbstadjungierte Elemente $a$ lässt sich mehr über den
  Absolutbetrag aussagen. Dann ist nämlich $C_{A}^{*}(a)$ ebenfalls
  eine kommutative $C^{*}$-Algebra und es gilt
  $\abs{a}\in C_{A}^{*}(a^{*}a)\leq C_{A}^{*}(a)$. Somit können wir
  die Gelfandtransformation in $C_{A}^{*}(a)$ auch auf $\abs{a}$
  anwenden und erhalten
  \begin{displaymath}
    \widehat{\abs{a}}=\widehat{(a^{*}a)^{1/2}}=\widehat{a^{*}a}^{1/2}=(\abs{\hat{a}}^{2})^{1/2}=\abs{\hat{a}}.
  \end{displaymath}
  Der für Funktionen $f\in C_{0}(M)$ offensichtliche Sachverhalt
  $\abs{f}=f^{+}+f^{-}$ impliziert somit
  \begin{displaymath}
    \widehat{\abs{a}}=\hat{a}^{+}+\hat{a}^{-}=\widehat{a^{+}}+\widehat{a^{-}}=\widehat{a^{+}+a^{-}}.
  \end{displaymath}
  Wir schließen auf $\abs{a}=a^{+}+a^{-}$ und weiter auf
  \begin{equation}\label{eq:absbetrag}
    a^{+}=\frac{1}{2}(\abs{a}+a) \qquad\text{sowie}\qquad a^{-}=\frac{1}{2}(\abs{a}-a).
  \end{equation}
\end{bemerkung}

Folgende Relation auf $A_{sa}$ stellt sich als Halbordnung heraus.
\begin{definition}\label{def:ho-sa}
  Für $a,b\in A_{sa}$ schreiben wir $a\leq b$, wenn $b-a$ positiv ist.
\end{definition}
\begin{lemma}\label{lem:ho-sa}
  Die Relation $\leq$ ist eine Halbordnung auf $A_{sa}$. Weiters ist
  $\leq$ translationsinvariant, d.~h. $a\leq b$ impliziert
  $a+c\leq b+c$.
\end{lemma}
\begin{proof}
  Die Reflexivität und Translationsinvarianz sind klar. Für die
  Transitivität seien $a,b,c\in A_{sa}$ mit $a\leq b$ und $b\leq c$
  gegeben, also $b-a,c-b\in A^{+}$. Nach
  Satz~\ref{satz:eig-pos-el}(\ref{item:eig-pos-el-i}) ist auch
  $c-a=(c-b)+(b-a)$ positiv, womit $a\leq c$ ist. Um die Antisymmetrie
  zu zeigen, sei $a\leq b$ und $b\leq a$. Die Elemente $b-a$ und
  $a-b=-(b-a)$ sind dann positiv, sodass aus dem
  Spektralabbildungssatz $\sigma(b-a)=\{0\}$ folgt. Wir erhalten
  $\norm{b-a}=r(b-a)=0$ bzw. $a=b$.
\end{proof}
\begin{bemerkung}\label{bem:ho}
  \hspace{0mm}
  \begin{enumerate}[label=(\roman*),ref=\roman*]
  \item\label{item:ho-i} Man kann die Halbordnung $\leq$ als
    translationsinvariante Fortsetzung der Schreibweise $a\geq 0$ für
    positive Elemente auffassen.
  \item\label{item:ho-ii} Die Ungleichung $s\leq t$ für $s,t\in\R$
    kann man auf zweierlei Art interpretieren, wenn $A$ ein
    Einselement $1$ enthält, denn neben der gewöhnlichen Ordnung auf
    $\R$ wäre auch die Beziehung $s1\leq t1$ in $A_{sa}$ eine mögliche
    Lesart. Wegen $\sigma(\lambda 1)=\{\lambda\}$ sind beide Varianten
    aber äquivalent, sodass diese Doppeldeutigkeit kein Problem
    darstellt.
  \item\label{item:ho-iii} Weiteren Interpretationsspielraum bietet
    eine beliebige Ungleichung $a\leq b$ mit $a,b\in A_{sa}$, wenn $A$
    kein Einselement enthält. Durch die isomorphe Einbettung
    $A\to\tilde{A}$, $a\mapsto (a,0)$ können wir $A$ als Teilmenge von
    $\tilde{A}$ auffassen und die Ungleichung sowohl in $A$ als auch
    in $\tilde{A}$ verstehen. Nach Definition der positiven Elemente
    in $A$ ist allerdings $b-a$ genau dann in $A$ positiv, wenn es in
    $\tilde{A}$ positiv ist. Somit bleibt diese Uneindeutigkeit in der
    Notation folgenlos. In einigen Beweisen kann man sogar einen
    expliziten Nutzen daraus ziehen: Um $a\leq b$ in $A$ zu zeigen,
    genügt es, die Ungleichung in $\tilde{A}$ nachzuweisen. Für eine
    Anwendung dieser Überlegung sei auf den ersten Beweisschritt von
    Satz~\ref{satz:char-pos-fkt} verwiesen.
  \item\label{item:ho-iv} Für $A=C(K)$ mit kompaktem $K$ ist das
    Spektrum von $h\in A$ nach Beispiel~\ref{bsp:spektrum-ck} gegeben
    durch $\sigma(h)=h(K)$. Daher gilt $f\leq g$ für $f,g\in A$ genau
    dann, wenn $f(t)\leq g(t)$ für alle $t\in K$ ist. Hier fügt sich
    auch die bereits verwendete Tatsache ein, dass eine Funktion in
    $C(K)$ genau dann positiv ist, wenn sie ausschließlich
    nichtnegative Werte annimmt.
  \end{enumerate}
\end{bemerkung}
Im nächsten Lemma sind einige Eigenschaften dieser Halbordnung
zusammengefasst.
\begin{lemma}\label{lem:eig-ho}
  Seien $a,b\in A_{sa}$.
  \begin{enumerate}[label=(\roman*),ref=\roman*]
  \item\label{item:eig-ho-i} $a\leq b$ ist äquivalent zu $-b\leq -a$
    und auch zu $ta\leq tb$ für eine beliebige reelle Zahl $t>0$.
  \item\label{item:eig-ho-ii} Gilt $a\leq b$ und ist $c\in A$
    beliebig, dann folgt $c^{*}ac\leq c^{*}bc$.
  \item\label{item:eig-ho-iii} Hat $A$ ein Einselement, so gilt
    $a\leq\norm{a}$. 
  \item\label{item:eig-ho-iv} Aus $0\leq a\leq b$ folgt
    $\norm{a}\leq\norm{b}$.
  \item\label{item:eig-ho-v} Sind $a,b\in S$ und $a,b\geq 0$, dann
    gilt $\norm{a-b}\leq 1$.
  \item\label{item:eig-ho-vi} Hat $A$ ein Einselement und sind $a$ und
    $b$ positiv sowie invertierbar, dann impliziert $a\leq b$ schon
    $0\leq b^{-1}\leq a^{-1}$.
  \end{enumerate}
\end{lemma}
\begin{proof}
  \hspace{0mm}
  \begin{enumerate}[label=(\roman*),ref=\roman*]
  \item Der erste Teil ist klar, der zweite folgt aus
    Satz~\ref{satz:eig-pos-el}(\ref{item:eig-pos-el-i}).
  \item Nach Voraussetzung gilt $b-a=q^{2}$ mit einem positiven
    $q$. Es folgt
    \begin{displaymath}
      c^{*}bc-c^{*}ac=c^{*}q^{2}c=(qc)^{*}(qc)\in A^{+}
    \end{displaymath}
    nach Satz~\ref{satz:eig-pos-el}(\ref{item:eig-pos-el-iii}).
  \item Für jedes $\lambda\in\sigma(a)$ gilt
    $\lambda=\abs{\lambda}\leq r(a)=\norm{a}$, sodass
    $\norm{a}-\lambda\geq 0$ ist. Diese Zahlen bilden das Spektrum des
    Elements $\norm{a}-a$, das somit als positiv nachgewiesen ist.
  \item Wir können annehmen, dass $A$ ein Einselement hat. Nach
    (\ref{item:eig-ho-iii}) gilt $b\leq\norm{b}$. Aus der
    Transitivität von $\leq$ folgt $0\leq a\leq\norm{b}$, also
    $\norm{b}-\sigma(a)\subseteq[0,+\infty)$. Anders formuliert gilt
    $\lambda=\abs{\lambda}\leq\norm{b}$ für jedes
    $\lambda\in\sigma(a)$. Daraus ergibt sich
    $\norm{a}=r(a)\leq\norm{b}$.
  \item Wieder nehmen wir an, dass $A$ ein Einselement hat. Erneut
    nach (\ref{item:eig-ho-iii}) gilt $a\leq\norm{a}\leq 1$ und analog
    $b\leq 1$. Mit (\ref{item:eig-ho-i}) folgt
    $-1\leq -b\leq a-b\leq a\leq 1$, sodass der Spektralabbildungssatz
    $\sigma(a-b)\subseteq [-1,1]$ liefert. Wir erhalten das Gewünschte
    aus $\norm{a-b}=r(a-b)\leq 1$.
  \item Die Ungleichung $0\leq b^{-1}$ folgt, wenn man die
    Gelfandtransformation in $C^{*}(b,1)$, der von $b$ erzeugten
    $C^{*}$-Unteralgebra mit Eins, verwendet, denn mit $\hat{b}$ nimmt
    auch $\widehat{b^{-1}}=\hat{b}^{-1}$ nur nichtnegative Werte
    an. Für die Ungleichung $b^{-1}\leq a^{-1}$ zeigen wir zunächst
    den Spezialfall, dass aus $b\geq 1$ schon $b^{-1}\leq 1$
    folgt. Die Ungleichung $b\geq 1$ bedeutet $\hat{b}(m)\geq 1$ für
    jedes $m\in M_{C^{*}(b,1)}$, was zu $1/\hat{b}(m)\leq 1$
    äquivalent ist. Es folgt
    $\widehat{b^{-1}}=\hat{b}^{-1}\leq \mathds{1}$, also
    $b^{-1}\leq 1$.

    Um den allgemeinen Fall zu beweisen, sei zunächst bemerkt, dass
    mit $a$ bzw. $b$ auch $a^{1/2}$ bzw. $b^{1/2}$ invertierbar sind,
    was aus $0\notin\sigma(a)=\sigma(a^{1/2})^{2}$ bzw. der analogen
    Tatsache für $b$ folgt. Mit (\ref{item:eig-ho-ii}) erhalten wir
    \begin{displaymath}
      1=(a^{1/2})^{-1}a(a^{1/2})^{-1}\leq (a^{1/2})^{-1}b(a^{1/2})^{-1}.
    \end{displaymath}
    Der erste Beweisteil liefert
    $((a^{1/2})^{-1}b(a^{1/2})^{-1})^{-1}\geq 1$, also
    $1\leq a^{1/2}b^{-1}a^{1/2}$.  Nochmals mit (\ref{item:eig-ho-ii})
    schließen wir auf
    \begin{displaymath}
      a^{-1}=(a^{1/2})^{-1}1(a^{1/2})^{-1}\leq
      (a^{1/2})^{-1}a^{1/2}b^{-1}a^{1/2}(a^{1/2})^{-1}=b^{-1}.
    \end{displaymath}
  \end{enumerate}
\end{proof}
Mit der Konstruktion aus Bemerkung~\ref{bem:b-alg-eins} und
Satz~\ref{satz:cstar-alg-eins} lässt sich zu jeder $C^{*}$-Algebra ein
Einselement adjungieren. In einigen Fällen ist dieses Vorgehen aber
nicht geeignet, da sich die algebraische Struktur der $C^{*}$-Algebra
drastisch ändern kann. Man kann sich dann mit einem anderen Konzept
behelfen.
\begin{definition}\label{def:approx-eins}
  Ein monoton wachsendes Netz $(u_{i})_{i\in I}$ positiver Elemente in
  $S$, aus $i\preccurlyeq j$ folgt also $0\leq u_{i}\leq u_{j}$, heißt
  \emph{approximatives Einselement}, wenn $a=\lim_{i\in I}au_{i}$ für
  alle $a\in A$ gilt.
\end{definition}
Da die $u_{i}$ aus dieser Definition selbstadjungiert sind und die
Operation $.^{*}$ stetig ist, ist eine äquivalente Bedingung gegeben
durch die Forderung $a=\lim_{i\in I}u_{i}a$ für alle $a\in A$.
\begin{satz}\label{satz:approx-eins}
  Sei $I:=\set{a\in U_{1}(0)}{a\geq 0}$.
  \begin{enumerate}[label=(\roman*),ref=\roman*]
  \item $I$ ist, versehen mit der Halbordnung aus
    Definition~\ref{def:ho-sa}, eine gerichtete Menge.
  \item Das Netz $(u_{a})_{a\in I}$, wobei $u_{a}:=a$, ist ein
    approximatives Einselement, das sogenannte \emph{kanonische
      approximative Einselement}.
  \end{enumerate}
\end{satz}
\begin{proof}
  \hspace{0mm}
  \begin{enumerate}[label=(\roman*),ref=\roman*]
  \item Nachzuprüfen ist nur die Richtungseigenschaft. Zu gegebenen
    $a,b\in I$ ist also ein $c\in I$ mit $a,b\leq c$ zu finden. Dazu
    zeigen wir die folgende Hilfsbehauptung\footnote{Für die rechte
      Ungleichung arbeiten wir in $\tilde{A}$. Die Inversen sind dabei
      wohldefiniert, da wegen $d_{k}\geq 0$ das Spektrum
      $\sigma(1+d_{k})=1+\sigma(d_{k})\subseteq[1,+\infty)$
      insbesondere $0$ nicht enthält.}:
    \begin{equation}\label{eq:approx-eins}
      0\leq d_{1}\leq d_{2}\Rightarrow d_{1}(1+d_{1})^{-1}\leq d_{2}(1+d_{2})^{-1}
    \end{equation}
    Klarerweise gilt $1+d_{1}\leq 1+d_{2}$, woraus mit
    Lemma~\ref{lem:eig-ho}(\ref{item:eig-ho-vi}) die Ungleichung
    $(1+d_{2})^{-1}\leq (1+d_{1})^{-1}$ folgt. Die zur trivialen
    Gleichung $(1+d_{k})(1+d_{k})^{-1}=1$ äquivalente Beziehung
    $d_{k}(1+d_{k})^{-1}=1-(1+d_{k})^{-1}$ ermöglicht in Kombination
    mit Lemma~\ref{lem:eig-ho}(\ref{item:eig-ho-i}) den Nachweis von
    \eqref{eq:approx-eins}:
    \begin{displaymath}
      d_{1}(1+d_{1})^{-1}=1-(1+d_{1})^{-1}\leq 1-(1+d_{2})^{-1}=d_{2}(1+d_{2})^{-1}.
    \end{displaymath}

    Seien weiterhin $a,b\in I$. Die Elemente $a':=a(1-a)^{-1}$ und
    $b':=b(1-b)^{-1}$ sind wohldefiniert, da
    \begin{displaymath}
      \sigma(1-a)=1-\sigma(a)\subseteq
      [1-r(a),1+r(a)]=[1-\norm{a},1+\norm{a}]\subseteq (0,2)
    \end{displaymath}
    die Zahl $0$ nicht enthält; Analoges gilt für $b$. Weiters liegen
    $a'$ und $b'$ in $A$ und nicht nur in $\tilde{A}$, da $A$ ein
    Ideal in $\tilde{A}$ ist. Zudem sind sie selbstadjungiert.
    
    Betrachten wir die Gelfandtransformation auf
    $C^{*}_{\tilde{A}}(a,1)$, so nimmt die Funktion
    $\widehat{a'}=\frac{\widehat{a}}{1-\widehat{a}}$ wegen
    $\widehat{a}(M_{C^{*}_{\tilde{A}}(a,1)})\subseteq [0,1)$ Werte in
    $[0,+\infty)$ an. Daraus folgt $a'\geq 0$ und auf analoge Weise
    $b'\geq 0$.

    Wegen $a'+b'\geq 0$ ist $c:=(a'+b')(1+a'+b')^{-1}$ wohldefiniert;
    da $A$ ein Ideal in $\tilde{A}$ ist, gilt sogar $c\in A$.
    Betrachten wir die Gelfandtransformation auf
    $C^{*}_{\tilde{A}}(a'+b',1)$, so ist einerseits
    $\hat{c}=\widehat{a'+b'}(1+\widehat{a'+b'})^{-1}$ das Produkt
    zweier Funktionen mit nichtnegativen Werten und hat daher dieselbe
    Eigenschaft, andererseits nimmt $\abs{\hat{c}}=\hat{c}$ ein
    Maximum an, also $\norm[\infty]{\hat{c}}=\lambda/(1+\lambda)<1$
    für ein nichtnegatives\footnote{Tatsächlich gilt wegen der
      Monotonie von $t\mapsto t/(1+t)$ genauer
      $\lambda=\norm[\infty]{\widehat{a'+b'}}$.} $\lambda$. Folglich
    gilt $c\geq 0$ sowie $\norm{c}<1$ und damit $c\in I$. Die Rechnung
    \begin{displaymath}
      (1+a')^{-1}=\big(\underbrace{(1-a)(1-a)^{-1}}_{=1}+a(1-a)^{-1}\big)^{-1}=\big((1-a+a)(1-a)^{-1}\big)^{-1}=1-a
    \end{displaymath}
    zeigt $a'(1+a')^{-1}=a(1-a)^{-1}(1-a)=a$; analog gilt
    $b'(1+b')^{-1}=b$. Die Hilfsbehauptung \eqref{eq:approx-eins} mit
    $d_{1}:=a'\leq a'+b'=:d_{2}$ liefert
    \begin{displaymath}
      a=d_{1}(1+d_{1})^{-1}\leq d_{2}(1+d_{2})^{-1}=c.
    \end{displaymath}
    Entsprechend zeigt man $b\leq c$, womit $I$ als gerichtete Menge
    identifiziert wurde.
  \item Wegen $A=\spn A^{+}$ genügt es, $b=\lim_{a\in I}u_{a}b$ für
    positive $b$ zu zeigen, wobei wir durch Skalieren zusätzlich
    $\norm{b}\leq 1$ annehmen können. Wir zeigen zunächst
    $\lim_{a\in I}bu_{a}b=b^{2}$ bzw. äquivalent dazu
    $\lim_{a\in I}b(1-u_{a})b=0$; man beachte, dass wir hier in
    $\tilde{A}$ rechnen.

    Aus $a\geq a_{0}$ mit $a,a_{0}\in I$ folgt mit
    Lemma~\ref{lem:eig-ho}(\ref{item:eig-ho-ii}) die Ungleichung
    \begin{displaymath}
      b(1-u_{a})b\leq b(1-u_{a_{0}})b
    \end{displaymath}
    und nach (\ref{item:eig-ho-iv}) aus demselben Lemma
    $\norm{b(1-u_{a})b}\leq\norm{b(1-u_{a_{0}})b}$. Der Beweis des
    Zwischenschritts reduziert sich also auf die Konstruktion eines
    $a_{0}\in I$ mit $\norm{b(1-u_{a_{0}})b}\leq\epsilon$ für
    gegebenes $\epsilon >0$, wobei wir ohne Beschränkung der
    Allgemeinheit $\epsilon<1$ annehmen. Setzt man zur Vereinfachung
    der Notation $f:=\hat{b}\in C_{0}(M)$, wobei $M$ den Gelfandraum
    von $C^{*}_{A}(b)$ bezeichnet, so ist
    $K:=\abs{f}^{-1}[\epsilon,\infty)$ kompakt. Nach dem Lemma von
    Urysohn für lokalkompakte Räume, siehe \cite[2.12 Urysohn's
    Lemma]{rudin:rca}, existiert eine stetige Funktion $g:M\to[0,1]$
    mit kompaktem Träger -- insbesondere gilt $g\in C_{0}(M)$ -- und
    $g(m)=1$ für alle $m\in K$. Wir wählen ein $\delta<1$ mit
    $1-\delta\leq\epsilon$. Für $m\in K$ gilt
    \begin{displaymath}
      \abs{f(m)-\delta g(m)f(m)}=(1-\delta)\abs{f(m)}\leq
      (1-\delta)\norm[\infty]{f}=(1-\delta)\norm{b}\leq 1-\delta\leq\epsilon,
    \end{displaymath}
    und für $m\in\compl{K}$
    \begin{displaymath}
      \abs{f(m)-\delta g(m)f(m)}=(1-\delta
      g(m))\abs{f(m)}\leq\abs{f(m)}<\epsilon.
    \end{displaymath}
    Insgesamt erhalten wir also
    $\norm[\infty]{f-\delta gf}\leq\epsilon$. Setzt man
    $a_{0}:=(\hat{.})^{-1}(\delta g)$, so ist $a_{0}$ positiv. Wegen
    $\norm{a_{0}}=\delta\norm[\infty]{g}\leq\delta<1$ gilt
    $a_{0}\in I$. Schließlich erhalten wir aus $\norm{b}\leq 1$ die
    Abschätzung
    \begin{displaymath}
      \norm{b(1-u_{a_{0}})b}\leq\norm{b-u_{a_{0}}b}=\norm[\infty]{\hat{b}-\widehat{a_{0}}\hat{b}}=\norm[\infty]{f-\delta
        gf}\leq\epsilon.
    \end{displaymath}
    
    Um das allgemeine Resultat zu zeigen, sei bemerkt, dass wegen der
    schon wiederholt aufgetretenen Überlegungen zum
    Spektralabbildungssatz die Elemente $1-u_{a}$ positiv sind mit
    $\norm{1-u_{a}}\leq 1$. Das bedeutet, dass $(1-u_{a})^{1/2}$
    wohldefiniert ist, wobei $\norm{(1-u_{a})^{1/2}}\leq 1$. Wegen
    \begin{displaymath}
      \norm{(1-u_{a})^{1/2}b}^{2}=\norm{\left((1-u_{a})^{1/2}b\right)^{*}(1-u_{a})^{1/2}b}=\norm{b(1-u_{a})b}\xrightarrow{a\in
        I} 0
    \end{displaymath}
    schließen wir auf
    \begin{displaymath}
      \norm{(1-u_{a})b}\leq\norm{(1-u_{a})^{1/2}}\cdot\norm{(1-u_{a})^{1/2}b}\leq\norm{(1-u_{a})^{1/2}b}\xrightarrow{a\in
        I} 0,
    \end{displaymath}
    und daher auf die gewünschte Beziehung $b=\lim_{a\in I}u_{a}b$.
  \end{enumerate}
\end{proof}

\section{Das Einselement als Extremalpunkt}
\label{sec:eins-und-extr}
Wir kommen zur schon angekündigten Behandlung von Extremalpunkten der
Einheitskugel. Primäres Ziel ist dabei die Tatsache, dass eine
$C^{*}$-Algebra genau dann ein Einselement hat, wenn die Einheitskugel
Extremalpunkte aufweist. Dazu erinnern wir an folgende Definition: Ist
$V$ ein Vektorraum und $M\subseteq V$ eine beliebige Teilmenge, dann
heißt $x\in M$ \emph{Extremalpunkt} von $M$, wenn aus
$tx_{1}+(1-t)x_{2}=x$ für $x_{1},x_{2}\in M$ und $t\in (0,1)$ schon
$x_{1}=x_{2}=x$ folgt.
\begin{bemerkung}\label{bem:char-extremalpkt}
  Ist $M$ konvex, so behaupten wir, dass $x\in M$ genau dann ein
  Extremalpunkt ist, wenn aus $\frac{1}{2}(x_{1}+x_{2})=x$ für
  $x_{1},x_{2}\in M$ bereits $x_{1}=x_{2}=x$ folgt; man kann sich also
  auf $t=1/2$ beschränken. Denn ist $tx_{1}+(1-t)x_{2}=x$, wobei wir
  ohne Beschränkung der Allgemeinheit $t\leq 1/2$ annehmen, dann ist
  wegen der Konvexität $\tilde{x}_{1}:=2tx_{1}+(1-2t)x_{2}$ ein
  Element von $M$ mit $x=\frac{1}{2}(\tilde{x}_{1}+x_{2})$. Aus der
  Voraussetzung folgt $x_{2}=x$, sodass wir
  $tx_{1}=x-(1-t)x_{2}=tx_{2}$ bzw. $x_{1}=x_{2}=x$ erhalten.
\end{bemerkung}
Zunächst betrachten wir kommutative $C^{*}$-Algebren, da dort die
Gelfandtransformation zur Verfügung steht.
\begin{lemma}\label{lem:komm-extremalpkt}
  Sei die $C^{*}$-Algebra $A\neq\{0\}$ kommutativ.
  \begin{enumerate}[label=(\roman*),ref=\roman*]
  \item\label{item:komm-extremalpkt-i} Die Extremalpunkte von $S$ sind
    genau die unitären Elemente von $A$. Insbesondere hat $S$ genau
    dann einen Extremalpunkt, wenn $A$ ein Einselement enthält.
  \item\label{item:komm-extremalpkt-ii} Die Extremalpunkte von
    $A_{sa}\cap S$ sind genau die selbstadjungierten, unitären
    Elemente von $A$.
  \item\label{item:komm-extremalpkt-iii} Die Extremalpunkte von
    $S\cap A^{+}$ sind genau die Projektionen -- also die
    selbstadjungierten idempotenten Elemente,
    vgl. Notation~\ref{not:cstar} -- von $A$. Ist $x\in S\cap A^{+}$
    nicht extrem, dann existiert sogar ein $a\in S\cap A^{+}$ mit
    \begin{equation}\label{eq:komm-extremalpkt-i}
      xa\neq 0\qquad \textnormal{und} \qquad x\pm xa\in S\cap A^{+}.
    \end{equation}
  \end{enumerate}
\end{lemma}
\begin{proof}
  Wir können $A=C_{0}(M)$ mit einem lokalkompakten Hausdorffraum $M$
  annehmen.
  \begin{enumerate}[label=(\roman*),ref=\roman*]
  \item Sei zunächst $f\in S$ unitär, also $\abs{f(m)}=1$ für alle
    $m\in M$. Wegen Bemerkung~\ref{bem:char-extremalpkt} haben wir aus
    $f=\frac{1}{2}(f_{1}+f_{2})$ mit Elementen $f_{1},f_{2}\in S$ auf
    $f_{1}=f_{2}=f$ zu schließen. Für beliebiges $m\in M$ gilt
    $f(m)=\frac{1}{2}(f_{1}(m)+f_{2}(m))$ und
    $\abs{f_{1}(m)},\abs{f_{2}(m)}\leq 1$. Wenn wir zeigen, dass nur
    Punkte der Einheitskreislinie $\T$ Extremalpunkte der
    abgeschlossenen Einheitsscheibe in $\C$ sein können, so folgt
    $f_{1}(m)=f_{2}(m)=f(m)$ und wegen der Beliebigkeit von $m\in M$
    die Gleichung $f_{1}=f_{2}=f$.

    Sei $\T\ni\zeta=\frac{1}{2}(\zeta_{1}+\zeta_{2})$ mit
    $\abs{\zeta_{1}},\abs{\zeta_{2}}\leq 1$. Dann gilt
    $1=\abs{\zeta}\leq\frac{1}{2}(\abs{\zeta_{1}}+\abs{\zeta_{2}})\leq
    1$, also $\abs{\zeta_{1}}+\abs{\zeta_{2}}=2$. Wegen
    $\abs{\zeta_{1}},\abs{\zeta_{2}}\leq 1$ folgt daraus
    $\zeta_{1},\zeta_{2}\in\T$. Schreibt man
    $\zeta_{j}=e^{i\theta_{j}}$ mit $\theta_{j}\in[0,2\pi)$, wobei
    ohne Beschränkung der Allgemeinheit $\theta_{1}\leq\theta_{2}$
    sei, dann ergibt sich
    $\zeta=\frac{1}{2}e^{i\theta_{1}}(1+e^{i(\theta_{2}-\theta_{1})})$. Wir
    erhalten
    \begin{displaymath}
      1=\abs{\zeta}=\frac{1}{2}\abs{1+e^{i(\theta_{2}-\theta_{1})}}=\frac{1}{2}\sqrt{(1+\cos(\theta_{2}-\theta_{1}))^{2}+\sin(\theta_{2}-\theta_{1})^{2}}
    \end{displaymath}
    bzw.
    \begin{displaymath}
      4=(1+\cos(\theta_{2}-\theta_{1}))^{2}+\sin(\theta_{2}-\theta_{1})^{2}=2+2\cos(\theta_{2}-\theta_{1})
    \end{displaymath}
    und schließen auf $\cos(\theta_{2}-\theta_{1})=1$. Wegen
    $\theta_{2}-\theta_{1}\in[0,2\pi)$ folgt $\theta_{1}=\theta_{2}$
    und $\zeta_{1}=\zeta_{2}$. Somit ist die Hilfsbehauptung gezeigt.
    
    Sei umgekehrt $f\in S$ nicht unitär, also $\abs{f(m_{0})}<1$ für
    ein $m_{0}\in M$. Da $M$ lokalkompakt und $f$ stetig ist, gibt es
    eine offene Umgebung $U$ von $m_{0}$ mit kompaktem Abschluss und
    $\abs{f(m)}<\frac{1}{2}(1+\abs{f(m_{0})})$ für alle $m\in U$. Es
    folgt
    \begin{displaymath}
      \alpha:=\max\set{\abs{f(m)}}{m\in\cl{U}}\leq\frac{1}{2}(1+\abs{f(m_{0})})<1.
    \end{displaymath}
    
    Nach dem Lemma von Urysohn für lokalkompakte Räume, \cite[2.12
    Urysohn's Lemma]{rudin:rca}, gibt es eine stetige Funktion
    $g:M\to [0,1]$ mit $g(m_{0})=1$ und
    $\supp g\subseteq U\subseteq\cl{U}$. Somit hat $g$ kompakten
    Träger, inbesondere liegt $g$ in $C_{0}(M)$. Für
    $h:=(1-\alpha)g\in C_{0}(M)$ wollen wir $\abs{f(m)\pm h(m)}\leq 1$
    für alle $m\in M$ nachweisen. Dazu unterscheiden wir die Fälle
    $m\in U$ und $m\notin U$. Im ersten Fall gilt
    \begin{displaymath}
      \abs{f(m)\pm h(m)}\leq\abs{f(m)}+(1-\alpha)g(m)\leq \alpha+(1-\alpha)=1.
    \end{displaymath}
    Im zweiten Fall folgt aus $h(m)=0$ die Abschätzung
    \begin{displaymath}
      \abs{f(m)\pm h(m)}=\abs{f(m)}\leq\norm[\infty]{f}\leq 1.
    \end{displaymath}
    
    Wir erhalten $f\pm h\in S$. Wegen $h(m_{0})=1-\alpha>0$ gilt
    $h\neq 0$, sodass die Funktionen $f\pm h$ von $f$ verschieden
    sind. Folglich ist $f=\frac{1}{2}\big((f+h)+(f-h)\big)$ nicht
    extremal in $S$.

    Zur Existenz von Einselementen: Wenn ein Extremalpunkt
    $f\in S\subseteq C_{0}(M)$ existiert, dann gilt nach dem
    Bewiesenen $\abs{f(m)}=1$ für alle $m\in M$. Da die Funktion im
    Unendlichen verschwindet, ist das nur möglich, wenn $M$ kompakt
    ist. In diesem Fall gilt $C_{0}(M)=C(M)$ und
    $\mathds{1}\in C_{0}(M)$. Diese Funktion stellt das Einselement
    dar. Enthält umgekehrt $C_{0}(M)$ ein Einselement, so ist dieses
    wie in jeder $C^{*}$-Algebra unitär und daher ein Extremalpunkt
    von $S$.
  \item Ein selbstadjungiertes, unitäres $f$ liegt klarerweise in
    $A_{sa}\cap S$ und ist nach (\ref{item:komm-extremalpkt-i}) sogar
    ein Extremalpunkt von $S$, also insbesondere von $A_{sa}\cap S$.

    Ist $f\in A_{sa}\cap S$ nicht unitär, so können wir denselben
    Beweis wie in (\ref{item:komm-extremalpkt-i}) verwenden, um zu
    zeigen, dass $f$ kein Extremalpunkt ist. Eine Funktion in
    $C_{0}(M)$ ist genau dann selbstadjungiert, wenn sie reellwertig
    ist. Da $f$ selbstadjungiert, also reellwertig, und die oben
    konstruierte Funktion $h:M\to [0,1]$ ebenfalls reellwertig ist,
    gilt $f\pm h\in A_{sa}$. Wir erhalten $f\pm h\in A_{sa}\cap S$.
  \item Sei $p\in C_{0}(M)$ eine Projektion. Die Projektionen in
    $C_{0}(M)$ sind genau die Funktionen, die höchstens die Werte
    $0,1$ annehmen, da ja $p(m)^{2}=p(m)$ gelten muss. Daraus folgt
    wegen $\sigma(p)=p(M)\subseteq\{0,1\}$ einmal $p\in S\cap A^{+}$.
    Schreibt man $p=\frac{1}{2}(h_{1}+h_{2})$ mit
    $h_{1},h_{2}\in S\cap A^{+}$, also
    $h_{1}(M),h_{2}(M)\subseteq [0,1]$, dann folgt aus $p(m)=0$ schon
    $h_{1}(m)=h_{2}(m)=0$. Aus $p(m)=1$ folgt die analoge Bedingung
    $h_{1}(m)=h_{2}(m)=1$. Mit anderen Worten gilt $h_{1}=h_{2}=p$ und
    $p$ ist ein Extremalpunkt.

    Für die Umkehrung sei $x\in S\cap A^{+}$ keine Projektion. Dies
    ist äquivalent dazu, dass $x(m_{0})\in (0,1)$ für ein $m_{0}\in M$
    gilt. Analog zu (\ref{item:komm-extremalpkt-i}) gibt es eine
    offene Umgebung $U$ von $m_{0}$ mit kompaktem Abschluss $\cl{U}$
    und $x(m)\in \big(x(m_{0})/2,(1+x(m_{0}))/2\big)\subseteq (0,1)$
    für alle $m\in U$. Aus dem Lemma von Urysohn erhalten wir eine
    Funktion $g:M\to [0,1]$ mit $g(m_{0})=1$ und
    $\supp g\subseteq U\subseteq\cl{U}$, insbesondere $g\in
    C_{0}(M)$. Wegen $x(m_{0})g(m_{0})=x(m_{0})>0$ gilt $xg\neq
    0$. Setzt man
    $\epsilon:=\min\big(1,\frac{2}{1+x(m_{0})}-1\big)>0$, so folgt
    $1-\epsilon g(m)\geq 0$. Außerdem gilt
    $x(m)(1+\epsilon g(m))\leq 1$ für alle $m\in M$: Dafür
    unterscheiden wir die Fälle $m\in U$ und $m\notin U$. Im ersten
    Fall folgt aus $g(m)\leq 1$ die Abschätzung
    \begin{displaymath}
      x(m)(1+\epsilon
      g(m))\leq\frac{1+x(m_{0})}{2}\left(1+\left(\frac{2}{1+x(m_{0})}-1\right)\cdot
        1\right)=1.
    \end{displaymath}
    Im zweiten Fall führt die Überlegung
    \begin{displaymath}
      x(m)(1+\epsilon g(m))=x(m)\leq 1
    \end{displaymath}
    zum Ziel. Zusammen mit den offensichtlichen Tatsachen
    \begin{displaymath}
      1+\epsilon g(m)\geq 1\geq 0\quad\text{und}\quad x(m)(1-\epsilon
      g(m))\leq x(m)\leq 1\quad\text{für alle}\ m\in M
    \end{displaymath}
    erhalten wir $x\pm x\epsilon g\in S\cap A^{+}$. Die Wahl
    $a:=\epsilon g$ liefert somit ein Element, das
    \eqref{eq:komm-extremalpkt-i} erfüllt. Wegen
    $x=\frac{1}{2}\big((x+xa)+(x-xa)\big)$ kann $x$ kein Extremalpunkt
    sein.
  \end{enumerate}
\end{proof}
Damit können wir die oben angepeilte Charakterisierung von
$C^{*}$-Algebren mit Einselement auch im nichtkommutativen Fall
beweisen.
\begin{satz}\label{satz:extremalpkt}
  $S$ hat genau dann einen Extremalpunkt, wenn in $A$ ein Einselement
  existiert. In diesem Fall ist das Einselement extremal.
\end{satz}
\begin{proof}
  Ist $1$ das Einselement von $A$, dann hat $S$ den Extremalpunkt $1$:
  In der Tat folgt aus $1=\frac{1}{2}(a+b)$ auch
  $1=\re 1=\frac{1}{2}(\re a+\re b)$. Aus $\re b=2-\re a$ erhalten
  wir, dass $\re a$ und $\re b$ kommutieren. Da $1$ sicher unitär ist,
  folgt aus
  Lemma~\ref{lem:komm-extremalpkt}(\ref{item:komm-extremalpkt-i})
  angewandt in $C^{*}_{A}(\re a,\re b,1)$, dass $\re a=\re b=1$
  ist. Wir erhalten $1=\re a=\frac{1}{2}(a+a^{*})$ bzw. $a^{*}=2-a$
  und schließen auf die Normalität von $a$. Erneute Anwendung von
  Lemma~\ref{lem:komm-extremalpkt}(\ref{item:komm-extremalpkt-i}),
  diesmal in $C^{*}(a,1)$, liefert $a=1$; analog zeigt man $b=1$.

  Habe nun $S$ umgekehrt einen Extremalpunkt $x$. Wir zeigen zunächst,
  dass $x^{*}x$ eine Projektion ist, indem wir das Gegenteil auf einen
  Widerspruch führen. Ist $x^{*}x$ keine Projektion, dann klarerweise
  $\abs{x}=(x^{*}x)^{1/2}$ auch nicht. Aus
  Lemma~\ref{lem:komm-extremalpkt}(\ref{item:komm-extremalpkt-iii}),
  angewandt in $B:=C_{\tilde{A}}^{*}(\abs{x},1)$, folgt die Existenz
  eines $a\in S_{B}\cap B^{+}$ mit
  \begin{displaymath}
    \abs{x}a\neq 0\qquad\textnormal{und}\qquad\norm{\abs{x}(1\pm
      a)}\leq 1.
  \end{displaymath}
  Offenbar gilt $x=\frac{1}{2}\big(x(1+a)+x(1-a)\big)$ und, da $A$ ein
  Ideal in $\tilde{A}$ ist, auch $x(1\pm a)\in A$. Weiters berechnen
  wir
  \begin{align}\label{eq:extremalpkt-i}
    \begin{split}
      \norm{x(1\pm a)}^{2}&=\norm{(x(1\pm a))^{*}(x(1\pm
        a))}=\norm{(1\pm a)\abs{x}^{2}(1\pm a)}\\
      &=\norm{(\abs{x}(1\pm a))^{*}(\abs{x}(1\pm
        a))}=\norm{\abs{x}(1\pm a)}^{2}\leq 1,
    \end{split}
  \end{align}
  womit $x(1\pm a)\in S$ ist. Da $x$ extremal ist, folgt
  $x=x(1\pm a)$, also $xa=0$. Mit zu \eqref{eq:extremalpkt-i} analoger
  Rechnung erhalten wir dazu im Widerspruch
  \begin{equation}\label{eq:extremalpkt-ii}
    \norm{xa}^{2}=\norm{\abs{x}a}^{2}\neq 0.
  \end{equation}
  Also muss $p:=x^{*}x$ eine Projektion sein. Daraus folgt (wir
  rechnen weiter in $\tilde{A}$)
  \begin{displaymath}
    \norm{x(1-p)}^{2}=\norm{(1-p)x^{*}x(1-p)}=\norm{(1-p)p(1-p)}=0
  \end{displaymath}
  und somit $x=xp$. Mit $q:=xx^{*}$ schließen wir auf $q=(xp)x^{*}$
  und
  \begin{displaymath}
    q^{2}=xpx^{*}xpx^{*}=xp^{3}x^{*}=xpx^{*}=q,
  \end{displaymath}
  wodurch sich $q$ ebenfalls als Projektion herausstellt. Dabei gilt
  \begin{displaymath}
    \norm{x^{*}(1-q)}^{2}=\norm{(1-q)xx^{*}(1-q)}=\norm{(1-q)q(1-q)}=0,
  \end{displaymath}
  also $x^{*}=x^{*}q$ bzw. durch Adjungieren $x=qx$.

  Als Nächstes zeigen wir, dass die Menge $(1-q)S(1-p)$ nur aus dem
  Nullelement besteht. Für $a\in(1-q)S(1-p)$ gilt $a=(1-q)a(1-p)$ und
  \begin{align*}
    \norm{x\pm a}^{2}&=\norm{(x\pm a)^{*}(x\pm a)}=\norm{x^{*}x\pm
                       a^{*}x\pm x^{*}a+a^{*}a}\\
                     &=\|p\pm(1-p)a^{*}\underbrace{(1-q)x}_{=0}\pm
                       \underbrace{x^{*}(1-q)}_{=0}a(1-p)+(1-p)\underbrace{a^{*}(1-q)^{2}a}_{=a^{*}a,\
                       \text{da}\ (1-q)a=a}(1-p)\|\\
                     &=\norm{p+(1-p)a^{*}a(1-p)},
  \end{align*}
  was nach Lemma~\ref{lem:norm-summe} mit
  $\max\big(\norm{p},\norm{(1-p)a^{*}a(1-p)}\big)$ übereinstimmt. Als
  Projektion er\-füllt $p$ die Ungleichungen
  $\norm{p},\norm{1-p}\leq 1$. Daraus folgt
  \begin{displaymath}
    \norm{(1-p)a^{*}a(1-p)}\leq\norm{1-p}\norm{a^{*}a}\norm{1-p}\leq\norm{a}
    ^{2}\leq 1,
  \end{displaymath}
  sodass die Elemente $x\pm a$ in der Einheitskugel $S$ enthalten
  sind. Aus der trivialen Gleichung
  $x=\frac{1}{2}\big((x+a)+(x-a)\big)$ folgt schließlich $a=0$, da $x$
  ein Extremalpunkt ist.

  Ist $(u_{i})_{i\in I}$ ein approximatives Einselement in $A$ --
  nach Satz~\ref{satz:approx-eins} gibt es ein solches -- dann folgt
  $(1-q)u_{i}(1-p)\in (1-q)S(1-p)=\{0\}$
  bzw. $u_{i}=qu_{i}+u_{i}p-qu_{i}p$. Die rechte Seite dieser
  Gleichung konvergiert nach Definition eines approximativen
  Einselements gegen $e:=q+p-qp$. Somit konvergiert das Netz
  $(u_{i})_{i\in I}$ gegen $e$. Der Grenzwert eines approximativen
  Einselements ist, wenn er existiert, das Einselement der
  $C^{*}$-Algebra, denn für beliebiges $a\in A$ gilt
  $ea=(\lim_{i\in I}u_{i})a=\lim_{i\in I}u_{i}a=a$. Die Gleichung
  $ae=a$ zeigt man analog.
\end{proof}

Schränken wir uns auf die selbstadjungierten Elemente ein, so können
wir die Aussage von
Lemma~\ref{lem:komm-extremalpkt}(\ref{item:komm-extremalpkt-ii}) auf
nicht notwendigerweise kommutative $C^{*}$-Algebren
verallgemeinern. Es steht über die übliche Technik der erzeugten
$C^{*}$-Unteralgebren nämlich die in kommutativen $C^{*}$-Algebren
entwickelte Theorie zur Verfügung.
\begin{satz}\label{satz:extremalpkt-selbstadj}
  Hat $A$ ein Einselement, so sind die Extremalpunkte von $A_{sa}\cap S$
  genau die selbstadjungierten, unitären Elemente von $A$.
\end{satz}
\begin{proof}
  Sei zunächst $u$ selbstadjungiert und unitär, d.~h. $u^{2}=1$. Die
  lineare Abbildung $T:A\to A$ definiert durch $Tx:=ux$ ist zu sich
  selbst invers und somit ein Isomorphismus. Die Ungleichungskette
  \begin{equation}\label{eq:bew-extremalpkt-selbstadj}
    \norm{x}=\norm{u^{2}x}\leq\norm{u}\cdot\norm{ux}=\norm{ux}\leq\norm{x}
  \end{equation}
  zeigt, dass $T$ zusätzlich isometrisch ist. Als isometrischer
  Isomorphismus\footnote{Hier betrachten wir $A$ als Banachraum.}
  bildet $T$ Extremalpunkte von $S$ auf ebensolche ab. Nach
  Satz~\ref{satz:extremalpkt} ist $1$ extremal in $S$, sodass $u=T1$
  ebenfalls ein Extremalpunkt von $S$ ist. Insbesondere ist $u$
  extremal in $A_{sa}\cap S$.

  Für die Umkehrung sei $u$ ein Extremalpunkt von $A_{sa}\cap
  S$. Zunächst einmal ist $u$ jedenfalls selbstadjungiert. Setzen wir
  $B:=C^{*}(u,1)$, so ist $u$ offenbar auch ein Extremalpunkt von
  $B_{sa}\cap
  S_{B}$. Lemma~\ref{lem:komm-extremalpkt}(\ref{item:komm-extremalpkt-ii}),
  angewandt in $B$, liefert, dass $u$ unitär in $B$ und daher unitär
  in $A$ ist.
\end{proof}
\begin{bemerkung}
  Der erste Teil des Beweises zeigt, dass unitäre Elemente in
  beliebigen $C^{*}$-Algebren Extremalpunkte der Einheitskugel
  sind. Die Zusatzvoraussetzung wird für die Umkehrung benötigt.
\end{bemerkung}

\section{Dichtheit von Projektionen}
\label{sec:dichth-von-proj}
Als Letztes soll in diesem Kapitel ein Satz bewiesen werden, der eine
hinreichende Bedingung dafür angibt, dass die lineare Hülle aller
Projektionen in einer kommutativen $C^{*}$-Algebra mit Einselement
dicht ist. Wegen der Gelfandtransformation genügt es, Algebren der
Form $C(K)$ mit einem kompakten Hausdorffraum $K$ zu betrachten. Dass
dieses Problem alles andere als trivial ist, zeigt das folgende
Beispiel, das zugleich die daran anschließende Definition motiviert.
\begin{beispiel}\label{bsp:proj-dicht}
  Sei $K=[0,1]$ versehen mit der euklidischen Topologie. Die
  Projektionen in $C(K)$ sind diejenigen stetigen reellwertigen
  Funktionen, die nur die Werte $0$ und $1$ annehmen. Da $K$
  zusammenhängend ist, muss für eine Projektion $p\in C(K)$ auch
  $p(K)\subseteq\{0,1\}$ zusammenhängend sein, womit $p\equiv 0$ oder
  $p\equiv 1$ gilt. Ihre lineare Hülle, der Raum aller konstanten
  Funktionen, ist schon abgeschlossen und daher sicher nicht dicht in
  $C(K)$.
\end{beispiel}
\begin{bemerkung}\label{bem:proj-dicht}
  Diese Argumentation zeigt auch, dass eine notwendige Bedingung für
  die Dichtheit von
  $\FF:=\spn\set{p\in C(K)}{p^{*}=p\ \text{und}\ p^{2}=p}$ in $C(K)$
  gegeben ist durch die Forderung, dass $K$ \emph{total
    unzusammenhängend} ist. Das bedeutet, dass alle
  Zusammenhangskomponenten einpunktig sind. In der Tat muss eine
  Projektion auf zusammenhängenden Mengen konstant sein. Damit sind
  auch alle Funktionen aus $\FF$ und folglich aus $\cl{\FF}$ auf
  zusammenhängenden Mengen konstant. Enthält eine
  Zusammenhangskomponente mindestens zwei Punkte, so kann man diese
  nach dem Lemma von Urysohn funktional trennen. Die erhaltene stetige
  Funktion kann somit nicht in $\cl{\FF}$ enthalten sein.
\end{bemerkung}
Wir fordern jedoch eine stärkere Eigenschaft.
\begin{definition}\label{def:extr-unzsh-stonesch}
  Ein topologischer Raum heißt \emph{extremal unzusammenhängend}, wenn
  der Abschluss jeder offenen Menge wieder offen und damit clopen
  ist. Ein kompakter, extremal unzusammenhängender Hausdorffraum heißt
  \emph{stonesch}.
\end{definition}
\begin{bemerkung}
  Ein extremal unzusammenhängender Hausdorffraum ist auch total
  unzusammenhängend. Denn wäre $C$ eine zusammenhängende Menge mit
  $x,y\in C$ und $x\neq y$, dann würden offene Mengen $O_{x},O_{y}$
  existieren mit $x\in O_{x}$, $y\in O_{y}$ und
  $O_{x}\cap O_{y}=\emptyset$. Insbesondere gilt
  $y\notin\cl{O_{x}}$. Der Abschluss $\cl{O_{x}}$ ist nach
  Voraussetzung clopen, womit $\cl{O_{x}}$ und $\compl{\cl{O_{x}}}$
  getrennt sind und den ganzen Raum überdecken. Da $C$ zusammenhängend
  ist, gilt $C\subseteq\cl{O_{x}}$ oder
  $C\subseteq\compl{\cl{O_{x}}}$. Beide Varianten führen aber auf
  einen Widerspruch, da $x\in\cl{O_{x}}$ und $y\in\compl{\cl{O_{x}}}$.
\end{bemerkung}

Zunächst zeigen wir eine Charakterisierung stonescher Räume. Dazu
müssen wir an das Konzept der Baire'schen Kategorien erinnern. Eine
Teilmenge $M$ eines topologischen Raums heißt \emph{nirgends dicht},
wenn $\intr{\cl{M}}=\emptyset$ ist. Abzählbare Vereinigungen nirgends
dichter Mengen werden \emph{von erster Kategorie} genannt, alle
anderen Teilmengen \emph{von zweiter Kategorie}.
\begin{satz}\label{satz:char-stonesch}
  Sei $K$ ein kompakter Hausdorffraum und $C(K,\R)$ die Menge der
  stetigen, reellwertigen Funktionen auf $K$. Dann sind die folgenden
  vier Aussagen äquivalent:
  \begin{enumerate}[label=(\roman*),ref=\roman*]
  \item\label{item:char-stonesch-i} $K$ ist extremal
    unzusammenhängend, also stonesch.
  \item\label{item:char-stonesch-ii} Jede beschränkte und bezüglich
    der punktweisen natürlichen Ordnung gerichtete Teilmenge von
    $C(K,\R)$ hat ein Supremum\footnote{Der Vollständigkeit halber sei
      bemerkt, dass das punktweise Supremum einer Familie stetiger
      Funktionen im Allgemeinen nicht mehr stetig ist.} in $C(K,\R)$.
  \item\label{item:char-stonesch-iii} Jede beschränkte Teilmenge von
    $C(K,\R)$ hat ein Supremum in $C(K,\R)$ bezüglich der punktweisen
    natürlichen Ordnung.
  \item\label{item:char-stonesch-iv} Jede beschränkte, reellwertige
    und von unten halbstetige Funktion\footnote{Eine Funktion
      $f:K\to\R$ heißt von unten halbstetig, wenn für alle $t\in\R$
      das Urbild $f^{-1}(t,+\infty)$ offen bzw. $f^{-1}(-\infty,t]$
      abgeschlossen ist.} auf $K$ stimmt bis auf eine Menge von erster
    Kategorie mit einer stetigen Funktion überein.
  \end{enumerate}
\end{satz}
\begin{proof}
  Der Beweis verläuft zyklisch:

  (\ref{item:char-stonesch-i})~$\Rightarrow$~(\ref{item:char-stonesch-iv}):
  Sei $f$ eine beschränkte, reellwertige und von unten halbstetige
  Funktion. Für die zu zeigende Behauptung können wir durch etwaiges
  Addieren einer positiven Konstanten und anschließendes Skalieren
  $0\leq f\leq 1$ annehmen. Wir definieren für $t\in\R$ die Mengen
  $F(t):=f^{-1}(-\infty,t]$ und $G(t):=\intr{F(t)}$. Dann sind die
  $F(t)$ wegen der Halbstetigkeit abgeschlossen und somit
  $\cl{\compl{F(t)}}$ clopen, da $K$ als extremal unzusammenhängend
  vorausgesetzt ist. Die Mengen $G(t)$ sind wegen
  $\intr{F(t)}=\intr{\big(\compl{(\compl{F(t)})}\big)}=\compl{\cl{\compl{F(t)}}}$
  ebenfalls clopen. Folglich sind die Indikatorfunktionen
  $\mathds{1}_{G(t)}$ stetig. Für $n\in\N$ definieren wir die stetige
  Funktion
  \begin{equation}\label{eq:bew-char-stonesch-i}
    f_{n}:=\sum_{k=1}^{2^{n}}\frac{k}{2^{n}}(\mathds{1}_{G(k/2^{n})}-\mathds{1}_{G((k-1)/2^{n})}).
  \end{equation}
  Sei ein beliebiges $x\in\R$ fixiert. Da die $G(t)$ monoton wachsend
  sind mit $G(1)=K$, gilt $f_{n}(x)=\frac{k}{2^{n}}$ für die Zahl
  \begin{equation}\label{eq:bew-char-stonesch-ii}
    k=\min\set{j\in\{0,\dots,2^{n}\}}{x\in G\left(\frac{j}{2^{n}}\right)}.
  \end{equation}
  Wegen $x\in G(\frac{k}{2^{n}})=G(\frac{2k}{2^{n+1}})$ gilt auch
  $f_{n+1}(x)\leq\frac{2k}{2^{n+1}}$. Wäre
  $f_{n+1}(x)\leq\frac{2k-2}{2^{n+1}}$, also
  $x\in G(\frac{2k-2}{2^{n+1}})=G(\frac{k-1}{2^{n}})$, so hätten wir
  einen Widerspruch zur Minimalität von $k$ in
  \eqref{eq:bew-char-stonesch-ii}. Wir erhalten
  $f_{n+1}(x)\in\{\frac{2k-1}{2^{n+1}},\frac{2k}{2^{n+1}}\}$ und
  daraus die Abschätzung
  $\norm[\infty]{f_{n+1}-f_{n}}\leq\frac{1}{2^{n+1}}$. Für $m\leq n$
  ergibt sich
  \begin{displaymath}
    \norm[\infty]{f_{n}-f_{m}}\leq\sum_{k=m}^{n-1}\norm[\infty]{f_{k+1}-f_{k}}\leq\sum_{k=m}^{n-1}\frac{1}{2^{k+1}}\leq\frac{1}{2^{m}},
  \end{displaymath}
  womit $(f_{n})_{n\in\N}$ eine Cauchy-Folge in $C(K,\R)$ ist und
  infolge gegen eine Funktion $f_{0}\in C(K,\R)$ konvergiert. Wir
  behaupten, dass $f$ und $f_{0}$ bis auf eine Menge von erster
  Kategorie überein\-stimmen. Dazu setzen wir
  \begin{displaymath}
    M:=\bigcup_{n=1}^{\infty}\bigcup_{k=1}^{2^{n}}F\left(\frac{k}{2^{n}}\right)\setminus G\left(\frac{k}{2^{n}}\right).
  \end{displaymath}
  Die Mengen
  $F\big(\frac{k}{2^{n}}\big)\setminus G\big(\frac{k}{2^{n}}\big)$
  sind abgeschlossen und haben leeres Inneres, da jede offene
  Teilmenge von $F\big(\frac{k}{2^{n}}\big)$ auch in
  $G\big(\frac{k}{2^{n}}\big)$ enthalten ist. Damit ist $M$ von erster
  Kategorie in $K$.

  Sind $x\in\compl{M}$ sowie $n\in\N$ fest und $k$ wie in
  \eqref{eq:bew-char-stonesch-ii}, so sind zwei Fälle zu
  unterscheiden: Ist $k=0$, so gilt $x\in G(0)$, also $f(x)=0$ und
  $f_{m}(x)=0$ für beliebiges $m\in\N$, folglich
  $f(x)=0=f_{0}(x)$. Für $k\geq 1$ hingegen gilt
  $x\notin G\big(\frac{k-1}{2^{n}}\big)$, womit wegen $x\in\compl{M}$
  auch $x\notin F\big(\frac{k-1}{2^{n}}\big)$ folgt. Zusammen mit
  $x\in G\big(\frac{k}{2^{n}}\big)\subseteq
  F\big(\frac{k}{2^{n}}\big)$ ergibt sich die Abschätzung
  $\frac{k-1}{2^{n}}<f(x)\leq\frac{k}{2^{n}}=f_{n}(x)$. Wir schließen
  auf $\abs{f_{n}(x)-f(x)}<\frac{1}{2^{n}}$ und erhalten tatsächlich
  $f(x)=f_{0}(x)$.

  \vspace{\baselineskip}
  (\ref{item:char-stonesch-iv})~$\Rightarrow$~(\ref{item:char-stonesch-iii}):
  Sei $\FF\subseteq C(K,\R)$ eine beschränkte Teilmenge. Definiert man
  $g(x):=\sup_{\phi\in\FF}\phi(x)$, so ist bekannt, dass $g$ von unten
  halbstetig ist. Aufgrund der Beschränktheit ist $g$ auch
  reellwertig, sodass wegen (\ref{item:char-stonesch-iv}) eine stetige
  Funktion $f:K\to\R$ existiert, die mit $g$ zumindest außerhalb einer
  Menge $M$ von erster Kategorie übereinstimmt. Da $g-f$ als Summe von
  unten halbstetiger Funktionen ebenfalls von unten halbstetig ist,
  ist die Menge $G:=\set{x\in K}{g(x)-f(x)>0}$ offen. Klarerweise ist
  $G$ auch in $M$ enthalten und daher selbst von erster Kategorie in
  $K$. Als Nächstes zeigen wir, dass $G$ auch als Teilmenge von sich
  von erster Kategorie ist. Dazu schreibt man
  \begin{displaymath}
    G=\bigcup_{n\in\N}G_{n}
  \end{displaymath}
  mit in $K$ nirgends dichten Mengen $G_{n}$, also
  $\left(\cl{G_{n}}^{K}\right)^{o,K}=\emptyset$. Es folgt
  \begin{equation}\label{eq:bew-char-stonesch-iii}
    \left(\cl{G_{n}}^{G}\right)^{o,K}=\left(\cl{G_{n}}^{K}\cap
      G\right)^{o,K}=\emptyset.
  \end{equation}
  Ist $O$ eine in $G$ offene, nichtleere Menge, so ist $O$ auch in $K$
  offen, da $G$ selbst offen ist. Nach
  \eqref{eq:bew-char-stonesch-iii} kann daher $O$ nicht in
  $\cl{G_{n}}^{G}$ enthalten sein und wir erhalten die gewünschte
  Tatsache
  \begin{displaymath}
    \left(\cl{G_{n}}^{G}\right)^{o,G}=\emptyset.
  \end{displaymath}
  Nach dem Satz von Baire für lokalkompakte Hausdorffräume, siehe
  \cite[2.2 Baire's theorem]{rudin:fana}, ist ein nichtleerer
  lokalkompakter Hausdorffraum niemals in sich selbst von erster
  Kategorie, sodass nur $G=\emptyset$ möglich ist. Wir schließen auf
  $g\leq f$. Folglich ist $f$ eine obere Schranke aller Funktionen aus
  $\FF$.  Ist umgekehrt $h\in C(K,\R)$ eine obere Schranke, so gilt
  $h(x)\geq g(x)=f(x)$ für $x\in\compl{M}$. Schreibt man
  $M=\bigcup_{n\in\N}M_{n}$ mit nirgends dichten Mengen $M_{n}$, dann
  gilt
  \begin{displaymath}
    \compl{M}=\bigcap_{n\in\N}\compl{M_{n}}\supseteq\bigcap_{n\in\N}\compl{\cl{M_{n}}}.
  \end{displaymath}
  Die Mengen $\compl{\cl{M_{n}}}$ sind offen und wegen
  $\cl{\compl{\cl{M_{n}}}}=\compl{\left(\intr{\cl{M_{n}}}\right)}=K$
  dicht. Nach dem Satz von Baire für lokalkompakte Räume, hier
  angewandt im kompakten Raum $K$, ist auch der Schnitt
  $\bigcap_{n\in\N}\compl{\cl{M_{n}}}$ und infolge $\compl{M}$ dicht
  in $K$. Also gilt die Ungleichung $h(x)\geq f(x)$ für alle $x\in
  K$. Wir erhalten, dass $f$ tatsächlich das Supremum von $\FF$ in
  $C(K,\R)$ ist.

  \vspace{\baselineskip}
  (\ref{item:char-stonesch-iii})~$\Rightarrow$~(\ref{item:char-stonesch-ii}):
  Klar.

  \vspace{\baselineskip}
  (\ref{item:char-stonesch-ii})~$\Rightarrow$~(\ref{item:char-stonesch-i}):
  Sei $G\subseteq K$ offen. Um zu beweisen, dass $\cl{G}$ clopen ist,
  zeigen wir die Stetigkeit der Indikatorfunktion
  $\mathds{1}_{\cl{G}}$.

  Die Menge $\FF:=\set{f\in C(K,\R)}{0\leq f\leq\mathds{1}_{G}}$ ist
  klarerweise beschränkt und nach oben gerichtet, da für
  $f_{1},f_{2}\in\FF$ auch $\max(f_{1},f_{2})$ in $\FF$ enthalten
  ist. Für $x\in G$ sind $\{x\}$ und $\compl{G}$ zwei disjunkte,
  abgeschlossene Mengen, sodass es nach dem Lemma von Urysohn eine
  stetige Funktion $f_{x}:K\to[0,1]$ gibt mit $f_{x}(x)=1$ und
  $f_{x}(\compl{G})\subseteq\{0\}$. Wegen $f_{x}\in\FF$ folgt
  sofort\footnote{Man beachte, dass hier noch kein Abschluss gebildet
    wird.}  $\mathds{1}_{G}(y)=\sup_{f\in\FF}f(y)$ für $y\in
  K$. Bezeichnet $g$ das nach Voraussetzung existierende Supremum von
  $\FF$ in $C(K,\R)$, so gilt einerseits $g\leq 1$, da die konstante
  Einsfunktion eine obere Schranke von $\FF$ ist, und andererseits
  $\mathds{1}_{G}\leq g$, da $g(y)$ eine obere Schranke aller $f(y)$
  für festes $y\in K$ ist. Die Funktion $g$ nimmt auf $G$ folglich den
  Wert $1$ an und wegen der Stetigkeit damit auch auf $\cl{G}$. Wir
  schließen auf die Ungleichung $\mathds{1}_{\cl{G}}\leq g$. Um
  $\mathds{1}_{\cl{G}}=g$ und damit die gewünschte Stetigkeit zu
  zeigen, müssen wir $g(x_{0})>0$ für ein $x_{0}\in\compl{\cl{G}}$ auf
  einen Widerspruch führen. In der Tat folgt aus dem Lemma von Urysohn
  die Existenz einer stetigen Funktion $h:K\to[0,1]$ mit $h(x_{0})=0$
  und $h\left(\cl{G}\right)\subseteq\{1\}$. Für die offenbar ebenfalls
  stetige und reellwertige Funktion $\tilde{g}:=gh$ und beliebige
  $f\in\FF$ und $y\in K$ gilt
  \begin{displaymath}
    \tilde{g}(y)\geq\mathds{1}_{\cl{G}}(y)\geq\mathds{1}_{G}(y)\geq f(y),
  \end{displaymath}
  sodass die stetige Funktion $\tilde{g}$ eine obere Schranke von
  $\FF$ ist. Da $g$ das Supremum ist, muss $g\leq\tilde{g}$ gelten,
  was zu dem Widerspruch $0<g(x_{0})\leq\tilde{g}(x_{0})=0$ führt.
\end{proof}
Nun kommen wir zum Dichtheitsresultat. Dessen Beweis ist relativ kurz,
da der Großteil der Arbeit schon im Beweis von
Satz~\ref{satz:char-stonesch} steckt.
\begin{satz}\label{satz:proj-dicht}
  Sei $K$ ein stonescher Raum und $A=C(K)$. Dann ist die lineare Hülle
  der Projektionen dicht in der $C^{*}$-Algebra $A$. Nimmt $f\in C(K)$
  nur nichtnegative Werte an, so ist $f$ sogar der Grenzwert einer
  Folge von Funktionen der Form
  $f_{n}:=\sum_{k=1}^{n}\gamma_{k,n}g_{k,n}$, wobei die $g_{k,n}$
  Projektionen und die Zahlen $\gamma_{k,n}$ nichtnegativ sind.
\end{satz}
\begin{proof}
  Sei $f\in C(K)$ beliebig. Durch Betrachten von Real- und
  Imaginärteil können wir annehmen, dass $f$ reellwertig ist. Da $f$
  beschränkt ist und die konstanten Funktionen in der linearen Hülle
  der Projektionen enthalten sind, können wir uns zusätzlich auf
  $0\leq f\leq 1$ einschränken. Wir führen die Konstruktion aus dem
  Beweisschritt
  (\ref{item:char-stonesch-i})~$\Rightarrow$~(\ref{item:char-stonesch-iv})
  in Satz~\ref{satz:char-stonesch} nochmals durch und erhalten, dass
  $f$ mit der Funktion $f_{0}=\lim_{n\to\infty}f_{n}$ überall außer
  möglicherweise auf
  \begin{displaymath}
    M=\bigcup_{n=1}^{\infty}\bigcup_{k=1}^{2^{n}}F\left(\frac{k}{2^{n}}\right)\setminus G\left(\frac{k}{2^{n}}\right)
  \end{displaymath}
  übereinstimmt, wenn die $f_{n}$ wie in
  \eqref{eq:bew-char-stonesch-i} definiert sind. Die Projektionen in
  $C(K)$ sind genau die stetigen Indikatorfunktionen, also
  Indikatorfunktionen von Mengen, die clopen sind. Somit sind die
  Funktionen $f_{n}$ Linearkombinationen von Projektionen und $f_{0}$
  ist in der abgeschlossenen linearen Hülle der Projektionen
  enthalten. Mit der Information, dass $f$ nicht nur von unten
  halbstetig sondern sogar stetig ist, können wir $f=f_{0}$ zeigen,
  was den Beweis der ersten Aussage abschließt. Wegen der Stetigkeit
  sind die Urbilder $f^{-1}(-\infty,t)$ für $t\in\R$ offen und daher
  in $G(t)$ enthalten. Weiters gilt
  $F(t)\subseteq f^{-1}(-\infty,t+\epsilon)\subseteq G(t+\epsilon)$
  für beliebiges $\epsilon>0$, sodass wir $F(t)\subseteq G(s)$ für
  $t<s$ erhalten. Für
  $x\in F\big(\frac{k}{2^{n}}\big)\setminus
  G\big(\frac{k}{2^{n}}\big)$ gilt daher einerseits
  $f(x)=\frac{k}{2^{n}}$ und andererseits
  $x\in G\big(\frac{k+1}{2^{n}}\big)\setminus
  G\big(\frac{k}{2^{n}}\big)$, also
  $f_{n}(x)=\frac{k+1}{2^{n}}$. Somit gilt auch auf $M$ die
  Abschätzung $\abs{f_{n}(x)-f(x)}\leq\frac{1}{2^{n}}$ und infolge wie
  behauptet $f=f_{0}$.

  Für die Zusatzaussage sei bemerkt, dass wir im Falle $f\geq 0$ nur
  um eine positive Konstante skalieren müssen, um $0\leq f\leq 1$ zu
  erhalten. Somit genügt es zu zeigen, dass diese spezielleren
  Funktionen als Grenzwert von Funktionen der im Satz angegebenen Form
  dargestellt werden können. Betrachtet man die im obigen Beweis für
  diese Situation konstruierten Funktionen $f_{n}$ aus
  \eqref{eq:bew-char-stonesch-i} und löst die Teleskopsumme auf, so
  ergibt sich
  \begin{displaymath}
    f_{n}=-\frac{1}{2^{n}}\mathds{1}_{G(0)}-\sum_{k=1}^{2^{n}-1}\frac{1}{2^{n}}\mathds{1}_{G(k/2^{n})}+\frac{2^{n}}{2^{n}}\mathds{1}_{G(2^{n}/2^{n})}=-\sum_{k=0}^{2^{n}-1}\frac{1}{2^{n}}\mathds{1}_{G(k/2^{n})}+1,
  \end{displaymath}
  wobei die zweite Gleichheit daraus folgt, dass wegen $f\leq 1$ schon
  $G(1)=K$ gilt. Betrachten wir die Komplemente $\compl{G(k/2^{n})}$,
  die ebenfalls clopen sind, so erhalten wir aus
  $\mathds{1}_{G(k/2^{n})}=1-\mathds{1}_{\compl{G(k/2^{n})}}$ die
  Darstellung
  \begin{displaymath}
    f_{n}=-\sum_{k=0}^{2^{n}-1}\frac{1}{2^{n}}\left(1-\mathds{1}_{\compl{G(k/2^{n})}}\right)+1=\sum_{k=0}^{2^{n}-1}\frac{1}{2^{n}}\mathds{1}_{\compl{G(k/2^{n})}}.
  \end{displaymath}
  Folglich haben die Funktionen $f_{n}$ die gewünschte Form.
\end{proof}

\clearpage
\chapter{Abstrakte $C^{*}$-Algebren}
\label{cha:abstrakte-c-algebren}
Es ist eine bekannte Tatsache, dass einerseits $L_{b}(H)$ eine
$C^{*}$-Algebra ist und andererseits normabgeschlossene
$*$-Unteralgebren von $C^{*}$-Algebren wieder $C^{*}$-Algebren
sind. Damit sind $\norm{\cdot}$-abgeschlossene $*$-Unteralgebren von
$L_{b}(H)$ schon $C^{*}$-Algebren. Das Hauptziel dieses Kapitels ist
mit dem Satz von Gelfand-Naimark die stärkstmögliche Umkehrung dieses
Sachverhalts. Dieses Resultat ist ohne Zweifel eines der wichtigsten
Ergebnisse der Theorie von $C^{*}$-Algebren.

\section{Positive Funktionale}
\label{sec:positive-funktionale}
\begin{notation}
  Sei weiterhin $A$ eine $C^{*}$-Algebra mit Einheitskugel $S$,
  selbstadjungierten Elementen $A_{sa}$ und positiven Elementen
  $A^{+}$.
\end{notation}
Im letzten Kapitel waren die positiven Elemente von $A$ das
wesentliche Hilfsmittel. Hier betrachten wir die mit diesen Elementen
verträglichen linearen Funktionale.
\begin{definition}\label{def:pos-fkt}
  Ein lineares Funktional $\varphi:A\to\C$ heißt \emph{positiv}, wenn
  $\varphi(A^{+})\subseteq[0,+\infty)$. Falls zusätzlich
  $\norm{\varphi}=1$ gilt, so nennen wir $\varphi$ einen
  \emph{Zustand}. Die Menge aller Zustände bezeichnen wir mit
  $\mathcal{S}(A)$.
\end{definition}
Zunächst fassen wir einige einfache Eigenschaften zusammen.
\begin{lemma}\label{lem:eig-pos-fkt}
  Sei $\varphi$ ein positives Funktional auf $A$.
  \begin{enumerate}[label=(\roman*),ref=\roman*]
  \item\label{item:eig-pos-fkt-i} Das Funktional $\varphi$ ist
    beschränkt.
  \item\label{item:eig-pos-fkt-ii} Für beliebige $a,b\in A$ gilt
    $\varphi(b^{*}a)=\overline{\varphi(a^{*}b)}$.
  \item\label{item:eig-pos-fkt-iii} (Cauchy-Schwarz'sche Ungleichung)
    Sind $a,b\in A$ gegeben, so folgt
    \begin{displaymath}
      \abs{\varphi(b^{*}a)}\leq\varphi(a^{*}a)^{1/2}\varphi(b^{*}b)^{1/2}.
    \end{displaymath}
  \item\label{item:eig-pos-fkt-iv} Die Funktion
    $a\mapsto\varphi(a^{*}a)^{1/2}$ ist eine Seminorm auf $A$.
  \end{enumerate}
\end{lemma}
\begin{proof}
  Wir zeigen zunächst
  $C:=\sup_{a\in A^{+},\norm{a}\leq 1}\varphi(a)<+\infty$.  Wäre diese
  Behauptung falsch, gäbe es zu jedem $n\in\N$ ein $a_{n}\in A^{+}$
  mit $\norm{a_{n}}\leq 1$ und $\varphi(a_{n})\geq 2^{n}$. Wegen der
  Unterpunkte (\ref{item:eig-pos-el-i}) und (\ref{item:eig-pos-el-iv})
  von Satz~\ref{satz:eig-pos-el} gilt für die absolut konvergente
  Reihe $a:=\sum_{n=1}^{\infty}\frac{1}{2^{n}}a_{n}$ und jedes
  $N\in\N$
  \begin{displaymath}
    a-\sum_{n=1}^{N}\frac{1}{2^{n}}a_{n}=\lim_{\stackrel{M\to\infty,}{M\geq
        N+1}}\,\,\sum_{n=N+1}^{M}\frac{1}{2^{n}}a_{n}\geq 0.
  \end{displaymath}
  Daraus folgt
  \begin{displaymath}
    \varphi(a)=\underbrace{\varphi\left(a-\sum_{n=1}^{N}\frac{1}{2^{n}}a_{n}\right)}_{\geq
      0, \,\text{da $\varphi$ positiv ist}}+\varphi\left(\sum_{n=1}^{N}\frac{1}{2^{n}}a_{n}\right)\geq\sum_{n=1}^{N}\underbrace{\varphi\left(\frac{1}{2^{n}}a_{n}\right)}_{\geq
      1}\geq N,
  \end{displaymath}
  was wegen der Beliebigkeit von $N$ einen Widerspruch darstellt. Für
  beliebiges $a\in S$ schreiben wir gemäß der
  Definitionen~\ref{def:re-im-cstar} und \ref{def:pos-neg-cstar}
  \begin{displaymath}
    a=(\re(a)^{+}-\re(a)^{-})+i(\im(a)^{+}-\im(a)^{-}).
  \end{displaymath}
  Wegen $\norm{\re(a)^{\pm}},\norm{\im(a)^{\pm}}\leq\norm{a}\leq 1$
  erhalten wir die Abschätzung
  \begin{displaymath}
    \abs{\varphi(a)}\leq\varphi(\re(a)^{+})+\varphi(\re(a)^{-})+\varphi(\im(a)^{+})+\varphi(\im(a)^{-})\leq
    4C.
  \end{displaymath}
  Daraus folgt $\norm{\varphi}\leq 4C$ und (\ref{item:eig-pos-fkt-i})
  ist gezeigt.

  Die Abbildung $\sigma:(a,b)\mapsto\varphi(b^{*}a)$ ist
  offensichtlich eine Sesquilinearform, die wegen der vorausgesetzten
  Positivität von $\varphi$ auch positiv semidefinit ist. Aus der
  Polarisationsformel
  \begin{displaymath}
    \sigma(a,b)=\frac{1}{4}\sum_{j=0}^{3}i^{j}\sigma(a+i^{j}b,a+i^{j}b)
  \end{displaymath}
  folgt die Hermitizität von $\sigma$, da die Ausdrücke
  $\sigma(a+i^{j}b,a+i^{j}b)$ nichtnegativ und insbesondere reell
  sind. Anders formuliert gilt $\sigma(a,b)=\overline{\sigma(b,a)}$
  und somit (\ref{item:eig-pos-fkt-ii}).

  Die Eigenschaften (\ref{item:eig-pos-fkt-iii}) und
  (\ref{item:eig-pos-fkt-iv}) gelten allgemein für hermitesche und
  positiv semidefinite Sesquilinearformen.
\end{proof}
Wir wollen ein im letzten Beweis implizit aufgetretenes Argument
explizit hervorheben: Ist $\varphi$ positiv, so folgt aus $a\leq b$
schon $\varphi(a)\leq\varphi(b)$, denn $\varphi$ bildet das Element
$b-a\geq 0$ auf eine nichtnegative reelle Zahl ab.

Wenn $A$ ein Einselement enthält, kann man
Lemma~\ref{lem:eig-pos-fkt}(\ref{item:eig-pos-fkt-ii}) prägnanter
formulieren, indem man $b=1$ setzt. Mithilfe der Konstruktion des
approximativen Einselements können wir dies aber auch für beliebige
$C^{*}$-Algebren erreichen.
\begin{lemma}\label{lem:pos-fkt-konj}
  Ist $\varphi$ ein positives Funktional auf $A$ und $a\in A$, so gilt
  $\varphi(a^{*})=\overline{\varphi(a)}$. Insbesondere ist
  $\varphi(a)$ für selbstadjungiertes $a$ reell.
\end{lemma}
\begin{proof}
  Sei $(u_{i})_{i\in I}$ ein approximatives Einselement. Dann gilt
  \begin{displaymath}
    \varphi(a^{*})=\lim_{i\in I}\varphi(a^{*}u_{i})=\lim_{i\in I}\overline{\varphi(u_{i}^{*}a)}=\lim_{i\in I}\overline{\varphi(u_{i}a)}=\overline{\varphi(a)},
  \end{displaymath}
  wobei die erste und letzte Gleichung aus der Beschränktheit von
  $\varphi$ und die zweite Gleichung aus
  Lemma~\ref{lem:eig-pos-fkt}(\ref{item:eig-pos-fkt-ii}) folgt.
\end{proof}

Auch die folgende Charakterisierung positiver Funktionale baut auf
approximativen Einselementen auf.
\begin{satz}\label{satz:char-pos-fkt}
  Für ein beschränktes lineares Funktional $\varphi$ auf $A$ sind die
  folgenden Aussagen äquivalent:
  \begin{enumerate}[label=(\roman*),ref=\roman*]
  \item\label{item:char-pos-fkt-i} $\varphi$ ist positiv.
  \item\label{item:char-pos-fkt-ii} Für \emph{jedes} approximative
    Einselement $(u_{i})_{i\in I}$ gilt
    $\norm{\varphi}=\lim_{i\in I}\varphi(u_{i})$.
  \item\label{item:char-pos-fkt-iii} Für \emph{ein} approximatives
    Einselement $(u_{i})_{i\in I}$ gilt
    $\norm{\varphi}=\lim_{i\in I}\varphi(u_{i})$.
  \end{enumerate}
\end{satz}
\begin{proof}
  Für $\varphi=0$ gelten trivialerweise alle Bedingungen, sodass wir
  durch Normieren $\norm{\varphi}=1$ annehmen können. Der Beweis
  verläuft zyklisch:
  
  (\ref{item:char-pos-fkt-i})~$\Rightarrow$~(\ref{item:char-pos-fkt-ii}):
  Sei $\varphi$ positiv und $(u_{i})_{i\in I}$ ein approximatives
  Einselement. Das Netz $(\varphi(u_{i}))_{i\in I}$ nichtnegativer
  reeller Zahlen ist monoton wachsend, durch $1$ beschränkt und daher
  gegen sein Supremum konvergent, wobei
  $\lim_{i\in I}\varphi(u_{i})\leq 1=\norm{\varphi}$. Für $a\in A$
  gilt nach der Cauchy-Schwarz'schen Ungleichung,
  Lemma~\ref{lem:eig-pos-fkt}(\ref{item:eig-pos-fkt-iii}),
  \begin{displaymath}
    \abs{\varphi(u_{i}a)}=\abs{\varphi(u_{i}^{*}a)}\leq\varphi(u_{i}^{2})^{1/2}\varphi(a^{*}a)^{1/2}.
  \end{displaymath}
  In $\tilde{A}$ (bzw. in $A$, wenn $A$ ein Einselement enthält) gilt
  $u_{i}\leq\norm{u_{i}}\leq 1$ aufgrund von
  Lemma~\ref{lem:eig-ho}(\ref{item:eig-ho-iii}), womit wegen
  Lemma~\ref{lem:eig-ho}(\ref{item:eig-ho-ii})
  \begin{displaymath}
    u_{i}^{2}=\left(u_{i}^{1/2}\right)^{*}u_{i}u_{i}^{1/2}\leq\left(u_{i}^{1/2}\right)^{*}1u_{i}^{1/2}=u_{i}
  \end{displaymath}
  folgt. Nach Bemerkung~\ref{bem:ho} gilt die Ungleichung
  $u_{i}^{2}\leq u_{i}$ auch in $A$. Daraus folgt
  $\varphi(u_{i}^{2})^{1/2}\leq\varphi(u_{i})^{1/2}$, sodass sich
  \begin{displaymath}
    \abs{\varphi(u_{i}a)}\leq\varphi(u_{i})^{1/2}\varphi(a^{*}a)^{1/2}\leq\left(\sup_{j\in
        I}\varphi(u_{j})\right)^{1/2}\underbrace{\norm{\varphi}^{1/2}}_{=1}\norm{a^{*}a}^{1/2}=\left(\lim_{j\in
        I}\varphi(u_{j})\right)^{1/2}\norm{a}
  \end{displaymath}
  ergibt. Bilden wir in dieser Ungleichung den Grenzwert $i\in I$, so
  schließen wir auf
  $\abs{\varphi(a)}\leq (\lim_{j\in I}\varphi(u_{j}))^{1/2}\norm{a}$
  und erhalten wegen der Beliebigkeit von $a$ die Abschätzung
  $1=\norm{\varphi}\leq(\lim_{j\in I}\varphi(u_{j}))^{1/2}$. Somit ist
  (\ref{item:char-pos-fkt-ii}) gezeigt.

  \vspace{\baselineskip}
  (\ref{item:char-pos-fkt-ii})~$\Rightarrow$~(\ref{item:char-pos-fkt-iii})
  folgt unmittelbar, wenn man beachtet, dass nach
  Satz~\ref{satz:approx-eins} überhaupt approximative Einselemente
  existieren.

  \vspace{\baselineskip}
  (\ref{item:char-pos-fkt-iii})~$\Rightarrow$~(\ref{item:char-pos-fkt-i}):
  Zunächst zeigen wir, dass $\varphi(a)$ für $a\in A_{sa}$ reell
  ist. Angenommen, es gilt $\varphi(a)=\alpha+i\beta$ mit
  $\beta\neq 0$, wobei wir durch etwaiges Übergehen zu $-a$ sicher
  $\beta<0$ und durch Normieren auch $\norm{a}\leq 1$ annehmen
  können. Für jedes $n\in\N$ und $j\in I$ berechnen wir
  \begin{align*}
    \norm{a-inu_{j}}^{2}&=\norm{(a-inu_{j})^{*}(a-inu_{j})}=\norm{a^{2}+n^{2}u_{j}^{2}-in(au_{j}-u_{j}a)}\\
                        &\leq\norm{a^{2}+n^{2}u_{j}^{2}}+n\norm{au_{j}-u_{j}a}\leq
                          1+n^{2}+n\norm{au_{j}-u_{j}a},
  \end{align*}
  sodass wir wegen $\norm{\varphi}=1$ auf
  $\abs{\varphi(a-inu_{j})}^{2}\leq 1+n^{2}+n\norm{au_{j}-u_{j}a}$
  schließen. Bilden wir den Grenzwert $j\in I$, so folgt
  $\varphi(a-inu_{j})\to \varphi(a)-in=\alpha+i(\beta-n)$ wegen der
  Voraussetzung (\ref{item:char-pos-fkt-iii}). Nach Definition eines
  approximativen Einselements gilt weiters
  $\norm{au_{j}-u_{j}a}\to 0$, sodass wir aus obiger Ungleichung
  \begin{displaymath}
    \alpha^{2}+\beta^{2}-2n\beta+n^{2}=\abs{\alpha+i(\beta-n)}^{2}\leq 1+n^{2}
  \end{displaymath}
  bzw. $-2n\beta\leq 1-\alpha^{2}-\beta^{2}$ erhalten. Wegen $\beta<0$
  ist die linke Seite in $n$ nach oben unbeschränkt, was zu einem
  Widerspruch führt. Also gilt tatsächlich $\varphi(a)\in\R$. Für die
  Positivität von $\varphi$ sei $a$ positiv, wobei wir
  $\norm{a}\leq 1$ annehmen können. Da auch $u_{j}-a$ selbstadjungiert
  ist, folgt
  $\varphi(u_{j}-a)\in\R$. Lemma~\ref{lem:eig-ho}(\ref{item:eig-ho-v})
  liefert $\norm{u_{j}-a}\leq 1$, womit wir
  $\varphi(u_{j}-a)\leq\abs{\varphi(u_{j}-a)}\leq\norm{\varphi}$ und
  nach dem Grenzübergang $j\in I$ schließlich
  $\norm{\varphi}-\varphi(a)\leq\norm{\varphi}$ bzw.
  $\varphi(a)\geq 0$ erhalten.
\end{proof}

\begin{korollar}\label{kor:char-pos-fkt}
  Hat $A$ ein Einselement, so ist ein beschränktes lineares Funktional
  $\varphi$ genau dann positiv, wenn $\varphi(1)=\norm{\varphi}$.
\end{korollar}
\begin{proof}
  Das konstante Netz $u_{i}:=1$, $i\in I$ ist für eine beliebige
  gerichtete Menge $I$ ein approximatives Einselement, sodass die
  Aussage aus Satz~\ref{satz:char-pos-fkt} folgt.
\end{proof}

Das nächste Lemma enthält eine der Cauchy-Schwarz'schen Ungleichung
stark ähnelnde Bedingung, die wir später benötigen werden.
\begin{lemma}\label{lem:kern-pos-fkt-unglg}
  Sei $\varphi:A\to\C$ ein positives Funktional.
  \begin{enumerate}[label=(\roman*),ref=\roman*]
  \item\label{item:kern-pos-fkt-unglg-i} Ist $c\in A$, so gilt
    $\varphi(c^{*}c)=0$ genau dann, wenn $\varphi(dc)=0$ für alle
    $d\in A$ ist.
  \item\label{item:kern-pos-fkt-unglg-ii} Für beliebige $a,b\in A$
    gilt $0\leq\varphi(b^{*}a^{*}ab)\leq\norm{a^{*}a}\varphi(b^{*}b)$.
  \end{enumerate}
\end{lemma}
\begin{proof}
  \hspace{0mm}
  \begin{enumerate}[label=(\roman*),ref=\roman*]
  \item Wegen der Cauchy-Schwarz'schen Ungleichung
    $\abs{\varphi(b^{*}a)}\leq\varphi(a^{*}a)^{1/2}\varphi(b^{*}b)^{1/2}$
    mit $a=c$ und $b=d^{*}$ folgt aus $\varphi(c^{*}c)=0$ schon
    $\varphi(dc)=0$ für alle $d\in A$. Die Umkehrung ist trivial.
  \item Die erste Ungleichung folgt aus
    $b^{*}a^{*}ab=(ab)^{*}(ab)\geq 0$. Die zweite Ungleichung ist nach
    (\ref{item:kern-pos-fkt-unglg-i}) angewandt auf $c=b$ und
    $d=b^{*}a^{*}a$ trivial, wenn $\varphi(b^{*}b)=0$ ist. Ansonsten
    ist
    \begin{displaymath}
      \psi:
      \begin{cases}
        \hfill A&\to\C\\
        \hfill c&\mapsto \frac{\varphi(b^{*}cb)}{\varphi(b^{*}b)}
      \end{cases}
    \end{displaymath}
    ein wohldefiniertes lineares Funktional, das offensichtlich
    positiv und nach
    Lemma~\ref{lem:eig-pos-fkt}(\ref{item:eig-pos-fkt-i}) beschränkt
    ist. Wählt man ein approximatives Einselement $(u_{i})_{i\in I}$,
    so liefert Satz~\ref{satz:char-pos-fkt}
    \begin{displaymath}
      \norm{\psi}=\lim_{i\in I}\psi(u_{i})=\lim_{i\in
        I}\frac{\varphi(b^{*}u_{i}b)}{\varphi(b^{*}b)}=\frac{\varphi(b^{*}b)}{\varphi(b^{*}b)}=1.
    \end{displaymath}
    Daraus folgt die gewünschte Ungleichung
    $\varphi(b^{*}a^{*}ab)=\psi(a^{*}a)\varphi(b^{*}b)\leq\norm{a^{*}a}\varphi(b^{*}b)$.
  \end{enumerate}
\end{proof}

Der folgende Satz zeigt, dass es ausreichend viele positive
Funktionale gibt, um sie zu Strukturanalysen heranziehen zu können.
\begin{satz}\label{satz:pos-fkt-trennend}
  Sei $a\neq 0$ ein normales Element von $A$. Dann gibt es einen
  Zustand $\varphi$ auf $A$ mit $\abs{\varphi(a)}=\norm{a}\neq
  0$.
  Insbesondere existieren auf $A$ Zustände, wenn $A$ nichttrivial ist.
\end{satz}
\begin{proof}
  Wir betrachten die von $a$ in $\tilde{A}$ erzeugte
  $C^{*}$-Unteralgebra mit Eins, $B:=C^{*}_{\tilde{A}}(a,1)$. Dann ist
  $B$ kommutativ und der Gelfandraum $M$ von $B$ kompakt. Da die
  Gelfandtransformierte $\hat{a}$ stetig ist, folgt
  \begin{displaymath}
    \norm{a}=\norm[\infty]{\hat{a}}=\sup_{m\in
      M}\abs{\hat{a}(m)}=\max_{m\in M}\abs{\hat{a}(m)}.
  \end{displaymath}
  Es gibt also ein multiplikatives Funktional $m_{0}$ mit
  $\norm{a}=\abs{\hat{a}(m_{0})}=\abs{m_{0}(a)}$. Nach dem Satz von
  Hahn-Banach existiert eine lineare Fortsetzung $f$ auf $\tilde{A}$
  mit $\norm{f}=\norm{m_{0}}=1$. Wegen $f(1)=m_{0}(1)=1=\norm{f}$
  folgt aus Korollar~\ref{kor:char-pos-fkt} die Positivität von
  $f$. Bezeichnen wir die Einschränkung von $f$ auf $A$ mit $\varphi$,
  so gilt $\norm{\varphi}\leq\norm{f}=1$ und
  $\norm{a}=\abs{m_{0}(a)}=\abs{\varphi(a)}\leq\norm{\varphi}\norm{a}$,
  sodass wir insgesamt $\norm{\varphi}=1$ erhalten. Das Funktional
  $\varphi$ ist positiv, da ein positives Element von $A$ nach
  Bemerkung~\ref{bem:ho}(\ref{item:ho-iii}) auch in $\tilde{A}$
  positiv ist. Somit ist $\varphi$ ein normiertes, positives
  Funktional, also ein Zustand.

  Für die Zusatzbehauptung sei $a\in A$, $a\neq 0$. Dann ist $a^{*}a$
  insbesondere normal mit $\norm{a^{*}a}=\norm{a}^{2}\neq 0$. Nach dem
  oben Bewiesenen gibt es einen Zustand $\varphi$ mit
  $\abs{\varphi(a^{*}a)}=\norm{a^{*}a}$.
\end{proof}

\section{Der Satz von Gelfand-Naimark}
\label{sec:gelfand-naimark}
Um das Wechselspiel zwischen $C^{*}$-Algebren und Räumen der Form
$L_{b}(H)$ analysieren zu kön\-nen, führen wir einen Begriff ein.
\begin{definition}\label{def:darstellung}
  Ein $*$-Algebrenhomomorphismus $\Phi:A\to L_{b}(H)$ heißt
  \emph{Darstellung}. Ist $\Phi$ injektiv, so nennen wir die
  Darstellung \emph{treu}.
\end{definition}

Zunächst zeigen wir (unter anderem), dass eine treue Darstellung
$\Phi$ ein isometrischer $*$-Al\-ge\-bren\-isomorphismus auf
$\ran\Phi$ ist. Man beachte, dass dies nicht durch Anwendung der
klassischen Tatsache~\cite[Satz~1.5.4]{kaltenb:fana2} folgt, dass
Isomorphismen \emph{zwischen} $C^{*}$-Algebren automatisch isometrisch
sind, denn es wäre zunächst auf andere Weise zu zeigen, dass
$\ran\Phi$ für sich eine $C^{*}$-Algebra ist.

Ein korrektes Argument macht Gebrauch vom
Funktionalkalkül in $C^{*}$-Algebren mit Einselement.
\begin{bemerkung}\label{bem:fktkalkuel}
  Wir erinnern an die Definition und einige Eigenschaften des
  Funktionalkalküls; vgl.~\cite[Satz~1.5.13]{kaltenb:fana2}. Ist $A$
  eine kommutative $C^{*}$-Algebra mit Einselement, $a\in A$ und $f$
  eine stetige Funktion auf $\sigma(a)$, so ist $f(a)$ das
  \emph{eindeutige} Element von $A$ mit $m(f(a))=f(m(a))$ für alle
  multiplikativen Funktionale $m$ auf $A$. Ist $A$ nicht kommutativ,
  kann man $f(a)$ immer noch definieren, wenn $a$ normal ist, indem
  man zu der von $a$ in $A$ erzeugten $C^{*}$-Unteralgebra mit
  Einselement übergeht, also zu $C^{*}(a,1)$. In dieser kommutativen
  $C^{*}$-Algebra liegt $f(a)$ für jede Funktion $f\in
  C(\sigma(a))$. Insbesondere kommutieren alle $f(a)$
  miteinander. Außerdem ist die Funktion $f\mapsto f(a)$ ein
  $C^{*}$-Algebrenisomorphismus von $C(\sigma(a))$ nach $C^{*}(a,1)$
  und daher nach dem bekannten Resultat isometrisch, also gilt
  $\norm{f(a)}=\norm[\infty]{f}$. Für ein Polynom
  $p(x)=\sum_{i=0}^{n}c_{i}x^{i}$ stimmt dabei $p(a)$ mit
  $\sum_{i=0}^{n}c_{i}a^{i}$ überein.

  Für den späteren Gebrauch behandeln wir an dieser Stelle noch den
  Fall, dass $A$ \emph{kein} Einselement enthält. In dieser Situation
  ist der Funktionalkalkül über die $C^{*}$-Algebra $\tilde{A}$
  definiert: Für ein normales Element $a\in A$ und eine stetige
  Funktion $f$ auf $\sigma_{A}(a)=\sigma_{\tilde{A}}(a)$ ist $a$ auch
  in $\tilde{A}$ normal. Somit können wir den obigen Funktionalkalkül
  in $\tilde{A}$ bzw. genauer in $C_{\tilde{A}}^{*}(a,1)$ verwenden,
  um das Element $f(a)\in\tilde{A}$ zu erhalten. Es stellt sich die
  Frage, wann $f(a)$ in $A$ enthalten ist. Als Motivation betrachten
  wir ein Polynom $p(x)=\sum_{i=0}^{n}c_{i}x^{i}$. Aus\footnote{Siehe
    auch den Beweis von Satz~\ref{satz:quwurzel}.}
  \begin{displaymath}
    p(a)=\sum_{i=0}^{n}c_{i}a^{i}=\left(\sum_{i=1}^{n}c_{i}a^{i},c_{0}\right)
  \end{displaymath}
  folgt, dass $p(a)\in A$ zum Verschwinden des konstanten
  Koeffizienten $c_{0}$ äquivalent ist, also zu $p(0)=0$. Dieselbe
  Aussage stimmt für allgemeine Funktionen $f$: Wir verwenden das
  multiplikative Funktional $m_{0}':\tilde{A}\to\C$,
  $m_{0}'((b,\lambda)):=\lambda$. Ein Element
  $(b,\lambda)\in\tilde{A}$ ist genau dann in $A$ enthalten, wenn
  $m_{0}'((b,\lambda))=0$ gilt. Wegen
  \begin{displaymath}
    m_{0}'(f(a))=f(m_{0}'(a))=f(0)
  \end{displaymath}
  ist also $f(a)\in A$ zu $f(0)=0$ äquivalent.
\end{bemerkung}
\begin{satz}\label{satz:inj-homo-iso}
  Ist $B$ eine weitere $C^{*}$-Algebra, so ist ein injektiver
  $*$-Algebrenhomomorphismus $\Phi:A\to B$ isometrisch. Insbesondere
  ist $\ran\Phi$ eine $C^{*}$-Unteralgebra von $B$ und
  $\Phi:A\to \ran\Phi$ ein isometrischer $*$-Algebrenisomorphismus.
\end{satz}
\begin{proof}
  Als Erstes sei angenommen, dass $A$ und $B$ kommutative
  $C^{*}$-Algebren mit Einselement sind, wobei wir noch
  $\Phi(1_{A})=1_{B}$ fordern. Für $\lambda\notin\sigma_{A}(a)$
  identifiziert man $\Phi((a-\lambda 1_{A})^{-1})$ leicht als Inverse
  von $\Phi(a)-\lambda 1_{B}$. Daraus folgt
  $\lambda\notin\sigma_{B}(\Phi(a))$ und somit
  \begin{displaymath}
    \sigma_{B}(\Phi(a))\subseteq\sigma_{A}(a).
  \end{displaymath}
  Unter der (bisher noch nicht verwendeten) Voraussetzung der
  Injektivität gilt hier sogar Gleichheit. Wäre dies nämlich nicht der
  Fall und $\lambda\in\sigma_{A}(a)\setminus\sigma_{B}(\Phi(a))$, so
  gäbe es nach dem Lemma von Urysohn eine stetige Funktion
  $f:\sigma_{A}(a)\to [0,1]$ mit $f(\lambda)=1$ und
  $f\big(\sigma_{B}(\Phi(a))\big)\subseteq\{0\}$. Ist $m$ ein
  multiplikatives Funktional auf $B$, so ist $m\circ\Phi$ ein solches
  auf $A$, sodass wir mit dem Funktionalkalkül in $A$
  \begin{displaymath}
    m(\Phi(f(a)))=(m\circ\Phi)(f(a))=f\big((m\circ\Phi)(a)\big)=f\big(\underbrace{m(\Phi(a))}_{\in\sigma_{B}(\Phi(a))}\big)=0
  \end{displaymath}
  erhalten. Damit ist die Gelfandtransformierte von $\Phi(f(a))$ die
  Nullfunktion, was uns wegen der Injektivität von $\hat{.}$ und
  $\Phi$ auf $f(a)=0$ schließen lässt. Wir erhalten den Widerspruch
  $0=\norm{f(a)}=\norm[\infty]{f}=1$ und haben somit
  $\sigma_{B}(\Phi(a))=\sigma_{A}(a)$ und infolge
  $r_{B}(\Phi(a))=r_{A}(a)$ gezeigt. Diese Tatsache angewandt auf
  $a^{*}a$ für beliebiges $a\in A$ liefert die Isometrieeigenschaft:
  \begin{displaymath}
    \norm{a}=\norm{a^{*}a}^{1/2}=r_{A}(a^{*}a)^{1/2}=r_{B}(\Phi(a^{*}a))^{1/2}=r_{B}(\Phi(a)^{*}\Phi(a))^{1/2}=\norm{\Phi(a)^{*}\Phi(a)}^{1/2}=\norm{\Phi(a)}
  \end{displaymath}
  
  Jetzt nehmen wir weiterhin an, dass $A$ und $B$ kommutativ sind,
  lassen aber die Forderung nach der Existenz eines Einselements
  weg. Wir betrachten die $C^{*}$-Algebren mit Einselement $\tilde{A}$
  und $\tilde{B}$ und setzen $\Phi$ gemäß
  $(a,\lambda)\mapsto (\Phi(a),\lambda)$ zu einem $*$-Homomorphismus
  $\widetilde{\Phi}$ fort. Diese Fortsetzung ist ebenfalls injektiv
  und bildet $1_{\tilde{A}}=(0_{A},1)$ auf $1_{\tilde{B}}=(0_{B},1)$
  ab. Nach dem ersten Teil ist $\widetilde{\Phi}$ isometrisch, also
  gilt
  \begin{displaymath}
    \norm{a}=\norm[\tilde{A}]{(a,0)}=\norm[\tilde{B}]{\widetilde{\Phi}((a,0))}=\norm[\tilde{B}]{(\Phi(a),0)}=\norm{\Phi(a)}.
  \end{displaymath}

  Um den allgemeinen Fall zu beweisen, beschränken wir uns zunächst
  auf $a\in A_{sa}$. Klarerweise ist dann $\Phi(a)\in B_{sa}$. Wir
  betrachten $C^{*}_{A}(a)$ sowie $C^{*}_{B}(\Phi(a))$. Das Urbild
  $\Phi^{-1}\big(C^{*}_{B}(\Phi(a))\big)$ ist eine $a$ enthaltende
  $*$-Unteralgebra von $A$. Da $\Phi$ als $*$-Homomorphismus zwischen
  $C^{*}$-Algebren beschränkt ist, ist
  $\Phi^{-1}\big(C^{*}_{B}(\Phi(a))\big)$ auch abgeschlossen, also
  eine $C^{*}$-Unteralgebra von $A$. Daraus folgt
  $C^{*}_{A}(a)\subseteq\Phi^{-1}\big(C^{*}_{B}(\Phi(a))\big)$ bzw.
  \begin{displaymath}
    \Phi(C^{*}_{A}(a))\subseteq C^{*}_{B}(\Phi(b)).
  \end{displaymath}
  Wir können also den offenbar injektiven $*$-Homomorphismus
  $\Phi|_{C^{*}_{A}(a)}:C^{*}_{A}(a)\to C^{*}_{B}(\Phi(b))$ zwischen
  kommutativen $C^{*}$-Algebren betrachten, der nach dem schon
  Bewiesenen isometrisch ist. Wir erhalten
  $\norm{a}=\norm{\Phi|_{C^{*}_{A}(a)}(a)}=\norm{\Phi(a)}$. Ist
  schließlich $a\in A$ beliebig, so ergibt sich
  \begin{displaymath}
    \norm{a}=\norm{a^{*}a}^{1/2}=\norm{\Phi(a^{*}a)}^{1/2}=\norm{\Phi(a)^{*}\Phi(a)}^{1/2}=\norm{\Phi(a)}.
  \end{displaymath}

  Nun ist die Hauptarbeit getan, denn als isometrisches Bild des
  vollständigen Raums $A$ ist auch $\ran\Phi$ vollständig und daher
  abgeschlossen. Dass $\Phi:A\to\ran\Phi$ ein isometrischer
  $*$-Algebrenisomorphismus ist, ist nur eine Umformulierung bereits
  bewiesener Tatsachen.
\end{proof}
Um die eingangs behauptete isometrische $*$-Algebrenisomorphie von $A$
und einer Unteralgebra von $L_{b}(H)$ zu zeigen, reicht es nach
Satz~\ref{satz:inj-homo-iso}, eine treue Darstellung zu konstruieren.

Zunächst induziert jedes positive Funktional $\varphi$ auf $A$ eine
Darstellung: Definiert man $N_{\varphi}$ als den isotropen Anteil der
Sesquilinearform $(a,b)\mapsto\varphi(b^{*}a)$, also
$N_{\varphi}=\set{a\in A}{\varphi(a^{*}a)=0}$, so folgt aus
Lemma~\ref{lem:kern-pos-fkt-unglg}(\ref{item:kern-pos-fkt-unglg-i})
\begin{equation}\label{eq:gns-darst-i}
  N_{\varphi}=\set{a\in A}{\varphi(b^{*}a)=0\ \text{für alle}\ b\in A},  
\end{equation}
womit $N_{\varphi}$ ein Unterraum ist. Nach
Lemma~\ref{lem:kern-pos-fkt-unglg}(\ref{item:kern-pos-fkt-unglg-ii})
mit vertauschten Rollen von $a$ und $b$ folgt aus $a\in N_{\varphi}$
und $b\in A$ schon $ba\in N_{\varphi}$, sodass der isotrope Anteil ein
Linksideal ist. Damit ist
\begin{equation}\label{eq:gns-darst-ii}
  (.,.)_{A/N_{\varphi}}:\begin{cases}
    \hfill(A/N_{\varphi})^{2}&\to\C\\
    (a+N_{\varphi},b+N_{\varphi})&\mapsto\varphi(b^{*}a)
  \end{cases}
\end{equation}
eine wohldefinierte und positiv semidefinite Sesquilinearform. Nach
Konstruktion ist $(.,.)_{A/N_{\varphi}}$ sogar positiv definit. Somit
ist $A/N_{\varphi}$ ein Prähilbertraum, der bekanntlich eine
Hilbertraumver\-vollständigung $H_{\varphi}$ hat. Da $N_{\varphi}$ ein
Linksideal ist, ist für jedes $a\in A$ die Abbildung
\begin{displaymath}
  \Phi_{\varphi}(a):\begin{cases}
    \hfill A/N_{\varphi}&\to A/N_{\varphi}\\
    b+N_{\varphi}&\mapsto ab+N_{\varphi}
  \end{cases}
\end{displaymath}
wohldefiniert. Klarerweise ist $\Phi_{\varphi}(a)$ linear und es gilt
nach
Lemma~\ref{lem:kern-pos-fkt-unglg}(\ref{item:kern-pos-fkt-unglg-ii})
\begin{displaymath}
  \norm[A/N_{\varphi}]{ab+N_{\varphi}}^{2}=\varphi(b^{*}a^{*}ab)\leq\norm{a^{*}a}\varphi(b^{*}b)=\norm{a}^{2}\norm[N_{\varphi}]{b+N_{\varphi}}^{2}.
\end{displaymath}
Damit ist $\Phi_{\varphi}(a)$ auch beschränkt mit
$\norm{\Phi_{\varphi}(a)}\leq\norm{a}$. Folglich existiert eine
eindeutige beschränkte Fortsetzung auf $H_{\varphi}$, die wir
ebenfalls mit $\Phi_{\varphi}(a)$ bezeichnen. Es bleibt nachzuprüfen,
dass $\Phi_{\varphi}:A\to L_{b}(H_{\varphi})$ tatsächlich ein
$*$-Homomorphismus ist. Da man die jeweiligen Bedingungen wegen der
Stetigkeit immer nur für Elemente der dichten Teilmenge
$A/N_{\varphi}\subseteq H_{\varphi}$ prüfen muss, sind die Linearität
und die Multiplikativität klar. Die Verträglichkeit mit $.^{*}$ folgt
aus
\begin{align*}
  \big(c+N_{\varphi},\Phi_{\varphi}(a)(b+N_{\varphi})\big)_{H_{\varphi}}&=(c+N_{\varphi},ab+N_{\varphi})_{H_{\varphi}}=(c+N_{\varphi},ab+N_{\varphi})_{A/N_{\varphi}}\\
                                                                        &=\varphi((ab)^{*}c)=\varphi(b^{*}(a^{*}c))\\
                                                                        &=(a^{*}c+N_{\varphi},b+N_{\varphi})_{A/N_{\varphi}}=\big(\Phi_{\varphi}(a^{*})(c+N_{\varphi}),b+N_{\varphi}\big)_{A/N_{\varphi}}.
\end{align*}
Wir fassen unsere Ergebnisse zusammen.
\begin{lemma}\label{lem:gns-darst}
  Ist $\varphi$ ein positives Funktional auf $A$ und
  $N_{\varphi}=\set{a\in A}{\varphi(a^{*}a)=0}$ das zugeordnete
  Linksideal, so ist die in \eqref{eq:gns-darst-ii} definierte
  Sesquilinearform positiv definit. Bezeichnet man die
  Hilbertraumvervollständigung von $A/N_{\varphi}$ mit $H_{\varphi}$,
  dann ist $\Phi_{\varphi}:A\to L_{b}(H_{\varphi})$ eine Darstellung,
  wobei $\Phi_{\varphi}(a)$ die stetige Fortsetzung des Operators
  $b+N_{\varphi}\mapsto ab+N_{\varphi}$ auf $H_{\varphi}$ ist. Diese
  Darstellung wird die \emph{von $\varphi$ induzierte GNS\footnote{Die
      Abkürzung steht für Gelfand, Naimark und Segal.}-Darstellung}
  genannt.
\end{lemma}

Nun können wir die gesuchte treue Darstellung von $A$ angeben. Dazu
betrachten wir die GNS-Darstellung
$\Phi_{\varphi}:A\to L_{b}(H_{\varphi})$ für jeden Zustand
$\varphi\in\mathcal{S}(A)$ und die direkte Summe
$H:=\bigoplus_{\varphi\in\mathcal{S}(A)}H_{\varphi}$. Der Operator
\begin{equation}\label{eq:univ-darst}
  \Phi(a):
  \begin{cases}
    \hfill H&\to H\\
    (x_{\varphi})_{\varphi\in\mathcal{S}(A)}&\mapsto
    (\Phi_{\varphi}(a)x_{\varphi})_{\varphi\in\mathcal{S}(A)}
  \end{cases}
\end{equation}
bildet tatsächlich nach $H$ ab, da $\Phi_{\varphi}$ als
$*$-Homomorphismus beschränkt mit $\norm{\Phi_{\varphi}}\leq 1$ ist.
Folglich gilt
$\norm[H_{\varphi}]{\Phi_{\varphi}(a)x_{\varphi}}^{2}\leq
\norm{a}^{2}\norm[H_{\varphi}]{x_{\varphi}}^{2}$ und aus der
Quadratsummierbarkeit von
$\left(\norm[H_{\varphi}]{x_{\varphi}}\right)_{\varphi\in\mathcal{S}(A)}$
erhalten wir die Quadratsummierbarkeit von
$\left(\norm[H_{\varphi}]{\Phi_{\varphi}(a)x_{\varphi}}\right)_{\varphi\in\mathcal{S}(A)}$. Direktes
Nachrechnen zeigt, dass $\Phi:A\to L_{b}(H)$ eine Darstellung ist. Man
spricht von der \emph{universellen Darstellung}. Damit lässt sich der
Satz von Gelfand-Naimark in folgender Form formulieren:
\begin{satz}[Gelfand-Naimark]\label{satz:gelfand-naimark}
  Die universelle Darstellung $\Phi$ von $A$ ist treu. Insbesondere
  ist $A$ isometrisch isomorph zu einer $C^{*}$-Unteralgebra von
  $L_{b}(H)$ für einen geeigneten Hilbertraum $H$.
\end{satz}
\begin{proof}
  Wäre $\Phi$ nicht treu, so gäbe es ein $a\in A$ mit $a\neq 0$ und
  $\Phi(a)=0$, also $\Phi_{\varphi}(a)=0$ für alle
  $\varphi\in\mathcal{S}(A)$. Da mit $\Phi(a)$ klarerweise auch
  $\Phi(a^{*}a)=\Phi(a)^{*}\Phi(a)$ der Nulloperator ist, können wir
  durch etwaigen Übergang zu $a^{*}a$ annehmen, dass $a$ positiv
  ist. Nach Satz~\ref{satz:pos-fkt-trennend} gibt es einen Zustand
  $\varphi\in\mathcal{S}(A)$ mit $\varphi(a)=\norm{a}$. Gemäß
  Satz~\ref{satz:quwurzel} existiert $b:=a^{1/4}=(a^{1/2})^{1/2}$. Das
  Element $\Phi_{\varphi}(b)$ ist positiv, denn mit $c=b^{1/2}$ gilt
  $\Phi_{\varphi}(b)=\Phi_{\varphi}(c)^{*}\Phi_{\varphi}(c)$. Außerdem
  folgt
  $\left(\Phi_{\varphi}(b)^{2}\right)^{2}=\Phi_{\varphi}(b^{4})=\Phi_{\varphi}(a)=0$
  und wegen der Eindeutigkeit der positiven Quadratwurzel
  $\Phi_{\varphi}(b)^{2}=0$. Mit demselben Argument ergibt sich
  $\Phi_{\varphi}(b)=0$. Daraus folgt der Widerspruch
  \begin{align*}
    \norm{a}&=\varphi(a)=\varphi\left(b^{4}\right)=\varphi\left(\left(b^{2}\right)^{*}b^{2}\right)=\left(b^{2}+N_{\varphi},b^{2}+N_{\varphi}\right)_{H_{\varphi}}=\norm[H_{\varphi}]{b^{2}+N_{\varphi}}^{2}\\
            &=\norm[H_{\varphi}]{\Phi_{\varphi}(b)(b+N_{\varphi})}^{2}=0.
  \end{align*}
  Somit ist $\Phi$ treu.

  Der Rest des Satzes folgt aus Satz~\ref{satz:inj-homo-iso}.
\end{proof}
\begin{bemerkung}\label{bem:gns-sichtweise-1}
  Abschließend sei noch bemerkt, dass $A$ in die Konstruktion der
  universellen Darstellung auf zwei verschiedene Arten eingeht:
  Einerseits bauen Tupel von Äquivalenzklassen den Hilbertraum $H$
  auf, andererseits -- diese Sichtweise stellt für den Satz von
  Gelfand-Naimark das Hauptinteresse dar -- wirkt $A$ vermöge $\Phi$
  auf dem Hilbertraum $H$.
\end{bemerkung}

\clearpage
\chapter{Topologien auf $L_{b}(H)$ und Von-Neumann-Algebren}
\label{cha:optop-von-neumann-algebren}
Im Mittelpunkt dieses Kapitels steht der Raum $L_{b}(H)$ von
Operatoren auf einem gegebenen Hilbertraum $H$. Von großem Interesse
sollen insbesondere mehrere kanonische Topologien sein, mit denen man
$L_{b}(H)$ versehen kann. Darauf aufbauend stellt sich auch die Frage,
unter welchen Bedingungen die Abgeschlossenheit einer Teilmenge
bezüglich all dieser Topologien äquivalent ist.

\section{Operatortopologien}
\label{sec:operatortopologien}
Wir beginnen mit der Definition der Topologien.  Dazu sei daran
erinnert, dass eine \emph{separierende} Menge $P$ von
Seminormen\footnote{Das bedeutet nach Definition, dass $p(x)=0$ für
  alle $p\in P$ schon $x=0$ impliziert.} auf einem Vektorraum $X$ eine
lokalkonvexe Vektorraum-Topologie auf $X$ induziert. Die Konvergenz
eines Netzes $(x_{i})_{i\in I}$ gegen einen Vektor $x\in X$ lässt sich
dabei charakterisieren durch
\begin{equation}\label{eq:konv-snormen-ind-top}
  \lim_{i\in I}x_{i}=x\Leftrightarrow \lim_{i\in I}p(x_{i}-x)=0\ \text{für alle}\ p\in P.
\end{equation}
\begin{definition}\label{def:optop}
  \hspace{0mm}
  \begin{enumerate}[label=(\roman*),ref=\roman*]
  \item Die von $P_{s}:=\set{(T\mapsto\norm{Tx})}{x\in H}$ auf
    $L_{b}(H)$ induzierte Topologie wird \emph{starke
      Operatortopologie} genannt und mit $\TT_{s}$ bezeichnet.
  \item Die von $P_{w}:=\set{(T\mapsto \abs{(Tx,y)})}{x,y\in H}$ auf
    $L_{b}(H)$ induzierte Topologie wird \emph{schwache
      Operatortopologie} genannt und mit $\TT_{w}$ bezeichnet.
  \item Die von $P_{uw}:=\set{(T\mapsto\abs{\tr(ST)})}{S\in L^{1}(H)}$
    auf $L_{b}(H)$ induzierte Topologie wird \emph{ultraschwache
      Operatortopologie} genannt und mit $\TT_{uw}$ bezeichnet.
  \end{enumerate}
\end{definition}
Nach \eqref{eq:konv-snormen-ind-top} bedeutet $T_{i}\to T$ bezüglich
$\TT_{s}$ genau $(T_{i}-T)x\to 0$ für alle $x\in H$ und $T_{i}\to T$
bezüglich $\TT_{w}$ genau $((T_{i}-T)x,y)\to 0$ für alle $x,y\in
H$. Somit ist $\TT_{s}$ die Topologie der punktweisen Konvergenz und
$\TT_{w}$ die Topologie der punktweise schwachen Konvergenz in $H$.

Der nächste Satz fasst einige einfache Aussagen über diese Topologien
zusammen.
\begin{satz}\label{satz:eig-optop}
  \hspace{0mm}
  \begin{enumerate}[label=(\roman*),ref=\roman*]
  \item\label{item:eig-optop-i} Die Mengen $P_{s},P_{w},P_{uw}$ sind
    jeweils separierend. Damit bildet $L_{b}(H)$ mit jeder der drei
    Topologien $\TT_{s},\TT_{w},\TT_{uw}$ tatsächlich einen
    lokalkonvexen topologischen Vektorraum.
  \item\label{item:eig-optop-ii} Mit der Bezeichnung
    $\TT(\norm{\cdot})$ für die von der Abbildungsnorm induzierte
    Topologie ist $\TT_{w}$ die gröbste und $\TT(\norm{\cdot})$ die
    feinste der hier diskutierten Topologien. Insbesondere gilt
    $\TT_{w}\subseteq\TT_{uw}$.
  \end{enumerate}
\end{satz}
\begin{bemerkung}
  Es sei explizit darauf hingewiesen, dass die leider übliche
  Bezeichnung \emph{ultraschwache Operatortopologie} irreführend ist,
  da sie tatsächlich feiner und damit stärker als die schwache
  Operatortopologie ist.
\end{bemerkung}
\begin{proof}[Beweis von Satz~\ref{satz:eig-optop}]
  \hspace{0mm}
  \begin{enumerate}[label=(\roman*),ref=\roman*]
  \item Aus $\norm{Tx}=0$ für alle $x\in H$ folgt offenbar $T=0$,
    womit $P_{s}$ als separierend nachgewiesen ist. Für $P_{w}$ ist
    die Überlegung ebenso einfach: Die Forderung $\abs{(Tx,y)}=0$ für
    alle $x,y\in H$ bedingt in einem ersten Schritt $Tx=0$ für alle
    $x\in H$ und folglich $T=0$. Für $P_{uw}$ genügt es,
    $P_{uw}\supseteq P_{w}$ zu zeigen. Dazu definieren wir einen
    Operator $S_{x,y}$ durch $S_{x,y}a:=(a,y)x$, wobei $x,y\in H$
    beliebig sind. Man rechnet unmittelbar nach, dass
    $S_{x,y}^{*}b=(b,x)y=S_{y,x}b$ gilt, womit wir
    \begin{equation}\label{eq:bew-eig-optop-i}
      S_{x,y}^{*}S_{x,y}a=(a,y)(x,x)y=\norm{x}^{2}(a,y)y=\norm{x}^{2}S_{y,y}a
    \end{equation}
    erhalten. Da man $x\neq 0$ wählen kann, schließen wir mit
    Satz~\ref{satz:eig-pos-el}(\ref{item:eig-pos-el-iii})
    insbesondere, dass der Operator $S_{y,y}$ für $y\in H$ positiv
    ist. Wegen $S_{y,y}^{*}=S_{y,y}$ folgt aus
    \eqref{eq:bew-eig-optop-i} für $x=y$ die Gleichung
    $S_{y,y}^{2}=\norm{y}^{2}S_{y,y}$. Sind $x,y\neq 0$, so erhalten
    wir
    $\left(\frac{\norm{x}}{\norm{y}}S_{y,y}\right)^{2}=\norm{x}^{2}S_{y,y}=S_{x,y}^{*}S_{x,y}$,
    also $\abs{S_{x,y}}=\frac{\norm{x}}{\norm{y}}S_{y,y}$. Erweitert
    man $\frac{1}{\norm{y}}y$ zu einer Orthonormalbasis $E$ von $H$,
    ergibt sich
    \begin{displaymath}
      \sum_{e\in E}(|S_{x,y}|(e),e)\stackrel{(*)}{=}\frac{\norm{x}}{\norm{y}}\left(S_{y,y}\frac{1}{\norm{y}}y,\frac{1}{\norm{y}}y\right)=\frac{\norm{x}}{\norm{y}}\left(\norm{y}y,\frac{1}{\norm{y}}y\right)=\norm{x}\norm{y}<+\infty,
    \end{displaymath}
    womit $S_{x,y}\in L^{1}(H)$ ist. Die Gleichheit $(*)$ folgt dabei
    daraus, dass $S_{y,y}$ nach $\spn y$ abbildet und daher die
    Summanden für $e\in E\setminus\{\frac{1}{\norm{y}}y\}$
    wegfallen. Für $x,y\neq 0$ und eine Orthonormalbasis $E$, die
    $\frac{1}{\norm{x}}x$ erweitert, rechnen wir
    \begin{align}\label{eq:bew-eig-optop-ii}
      \begin{split}
        \tr(S_{x,y}T)&=\sum_{e\in
          E}(S_{x,y}Te,e)\stackrel{(**)}{=}\left(S_{x,y}T\left(\frac{1}{\norm{x}}x\right),\frac{1}{\norm{x}}x\right)\\
        &=\left(T\left(\frac{1}{\norm{x}}x\right),y\right)\left(x,\frac{1}{\norm{x}}x\right)=(Tx,y),
      \end{split}
    \end{align}
    wobei die Gleichheit $(**)$ analog zu $(*)$ aus
    $\ran S_{x,y}\leq\spn x$ folgt. Für $x=0$ oder $y=0$ gilt
    klarerweise $\tr(S_{x,y}T)=0=(Tx,y)$. Damit ist
    $P_{uw}\supseteq P_{w}$ nachgewiesen.
  \item Für Topologien $\TT_{1},\TT_{2}$ auf einer Menge $M$ gilt
    $\TT_{1}\subseteq\TT_{2}$ genau dann, wenn jedes bezüglich
    $\TT_{2}$ gegen ein $m\in M$ konvergente Netz $(m_{i})_{i\in I}$
    auch bezüglich $\TT_{1}$ gegen $m$ konvergiert.

    Wegen $\abs{((T_{i}-T)x,y)}\leq\norm{(T_{i}-T)x}\norm{y}$ und
    $\norm{(T_{i}-T)x}\leq\norm{T_{i}-T}\norm{x}$ können wir nach
    \eqref{eq:konv-snormen-ind-top} auf
    $\TT_{w}\subseteq\TT_{s}\subseteq\TT(\norm{\cdot})$ schließen. Die
    aus Gleichung \eqref{eq:bew-eig-optop-ii} folgende Inklusion der
    induzierenden Seminormen liefert
    $\TT_{w}\subseteq\TT_{uw}$. Schließlich ergibt sich
    $\TT_{uw}\subseteq\TT(\norm{\cdot})$ ebenfalls aus
    \eqref{eq:konv-snormen-ind-top}, wenn wir
    $\tr(S(T_{i}-T))\leq\norm[1]{S}\norm{T_{i}-T}$ berücksichtigen;
    siehe
    Lemma~\ref{lem:spur-wohldef+eig}(\ref{item:spur-wohldef+eig-i})
    und \eqref{eq:spklasse-ideal}.
  \end{enumerate}
\end{proof}

Entscheidend für den weiteren Verlauf dieser Arbeit wird der nächste Satz sein.
\begin{satz}\label{satz:uw-optop-wstar}
  Der Raum $L_{b}(H)$ versehen mit $\TT_{uw}$ ist linear
  homöomorph\footnote{Das bedeutet, dass es eine lineare und in beide
    Richtungen stetige Bijektion gibt.} zu $L^{1}(H)'$ versehen mit
  der schwach-*-Topologie $\sigma(L^{1}(H)',L^{1}(H))$. Ein linearer
  Homöomorphismus ist dabei gegeben durch die kanonische Abbildung
  \begin{displaymath}
    \theta:
    \begin{cases}
      L_{b}(H)&\to L^{1}(H)' \\
      \hfill T&\mapsto\tr(.T)
    \end{cases}
  \end{displaymath}
  aus
  Satz~\ref{satz:beschrop-spklasse}(\ref{item:beschrop-spklasse-iii}).
\end{satz}
\begin{proof}
  Es ist nur mehr nachzuprüfen, dass $\theta$ und $\theta^{-1}$ stetig
  sind. Dazu verwenden wir die Grenzwertcharakterisierung der
  Stetigkeit und zeigen
  \begin{displaymath}
    T_{j}\xrightarrow{\TT_{uw}}T \Leftrightarrow\theta(T_{j})\xrightarrow{\sigma(L^{1}(H)',L^{1}(H))}\theta(T).
  \end{displaymath}
  Mit der Bezeichung $\TT_{\C}$ für die euklidische Topologie auf $\C$
  ergibt sich das aus der folgenden Äquivalenzenkette:
  \begin{align}\label{eq:bew-uw-optop-wstar}
    \begin{split}
      T_{j}\xrightarrow{\TT_{uw}}T &\Leftrightarrow\forall S\in L^{1}(H):
                                   \abs{\tr(S(T_{j}-T))}\xrightarrow{\TT_{\C}}
                                   0\\
                                 &\Leftrightarrow\forall S\in L^{1}(H):
                                   \tr(S(T_{j}-T))\xrightarrow{\TT_{\C}} 0\\
                                 &\Leftrightarrow\forall S\in
                                   L^{1}(H): \chevr{S}{\theta(T_{j}-T)}\xrightarrow{\TT_{\C}} 0\\
                                 &\Leftrightarrow\theta(T_{j}-T)\xrightarrow{\sigma(L^{1}(H)',L^{1}(H))} 0\\
                                 &\Leftrightarrow\theta(T_{j})\xrightarrow{\sigma(L^{1}(H)',L^{1}(H))}\theta(T)
    \end{split}
  \end{align}
\end{proof}
Auf beschränkten Mengen fallen die schwache und die ultraschwache
Topologie zusammen:
\begin{korollar}\label{kor:uw-optop-wstar}
  \hspace{0mm}
  \begin{enumerate}[label=(\roman*),ref=\roman*]
  \item\label{item:uw-optop-wstar-i} Die Einheitskugel $S$ von
    $L_{b}(H)$ ist $\TT_{uw}$-kompakt.
  \item\label{item:uw-optop-wstar-ii} Für jedes $r>0$, insbesondere
    für $r=1$, sind die Spurtopologien $(\TT_{uw})|_{rS}$ und
    $(\TT_{w})|_{rS}$ gleich. Insbesondere ist $rS$ auch bezüglich der
    schwachen Operatortopologie kompakt.
  \item\label{item:uw-optop-wstar-iii} Ist $M\subseteq L_{b}(H)$ eine
    $\norm{\cdot}$-beschränkte Menge, so stimmen die Spurtopologien
    $(\TT_{uw})|_{M}$ und $(\TT_{w})|_{M}$ überein.
  \end{enumerate}
\end{korollar}
\begin{proof}
  \hspace{0mm}
  \begin{enumerate}[label=(\roman*),ref=\roman*]
  \item Da $\theta^{-1}$ isometrisch ist, gilt
    $\theta^{-1}(S_{L^{1}(H)'})=S$. Nach dem Satz von Banach-Alaoglu
    ist die Einheitskugel in $L^{1}(H)'$ schwach-*-kompakt, sodass die
    Aussage aus Satz~\ref{satz:uw-optop-wstar} folgt.
  \item Da $\TT_{uw}$ feiner als $\TT_{w}$ ist, ist die Identität eine
    stetige und bijektive Abbildung vom kompakten topologischen Raum
    $(rS,(\TT_{uw})|_{rS})$ in den Hausdorffraum
    $(rS,(\TT_{w})|_{rS})$ und damit ein Homöomorphismus.
  \item Sei $M$ beschränkt durch $r>0$, d.~h. $M\subseteq rS$.

    Schränken wir den Homöomorphismus
    $\id_{rS}:(rS,(\TT_{uw})|_{rS})\to (rS,(\TT_{w})|_{rS})$ aus
    (\ref{item:uw-optop-wstar-ii}) auf $M$ ein, so erhalten wir, dass
    die Identität auf $M$ ein Homöomorphismus zwischen den
    Spurtopologien von $\TT_{uw}$ und $\TT_{w}$ ist.
  \end{enumerate}
\end{proof}

Als Nächstes untersuchen wir, welche über die Struktur eines
topologischen Vektorraums hinausgehenden Operationen auf $L_{b}(H)$
bezüglich welcher der hier betrachteten Topologien stetig
sind. Genauer soll es um die Adjungiertenbildung $.^{*}$ und die
Multiplikation von Operatoren gehen, wobei wir uns nur für den Fall
eines unendlichdimensionalen Hilbertraums $H$
interessieren\footnote{Bei endlicher Dimension sind alle betrachteten
  Abbildungen stetig.}. Dabei werden wir sehen, dass die
Multiplikation $(S,T)\mapsto ST$ einzig für die Normtopologie simultan
stetig ist, sodass wir zusätzlich separate Stetigkeit betrachten. Mit
anderen Worten analysieren wir die Stetigkeit der \emph{Translationen}
$T\mapsto ST$ bzw. $S\mapsto ST$ bei festgehaltenen Operatoren $S$
bzw. $T$.

\begin{satz}\label{satz:adj-mult-stetig-optop}
  Für einen unendlichdimensionalen Hilbertraum $H$ gelten folgende
  Stetigkeitsaussagen:
  \begin{enumerate}[label=(\roman*),ref=\roman*]
  \item\label{item:adj-mult-stetig-optop-i} Bezüglich der
    Normtopologie $\TT(\norm{\cdot})$ sind
    \begin{enumerate}[label=(\alph*),ref=(\theenumi)(\alph*)]
    \item\label{item:adj-mult-stetig-optop-i-a} die
      Adjungiertenbildung $.^{*}:T\mapsto T^{*}$ \textbf{stetig},
    \item\label{item:adj-mult-stetig-optop-i-b} die Multiplikation
      $(S,T)\mapsto ST$ \textbf{stetig},
    \item\label{item:adj-mult-stetig-optop-i-c} die Translationen
      $T\mapsto ST$ bzw. $S\mapsto ST$ für feste Operatoren
      $S\in L_{b}(H)$ bzw. $T\in L_{b}(H)$ \textbf{stetig}.
    \end{enumerate}
  \item\label{item:adj-mult-stetig-optop-ii} Bezüglich der starken
    Operatortopologie $\TT_{s}$ sind
    \begin{enumerate}[label=(\alph*),ref=(\theenumi)(\alph*)]
    \item\label{item:adj-mult-stetig-optop-ii-a} die
      Adjungiertenbildung $.^{*}:T\mapsto T^{*}$ \textbf{nicht
        stetig},
    \item\label{item:adj-mult-stetig-optop-ii-b} die Multiplikation
      $(S,T)\mapsto ST$ \textbf{nicht stetig},
    \item\label{item:adj-mult-stetig-optop-ii-c} die Translationen
      $T\mapsto ST$ bzw. $S\mapsto ST$ für feste Operatoren
      $S\in L_{b}(H)$ bzw. $T\in L_{b}(H)$ \textbf{stetig}.
    \end{enumerate}
  \item\label{item:adj-mult-stetig-optop-iii} Bezüglich der schwachen
    Operatortopologie $\TT_{w}$ sind
    \begin{enumerate}[label=(\alph*),ref=(\theenumi)(\alph*)]
    \item\label{item:adj-mult-stetig-optop-iii-a} die
      Adjungiertenbildung $.^{*}:T\mapsto T^{*}$ \textbf{stetig},
    \item\label{item:adj-mult-stetig-optop-iii-b} die Multiplikation
      $(S,T)\mapsto ST$ \textbf{nicht stetig},
    \item\label{item:adj-mult-stetig-optop-iii-c} die Translationen
      $T\mapsto ST$ bzw. $S\mapsto ST$ für feste Operatoren
      $S\in L_{b}(H)$ bzw. $T\in L_{b}(H)$ \textbf{stetig}.
    \end{enumerate}
  \item\label{item:adj-mult-stetig-optop-iv} Bezüglich der
    ultraschwachen Operatortopologie $\TT_{uw}$ sind
    \begin{enumerate}[label=(\alph*),ref=(\theenumi)(\alph*)]
    \item\label{item:adj-mult-stetig-optop-iv-a} die
      Adjungiertenbildung $.^{*}:T\mapsto T^{*}$ \textbf{stetig},
    \item\label{item:adj-mult-stetig-optop-iv-b} die Multiplikation
      $(S,T)\mapsto ST$ \textbf{nicht stetig},
    \item\label{item:adj-mult-stetig-optop-iv-c} die Translationen
      $T\mapsto ST$ bzw. $S\mapsto ST$ für feste Operatoren
      $S\in L_{b}(H)$ bzw. $T\in L_{b}(H)$ \textbf{stetig}.
    \end{enumerate}
  \end{enumerate}
\end{satz}
\begin{proof}
  Im gesamten Beweis sei $\set{e_{n}}{n\in\N}$ ein abzählbar
  unendliches Orthonormalsystem in $H$.
  \begin{enumerate}[label=(\roman*)]
  \item Der Beweis ist hinlänglich bekannt.
  \item
    \begin{enumerate}[label=(\alph*)]
    \item Wir verwenden die Notation aus dem Beweis von
      Satz~\ref{satz:eig-optop} und betrachten die Operatoren
      $S_{e_{1},e_{n}}$. Wegen $S_{e_{1},e_{n}}x=(x,e_{n})e_{1}$ gilt
      $\norm{S_{e_{1},e_{n}}x}=\abs{(x,e_{n})}\to 0$, da
      $(\abs{(x,e_{n})})_{n\in\N}$ sogar quadratsummierbar ist. Daher
      konvergiert $S_{e_{1},e_{n}}$ bezüglich der starken
      Operatortopologie gegen den Nulloperator. Für den adjungierten
      Operator gilt dies aber nicht, weil wegen
      $S_{e_{1},e_{n}}^{*}=S_{e_{n},e_{1}}$ der Ausdruck
      $\norm{S_{e_{1},e_{n}}^{*}e_{1}}=(e_{1},e_{1})=1$ nicht gegen
      $0$ konvergiert.
    \item Sei $V$ der linksseitige Shiftoperator auf dem
      Orthonormalsystem $\set{e_{n}}{n\in\N}$, also\footnote{Man
        beachte, dass diese Reihe wegen
        $((x,e_{n}))_{n\in\N}\in\ell^{2}(\N)$ konvergiert.}
      \begin{displaymath}
        Vx:=\sum_{n\in\N}(x,e_{n+1})e_{n}.
      \end{displaymath}
      Bekanntermaßen gilt $\norm{V}\leq 1$. Aus
      \begin{displaymath}
        V^{m}x=\sum_{n\in\N}(x,e_{n+m})e_{n}
      \end{displaymath}
      für beliebiges $m\in\N$ erhalten wir
      \begin{displaymath}
        \norm{V^{m}x}^{2}=\sum_{n\in\N}\abs{(x,e_{n+m})}^{2}=\sum_{n=m+1}^{\infty}\abs{(x,e_{n})}^{2}\to
        0\ \text{für}\ m\to\infty.
      \end{displaymath}
      Also konvergiert die Folge $(V^{m})_{m\in\N}$ bezüglich der
      starken Operatortopologie gegen den Nulloperator. Für eine
      \emph{feste} natürliche Zahl $k$ konvergiert daher auch
      $(kV^{m})_{m\in\N}$ gegen $0$.

      Diese Tatsache verwenden wir, um zwei Netze von Operatoren zu
      konstruieren, die beide bezüglich $\TT_{s}$ gegen den
      Nulloperator konvergieren, für die das Netz der Produkte aber
      konstant und ungleich $0$ ist.

      Dazu sei $\FrU^{\TT_{s}}(0)$ der Nullumgebungsfilter bezüglich
      der starken Operatortopologie und
      \begin{displaymath}
        I:=\set{(k,m,U)\in\N\times\N\times\FrU^{\TT_{s}}(0)}{kV^{m}\in U}.
      \end{displaymath}
      Auf $I$ definieren wir die Relation $\preccurlyeq$ durch
      \begin{displaymath}
        (k,m,U)\preccurlyeq (k',m',U'):\Leftrightarrow k\leq k'\
        \text{und}\ U'\subseteq U.
      \end{displaymath}

      Die gerade gezeigte Konvergenz impliziert, dass es für jede
      Umgebung $U\in\FrU^{\TT_{s}}(0)$ und jede natürliche Zahl
      $k\in\N$ ein $m\in\N$ gibt\footnote{Es würde genügen, dass der
        Nulloperator ein Häufungspunkt von $(V^{m})_{m\in\N}$ ist.}
      mit $(k,m,U)\in I$.

      Klarerweise ist $\preccurlyeq$ eine reflexive und transitive
      Relation. Für zwei vorgegebene Tripel
      $(k_{1},m_{1},U_{1}),(k_{2},m_{2},U_{2})\in I$ setzen wir
      $U:=U_{1}\cap U_{2}$ sowie $k:=\max(k_{1},k_{2})$ und wählen
      eine natürliche Zahl $m$ mit $(k,m,U)\in I$. Dieses Tripel
      erfüllt
      \begin{displaymath}
        (k_{1},m_{1},U_{1}),(k_{2},m_{2},U_{2})\preccurlyeq (k,m,U),
      \end{displaymath}
      womit $I$ eine gerichtete Menge ist.

      Weiters definieren wir Netze $(S_{(k,m,U)})_{(k,m,U)\in I}$ und
      $(T_{(k,m,U)})_{(k,m,U)\in I}$ durch
      \begin{displaymath}
        S_{(k,m,U)}:=kV^{m}\quad\text{und}\quad T_{(k,m,U)}:=\frac{1}{k}(V^{*})^{m}.
      \end{displaymath}
      Beide Netze konvergieren bezüglich der starken Operatortopologie
      gegen den Nulloperator: Sei eine Umgebung
      $U\in\FrU^{\TT_{s}}(0)$ gegeben. Wählen wir die natürliche Zahl
      $m$ so, dass $(1,m,U)\in I$, so gilt für alle Tripel
      $(k',m',U')\succcurlyeq (1,m,U)$
      \begin{displaymath}
        S_{(k',m',U')}=k'V^{m'}\in U'\subseteq U.
      \end{displaymath}
      Für das zweite Netz können wir sogar Konvergenz bezüglich der
      Abbildungsnorm zeigen. Dazu sei $\epsilon>0$ und $k\in\N$ mit
      $1/k<\epsilon$. Es gilt $(k,1,L_{b}(H))\in I$ und für
      $(k',m',U')\succcurlyeq (k,1,L_{b}(H))$ folgt wegen
      $\norm{(V^{*})^{m}}\leq 1$
      \begin{displaymath}
        \norm{T_{(k',m',U')}}=\frac{1}{k'}\norm{(V^{*})^{m}}\leq\frac{1}{k'}\leq\frac{1}{k}<\epsilon.
      \end{displaymath}
      Da $V^{*}$ der rechtsseitige Shiftoperator auf
      $\set{e_{n}}{n\in\N}$ ist, also
      \begin{displaymath}
        V^{*}x=\sum_{n\in\N}(x,e_{n})e_{n+1},
      \end{displaymath}
      ergibt sich
      \begin{displaymath}
        (V^{*})^{m}x=\sum_{n\in\N}(x,e_{n})e_{n+m}.
      \end{displaymath}
      Man prüft leicht nach, dass dann
      \begin{displaymath}
        V^{m}(V^{*})^{m}x=\sum_{n\in\N}(x,e_{n})e_{n}=Px
      \end{displaymath}
      gilt, wobei $P$ die Orthogonalprojektion auf
      $\overline{\spn\set{e_{n}}{n\in\N}}$ bezeichnet. Daraus folgt
      \begin{displaymath}
        S_{(k,m,U)}T_{(k,m,U)}=k\frac{1}{k}V^{m}(V^{*})^{m}=P.
      \end{displaymath}
      Wie oben angekündigt erhalten wir bezüglich der starken
      Operatortopologie
      \begin{displaymath}
        S_{(k,m,U)}\to 0,\quad T_{(k,m,U)}\to 0,\quad
        S_{(k,m,U)}T_{(k,m,U)}\not\to 0.
      \end{displaymath}
      Somit ist die Multiplikation nicht stetig.
    \item Da die Translationen linear sind, müssen wir nur zeigen,
      dass sie bei $0$ stetig sind.

      Sei zunächst $S\in L_{b}(H)$ fest und $(T_{i})_{i\in I}$ ein
      bezüglich der starken Operatortopologie gegen $0$ konvergentes
      Netz. Für jedes $x\in H$ gilt also $\norm{T_{i}x}^{2}\to 0$
      bzw. $T_{i}x\to 0$. Da $S$ beschränkt ist, folgt $ST_{i}x\to 0$
      und wir erhalten $ST_{i}\xrightarrow{\TT_{s}} 0$.

      Ist andererseits $(S_{i})_{i\in I}$ ein gegen $0$ konvergentes
      Netz und $T\in L_{b}(H)$ fest, so gilt wie oben $S_{i}x\to 0$
      für alle $x\in H$. Wir setzen speziell $x=Ty$ für $y\in H$ und
      erhalten $S_{i}Ty\to 0$ für alle $y\in H$. Daraus folgt
      $S_{i}T\xrightarrow{\TT_{s}} 0$.
    \end{enumerate}
  \item
    \begin{enumerate}[label=(\alph*)]
    \item Sei $(T_{i})_{i\in I}$ ein Netz in $L_{b}(H)$, das bezüglich
      der schwachen Operatortopologie gegen $T$ konvergiert. Da
      $.^{*}$ eine konjugiert lineare Abbildung ist, konvergiert
      $T_{i}^{*}$ genau dann gegen $T^{*}$, wenn
      $((T_{i}-T)^{*})_{i\in I}=(T_{i}^{*}-T^{*})_{i\in I}$ gegen $0$
      konvergiert. Somit können wir ohne Beschränkung der
      Allgemeinheit $T=0$ annehmen und müssen nur zeigen, dass
      $(T_{i}^{*})_{i\in I}$ bezüglich $\TT_{w}$ gegen den
      Nulloperator konvergiert, wenn $(T_{i})_{i\in I}$ es tut. Dies
      folgt unmittelbar aus
      \begin{displaymath}
        \abs{(T_{i}^{*}x,y)}=\abs{(x,T_{i}y)}=\abs{(T_{i}y,x)}\to 0
      \end{displaymath}
      für alle $x,y\in H$.
    \item Da die starke Operatortopologie feiner als die schwache
      Operatortopologie ist, konvergiert die Folge $(V^{m})_{m\in\N}$
      aus dem Beweis von \ref{item:adj-mult-stetig-optop-ii-b} auch
      bezüglich $\TT_{w}$ gegen den Nulloperator. Daher kann man die
      gleiche Konstruktion mit $\FrU^{\TT_{w}}(0)$ anstelle von
      $\FrU^{\TT_{s}}(0)$ durchführen, um Netze
      $(S_{(k,m,U)})_{(k,m,U)\in I}$ und
      $(T_{(k,m,U)})_{(k,m,U)\in I}$ mit
      \begin{displaymath}
        S_{(k,m,U)}\to 0,\quad T_{(k,m,U)}\to 0,\quad
        S_{(k,m,U)}T_{(k,m,U)}\not\to 0
      \end{displaymath}
      bezüglich der schwachen Operatortopologie zu erhalten.
    \item Sei $S\in L_{b}(H)$ fest und $(T_{i})_{i\in I}$ ein
      bezüglich der schwachen Operatortopologie gegen $0$ konvergentes
      Netz, also $\abs{(T_{i}x,y)}\to 0$ für alle $x,y\in H$. Wegen
      \begin{displaymath}
        \abs{(ST_{i}x,y)}=\abs{(T_{i}x,S^{*}y)}\to 0
      \end{displaymath}
      erhalten wir $ST_{i}\xrightarrow{\TT_{w}} 0$.

      Den Beweis der Stetigkeit von $S\mapsto ST$ kann man ähnlich
      führen; kürzer kann man diese Translation auch als Verkettung
      von $.^{*}$ und der bereits als stetig nachgewiesenen
      Translation schreiben. Aus $ST=(T^{*}S^{*})^{*}$ folgt nämlich
      \begin{equation}\label{eq:bew-adj-mult-stetig-optop}
        (S\mapsto ST)=.^{*}\circ (S\mapsto T^{*}S)\circ .^{*}
      \end{equation}
    \end{enumerate}
  \item
    \begin{enumerate}[label=(\alph*)]
    \item Wie in \ref{item:adj-mult-stetig-optop-iii-a} können wir uns
      auf ein gegen $0$ konvergentes Netz $(T_{i})_{i\in I}$
      beschränken. Es gelte also $\abs{\tr(RT_{i})}\to 0$ für jeden
      Spurklasseoperator $R\in L^{1}(H)$. Aus
      Lemma~\ref{lem:spur-wohldef+eig}(\ref{item:spur-wohldef+eig-ii})
      und Lemma~\ref{lem:tr-komm} folgt
      \begin{displaymath}
        \abs{\tr(RT_{i}^{*})}=\abs{\overline{\tr(RT_{i}^{*})}}\stackrel{\ref{lem:spur-wohldef+eig}}{=}\abs{\tr((RT_{i}^{*})^{*})}=\abs{\tr(T_{i}R^{*})}\stackrel{\ref{lem:tr-komm}}{=}\abs{\tr(R^{*}T_{i})}\to 0,
      \end{displaymath}
      da $R^{*}$ nach
      Lemma~\ref{lem:eig-spklasse}(\ref{item:eig-spklasse-ii}) ein
      Spurklasseoperator ist.
    \item Auch für die ultraschwache Operatortopologie funkioniert die
      Konstruktion aus dem Beweis von
      \ref{item:adj-mult-stetig-optop-ii-b}, wenn wir
      $\FrU^{\TT_{s}}(0)$ durch $\FrU^{\TT_{uw}}(0)$ ersetzen. Wir
      müssen nämlich nur nachweisen, dass die Folge $(V^{m})_{m\in\N}$
      auch bezüglich $\TT_{uw}$ gegen $0$ konvergiert. Dies erhalten
      wir aus
      Korollar~\ref{kor:uw-optop-wstar}(\ref{item:uw-optop-wstar-ii}),
      denn die Operatoren $V^{m}$ sind in $S$ enthalten und
      konvergieren bezüglich $\TT_{w}$ gegen $0$. Somit konvergieren
      sie bezüglich der Spurtopologie $(\TT_{w})|_{S}=(\TT_{uw})|_{S}$
      und folglich bezüglich $\TT_{uw}$.
    \item Sei wieder $S\in L_{b}(H)$ fest und $(T_{i})_{i\in I}$ ein
      Netz, das bezüglich der ultraschwachen Operatortopologie gegen
      den Nulloperator konvergiert. Nach
      Lemma~\ref{lem:eig-spklasse}(\ref{item:eig-spklasse-ii}) liegt
      für jeden Spurklasseoperator $R\in L^{1}(H)$ auch $RS$ in
      $L^{1}(H)$. Daraus folgt
      \begin{displaymath}
        \abs{\tr(R(ST_{i}))}=\abs{\tr((RS)T_{i})}\to 0
      \end{displaymath}
      und wir erhalten $ST_{i}\xrightarrow{\TT_{uw}} 0$.

      Die Stetigkeit der Translation $S\mapsto ST$ folgt analog zu
      \ref{item:adj-mult-stetig-optop-iii-c} aus
      \eqref{eq:bew-adj-mult-stetig-optop} und dem schon Bewiesenen.
    \end{enumerate}
  \end{enumerate}
\end{proof}

Mithilfe des letzten Satzes können wir auf einen Blick ablesen, dass
für einen unendlichdimensionalen Hilbertraum beispielsweise die
Normtopologie niemals mit einer der drei Operatortopologien
zusammenfällt. Tatsächlich gilt:

\begin{satz}\label{satz:optop-versch}
  Wenn $H$ unendlichdimensional ist, sind die Normtopologie sowie die
  starke, schwache und ultraschwache Operatortopologie paarweise
  verschieden.
\end{satz}
\begin{proof}
  Die Multiplikation ist stetig bezüglich der Normtopologie, aber
  nicht bezüglich der starken, schwachen oder ultraschwachen
  Operatortopologie, woraus wir
  \begin{displaymath}
    \TT(\norm{\cdot})\neq\TT_{s}, \TT_{w}, \TT_{uw}
  \end{displaymath}
  schließen. Analog folgt
  \begin{displaymath}
    \TT_{s}\neq\TT(\norm{\cdot}),\TT_{w},\TT_{uw}
  \end{displaymath}
  durch Betrachten der Adjungiertenbildung $.^{*}$. Somit bleibt nur
  noch
  \begin{displaymath}
    \TT_{w}\neq\TT_{uw}
  \end{displaymath}
  zu zeigen. Dazu werden wir ein unbeschränktes\footnote{Nach
    Korollar~\ref{kor:uw-optop-wstar} stimmen $\TT_{w}$ und $\TT_{uw}$
    auf beschränkten Mengen ja überein. Aus diesem Grund müssen wir
    auch zwangsläufig ein Netz konstruieren: Eine
    $\TT_{w}$-konvergente \emph{Folge} ist immer beschränkt, wie man
    mit dem Satz von Banach-Steinhaus zeigen kann, sodass eine Folge
    von Operatoren genau dann bezüglich der schwachen
    Operatortopologie konvergent ist, wenn sie bezüglich der
    ultraschwachen Operatortopologie konvergiert.} Netz angeben, das
  bezüglich der schwachen Operatortopologie gegen den Nulloperator
  konvergiert, nicht aber bezüglich der ultraschwachen
  Operatortopologie.

  Wir betrachten die Menge $\Sub_{\text{fin}}(H)$ der
  endlichdimensionalen Unterräume von $H$. Versehen mit der Inklusion
  von Unterräumen, die wir mit $\leq$ notieren, wird
  $\Sub_{\text{fin}}(H)$ offenbar zu einer gerichteten
  Menge\footnote{$(\Sub_{\text{fin}}(H),\leq)$ ist sogar eine
    Verbandshalbordnung.}. Für einen endlichdimensionalen Unterraum
  $U\in\Sub_{\text{fin}}(H)$ sei $P_{\orth{U}}$ die
  Orthogonalprojektion auf den Orthogonalraum $\orth{U}$ und
  \begin{displaymath}
    T_{U}:=\dim(U)\cdot P_{\orth{U}}\in L_{b}(H).
  \end{displaymath}
  Wir behaupten, dass das Netz $(T_{U})_{U\in\Sub_{\text{fin}}(H)}$
  ein $\TT_{w}$-, aber kein $\TT_{uw}$-Nullnetz ist.

  Die Konvergenz $T_{U}\xrightarrow{\TT_{w}} 0$ folgt daraus, dass für
  $U\geq U_{0}:=\spn x$ sicher $T_{U}x=0$ und daher
  $\abs{(T_{U}x,y)}=0$ gilt.

  Wenn wir einen Spurklasseoperator $S\in L^{1}(H)$ finden können, für
  den $\tr(ST_{U})$ nicht gegen $0$ konvergiert, so haben wir
  $T_{U}\not\to 0$ bezüglich $\TT_{uw}$ gezeigt. Dazu wählen wir ein
  abzählbar unendliches Orthonormalsystem $\set{e_{n}}{n\in\N}$ in $H$
  und definieren den Operator
  \begin{displaymath}
    Sx:=\sum_{n=1}^{\infty}\frac{1}{n^{2}}(x,e_{2n})e_{2n},
  \end{displaymath}
  der wegen $(\frac{1}{n^{2}}(x,e_{n}))_{n\in\N}\in\ell^{2}(\N)$
  wohldefiniert ist. Diese Reihe konvergiert sogar bezüglich der
  Operatortopologie, da aus
  \begin{align*}
    \norm{Sx-\sum_{n=1}^{N}\frac{1}{n^{2}}(x,e_{2n})e_{2n}}^{2}&=\norm{\sum_{n=N+1}^{\infty}\frac{1}{n^{2}}(x,e_{2n})e_{2n}}^{2}\\
                                                               &=\sum_{n=N+1}^{\infty}\norm{\frac{1}{n^{2}}(x,e_{2n})e_{2n}}^{2}=\sum_{n=N+1}^{\infty}\frac{1}{n^{4
                                                                }}\abs{(x,e_{2n})}^{2}\\
                                                               &\leq\frac{1}{(N+1)^{4}}\sum_{n=N+1}^{\infty}\abs{(x,e_{2n})}^{2}\leq\frac{1}{(N+1)^{4}}\norm{x}^{2}
  \end{align*}
  sofort
  \begin{displaymath}
    \norm{S-\sum_{n=1}^{N}\frac{1}{n^{2}}(.,e_{2n})e_{2n}}\leq\frac{1}{(N+1)^{2}}
  \end{displaymath}
  folgt. Mit der Notation aus dem Beweis von Satz~\ref{satz:eig-optop}
  erhalten wir also
  \begin{displaymath}
    S=\lim_{N\to\infty}\sum_{n=1}^{N}\frac{1}{n^{2}}S_{e_{2n},e_{2n}}.
  \end{displaymath}
  In besagtem Beweis haben wir gezeigt, dass $S_{y,y}$ für jedes
  $y\in H$ positiv ist. Nach
  Satz~\ref{satz:eig-pos-el}(\ref{item:eig-pos-el-i}) und
  (\ref{item:eig-pos-el-iv}) ist somit auch $S$ positiv. Aus
  $Se_{2n-1}=0$ folgt
  \begin{displaymath}
    \sum_{e\in E}(|S|e,e)=\sum_{e\in E}(Se,e)=\sum_{n=1}^{\infty}(Se_{2n},e_{2n})=\sum_{n=1}^{\infty}\frac{1}{n^{2}}<\infty,
  \end{displaymath}
  also $S\in L^{1}(H)$.

  Angenommen, das Netz $(\tr(ST_{U}))_{U\in\Sub_{\text{fin}}(H)}$
  konvergiert gegen $0$. Dann gibt es einen endlichdimensionalen
  Unterraum $U_{0}$ mit
  \begin{equation}\label{eq:bew-optop-versch-i}
    \abs{\tr(ST_{U})}\leq 1\quad\text{für alle}\ U\geq U_{0}.
  \end{equation}
  Da $U_{0}$ endlichdimensional ist, gibt es ein $n_{0}\in\N$ mit
  $e_{2n_{0}}\notin U_{0}=\ker P_{\orth{U_{0}}}$ bzw.
  $P_{\orth{U_{0}}}e_{2n_{0}}\neq 0$. Aus
  \begin{displaymath}
    (SP_{\orth{U_{0}}}e_{2n},e_{2n})=\left(\sum_{m=1}^{\infty}\frac{1}{m^{2}}(P_{\orth{U_{0}}}e_{2n},e_{2m})e_{2m},e_{2n}\right)=\frac{1}{n^{2}}(P_{\orth{U_{0}}}e_{2n},e_{2n})=\frac{1}{n^{2}}\norm{P_{\orth{U_{0}}}e_{2n}}^{2}
  \end{displaymath}
  folgt
  \begin{align}
    \begin{split}
      \label{eq:bew-optop-versch-ii}
      \tr(SP_{\orth{U_{0}}})&=\sum_{e\in
        E}(SP_{\orth{U_{0}}}e,e)\stackrel{(*)}{=}\sum_{n=1}^{\infty}(SP_{\orth{U_{0}}}e_{2n},e_{2n})\\
      &=\sum_{n=1}^{\infty}\frac{1}{n^{2}}\norm{P_{\orth{U_{0}}}e_{2n}}^{2}\geq\frac{1}{n_{0}^{2}}\norm{P_{\orth{U_{0}}}e_{2n_{0}}}^{2}>0,
    \end{split}
  \end{align}
  wobei sich die Gleichung $(*)$ aus
  \begin{displaymath}
    \ran S\leq\overline{\spn\set{e_{2n}}{n\in\N}}\leq\orth{\big(\spn(E\setminus\set{e_{2n}}{n\in\N})\big)}
  \end{displaymath}
  ergibt. Insbesondere ist $\abs{\tr(SP_{\orth{U_{0}}})}\neq 0$. Wenn
  wir Unterräume $U\geq U_{0}$ mit beliebig großer endlicher Dimension
  finden können, für die $\tr(SP_{\orth{U}})=\tr(SP_{\orth{U_{0}}})$
  gilt, so erhalten wir für einen Unterraum mit
  $\dim(U)>\frac{1}{\abs{\tr(SP_{\orth{U_{0}}})}}$ die Ungleichung
  \begin{displaymath}
    \abs{\tr(ST_{U})}=\dim(U)\abs{\tr(SP_{\orth{U}})}=\dim(U)\abs{\tr(SP_{\orth{U_{0}}})}>1,
  \end{displaymath}
  im Widerspruch zu \eqref{eq:bew-optop-versch-i}.

  Analog zu \eqref{eq:bew-optop-versch-ii} zeigt man mithilfe von
  $(*)$
  \begin{displaymath}
    \tr(SP_{\orth{U}})=\sum_{n=1}^{\infty}\frac{1}{n^{2}}\norm{P_{\orth{U}}e_{2n}}^{2},
  \end{displaymath}
  sodass es für den Beweis von
  $\tr(SP_{\orth{U}})=\tr(SP_{\orth{U_{0}}})$ reicht,
  $P_{\orth{U}}e_{2n}=P_{\orth{U_{0}}}e_{2n}$ für alle $n\in\N$ zu
  garantieren. Der Orthogonalraum
  \begin{displaymath}
    W:=\orth{(U_{0}+\spn\set{e_{2n}}{n\in\N})}
  \end{displaymath}
  hat unendliche Dimension, denn wegen
  \begin{align*}
    H&=\big(U_{0}+\spn\set{e_{2n}}{n\in\N}\big)+\orth{(U_{0}+\spn\set{e_{2n}}{n\in\N})}\\
     &=\spn\set{e_{2n}}{n\in\N}+\big(U_{0}+\orth{(U_{0}+\spn\set{e_{2n}}{n\in\N})}\big)
  \end{align*}
  hätte sonst $\spn\set{e_{2n}}{n\in\N}$ endliche Kodimension. Dies
  würde
  \begin{displaymath}
    \spn\set{e_{2n}}{n\in\N}\cap\spn\set{e_{2n-1}}{n\in\N}=\{0\}
  \end{displaymath}
  widersprechen.
  
  Sei $\set{f_{n}}{n\in\N}$ eine abzählbar unendliche und linear
  unabhängige Teilmenge von $W$. Wir definieren
  \begin{displaymath}
    U_{k}:=U_{0}+\spn\{f_{1},\dots,f_{k}\}
  \end{displaymath}
  für $k\in\N$ und zeigen
  $P_{\orth{U_{k}}}e_{2n}=P_{\orth{U_{0}}}e_{2n}$. Dazu müssen wir
  \begin{equation}\label{eq:bew-optop-versch-iii}
    P_{\orth{U_{0}}}e_{2n}\in\orth{U_{k}}
  \end{equation}
  und
  \begin{equation}\label{eq:bew-optop-versch-iv}
    e_{2n}-P_{\orth{U_{0}}}e_{2n}\in U_{k}
  \end{equation}
  nachweisen.

  Die Eigenschaft \eqref{eq:bew-optop-versch-iv} folgt aus
  \begin{displaymath}
    e_{2n}-P_{\orth{U_{0}}}e_{2n}\in U_{0}\leq U_{k}.
  \end{displaymath}
  Für \eqref{eq:bew-optop-versch-iii} bemerken wir
  \begin{displaymath}
    \orth{U_{k}}=\orth{U_{0}}\cap\orth{\spn\{f_{1},\dots,f_{k}\}},
  \end{displaymath}
  wobei $P_{\orth{U_{0}}}e_{2n}\in\orth{U_{0}}$ wieder unmittelbar
  klar ist. Aus $e_{2n}\in\orth{W}$ und
  $e_{2n}-P_{\orth{U_{0}}}e_{2n}\in U_{0}\leq\orth{W}$ folgt
  \begin{displaymath}
    P_{U_{0}}e_{2n}=e_{2n}-(e_{2n}-P_{\orth{U_{0}}}e_{2n})\in\orth{W}\leq\orth{\spn\{f_{1},\dots,f_{k}\}},
  \end{displaymath}
  sodass wir insgesamt \eqref{eq:bew-optop-versch-iii} erhalten. Da
  $\dim(U_{k})=\dim(U_{0})+k$ in $k$ unbeschränkt ist, haben wir unser
  Ziel erreicht und den Beweis abgeschlossen.
\end{proof}

\section{Abgeschlossenheit von $*$-Unteralgebren}
\label{sec:abg-star-unteralg}
Wir können nun das Hauptresultat dieses Kapitels formulieren, das wir
anschließend in einigen Schritten beweisen
werden. Satz~\ref{satz:optop-versch} einerseits und
Korollar~\ref{kor:uw-optop-wstar} andererseits zeigen, dass die hier
behandelten Operatortopologien für $\dim H=\infty$ zwar alle
verschieden, in mancherlei Hinsicht aber auch sehr ähnlich sind. Der
nächste Satz kreist ebenfalls um die Frage, wann die verschiedenen
Operatortopologien für bestimmte Teilmengen in gewisser Weise gleich
sind. Dabei soll "`Gleichheit"' bedeuten, dass die Topologien den
gleichen Begriff von Abgeschlossenheit für gewisse Mengen liefern. Es
zeigt sich, dass für $*$-Unteralgebren die stärkstmögliche Aussage
gilt.
\begin{satz}\label{satz:optop-abg-unteralg}
  Für eine $*$-Unteralgebra $A$ von $L_{b}(H)$ sind folgende Aussagen
  äquivalent:
  \begin{enumerate}[label=(\roman*),ref=\roman*]
  \item\label{item:optop-abg-unteralg-i} $A$ ist
    $\TT_{s}$-abgeschlossen.
  \item\label{item:optop-abg-unteralg-ii} $A$ ist
    $\TT_{w}$-abgeschlossen.
  \item\label{item:optop-abg-unteralg-iii} $A$ ist
    $\TT_{uw}$-abgeschlossen.
  \end{enumerate}
\end{satz}
Zunächst wollen wir die Äquivalenz von
(\ref{item:optop-abg-unteralg-i}) und
(\ref{item:optop-abg-unteralg-ii}) zeigen. Dabei werden wir beweisen,
dass diese Aussage sogar für beliebige konvexe Mengen gilt, indem wir
die Gleichheit der Dualräume bezüglich $\TT_{s}$ und $\TT_{w}$
zeigen. Als Vorarbeit für Kapitel~\ref{cha:eindeutigkeit}
identifizieren wir zusätzlich den Dualraum bezüglich der
ultraschwachen Operatortopologie.
\begin{lemma}\label{lem:optop-dualraeume}
  Für die Dualräume von $L_{b}(H)$ versehen mit der starken, schwachen
  bzw. ultraschwachen Operatortopologie gilt
  \begin{subequations}
    \begin{align}\label{eq:optop-dualraeume-a}
      \big(L_{b}(H),\TT_{s}\big)'=\big(L_{b}(H),\TT_{w}\big)'&=\set{\left(T\mapsto\sum_{k=1}^{n}(Tx_{k},y_{k})\right)}{n\in\N,x_{k},y_{k}\in
                                   H}\\\label{eq:optop-dualraeume-b}
      \big(L_{b}(H),\TT_{uw}\big)'&=\set{(T\mapsto\tr(ST))}{S\in
                                    L^{1}(H)}
    \end{align}
  \end{subequations}
\end{lemma}
\begin{proof}
  Wir führen den Beweis in drei Schritten.
  \begin{enumerate}[label=(\roman*),ref=\roman*]
  \item Zunächst zeigen wir
    \begin{equation}\label{eq:bew-optop-dualraeume}
      \big(L_{b}(H),\TT_{w}\big)'=\set{\left(T\mapsto\sum_{k=1}^{n}(Tx_{k},y_{k})\right)}{n\in\N,x_{k},y_{k}\in
        H}.
    \end{equation}
    Ist $Y$ ein punktetrennender Unterraum des Raums der linearen
    Funktionale auf einem Vektorraum $X$, so kann man die schwache
    Topologie $\sigma(X,Y)$ definieren. Bekanntermaßen wird
    $\sigma(X,Y)$ auch durch die Menge $P=\set{p_{f}}{f\in Y}$ von
    Seminormen erzeugt, wenn wir $p_{f}(z):=\abs{\chevr{z}{f}}$ für
    $z\in X$ setzen. Die rechte Seite von
    \eqref{eq:bew-optop-dualraeume} ist ein Unterraum des
    algebraischen Dualraums von $L_{b}(H)$, der nach dem Beweis von
    Satz~\ref{satz:eig-optop}(\ref{item:eig-optop-i}) punktetrennend
    ist. Somit ist die schwache Topologie bezüglich dieser rechten
    Seite wohldefiniert. Sei nun ein Funktional
    $\chevr{T}{\varphi}:=\sum_{k=1}^{n}(Tx_{k},y_{k})$ und ein Netz
    $(T_{i})_{i\in I}$ von beschränkten Operatoren gegeben. Definieren
    wir die Funktionale $\varphi_{k}$ durch
    $\chevr{T}{\varphi_{k}}:=(Tx_{k},y_{k})$, so folgt
    \begin{displaymath}
      p_{\varphi}(T_{i})=\abs{\sum_{k=1}^{n}(Tx_{k},y_{k})}\leq\sum_{k=1}^{n}\abs{(Tx_{k},y_{k})}=\sum_{k=1}^{n}p_{\varphi_{k}}(T_{i})
    \end{displaymath}
    und daraus, dass $p_{\varphi}(T_{i})$ für $i\in I$ gegen Null
    konvergiert, wenn alle $p_{\varphi_{k}}(T_{i})$ gegen Null
    konvergieren. Betrachten wir die Mengen
    \begin{align*}
      P_{w}&:=\set{(T\mapsto \abs{(Tx,y)})}{x,y\in H}
             \quad\text{und}\\
      P&:=\set{\left(T\mapsto\abs{\sum_{k=1}^{n}(Tx_{k},y_{k})}\right)}{n\in\N,x_{k},y_{k}\in
         H}
    \end{align*}
    von Seminormen, so zeigt die Grenzwertcharakterisierung
    \eqref{eq:konv-snormen-ind-top}, dass in den von $P_{w}$ bzw. $P$
    erzeugten Topologien die gleichen Netze gegen Null
    konvergieren. Somit sind die Topologien gleich. Die erste
    Topologie ist genau die schwache Operatortopologie, die zweite
    Topologie stimmt mit der schwachen Topologie bezüglich der rechten
    Seite von \eqref{eq:bew-optop-dualraeume} überein. Die
    wohlbekannte Tatsache $\big(X,\sigma(X,Y)\big)'=Y$ schließt den
    Beweis von \eqref{eq:bew-optop-dualraeume} ab.
  \item Als Nächstes beweisen wir
    $(L_{b}(H),\TT_{s})'=(L_{b}(H),\TT_{w})'$.

    Aus der Inklusion $\TT_{w}\subseteq\TT_{s}$ folgt sofort
    $(L_{b}(H),\TT_{w})'\subseteq(L_{b}(H),\TT_{s})'$.

    Sei umgekehrt $f\in L_{b}(H)^{*}$ stetig bezüglich der starken
    Operatortopologie. Dann ist $f$ auf einer $\TT_{s}$-Nullumgebung
    $U$ durch eine Konstante $D>0$ beschränkt, wobei wir durch
    Skalieren annehmen können, dass $D=1$ ist. Außerdem können wir $U$
    als Element einer vorgegebenen Nullumgebungsbasis wählen. Eine
    Nullumgebungsbasis einer von Seminormen $p\in P$ induzierten
    Topologie ist gegeben durch die Mengen
    \begin{displaymath}
      \set{x\in X}{p_{k}(x)<\epsilon\ \text{für alle}\
        k=1,\dots,n},
    \end{displaymath}
    wobei $\epsilon>0$, $n\in\N$ und $p_{1},\dots,p_{n}\in P$ beliebig
    gewählt werden können. In unserem Fall bedeutet das
    \begin{displaymath}
      U=\set{T\in L_{b}(H)}{\norm{Tx_{k}}<\epsilon\ \text{für alle}\ k=1,\dots,n}
    \end{displaymath}
    für ein positives $\epsilon$ und Vektoren
    $x_{1},\dots,x_{n}\in H$.  \newline Ist $T\in L_{b}(H)$ beliebig
    und $C>\max(\norm{Tx_{1}}\,\dots,\norm{Tx_{n}})$, so gilt
    $\frac{\epsilon}{C}T\in U$, sodass
    $\abs{\chevr{\frac{\epsilon}{C}T}{f}}\leq 1$
    bzw. $\abs{\chevr{T}{f}}\leq\frac{C}{\epsilon}$ folgt. Da $C$
    beliebig war, schließen wir auf
    \begin{equation}\label{eq:w-s-optop-funkt-i}
      \abs{\chevr{T}{f}}\leq\frac{1}{\epsilon}\max(\norm{Tx_{1}},\dots,\norm{Tx_{n}})
    \end{equation}
    für jedes $T\in L_{b}(H)$.

    Wir betrachten nun das kartesische Produkt $H^{n}$ und dessen
    Unterraum
    \begin{displaymath}
      Y:=\set{(Tx_{1},\dots,Tx_{n})^{T}\in H^{n}}{T\in L_{b}(H)}.
    \end{displaymath}
    Die Abbildung
    \begin{displaymath}
      \varphi:
      \begin{cases}
        \hfill Y&\to \C\\
        (Tx_{1},\dots,Tx_{n})^{T}&\mapsto\chevr{T}{f}
      \end{cases}
    \end{displaymath}
    ist wohldefiniert, da aus
    $(Tx_{1},\dots,Tx_{n})^{T}=(Sx_{1},\dots,Sx_{n})^{T}$ wegen
    \eqref{eq:w-s-optop-funkt-i} angewandt auf $T-S$ die Gleichheit
    $\chevr{T}{f}=\chevr{S}{f}$ folgt. Klarerweise ist $\varphi$ ein
    lineares Funktional und erfüllt wegen \eqref{eq:w-s-optop-funkt-i}
    die Ungleichung
    \begin{displaymath}
      \abs{\chevr{(Tx_{1},\dots,Tx_{n})^{T}}{\varphi}}\leq\frac{1}{\epsilon}\max(\norm{Tx_{1}},\dots,\norm{Tx_{n}})\leq\frac{1}{\epsilon}\norm[H^{n}]{(Tx_{1},\dots,Tx_{n})^{T}}.
    \end{displaymath}
    Das Funktional $\varphi$ ist also beschränkt und hat folglich eine
    -- ebenfalls mit $\varphi$ bezeichnete -- beschränkte Fortsetzung
    auf den Hilbertraum $\cl{Y}\leq H^{n}$. Der Darstellungssatz von
    Riesz in $\cl{Y}$ liefert die Existenz eines Elements
    $(y_{1},\dots,y_{n})\in\cl{Y}$ mit
    \begin{displaymath}
      \chevr{T}{f}=\chevr{(Tx_{1},\dots,Tx_{n})^{T}}{\varphi}=\left((Tx_{1},\dots,Tx_{n})^{T},(y_{1},\dots,y_{n})^{T}\right)_{H^{n}}=\sum_{k=1}^{n}(Tx_{k},y_{k}).
    \end{displaymath}
    Nach \eqref{eq:bew-optop-dualraeume} ist $f$ stetig bezüglich der
    schwachen Operatortopologie, was den Beweis von
    \eqref{eq:optop-dualraeume-a} vollendet.
  \item Zuletzt folgt \eqref{eq:optop-dualraeume-b} analog (aber auf
    direktere Weise) zu \eqref{eq:bew-optop-dualraeume}; siehe auch
    \eqref{eq:bew-uw-optop-wstar}.
  \end{enumerate}
\end{proof}

Insbesondere sind also die Dualräume bezüglich der starken und der
schwachen Operatortopologie gleich. Mit der Folgerung aus dem Satz von
Hahn-Banach, dass der Abschluss einer konvexen Menge nur vom
topologischen Dualraum und nicht von der genauen Topologie abhängt,
erhalten wir aus diesem Resultat sofort:
\begin{korollar}\label{kor:optop-konv-abg}
  Eine konvexe Teilmenge -- insbesondere eine $*$-Unteralgebra -- von
  $L_{b}(H)$ ist genau dann $\TT_{s}$-abgeschlossen, wenn sie
  $\TT_{w}$-abgeschlossen ist.
\end{korollar}

Nach
Satz~\ref{satz:adj-mult-stetig-optop}(\ref{item:adj-mult-stetig-optop-ii})
sind weder die Adjungiertenbildung auf $L_{b}(H)$ noch die
Multiplikation von Operatoren $\TT_{s}$-stetig, wenn $H$
unendlichdimensional ist. Um den Beweis von
Satz~\ref{satz:optop-abg-unteralg} vervollständigen zu können,
benötigen wir Voraussetzungen, unter denen man sehr wohl die
Adjungiertenbildung bzw. die Multiplikation mit Grenzwerten bezüglich
$\TT_{s}$ vertauschen kann.
\begin{lemma}\label{lem:tech-st-stetig}
  \hspace{0mm}
  \begin{enumerate}[label=(\roman*),ref=\roman*]
  \item\label{item:tech-st-stetig-i} Die Adjungiertenbildung $.^{*}$
    ist auf den normalen Operatoren in $L_{b}(H)$ stetig bezüglich der
    starken Operatortopologie.
  \item\label{item:tech-st-stetig-ii} Ist $M\subseteq L_{b}(H)$ eine
    $\norm{\cdot}$-beschränkte Menge, so ist die Muliplikation
    $(S,T)\mapsto ST$ auf $M\times L_{b}(H)$ stetig bezüglich der
    starken Operatortopologie.
  \end{enumerate}
\end{lemma}
\begin{proof}
  \hspace{0mm}
  \begin{enumerate}[label=(\roman*),ref=\roman*]
  \item Seien $S,T\in L_{b}(H)$ normale Operatoren. Für einen Vektor
    $x\in H$ gilt
    \begin{align*}
      \norm{(T^{*}-S^{*})x}^{2}&=(T^{*}x-S^{*}x,T^{*}x-S^{*}x)\\
                               &=(TT^{*}x,x)-(TS^{*}x,x)-(x,TS^{*}x)+(SS^{*}x,x)\\
                               &\stackrel{\mathclap{(*)}}{=}(T^{*}Tx,x)+((S-T)S^{*}x,x)-(x,TS^{*}x)\\
                               &\stackrel{\mathclap{(**)}}{=}(T^{*}Tx,x)+((S-T)S^{*}x,x)-(x,TS^{*}x)+\underbrace{(x,SS^{*}x)-(x,S^{*}Sx)}_{=\,0}\\
                               &=(Tx,Tx)+((S-T)S^{*}x,x)+(x,(S-T)S^{*}x)-(Sx,Sx)\\
                               &=\norm{Tx}^{2}-\norm{Sx}^{2}+((S-T)S^{*}x,x)+(x,(S-T)S^{*}x),
    \end{align*}
    wobei in $(*)$ bzw. $(**)$ die Normalität von $T$ bzw. $S$
    eingeht. Diesen Ausdruck schätzen wir mit der Cauchy-Schwarz'schen
    Ungleichung ab und erhalten
    \begin{displaymath}
      \norm{(T^{*}-S^{*})x}^{2}\leq\norm{Tx}^{2}-\norm{Sx}^{2}+2\norm{x}\norm{(S-T)S^{*}x}.
    \end{displaymath}
    Setzen wir in diese Gleichung $T=S_{i}$ ein für ein Netz
    $(S_{i})_{i\in I}$, das bezüglich $\TT_{s}$ gegen $S$ konvergiert,
    so folgt, dass $S_{i}^{*}$ gegen $S^{*}$ konvergiert. Hierbei ist
    zu beachten, dass $S^{*}x$ ein fester Vektor ist, womit
    $(S-S_{i})S^{*}x$ gegen $0$ konvergiert.
  \item Seien $x\in H$ und
    $(S_{1},T_{1}),(S_{2},T_{2})\in M\times L_{b}(H)$ gegeben. Ist $M$
    beschränkt durch $C$, dann gilt
    \begin{align*}
      \norm{(S_{1}T_{1}-S_{2}T_{2})x}&\leq\norm{S_{1}(T_{1}-T_{2})x}+\norm{(S_{1}-S_{2})T_{2}x}\\
                                     &\leq C\norm{(T_{1}-T_{2})x}+\norm{(S_{1}-S_{2})T_{2}x}.
    \end{align*}
    Aus dieser Ungleichung folgt, dass $(S_{1,i}T_{1,i})_{i\in I}$
    gegen $ST$ konvergiert, wenn $S_{1,i}$ bzw. $T_{1,i}$ gegen
    $S_{2}:=S$ bzw. $T_{2}:=T$ konvergiert.
  \end{enumerate}
\end{proof}

Mithilfe des letzten Lemmas können wir eine -- leider etwas technische
-- Stetigkeitsbedingung beweisen. Dazu sei bemerkt, dass trotz der
fehlenden $\TT_{s}$-Stetigkeit von $.^{*}$ der $\TT_{s}$-Grenzwert von
selbstadjungierten Operatoren wieder selbstadjungiert ist: Ein
$\TT_{s}$-konvergentes Netz von Operatoren ist auch bezüglich
$\TT_{w}$ konvergent mit demselben Grenzwert, sodass die Behauptung
aus der $\TT_{w}$-Stetigkeit von $.^{*}$ folgt.

\begin{lemma}\label{lem:st-stetig-fkt}
  Sei $f:\R\to\C$ eine beschränkte und stetige Funktion. Für alle
  Netze $(S_{i})_{i\in I}$ selbstadjungierter Operatoren auf $H$, die
  bezüglich $\TT_{s}$ gegen $S$ konvergieren, konvergieren die nach
  dem Funktionalkalkül definierten Ausdrücke $f(S_{i})$ bezüglich
  $\TT_{s}$ gegen $f(S)$.
\end{lemma}
\begin{proof}
  Sei $\FF$ die Menge aller stetigen, aber nicht notwendigerweise
  beschränkten Funktionen $f:\R\to\C$, die die gewünschte
  Eigenschaft\footnote{Für einen selbstadjungierten Operator
    $S\in L_{b}(H)$ ist $f|_{\sigma(S)}$ wegen der Kompaktheit von
    $\sigma(S)$ beschränkt. Somit ist $f(S):=f|_{\sigma(S)}(S)$
    wohldefiniert und es gilt
    $\norm{f(S)}=\norm[\infty]{f|_{\sigma(S)}}$.}  haben. Zu zeigen
  ist $C_{b}(\R)\subseteq\FF$. Im gesamten Beweis bezeichne
  $(S_{i})_{i\in I}$ ein bezüglich der starken Operatortopologie
  konvergentes Netz selbstadjungierter Operatoren mit Grenzwert
  $S=S^{*}$ und $T$ einen weiteren selbstadjungierten Operator.

  Klarerweise ist $\FF$ ein Unterraum von $C(\R)$. Für Funktionen
  $f,g\in\FF$, wobei $g$ beschränkt ist, gilt $gf=fg\in\FF$ nach
  Lemma~\ref{lem:tech-st-stetig}(\ref{item:tech-st-stetig-ii}), da die
  Operatoren $g(S_{i})$ gleichmäßig in $i$ durch $\norm[\infty]{g}$
  beschränkt sind. Nach Bemerkung~\ref{bem:fktkalkuel} kommutieren
  $f(S)$ und $\overline{f}(S)=f(S)^{*}$, sodass $f(S)$ normal
  ist. Daraus folgt mit
  Lemma~\ref{lem:tech-st-stetig}(\ref{item:tech-st-stetig-i})
  \begin{displaymath}
    \overline{f}(S_{i})=f(S_{i})^{*}\xrightarrow{\TT_{s}}f(S)^{*}=\overline{f}(S).
  \end{displaymath}
  Somit ist $\FF$ auch bezüglich der Konjugation abgeschlossen.

  Wir betrachten nun die nach diesen Argumenten bezüglich der
  Konjugation abgeschlossene Funktionenalgebra
  $\FF_{0}:=\FF\cap C_{0}(\R)$. Diese Algebra ist bezüglich der
  Supremumsnorm $\norm[\infty]{\cdot}$ auf $C_{0}(\R)$ abgeschlossen:
  Sei $(f_{n})_{n\in\N}$ eine Folge aus $\FF_{0}$, die bezüglich
  $\norm[\infty]{\cdot}$ gegen $f\in C_{0}(\R)$ konvergiert. Für jedes
  $n\in\N$, $i\in I$ und $x\in H$ gilt
  \begin{align*}
    \norm{(f(S_{i})-f(S))(x)}&\leq\norm{(f(S_{i})-f_{n}(S_{i}))(x)}+\norm{(f_{n}(S_{i})-f_{n}(S))(x)}+\norm{(f_{n}(S)-f(S))(x)}\\
                             &=\norm{(f-f_{n})(S_{i})(x)}+\norm{(f_{n}(S_{i})-f_{n}(S))(x)}+\norm{(f_{n}-f)(S)(x)}\\
                             &\leq\norm{(f-f_{n})(S_{i})}\norm{x}+\norm{(f_{n}(S_{i})-f_{n}(S))(x)}+\norm{(f_{n}-f)(S)}\norm{x} \\
                             &=\norm[\infty]{(f-f_{n})|_{\sigma(S_{i})}}\norm{x}+\norm{(f_{n}(S_{i})-f_{n}(S))(x)}+\norm[\infty]{(f_{n}-f)|_{\sigma(S)}}\norm{x}\\
                             &\leq 2\norm[\infty]{f-f_{n}}\norm{x}+\norm{(f_{n}(S_{i})-f_{n}(S))(x)}.
  \end{align*}
  Zu gegebenem $\epsilon>0$ wählt man einen Index $n_{0}$ mit
  $\norm[\infty]{f-f_{n_{0}}}\norm{x}<\epsilon$ und dazu $i_{0}\in I$
  mit $\norm{(f_{n_{0}}(S_{i})-f_{n_{0}}(S))(x)}<\epsilon$ für alle
  $i\succcurlyeq i_{0}$. Obige Abschätzung für $n:=n_{0}$ liefert
  $\norm{(f(S_{i})-f(S))(x)}<3\epsilon$ für alle
  $i\succcurlyeq i_{0}$, sodass die Konvergenz $f(S_{i})\to f(S)$
  bezüglich der starken Operatortopologie bewiesen ist.

  Als Nächstes zeigen wir $\FF_{0}=C_{0}(\R)$, also
  $C_{0}(\R)\subseteq\FF$. Wegen der Abgeschlossenheit von $\FF_{0}$
  ist dieses Ziel erreicht, wenn wir zusätzlich $\FF_{0}$ als dicht in
  $C_{0}(\R)$ nachweisen. Nach dem Satz von Stone-Weierstraß für
  lokalkompakte Räume genügt es zu überprüfen, dass $\FF_{0}$
  punktetrennend und nirgends identisch verschwindend ist. Dazu zeigen
  wir, dass die injektive Funktion $g(t):=\frac{t}{1+t^{2}}$ und die
  nullstellenfreie Funktion $f(t):=\frac{1}{1+t^{2}}$ in $\FF_{0}$
  liegen, wobei $f,g\in C_{0}(\R)$ klar ist. Mit den Rechenregeln für
  den Funktionalkalkül gilt
  \begin{equation}\label{eq:st-stetig-fkt-i}
    g(T)(1+T^{2})=(t\mapsto g(t)(1+t^{2}))(T)=(t\mapsto t)(T)=T.
  \end{equation}
  Nach Satz~\ref{satz:eig-pos-el}(\ref{item:eig-pos-el-iii}) ist
  $T^{2}$ positiv, sodass $\sigma(1+T^{2})\subseteq [1,+\infty)$
  insbesondere Null nicht enthält. Multipliziert man beide Seiten von
  \eqref{eq:st-stetig-fkt-i} mit der daher existierenden Inversen von
  $1+T^{2}$, so folgt $g(T)=T(1+T^{2})^{-1}$. Analog zeigt man
  $g(T)=(1+T^{2})^{-1}T$ und $f(T)=(1+T^{2})^{-1}$. Damit berechnen
  wir
  \begin{align*}
    g(S_{i})-g(S)&=(1+S_{i}^{2})^{-1}S_{i}-S(1+S^{2})^{-1}\\
                 &=(1+S_{i}^{2})^{-1}\big(S_{i}(1+S^{2})-(1+S_{i}^{2})S\big)(1+S^{2})^{-1}\\
                 &=(1+S_{i}^{2})^{-1}\big((S_{i}-S)+S_{i}(S-S_{i})S\big)(1+S^{2})^{-1}.
  \end{align*}
  Wegen $\norm{f(S_{i})}\leq\norm[\infty]{f}\leq 1$ und
  $\norm{g(S_{i})}\leq\norm[\infty]{g}\leq 1$ folgt
  \begin{align*}
    \|(g(S_{i})-g&(S))(x)\|\\
                 &\leq\big\|\underbrace{(1+S_{i}^{2})^{-1}}_{=f(S_{i})}\big(S_{i}-S\big)(1+S^{2})^{-1}x\big\|+\big\|\underbrace{(1+S_{i}^{2})^{-1}S_{i}}_{=g(S_{i})}(S-S_{i})S(1+S^{2})^{-1}x\|\\
                 &\leq\norm{f(S_{i})}\norm{(S_{i}-S)(1+S^{2})^{-1}x}+\norm{g(S_{i})}\norm{(S-S_{i})S(1+S^{2})^{-1}x}\\
                 &\leq\norm{(S_{i}-S)(1+S^{2})^{-1}x}+\norm{(S-S_{i})S(1+S^{2})^{-1}x}.
  \end{align*}
  Da $(1+S^{2})^{-1}x$ sowie $S(1+S^{2})^{-1}x$ feste Vektoren sind,
  ergibt sich daraus $g\in\FF$. Wir schließen, dass auch
  $f=(t\mapsto 1-tg(t))$ in $\FF$ enthalten ist, da sowohl
  $\mathds{1}$ als auch $(t\mapsto t)$ in $\FF$ liegen und $g$
  beschränkt ist.

  Nun können wir den Beweis abschließen. Ist $h\in C_{b}(\R)$
  beliebig, so sind die Funktionen $h_{1}:=hf$ und $h_{2}:=hg$ in
  $C_{0}(\R)\subseteq\FF$ enthalten. Da $h_{2}$ beschränkt ist, gilt
  auch $h_{3}:=(t\mapsto th_{2}(t))\in\FF$. Daraus folgt das
  gewünschte Resultat $h=h_{1}+h_{3}\in\FF$.
\end{proof}

Um die ultraschwache Topologie ins Spiel bringen zu können, benötigen
wir -- obwohl auf den ersten Blick nicht ersichtlich -- den nützlichen
und auch für sich ästhetisch sehr ansprechenden \emph{Dichtesatz von
  Kaplansky}. Man kann es nicht schöner ausdrücken als G. Pedersen in
\cite[2.3.4. Notes and remarks.]{pedersen:autgrps}:
\begin{quote}
  The density theorem is Kaplansky's great gift to
mankind [...]. It can be used every day, and twice on Sundays.
\end{quote}

\begin{satz}[Kaplansky]\label{satz:kaplansky}
  Sei $A$ eine $C^{*}$-Unteralgebra von $L_{b}(H)$ und $B$ ihr
  $\TT_{s}$-Abschluss. Dann gelten die folgenden drei Aussagen:
  \begin{enumerate}[label=(\roman*),ref=\roman*]
  \item\label{item:kaplansky-i} $\cl{A_{sa}}^{\TT_{s}}=B_{sa}$
  \item\label{item:kaplansky-ii}
    $\cl{S_{A_{sa}}}^{\TT_{s}}=S_{B_{sa}}$
  \item\label{item:kaplansky-iii} $\cl{S_{A}}^{\TT_{s}}=S_{B}$
  \end{enumerate}
\end{satz}
\begin{proof}
  \hspace{0mm}
  \begin{enumerate}[label=(\roman*),ref=\roman*]
  \item Sei $S\in B_{sa}$ beliebig. Nach Definition gibt es ein Netz
    $(S_{i})_{i\in I}$ aus $A$, das bezüglich $\TT_{s}$ und damit auch
    bezüglich $\TT_{w}$ gegen $S$ konvergiert. Wegen der
    $\TT_{w}$-Stetigkeit von $.^{*}$ konvergiert
    $A_{sa}\ni\re(S_{i})=\frac{1}{2}(S_{i}+S_{i}^{*})$ gegen
    $\re(S)=S$ bezüglich der schwachen Operatortopologie. Der Operator
    $S$ liegt daher im $\TT_{w}$-Abschluss von $A_{sa}$. Da $A_{sa}$
    konvex ist, stimmen $\TT_{w}$- und $\TT_{s}$-Abschluss nach
    Korollar~\ref{kor:optop-konv-abg} aber überein.

    Umgekehrt ist ein Element von
    $\cl{A_{sa}}^{\TT_{s}}=\cl{A_{sa}}^{\TT_{w}}$ einerseits in $B$
    enthalten und andererseits selbstadjungiert, da $.^{*}$ bezüglich
    der schwachen Operatortopologie stetig ist.
  \item Sei $S$ beliebig in der Einheitskugel $S_{B_{sa}}$. Nach
    (\ref{item:kaplansky-i}) existiert ein Netz $(S_{i})_{i\in I}$ in
    $A_{sa}$, das bezüglich der starken Operatortopologie gegen $S$
    konvergiert. Die Funktion $f:\R\to\C$ definiert durch
    \begin{displaymath}
      f(t):=
      \begin{cases}
        -1\quad &t<-1, \\
        t\quad &t\in [-1,1], \\
        1\quad &t>1
      \end{cases}
    \end{displaymath}
    ist stetig und beschränkt mit $\norm[\infty]{f}=1$. Nach
    Lemma~\ref{lem:st-stetig-fkt} konvergiert das Netz
    $(f(S_{i}))_{i\in I}$ gegen $f(S)$. Wegen $\norm{S}\leq 1$ gilt
    $\sigma(S)\subseteq [-1,1]$ und somit
    \begin{displaymath}
      f(S)=f|_{\sigma(S)}(S)=(t\mapsto t)|_{\sigma(S)}(S)=S.
    \end{displaymath}
    Verbleibt noch zu zeigen, dass die Operatoren $f(S_{i})$ alle in
    $S_{A_{sa}}$ liegen. Wegen $f(0)=0$ gilt nach
    Bemerkung~\ref{bem:fktkalkuel} einmal $f(S_{i})\in A$. Da $f$
    reellwertig ist, sind die Operatoren $f(S_{i})$ auch
    selbstadjungiert. Schließlich gilt
    $\norm{f(S_{i})}\leq\norm[\infty]{f}=1$.

    Ist umgekehrt $S$ ein Element von $\cl{S_{A_{sa}}}^{\TT_{s}}$, so
    gilt analog zum letzten Punkt $S\in B_{sa}$. Weiters liegt $S$ in
    der Einheitskugel: Es existiert ein Netz $(S_{i})_{i\in I}$ aus
    $S_{A_{sa}}$, das bezüglich $\TT_{s}$ gegen $S$ konvergiert. Für
    jedes $i\in I$ und $x\in H$ gilt
    $\norm{S_{i}x}\leq\norm{x}$. Bildet man den Grenzwert $i\in I$, so
    folgt $\norm{Sx}\leq\norm{x}$, also $\norm{S}\leq 1$.
  \item Wir betrachten die Matrizenalgebren $M_{2}(A)$ und $M_{2}(B)$
    aus Bemerkung~\ref{bem:kart-prod-matrix-hr} und zeigen zunächst,
    dass $M_{2}(B)$ mit dem Abschluss von $M_{2}(A)$ bezüglich der
    starken Operatortopologie $\TT_{s,H^{2}}$ in $L_{b}(H^{2})$
    übereinstimmt. Für $\left(
      \begin{smallmatrix}
        T_{11} & T_{12}\\
        T_{21} & T_{22}
      \end{smallmatrix}
    \right)\in M_{2}(B)$ gibt es ein Netz $(T_{11,i})_{i\in I}$ aus
    $A$ mit $T_{11,i}\to T_{11}$ bezüglich $\TT_{s,H}$. Klarerweise
    konvergiert dann $\left(
      \begin{smallmatrix}
        T_{11,i} & T_{12}\\
        T_{21} & T_{22}
      \end{smallmatrix}
    \right)$ bezüglich $\TT_{s,H^{2}}$ gegen $\left(
      \begin{smallmatrix}
        T_{11} & T_{12}\\
        T_{21} & T_{22}
      \end{smallmatrix}
    \right)$, sodass es ausreicht zu zeigen, dass die Matrizen $\left(
      \begin{smallmatrix}
        T_{11,i} & T_{12}\\
        T_{21} & T_{22}
      \end{smallmatrix}
    \right)$ im $\TT_{s,H^{2}}$-Abschluss von $M_{2}(A)$ liegen. Nun
    wählt man ein gegen $T_{12}$ konvergentes Netz und argumentiert
    analog; nach zwei weiteren Schritten erhält man, dass es
    ausreicht,
    \begin{displaymath}
      \begin{pmatrix}
        T_{11,i} & T_{12,j}\\
        T_{21,k} & T_{22,l}
      \end{pmatrix}
      \in\cl{M_{2}(A)}^{\TT_{s,H^{2}}}
    \end{displaymath}
    für $T_{11,i},T_{12,j},T_{21,k},T_{22,l}\in A$ zu zeigen. Dies ist
    trivialerweise richtig. Ist umgekehrt $\left(
      \begin{smallmatrix}
        T_{11} & T_{12}\\
        T_{21} & T_{22}
      \end{smallmatrix}
    \right)$ im Abschluss von $M_{2}(A)$ enthalten, so sei
    $\left(\left(
        \begin{smallmatrix}
          T_{11,k} & T_{12,k}\\
          T_{21,k} & T_{22,k}
        \end{smallmatrix}
      \right)\right)_{k\in I}$ ein Netz aus $M_{2}(A)$, das bezüglich
    $\TT_{s,H^{2}}$ gegen diese Matrix konvergiert. Indem man die
    Vektoren $(x,0)^{T}$ bzw. $(0,y)^{T}$ aus $H^{2}$ einsetzt, sieht
    man, dass die Einträge $T_{ij,k}$ bezüglich der starken
    Operatortopologie $\TT_{s,H}$ gegen $T_{ij}$ konvergieren, womit
    $T_{ij}\in B$ folgt.

    Jetzt können wir die Aussage (\ref{item:kaplansky-iii})
    beweisen. Dazu sei $T\in S_{B}$ gegeben. Die Matrix $\left(
      \begin{smallmatrix}
        0 & T\\
        T^{*} & 0
      \end{smallmatrix}
    \right)$ liegt in $M_{2}(B)_{sa}$ und ist dort sogar in der
    Einheitskugel\footnote{Es sei explizit darauf hingewiesen, dass
      wir nicht mit \eqref{eq:kart-prod-matrix-hr-ii} argumentieren
      können.} enthalten:
    \begin{align*}
      \norm{\begin{pmatrix}
          0 & T \\
          T^{*} & 0
        \end{pmatrix}
                  \begin{pmatrix}
                    x \\
                    y
                  \end{pmatrix}}^{2}&= \norm{\begin{pmatrix}
                    Ty \\
                    T^{*}x
                  \end{pmatrix}}^{2}=\norm{Ty}^{2}+\norm{T^{*}x}^{2}\leq\norm{T}^{2}\norm{x}^{2}+\norm{T^{*}}^{2}\norm{y}^{2}\\
            &\leq\norm{x}^{2}+\norm{y}^{2}=\norm{\begin{pmatrix}
                x \\
                y
              \end{pmatrix}}^{2}.
    \end{align*}
    Nach (\ref{item:kaplansky-ii}), angewandt auf die Algebren
    $M_{2}(A)$ und $M_{2}(B)$, existiert ein Netz von Matrizen
    $\left(\left(
        \begin{smallmatrix}
          T_{11,k} & T_{12,k}\\
          T_{21,k} & T_{22,k}
        \end{smallmatrix}
      \right)\right)_{k\in I}$ aus der Einheitskugel von
    $M_{2}(A)_{sa}$, das gegen $\left(
      \begin{smallmatrix}
        0 & T\\
        T^{*} & 0
      \end{smallmatrix}
    \right)$ konvergiert. Mit der ersten Ungleichung in
    \eqref{eq:kart-prod-matrix-hr-ii} folgt $\norm{T_{12,k}}\leq 1$,
    also $T_{12,k}\in S_{A}$. Durch Einsetzen von $(x,0)^{T}$ ergibt
    sich, dass $(T_{12,k})_{k\in I}$ bezüglich $\TT_{s}$ gegen $T$
    konvergiert, womit $T\in \cl{S_{A}}^{\TT_{s}}$ gezeigt ist.

    Ist umgekehrt $T\in \cl{S_{A}}^{\TT_{s}}$, so folgt $T\in S_{B}$
    wie in (\ref{item:kaplansky-ii}).
  \end{enumerate}
\end{proof}

Nach all diesen Überlegungen können wir
Satz~\ref{satz:optop-abg-unteralg} beweisen.
\begin{proof}[Beweis (von Satz~\ref{satz:optop-abg-unteralg}).]
  Die Äquivalenz von (\ref{item:optop-abg-unteralg-i}) und
  (\ref{item:optop-abg-unteralg-ii}) folgt aus
  Korollar~\ref{kor:optop-konv-abg}. Wegen $\TT_{w}\subseteq\TT_{uw}$
  ist die Implikation
  (\ref{item:optop-abg-unteralg-ii})~$\Rightarrow$~(\ref{item:optop-abg-unteralg-iii})
  klar.

  Bleibt also
  (\ref{item:optop-abg-unteralg-iii})~$\Rightarrow$~(\ref{item:optop-abg-unteralg-ii})
  zu zeigen. Sei dazu $A$ abgeschlossen bezüglich $\TT_{uw}$. Wir
  setzen $B:=\cl{A}^{\TT_{w}}$ und schließen mit
  Korollar~\ref{kor:optop-konv-abg} auf $B=\cl{A}^{\TT_{s}}$. Da die
  ultraschwache Operatortopologie gröber ist als die Normtopologie,
  ist $A$ eine $C^{*}$-Unteralgebra von $L_{b}(H)$. Der Dichtesatz von
  Kaplansky liefert somit $\cl{S_{A}}^{\TT_{s}}=S_{B}$. Aus
  Korollar~\ref{kor:uw-optop-wstar}(\ref{item:uw-optop-wstar-i})
  folgt, dass die Einheitskugel $S$ von $L_{b}(H)$ bezüglich
  $\TT_{uw}$ kompakt und daher abgeschlossen ist. Damit ist auch
  $S_{A}=A\cap S$ bezüglich $\TT_{uw}$ abgeschlossen und als Teilmenge
  von $S$ folglich $\TT_{w}$-abgeschlossen;
  vgl. Korollar~\ref{kor:uw-optop-wstar}(\ref{item:uw-optop-wstar-ii}). Erneute
  Anwendung von Korollar~\ref{kor:optop-konv-abg} liefert die
  Abgeschlossenheit von $S_{A}$ bzgl. $\TT_{s}$. Daraus schließen wir
  $S_{A}=S_{B}$ und in weiterer Folge $A=B$, was genau bedeutet, dass
  $A$ bezüglich $\TT_{w}$ abgeschlossen ist.
\end{proof}

Wir schließen das Kapitel mit einer Definition ab.
\begin{definition}\label{def:vn-alg}
  Eine $*$-Unteralgebra von $L_{b}(H)$ heißt
  \emph{Von-Neumann-Algebra}, wenn sie be\-züglich einer und damit
  aller der drei Operatortopologien $\TT_{s},\TT_{w},\TT_{uw}$
  abgeschlossen ist.
\end{definition}
Es sei betont, dass Von-Neumann-Algebren immer $C^{*}$-Unteralgebren
von $L_{b}(H)$ sind. Die Stärke und Schönheit der Theorie der
Von-Neumann-Algebren liegt darin, dass man in vielen Fällen beliebig
zwischen den drei Operatortopologien wechseln kann, je nachdem, welche
in der aktuellen Situation die nützlichsten Eigenschaften hat. Im
Beweis des Dichtesatzes von Kaplansky kann man dieses Wechselspiel gut
beobachten.

\clearpage
\chapter{$W^{*}$-Algebren}
\label{cha:w-algebren}
Im vorliegenden Kapitel wird mit den \emph{$W^{*}$-Algebren} die
hilbertraumfreie Axiomatisierung der Von-Neumann-Algebren eingeführt
sowie einige ihrer Eigenschaften bewiesen. Beispielsweise haben
$W^{*}$-Algebren stets ein Einselement. Als wesentliche neue
Definition ist die schwach-*-Topologie auf einer $W^{*}$-Algebra zu
nennen, die als kanonische Topologie eine entscheidende Rolle in der
restlichen Arbeit spielen wird. Den Großteil dieses Kapitels machen
Aussagen zur schwach-*-Stetigkeit der Algebrenoperationen aus.

\section{Definition}
\label{sec:def}
\begin{bemerkung}\label{bem:motiv-wstar}
  Um das Axiom, das $W^{*}$-Algebren unter $C^{*}$-Algebren
  auszeichnet, zu motivieren, betrachten wir eine Von-Neumann-Algebra
  $A\leq L_{b}(H)$. Der Raum $L_{b}(H)$ ist nach
  Satz~\ref{satz:beschrop-spklasse} isometrisch isomorph zum Dualraum
  $L^{1}(H)'$. Als Von-Neumann-Algebra ist $A$ abgeschlossen bezüglich
  $\TT_{uw}$. Betrachtet man den linearen Homöomorphismus
  \begin{displaymath}
    \theta:(L_{b}(H),\TT_{uw})\to\big(L^{1}(H)',\sigma(L^{1}(H)',L^{1}(H))\big)
  \end{displaymath}
  aus Satz~\ref{satz:uw-optop-wstar}, so stellt sich $\theta(A)$ als
  schwach-*-abgeschlossen in $L^{1}(H)'$ heraus. Als Linksannihilator
  bezüglich $(L^{1}(H),L^{1}(H)')$ ist
  \begin{align*}
    M:=\lanh{(\theta(A))}&=\set{S\in L^{1}(H)}{\chevr{S}{\theta(T)}=0\
                           \text{für alle}\ T\in A}\\
                         &=\set{S\in L^{1}(H)}{\tr(ST)=0\
                           \text{für alle}\ T\in A}
  \end{align*}
  abgeschlossen bezüglich der schwachen Topologie
  $\sigma\left(L^{1}(H),L^{1}(H)'\right)$ und daher bezüglich der
  Normtopologie auf $L^{1}(H)$. Bekanntermaßen ist in dieser Situation
  \begin{equation}\label{eq:motiv-wstar}
    \tau:
    \begin{cases}
      \left(L^{1}(H)/M\right)'&\to\ranh{M} \\
      \hfill f&\mapsto f\circ\pi
    \end{cases}
  \end{equation}
  ein isometrischer Isomorphismus, wobei $\pi:L^{1}(H)\to L^{1}(H)/M$
  die Faktorisierungsabbildung bezeichnet und der Dualraum
  $\left(L^{1}(H)/M\right)'$ bezüglich der Faktorraum-Norm auf
  $L^{1}(H)/M$ gebildet wird. Nun gilt
  \begin{displaymath}
    \ranh{M}=\ranh{\left(\lanh{\theta(A)}\right)}=\cl{\spn\theta(A)}^{\sigma\left(L^{1}(H)',L^{1}(H)\right)}=\theta(A),
  \end{displaymath}
  da $\theta(A)$ schwach-*-abgeschlossen ist. Wir erhalten, dass
  $\tau^{-1}\circ\theta|_{A}$ ein Isomorphismus von $A$ auf
  $\left(L^{1}(H)/M\right)'$ ist, der wegen
  Satz~\ref{satz:beschrop-spklasse}(\ref{item:beschrop-spklasse-iii})
  sogar isometrisch ist. Also ist $A$ bis auf isometrische Isomorphie
  der Dualraum eines Banachraums.
\end{bemerkung}

Genau diese Eigenschaft ist nun das gesuchte Axiom, um
Von-Neumann-Algebren zu beschreiben.

\begin{definition}\label{def:wstar-alg}
  Eine $C^{*}$-Algebra $A$ heißt \emph{$W^{*}$-Algebra}, wenn es einen
  Banachraum $\pred{A}$ und einen isometrischen Isomorphismus
  $j:A\to\pred{A}'$ gibt. In diesem Fall heißt $\pred{A}$ ein
  \emph{Prädualraum} von $A$.
\end{definition}
\begin{bemerkung}
  Unsere Notation $\pred{A}$ ist etwas problematisch, da zumindest ad
  hoc nicht klar ist, ob $\pred{A}$ in geeignetem Sinne eindeutig
  bestimmt ist. Andernfalls wäre der Prädual\-raum nicht in Abhängigkeit
  von $A$ wohldefiniert. Tatsächlich ist der Prädualraum einer
  $W^{*}$-Algebra bis auf isometrische Isomorphie eindeutig, das
  werden wir aber erst in Satz~\ref{satz:pd-wstar-top-eindeut}
  beweisen können.
\end{bemerkung}
Das Axiom aus Definition~\ref{def:wstar-alg} ermöglicht die
Konstruktion einer Topologie auf der $W^{*}$-Algebra $A$, nämlich die
initiale Topologie bezüglich $j$ als Abbildung von $A$ nach
$\pred{A}'$ versehen mit der schwach-*-Topologie
$\sigma\big(\pred{A}',\pred{A}\big)$.

\begin{definition}\label{def:wstar-top-wstar}
  Das Mengensystem
  $\TT_{w^{*},A}:=\set{j^{-1}(O)}{O\in\sigma\big(\pred{A}',\pred{A}\big)}$
  bezeichnet man als \emph{schwach-*-Topologie auf $A$}. Wenn keine
  Verwechslungsgefahr besteht, werden wir zur einfacheren Notation nur
  $\TT_{w^{*}}$ schreiben.
\end{definition}

\begin{notation}
  Sei $A$ im Rest dieses Kapitels stets eine $W^{*}$-Algebra und $S$
  ihre Einheitskugel. Wenn nichts Näheres spezifiziert wird, beziehen
  sich sämtliche topologische Aussagen dabei auf die
  schwach-*-Topologie $\TT_{w^{*}}$.
\end{notation}

\begin{bemerkung}\label{bem:wstar-top}
  \hspace{0mm}
  \begin{enumerate}[label=(\roman*),ref=\roman*]
  \item\label{item:wstar-top-i} Mit der Notation aus
    Definition~\ref{def:wstar-alg} kann man $\pred{L_{b}(H)}=L^{1}(H)$
    wählen. Für diesen Prädualraum, gemeinsam mit dem isometrischen
    Isomorphismus $j:=\theta:L_{b}(H)\to L^{1}(H)'$ aus
    Satz~\ref{satz:beschrop-spklasse}(\ref{item:beschrop-spklasse-iii}),
    stimmt die schwach-*-Topologie $\TT_{w^{*},L_{b}(H)}$ nach
    Satz~\ref{satz:uw-optop-wstar} mit der ultraschwachen
    Operatortopologie $\TT_{uw}$ überein.
  \item\label{item:wstar-top-ii} Da $j$ bijektiv ist, ist
    \begin{displaymath}
      j:(A,\TT_{w^{*}})\to
      \left(\pred{A}',\sigma\big(\pred{A}',\pred{A}\big)\right)
    \end{displaymath}
    ein linearer Homöomorphismus. Somit erfüllt die
    schwach-*-Topologie $\TT_{w^{*}}$ alle Eigenschaften, die ein
    Banachraum als topologischer Vektorraum versehen mit der bekannten
    schwach-*-Topologie hat. Beispielsweise gilt der Satz von
    Krein-Milman sowie Lemma~\ref{lem:proj-wstar-stetig} über die
    schwach-*-Stetigkeit von Projektionen, außerdem ist $\TT_{w^{*}}$
    die initiale Topologie bezüglich aller $\TT_{w^{*}}$-stetigen
    linearen Funktionale auf $A$. Da $j$ isometrisch ist, bleiben auch
    alle Eigenschaften, die zusätzlich mit Einheitskugeln operieren,
    erhalten. Hervorzuheben sind hier einerseits der Satz von
    Banach-Alaoglu und andererseits die Sätze von Krein-Smulian und
    Banach-Dieudonné, Satz~\ref{satz:krein-sm}
    bzw. Satz~\ref{satz:banach-dieud}.
  \item\label{item:wstar-top-iii} Da
    \begin{displaymath}
      j:(A,\TT(\norm[A]{\cdot}))\to\big(\pred{A}',\TT(\norm[\pred{A}']{\cdot})\big)
    \end{displaymath}
    als Isometrie stetig ist, ist wegen
    $\sigma\big(\pred{A}',\pred{A}\big)\subseteq\TT(\norm[\pred{A}']{\cdot})$
    auch
    \begin{displaymath}
      j:(A,\TT(\norm[A]{\cdot}))\to\big(\pred{A}',\sigma\big(\pred{A}',\pred{A}\big)\big)
    \end{displaymath}
    stetig. Folglich gilt $\TT_{w^{*}}\subseteq\TT(\norm[A]{\cdot})$
    und ein schwach-*-stetiges Funktional $\varphi:A\to\C$ ist
    automatisch beschränkt.
  \end{enumerate}
\end{bemerkung}

\section{Einselement}
\label{sec:einselement}
Lässt man einige Sätze der bisherigen Arbeit zusammenwirken und
kombiniert sie mit klassischen Resultaten der Funktionalanalysis, so
erhalten wir folgende Aussage.
\begin{satz}\label{satz:wstar_eins}
  Eine $W^{*}$-Algebra hat stets ein Einselement.
\end{satz}
\begin{proof}
  Nach dem Satz von Banach-Alaoglu ist die Einheitskugel $S$ kompakt
  bezüglich der schwach-*-Topologie. Wir können daher den Satz von
  Krein-Milman im Vektorraum $(A,\TT_{w^{*}})$ auf die kompakte und
  konvexe Menge $S$ anwenden und erhalten, dass $S$ die abgeschlossene
  konvexe Hülle ihrer Extremalpunkte ist. Insbesondere hat $S$
  Extremalpunkte, sodass Satz~\ref{satz:extremalpkt} die gewünschte
  Aussage liefert.
\end{proof}

\section{Stetigkeit der Operationen}
\label{sec:stetigkeit-operationen}
Aus
Satz~\ref{satz:adj-mult-stetig-optop}(\ref{item:adj-mult-stetig-optop-iv})
wissen wir, dass die Adjungiertenbildung sowie die Translationen
stetig bezüglich $\TT_{uw}$ sind. Betrachten wir $L_{b}(H)$ als
$W^{*}$-Algebra, so ist die ultraschwache Operatortopologie genau die
schwach-*-Topologie. Daher ist es nicht abwegig zu erwarten, dass
diese algebraischen Operationen auch in einer allgemeinen
$W^{*}$-Algebra $A$ stetig bezüglich $\TT_{w^{*}}$ sind. Das restliche
Kapitel ist dem Beweis dieser Aussagen gewidmet. Dabei sind der
Satz~\ref{satz:banach-dieud} von Banach-Dieudonné und
Lemma~\ref{lem:proj-wstar-stetig} über die Stetigkeit von Projektionen
die wesentlichen Hilfsmittel.

Wir beginnen mit der Adjungiertenbildung.
\begin{lemma}\label{lem:sa-pos-wstar-abg}
  $A_{sa}$ und $A^{+}$ sind $\TT_{w^{*}}$-abgeschlossen.
\end{lemma}
\begin{proof}
  Nach dem Satz von Banach-Dieudonné, Satz~\ref{satz:banach-dieud},
  genügt es, die $\TT_{w^{*}}$-Abgeschlossenheit von $A_{sa}\cap S$
  und $A^{+}\cap S$ zu zeigen.

  Angenommen, $A_{sa}\cap S$ ist nicht abgeschlossen. Dann existiert
  ein Netz $(x_{j})_{j\in I}$ selbstadjungierter Elemente in $S$, das
  bezüglich $\TT_{w^{*}}$ gegen $x\notin A_{sa}\cap S$ konvergiert. Da
  $S$ schwach-*-abgeschlossen ist, gilt $x\notin A_{sa}$. Schreibt man
  $x=a+ib$, $a,b\in A_{sa}$, so gilt $b\neq 0$ bzw. $r(b)>0$. Indem
  wir notfalls zu $(-x_{j})_{j\in I}$ und $-x$ übergehen, können wir
  annehmen, dass es ein positives $0<\lambda\in\sigma(b)$ gibt. Durch
  Quadrieren sieht man, dass für hinreichend großes $n\in\N$ die
  Ungleichung $(1+n^{2})^{1/2}<\lambda+n$ gilt, womit sich wegen
  $\|x_{j}^{2}\|\leq 1$ die Ungleichungskette
  \begin{align*}
    \norm{x_{j}+in}&=\norm{(x_{j}+in)^{*}(x_{j}+in)}^{1/2}=\norm{x_{j}^{2}+n^{2}}^{1/2}\leq
                     (1+n^{2})^{1/2}\\
                   &<\underbrace{\lambda+n}_{\in\sigma(b+n)}\leq
                     r(b+n)=\norm{b+n}\stackrel{(*)}{\leq}\norm{a+ib+in}
  \end{align*}
  ergibt. Die mit $(*)$ gekennzeichnete Ungleichung folgt dabei aus
  Lemma~\ref{lem:eig-re-im}. Wir erhalten also
  \begin{equation}\label{eq:sa-pos-wstar-abg-i}
    \norm{x_{j}+in}\leq (1+n^{2})^{1/2}<\norm{a+ib+in}.
  \end{equation}
  Das Netz $(x_{j}+in)_{j\in I}$ liegt wegen
  \eqref{eq:sa-pos-wstar-abg-i} in der kompakten und somit
  abgeschlossenen Menge $(1+n^{2})^{1/2}S$. Folglich ist auch
  $a+ib+in$ als Grenzwert des Netzes in $(1+n^{2})^{1/2}S$
  enthalten. Dies widerspricht aber dem zweiten Teil von
  \eqref{eq:sa-pos-wstar-abg-i}, womit $A_{sa}\cap S$ abgeschlossen
  sein muss.

  Für den zweiten Teil sei bemerkt, dass für $a\in S$ das Spektrum
  $\sigma(a)\subseteq K_{1}^{\C}(0)$ erfüllt. Somit ist $a$ genau dann
  positiv, wenn $\sigma(a)$ in $[0,1]$ enthalten ist. Nach dem
  Spektralabbildungssatz ist das äquivalent zu
  $\sigma(a-1)\subseteq[-1,0]$. Wegen $\norm{a}=r(a)$ sowie
  $\norm{a-1}=r(a-1)$ sind sowohl $a$ als auch $a-1$ im Falle der
  Positivität von $a$ daher in $A_{sa}\cap S$ enthalten. Gilt
  umgekehrt $a,a-1\in A_{sa}\cap S$, so folgt
  $\sigma(a)\subseteq [-1,1]\cap([-1,1]+1)=[0,1]$, womit $a$ positiv
  ist. Insgesamt gilt
  \begin{displaymath}
    A^{+}\cap S=(A_{sa}\cap S)\cap ((A_{sa}\cap S)+1),
  \end{displaymath}
  sodass die Abgeschlossenheit von $A^{+}\cap S$ aus dem ersten
  Beweisteil folgt.
\end{proof}
\begin{korollar}\label{kor:wstar-adj-stetig}
  Die Operation $.^{*}$ auf $A$ ist $\TT_{w^{*}}$-stetig.
\end{korollar}
\begin{proof}
  Wegen der Zerlegung in Real- und Imaginärteil gilt
  $A=A_{sa}+iA_{sa}$. Betrachten wir $A$ als reellen Vektorraum, so
  liegt eine Summe von Unterräumen vor. Diese Summe ist sogar direkt,
  denn sind $a,b$ selbstadjungiert mit $a=ib$, dann folgt
  $a=a^{*}=(ib)^{*}=-ib=-a$, also $a=0$. Die Funktion $x\mapsto\re(x)$
  ist dabei genau die Projektion auf die erste Komponente dieser
  direkten Zerlegung. Die Räume $\ran\re(.)=A_{sa}$ und
  $\ker\re(.)=iA_{sa}$ sind abgeschlossen, sodass $\re(.)$ nach
  Lemma~\ref{lem:proj-wstar-stetig} und Bemerkung~\ref{bem:reeller-br}
  stetig ist. Wegen $x^{*}=2\re(x)-x$ ist damit auch $.^{*}$ stetig.
\end{proof}
Aus Lemma~\ref{lem:sa-pos-wstar-abg} erhalten wir ein weiteres
Korollar, das zeigt, dass die schwach-*-Topologie auf $A$ mit der
$C^{*}$-Algebren-Struktur gut harmoniert. Es sei auf die enge Analogie
zu Satz~\ref{satz:pos-fkt-trennend} hingewiesen.
\begin{korollar}\label{kor:wstar-pos-fkt-trennend}
  \hspace{0mm}
  \begin{enumerate}[label=(\roman*),ref=\roman*]
  \item\label{item:wstar-pos-fkt-trennend-i} Ist $a$ ein
    selbstadjungiertes aber nicht positives Element von $A$, so gibt
    es ein $\TT_{w^{*}}$-stetiges, positives Funktional $\varphi$ mit
    $\varphi(a)<0$.
  \item\label{item:wstar-pos-fkt-trennend-ii} Ist $b\in A$ mit
    $\varphi(b)=0$ für alle $\TT_{w^{*}}$-stetigen, positiven
    Funktionale auf $A$, dann gilt $b=0$.
  \end{enumerate}
\end{korollar}
\begin{proof}
  \hspace{0mm}
  \begin{enumerate}[label=(\roman*),ref=\roman*]
  \item Die positiven Elemente $A^{+}$ bilden eine konvexe und
    abgeschlossene Teilmenge des lokalkonvexen, reellen Vektorraums
    $(A_{sa},\TT_{w^{*}}\cap A_{sa})$. Also folgt für
    $a\in A_{sa}\setminus A^{+}$ aus dem Satz von Hahn-Banach die
    Existenz eines $\lambda\in\R$ und eines bezüglich der
    Spurtopologie schwach-*-stetigen $\R$-linearen Funktionals $f$ auf
    $A_{sa}$ mit
    \begin{equation}\label{eq:wstar-pos-fkt-trennend-i}
      f(a)<\lambda<f(c)\quad\text{für alle}\,\, c\in A^{+}.
    \end{equation}
    Wäre $f(c_{0})<0$ für ein $c_{0}\in A^{+}$, so würde $f$ auf
    $A^{+}$ wegen $tc_{0}\in A^{+}$ für jedes $t>0$ beliebig kleine
    Werte annehmen, was \eqref{eq:wstar-pos-fkt-trennend-i}
    widerspricht. Es gilt also $f(c)\geq 0$ für alle $c\in A^{+}$.
    Zudem ist $0\in A^{+}$, womit wir $f(a)<\lambda<f(0)=0$
    erhalten. Definiert man nun $\varphi(x):=f(\re(x))+if(\im(x))$, so
    rechnet man nach, dass $\varphi$ ein $\C$-lineares Funktional auf
    $A$ ist. Wegen Korollar~\ref{kor:wstar-adj-stetig} ist $\varphi$
    auch $\TT_{w^{*}}$-stetig und wegen $\varphi(c)=f(c)\geq 0$ für
    $c\in A^{+}\subseteq A_{sa}$ positiv. Schließlich gilt
    $\varphi(a)=f(a)<0$.
  \item Ist $\varphi$ ein $\TT_{w^{*}}$-stetiges, positives
    Funktional, so gilt $\varphi(b^{*})=\overline{\varphi(b)}=0$ nach
    Lemma~\ref{lem:pos-fkt-konj}. Daraus folgt
    $\varphi(\pm\re(b))=\pm\frac{1}{2}(\varphi(b)+\varphi(b^{*}))=0$. Nach
    (\ref{item:wstar-pos-fkt-trennend-i}) müssen $\re(b)$ und
    $-\re(b)$ positiv sein, sodass sich $\re(b)=0$ ergibt. Für
    $\im(b)$ schließt man analog und erhält tatsächlich $b=0$.
  \end{enumerate}
\end{proof}

Bevor wir uns der Multiplikation zuwenden, zeigen wir, dass in einer
$W^{*}$-Algebra die lineare Hülle der Projektionen dicht bezüglich der
Normtopologie und damit auch bezüglich der schwach-*-Topologie
ist. Dazu verwenden wir ein Resultat, das gleichzeitig in Richtung
Stetigkeit der Multiplikation weist.
\begin{satz}\label{satz:wstar-sup-stetig}
  Ist $(x_{i})_{i\in I}$ ein monoton wachsendes und gleichmäßig
  beschränktes Netz in $A_{sa}$, so ist es bezüglich $\TT_{w^{*}}$
  konvergent und es gilt $\lim_{i\in I}x_{i}=\sup_{i\in
    I}x_{i}$. Insbesondere existiert dieses Supremum. Für beliebiges
  $c\in A$ gilt zudem
  \begin{displaymath}
    c^{*}(\sup_{i\in I}x_{i})c=c^{*}(\lim_{i\in I}x_{i})c=\lim_{i\in
      I}c^{*}x_{i}c=\sup_{i\in I}c^{*}x_{i}c.
  \end{displaymath}
\end{satz}
\begin{proof}
  Sei $E$ die lineare Hülle der $\TT_{w^{*}}$-stetigen, positiven
  Funktionale. Nach
  Korollar~\ref{kor:wstar-pos-fkt-trennend}(\ref{item:wstar-pos-fkt-trennend-ii})
  ist $E$ punktetrennend, womit die schwache Topologie $\sigma(A,E)$
  wohldefiniert ist. Klarerweise ist $\sigma(A,E)$ gröber als
  $\TT_{w^{*}}$, sodass die beiden Topologien analog zum Beweis von
  Korollar~\ref{kor:uw-optop-wstar}(\ref{item:uw-optop-wstar-ii}) auf
  $\TT_{w^{*}}$-Kompakta übereinstimmen\footnote{In
    Satz~\ref{satz:wstar-fkt-zerl} werden wir zeigen, dass $E$ mit der
    Menge \emph{aller} schwach-*-stetigen Funktionale
    übereinstimmt. Es wird sich außerdem herausstellen, dass
    $\sigma(A,E)$ genau die schwach-*-Topologie ist; siehe den Beweis
    von Satz~\ref{satz:pd-wstar-top-eindeut}.}. Insbesondere gilt dies
  für jenes hinreichend große Vielfache der Einheitskugel, in dem alle
  $x_{i}$ nach Voraussetzung enthalten sind. Um zu zeigen, dass
  $(x_{i})_{i\in I}$ bezüglich $\TT_{w^{*}}$-konvergent ist, genügt es
  wegen der Kompaktheit zu zeigen, dass es sich um ein
  $\TT_{w^{*}}$-Cauchy-Netz bzw. äquivalent um ein
  $\sigma(A,E)$-Cauchy-Netz handelt. Bekanntlich bilden die Mengen
  \begin{displaymath}
    \set{x\in A}{\abs{\chevr{x}{\psi_{k}}}<\epsilon\ \text{für alle}\ k=1,\dots,n},
  \end{displaymath}
  wobei $\epsilon>0$, $n\in\N$ und $\psi_{1},\dots,\psi_{n}\in E$
  beliebig zu wählen sind, eine Nullumgebungsbasis von
  $\sigma(A,E)$. Ist eine derartige Menge gegeben, so haben wir einen
  Index $i_{0}\in I$ zu finden mit
  \begin{equation}\label{eq:bew-wstar-sup-stetig-i}
    \abs{\chevr{x_{i}-x_{j}}{\psi_{k}}}<\epsilon \quad\text{für alle}\ i,j\succcurlyeq i_{0}\ \text{und alle}\ k=1,\dots,n.
  \end{equation}
  Es gilt
  $\psi_{k}=\sum_{\ell=1}^{m_{k}}\lambda_{\ell,k}\varphi_{\ell,k}$ für
  Konstanten $\lambda_{\ell,k}\in\C$ und $\TT_{w^{*}}$-stetige
  positive Funktionale $\varphi_{\ell,k}$. Dabei können wir
  $\lambda_{\ell,k}\neq 0$ und $m_{k}\geq 1$ annehmen, da das
  Nullfunktional stets $\abs{\chevr{x}{0}}=0<\epsilon$ erfüllt. Wegen
  \begin{displaymath}
    \bigcap_{k=1}^{n}\bigcap_{\ell=1}^{m_{k}}\set{x\in
      A}{\abs{\chevr{x}{\varphi_{\ell,k}}}<\frac{\epsilon}{m_{k}\abs{\lambda_{\ell,k}}}}\subseteq\set{x\in
      A}{\abs{\chevr{x}{\psi_{k}}}<\epsilon\ \text{für alle}\ k=1,\dots,n}
  \end{displaymath}
  genügt es, die Aussage für die $\varphi_{\ell,k}$ anstelle der
  $\psi_{k}$ zu zeigen. Anders formuliert können wir annehmen, dass
  die Funktionale $\psi_{k}$ bereits positiv sind. Für festes $k$ ist
  daher das reelle Netz $(\chevr{x_{i}}{\psi_{k}})_{i\in I}$ monoton
  wachsend, beschränkt und folglich konvergent. Insbesondere ist es
  ein Cauchy-Netz, sodass für jedes $k$ ein $i_{0,k}\in I$ existiert,
  für das \eqref{eq:bew-wstar-sup-stetig-i} sinngemäß gilt. Wählen wir
  nun gemäß der Richtungseigenschaft noch einen Index
  $i_{0}\succcurlyeq i_{0,1},\dots,i_{0,n}$, so folgt
  \eqref{eq:bew-wstar-sup-stetig-i}. Damit ist einmal die Existenz von
  $x:=\lim_{i\in I}x_{i}$ gezeigt. Wegen der
  $\TT_{w^{*}}$-Abgeschlossenheit von $A_{sa}$ ist $x$
  selbstadjungiert. Für beliebiges $i_{0}\in I$ gilt
  $x_{i}\geq x_{i_{0}}$ für $i\succcurlyeq i_{0}$, also
  $x_{i}-x_{i_{0}}\in A^{+}$. Daraus folgt
  $x-x_{i_{0}}=\displaystyle{\lim_{i\in I, i\succcurlyeq
      i_{0}}x_{i}-x_{i_{0}}}\in A^{+}$ aufgrund der Abgeschlossenheit
  von $A^{+}$, sodass wir auf $x\geq x_{i_{0}}$ schließen können. Der
  Grenzwert $x$ ist somit eine obere Schranke von $(x_{i})_{i\in
    I}$. Ist $y$ eine weitere obere Schranke, so gilt
  $y-x_{i}\in A^{+}$ für alle $i\in I$. Wie eben erhalten wir daraus
  $x\leq y$. Somit ist $x=\lim_{i\in I}x_{i}$ sogar die kleinste obere
  Schranke von $(x_{i})_{i\in I}$ und es folgt
  $\lim_{i\in I}x_{i}=\sup_{i\in I}x_{i}$.

  Sei jetzt $c\in A$ gegeben. Das Netz $(c^{*}x_{i}c)_{i\in I}$ ist
  wegen Lemma~\ref{lem:eig-ho}(\ref{item:eig-ho-ii}) monoton wachsend
  und offenbar gleichmäßig beschränkt. Aus dem ersten Teil des
  Beweises folgt die Konvergenz
  \begin{equation}\label{eq:bew-wstar-sup-stetig-ii}
    \lim_{i\in I}c^{*}x_{i}c=\sup_{i\in I}c^{*}x_{i}c.
  \end{equation}
  Der Beweis ist abgeschlossen, wenn wir
  \begin{displaymath}
    c^{*}(\sup_{i\in I}x_{i})c=\sup_{i\in
                                I}c^{*}x_{i}c
                              \qquad\textit{oder}\qquad
                              c^{*}(\lim_{i\in I}x_{i})c=\lim_{i\in
                                I}c^{*}x_{i}c
  \end{displaymath}
  zeigen können.

  Im Falle eines invertierbaren $c$ zeigen wir die Gleichheit der
  Supremumsausdrücke. Aus Lemma~\ref{lem:eig-ho}(\ref{item:eig-ho-ii})
  folgt die Ungleichung $c^{*}x_{i}c\leq c^{*}(\sup_{j\in I}x_{j})c$
  und damit
  \begin{equation}\label{eq:bew-wstar-sup-stetig-iii}
    \sup_{i\in I}c^{*}x_{i}c\leq c^{*}(\sup_{j\in I}x_{j})c.
  \end{equation}
  Analog erhalten wir
  \begin{displaymath}
    \sup_{i\in I}x_{i}=\sup_{i\in
      I}(c^{-1})^{*}(c^{*}x_{i}c)c^{-1}\leq (c^{-1})^{*}(\sup_{i\in I}c^{*}x_{i}c)c^{-1}.
  \end{displaymath}
  Eine weitere Anwendung von
  Lemma~\ref{lem:eig-ho}(\ref{item:eig-ho-ii}) liefert
  \begin{equation}\label{eq:bew-wstar-sup-stetig-iv}
    c^{*}(\sup_{i\in I}x_{i})c\leq c^{*}\left((c^{-1})^{*}(\sup_{i\in I}c^{*}x_{i}c)c^{-1}\right)c=\sup_{i\in I}c^{*}x_{i}c.
  \end{equation}
  Die Ungleichungen \eqref{eq:bew-wstar-sup-stetig-iii} und
  \eqref{eq:bew-wstar-sup-stetig-iv} ergeben das Resultat für diesen
  Spezialfall.

  Im Falle eines nicht invertierbaren $c$ beweisen wir die Gleichheit
  der Limesausdrücke. Für hinreichend großes $\lambda>0$ ist
  $c+\lambda$ invertierbar, beispielsweise für $\lambda=\norm{c}+1$.
  Aus dem Bisherigen folgt die Gleichung
  \begin{displaymath}
    \lim_{i\in
      I}(c+\lambda)^{*}x_{i}(c+\lambda)=\sup_{i\in
      I}(c+\lambda)^{*}x_{i}(c+\lambda)=(c+\lambda)^{*}(\sup_{i\in I}x_{i})(c+\lambda)=(c+\lambda)^{*}x(c+\lambda)
  \end{displaymath}
  bzw., wenn man den ersten und letzten Ausdruck ausmultipliziert,
  \begin{displaymath}
    \lim_{i\in
      I}c^{*}x_{i}c+\lambda(x_{i}c+c^{*}x_{i})+\lambda^{2}x_{i}=c^{*}xc+\lambda(xc+c^{*}x)+\lambda^{2}x.
  \end{displaymath}
  Der Term $\lambda^{2}x_{i}$ konvergiert gegen $\lambda^{2}x$, sodass
  die Aussage folgt, wenn wir $x_{i}c+c^{*}x_{i}\to xc+c^{*}x$ gezeigt
  haben. Wie am Anfang des Beweises genügt es, statt der
  schwach-*-Topologie $\TT_{w^{*}}$ die schwache Topologie
  $\sigma(A,E)$ zu betrachten, wobei wir auch diesmal
  $\varphi(x_{i}c+c^{*}x_{i})\to\varphi(xc+c^{*}x)$ nur für alle
  $\TT_{w^{*}}$-stetigen und positiven Funktionale $\varphi$ zeigen
  müssen. Wegen der Cauchy-Schwarz'schen Ungleichung,
  Lemma~\ref{lem:eig-pos-fkt}(\ref{item:eig-pos-fkt-iii}), gilt
  \begin{align*}
    \abs{\varphi(c^{*}(x-x_{i}))}&=\abs{\varphi\big(c^{*}(x-x_{i})^{1/2}(x-x_{i})^{1/2}\big)}=\abs{\varphi\big(((x-x_{i})^{1/2}c)^{*}(x-x_{i})^{1/2}\big)}\\
                                 &\leq\varphi\left(c^{*}(x-x_{i})^{1/2}(x-x_{i})^{1/2}c\right)^{1/2}\varphi\left(((x-x_{i})^{1/2})^{*}(x-x_{i})^{1/2}\right)^{1/2}\\
                                 &=\varphi\left(c^{*}(x-x_{i})c\right)^{1/2}\varphi(x-x_{i})^{1/2}\leq
                                   (\norm{\varphi}M)^{1/2}\varphi(x-x_{i})^{1/2}\to 0,
  \end{align*}
  weil $(c^{*}(x-x_{i})c)_{i\in I}$ gleichmäßig durch eine Konstante
  $M$ beschränkt und $\varphi$ schwach-*-stetig ist. Daraus folgt
  $\varphi(c^{*}x_{i})\to\varphi(c^{*}x)$. Konjugiert man beide
  Seiten, so ergibt sich wegen der Stetigkeit von $.^{*}$ kombiniert
  mit Lemma~\ref{lem:pos-fkt-konj} die Konvergenz
  $\varphi(x_{i}c)\to\varphi(xc)$ und damit die gewünschte Aussage.
\end{proof}
\begin{korollar}\label{kor:wstar-proj-dicht}
  Die lineare Hülle der Projektionen in $A$ ist dicht bezüglich
  $\TT(\norm{\cdot})$. Ist $a\in A$ positiv, so kann man $a$ sogar als
  $\norm{\cdot}$-Grenzwert von Elementen der Form
  $a_{n}:=\sum_{k=1}^{n}\gamma_{k,n}p_{k,n}$ darstellen, wobei die
  $p_{k,n}$ Projektionen und die Zahlen $\gamma_{k,n}$ nichtnegativ
  sind.
\end{korollar}
\begin{proof}
  Klarerweise genügt es zu zeigen, dass jedes selbstadjungierte
  Element $a\in A_{sa}$ in der $\TT(\norm{\cdot})$-abgeschlossenen
  linearen Hülle der Projektionen in $A$ enthalten ist. Sei dazu $C$
  eine bezüglich Mengeninklusion maximale kommutative
  $C^{*}$-Unteralgebra mit Einselement von $A$, die $a$ enthält. Diese
  existiert nach dem Lemma von Zorn: In der Tat ist die von $a$
  erzeugte $C^{*}$-Unteralgebra mit Einselement von $A$
  kommutativ. Folglich ist das Mengensystem $\mathcal{S}$ aller $a$
  enthaltenden kommutativen $C^{*}$-Unteralgebren mit Einselement
  nicht leer. Ist $\MM\subseteq\mathcal{S}$ eine Kette, so ist
  $\cl{\bigcup\MM}^{\TT(\norm{\cdot})}$ wieder in $\mathcal{S}$
  enthalten und klarerweise eine obere Schranke von $\MM$.

  Ist $\BB$ eine beschränkte und durch die kanonische Halbordnung
  $\leq$ gerichtete Teilmenge von $C_{sa}$, so sei
  $b_{0}:=\sup_{b\in\BB}b$ ihr Supremum; dieses existiert wegen
  Satz~\ref{satz:wstar-sup-stetig}. Für ein unitäres Element $u\in C$
  betrachten wir das Netz $(b)_{b\in\BB}$ und schließen aus
  Satz~\ref{satz:wstar-sup-stetig}
  \begin{displaymath}
    u^{*}b_{0}u=\sup_{b\in\BB}\underbrace{u^{*}bu}_{=u^{*}ub=b}=b_{0},
  \end{displaymath}
  da $u$ mit $b$ kommutiert. Anders formuliert erhalten wir
  $b_{0}u=ub_{0}$. Unitäre Elemente von $C$ kommutieren also mit
  $b_{0}$. Da die unitären Elemente nach
  Bemerkung~\ref{bem:pos-negteil-span}(\ref{item:pos-negteil-span-iii})
  ganz $C$ linear aufspannen, folgt, dass $b_{0}$ mit jedem Element
  von $C$ kommutiert. Somit ist auch $C^{*}_{A}(C\cup\{b_{0}\})$
  kommutativ. Wegen der Maximalität muss schon $b_{0}\in C$
  gelten. Übersetzt man diese Überlegung mit der Gelfandtransformation
  in den Gelfandraum $M$ von $C$, so folgt, dass jede beschränkte und
  bezüglich der punktweisen Ordnung gerichtete Teilmenge von $C(M,\R)$
  ein Supremum in $C(M,\R)$ hat. Man beachte hierzu, dass die Ordnung
  $\leq$ in $C_{sa}$ der punktweisen Ordnung im Raum $C(M,\R)$
  entspricht. Nach Satz~\ref{satz:char-stonesch} ist $M$ stonesch. Aus
  Satz~\ref{satz:proj-dicht} folgt, dass die lineare Hülle der
  Projektionen von $C(M,\R)$ dicht in $C(M,\R)$ bezüglich
  $\norm[\infty]{\cdot}$ ist. Zurück übersetzt ergibt sich die
  entsprechende Aussage in $C_{sa}$ mit der Normtopologie. Wir können
  also $a$ durch Linearkombinationen von Projektionen in $C$ beliebig
  genau annähern, und damit auch durch Linearkombinationen von
  Projektionen in $A$.

  Die Zusatzaussage folgt sofort aus dem obigen Beweis und der
  Zusatzaussage von Satz~\ref{satz:proj-dicht}, da ein $a\geq 0$ auch
  in $C$ positiv ist.
\end{proof}

Nach diesen Vorarbeiten kommen wir zur Multiplikation. Wir werden
zeigen, dass die Translationen $x\mapsto ax$ und $x\mapsto xa$
schwach-*-stetig sind.

Bevor wir zum technisch aufwendigen Beweis kommen,
skizzieren wir das Vorgehen in groben Zügen. Zunächst zeigen wir die
Stetigkeit der Abbildung $x\mapsto pxp$ für eine Projektion
$p\in A\setminus\{0,1\}$, indem wir die Abbildung als zur direkten
Summe $A=pAp\dotplus\big((1-p)Ap+A(1-p)\big)$ gehörige
Projektion\footnote{Man beachte, dass in diesen Überlegungen zwei
  verschiedene Arten von Projektionen eine Rolle spielen, nämlich
  selbstadjungierte und idempotente \emph{Elemente von} $A$ und
  idempotente \emph{Abbildungen auf} $A$.} deuten. Dazu müssen wir die
Abgeschlossenheit der direkten Summanden zeigen, wozu der Satz von
Banach-Dieudonné unverzichtbar sein wird. Anschließend führen wir
dieselben Überlegungen für die Abbildung $x\mapsto px(1-p)$ und die
Zerlegung $A=pA(1-p)\dotplus\big((1-p)A+pAp\big)$ durch. Summenbildung
der beiden Funktionen liefert die Stetigkeit von $x\mapsto px$, was
wir mithilfe der Dichtheit der linearen Hülle der Projektionen zur
Stetigkeit von $x\mapsto ax$ ausdehnen können. Die Ausnutzung der
Stetigkeit von $.^{*}$ liefert schließlich auch die Stetigkeit von
$x\mapsto xa$.

Wir starten mit einer Hilfsaussage.
\begin{lemma}\label{lem:mult-wstar-stetig}
  Ist $p\in A$ eine Projektion, so gilt $\cl{p(A_{sa}\cap
    S)(1-p)}^{\TT_{w^{*}}}\subseteq pS(1-p)$.
\end{lemma}
\begin{proof}
  Sei $(px_{i}(1-p))_{i\in I}$ ein bezüglich der schwach-*-Topologie
  gegen $x\in A$ konvergentes Netz, wobei $x_{i}$ in $A_{sa}\cap S$
  liege. Der Kern des Beweises ist es, $x=px(1-p)$ zu zeigen, indem
  wir zuerst $x=px(1-p)+(1-p)xp$ und danach $(1-p)xp=0$ beweisen. Ist
  $p$ eine triviale Projektion, also $p\in\{0,1\}$, so sind beide
  Aussagen klar. Daher nehmen wir im Folgenden $p\neq 0,1$ an.
  \begin{enumerate}[label=(\roman*),ref=\roman*]
  \item\label{item:bew-lem-mult-wstar-stetig-i}\emph{Es gilt
      $x=px(1-p)+(1-p)xp$}:

    Zunächst zeigen wir $pxp=(1-p)x(1-p)=0$. Dazu sei
    $\frac{1}{2}(pxp+px^{*}p)=\re(pxp)\neq 0$ angenommen. Indem wir
    gegebenenfalls zu $-px_{i}(1-p)$ und $-x$ übergehen, können wir
    annehmen, dass es ein positives $0<\lambda\in\sigma(\re(pxp))$
    gibt. Dann gilt für jedes $n\in\N$ und $i\in I$
    \begin{align}
      \begin{split}
        \norm{px_{i}(1-p)+np}&=\norm{(px_{i}(1-p)+np)^{*}}=\norm{(px_{i}(1-p)+np)(px_{i}(1-p)+np)^{*}}^{1/2}\\
        &=\norm{px_{i}(1-p)x_{i}p+n^{2}p}^{1/2}\leq(1+n^{2})^{1/2}.
      \end{split}
    \end{align}
    Somit ist $px_{i}(1-p)+np$ und folglich der Grenzwert $x+np$ in
    $(1+n^{2})^{1/2}S$ enthalten, da letztere Menge abgeschlossen
    ist. Andererseits erhalten wir aus $p\re(pxp)p=\re(pxp)$ die
    Abschätzung\footnote{Hier geht $p\neq 0$, folglich $\norm{p}=1$,
      ein.}
    \begin{align}
      \begin{split}\label{eq:bew-lem-mult-wstar-stetig-i}
        \norm{x+np}&=\norm{p}\norm{x+np}\norm{p}\geq\norm{p(x+np)p}\geq\norm{\re(p(x+np)p)}\\
        &=\norm{\re(pxp)+np}=\norm{p\re(pxp)p+np}.
      \end{split}
    \end{align}
    Die $C^{*}$-Algebra $C^{*}(\re(pxp),p,1)$ ist nach
    Bemerkung~\ref{bem:erz-unteralg} kommutativ, da die Erzeuger
    selbstadjungiert sind und miteinander kommutieren. Geht man mit
    der Gelfandtransformation zum Raum der stetigen Funktionen auf dem
    Gelfandraum $M$ dieser $C^{*}$-Algebra über, so gilt
    \begin{displaymath}
      \norm{p\re(pxp)p+np}=\norm[\infty]{\widehat{p}\widehat{\re(pxp)}\widehat{p}+n\widehat{p}}
    \end{displaymath}
    und
    \begin{displaymath}
      \lambda\in\sigma(\re(pxp))=\sigma(p\re(pxp)p)=\sigma\left(\widehat{p}\widehat{\re(pxp)}\widehat{p}\right)=\left(\widehat{p}\widehat{\re(pxp)}\widehat{p}\right)(M),
    \end{displaymath}
    weshalb es ein $m\in M$ mit
    $\lambda=\widehat{p}(m)\widehat{\re(pxp)}(m)\widehat{p}(m)$
    gibt. Die Funktion $\widehat{p}$ nimmt als Projektion nur die
    Werte $0,1$ an; wir setzen
    $E:=\widehat{p}^{-1}(\{1\})\subseteq M$. Mit dieser Notation gilt
    $m\in E$ wegen $\lambda>0$. Daraus folgt
    \begin{displaymath}
      \norm[\infty]{\widehat{p}\widehat{\re(pxp)}\widehat{p}+n\widehat{p}}=\norm[\infty,E]{\widehat{p}\widehat{\re(pxp)}\widehat{p}+n\widehat{p}}\geq\abs{\widehat{p}(m)\widehat{\re(pxp)}(m)\widehat{p}(m)+n\widehat{p}(m)}=\lambda+n.
    \end{displaymath}
    Mit \eqref{eq:bew-lem-mult-wstar-stetig-i} ergibt sich also
    \begin{equation}\label{eq:bew-lem-mult-wstar-stetig-ii}
      \norm{x+np}\geq\lambda+n.
    \end{equation}
    Wählt man $n$ so groß, dass $\lambda+n>(1+n^{2})^{1/2}$ ist, so
    erhalten wir aus \eqref{eq:bew-lem-mult-wstar-stetig-ii} den
    Widerspruch $x+np\notin (1+n^{2})^{1/2}S$.

    Die Annahme $\im(pxp)\neq 0$ führt man analog auf einen
    Widerspruch, indem man die Elemente $px_{i}(1-p)+nip$ und $x+nip$
    betrachtet und in \eqref{eq:bew-lem-mult-wstar-stetig-i} durch die
    Norm des Imaginärteils abschätzt. Also folgt tatsächlich
    $pxp=0$. Für $(1-p)x(1-p)=0$ wendet man das soeben Bewiesene auf
    $\tilde{p}:=1-p$ an: Es gilt
    $\tilde{p}x_{i}(1-\tilde{p})=(px_{i}(1-p))^{*}\to x^{*}$, sodass
    wir auf $\tilde{p}x^{*}\tilde{p}=0$ schließen, also auf
    $(1-p)x^{*}(1-p)=0$. Adjungieren dieser Gleichung liefert
    $(1-p)x(1-p)=0$.

    Insgesamt folgt
    \begin{displaymath}
      0=(1-p)x(1-p)=x-px-xp+pxp\,\underbrace{+pxp}_{=0}=x-px(1-p)-(1-p)xp,
    \end{displaymath}
    also $x=px(1-p)+(1-p)xp$.
  \item\label{item:bew-lem-mult-wstar-stetig-ii}\emph{Das Element
      $(1-p)xp$ verschwindet}:
    
    Für $n\in\N$ und $i\in I$ gilt
    \begin{align*}
      \norm{px_{i}(1-p)+n(1-p)xp}^{2}&=\norm{(px_{i}(1-p)+n(1-p)xp)^{*}(px_{i}(1-p)+n(1-p)xp)}\\
                                     &=\|\underbrace{(px_{i}(1-p))^{*}(px_{i}(1-p))}_{\quad
                                       =:a}+\underbrace{n^{2}((1-p)xp)^{*}((1-p)xp)}_{\quad
                                       =:b}\|,
    \end{align*}
    was nach Lemma~\ref{lem:norm-summe} mit
    \begin{align*}
      \max(\|a\|,\|b\|)&=\max\left(\norm{(px_{i}(1-p))^{*}(px_{i}(1-p))},n^{2}\norm{((1-p)xp)^{*}((1-p)xp)}\right)\\
                       &=\max\left(\norm{px_{i}(1-p)}^{2},n^{2}\norm{(1-p)xp}^{2}\right)
    \end{align*}
    übereinstimmt. Damit erhalten wir
    \begin{align}\label{eq:bew-lem-mult-wstar-stetig-iii}
      \begin{split}
        \norm{px_{i}(1-p)+n(1-p)xp}&=\max\big(\underbrace{\|px_{i}(1-p)\|}_{\leq
          1},n\|(1-p)xp\|\big)\\
        &\leq\max\big(1,n\|(1-p)xp\|\big)=:\alpha.
      \end{split}
    \end{align}
    Wegen $x=px(1-p)+(1-p)xp$ gilt ebenfalls für beliebiges $n\in\N$
    \begin{displaymath}
      \norm{x+n(1-p)xp}=\norm{px(1-p)+(n+1)(1-p)xp},
    \end{displaymath}
    was analog zur ersten Hälfte von
    \eqref{eq:bew-lem-mult-wstar-stetig-iii} die Gleichung
    \begin{equation}\label{eq:bew-lem-mult-wstar-stetig-iv}
      \norm{x+n(1-p)xp}=\max\big(\|px(1-p)\|,(n+1)\|(1-p)xp\|\big)
    \end{equation}
    nach sich zieht. Aus \eqref{eq:bew-lem-mult-wstar-stetig-iii} und
    der Abgeschlossenheit von $\alpha S$ folgt
    \begin{displaymath}
      \norm{x+n(1-p)xp}\leq\alpha=\max\big(1,n\|(1-p)xp\|\big),
    \end{displaymath}
    und wegen \eqref{eq:bew-lem-mult-wstar-stetig-iv}
    \begin{equation}\label{eq:bew-lem-mult-wstar-stetig-v}
      \max\big(\|px(1-p)\|,(n+1)\|(1-p)xp\|\big)\leq\max\big(1,n\|(1-p)xp\|\big).
    \end{equation}
    Im Falle $\norm{(1-p)xp}\neq 0$ gäbe es ein $n\in\N$ mit
    $(n+1)\|(1-p)xp\|\geq\norm{px(1-p)}$ und $n\|(1-p)xp\|\geq 1$,
    sodass aus \eqref{eq:bew-lem-mult-wstar-stetig-v} der Widerspruch
    \begin{displaymath}
      (n+1)\norm{(1-p)xp}\leq n\norm{(1-p)xp}
    \end{displaymath}
    folgte.
  \end{enumerate}
  Mit (\ref{item:bew-lem-mult-wstar-stetig-i}) und
  (\ref{item:bew-lem-mult-wstar-stetig-ii}) haben wir
  $x=px(1-p)\in pA(1-p)$ gezeigt. Weiters gilt
  \begin{displaymath}
    \norm{px_{i}(1-p)}\leq\norm{p}\norm{x_{i}}\norm{1-p}\leq 1,
  \end{displaymath}
  also $px_{i}(1-p)\in S$. Da die Einheitskugel abgeschlossen ist,
  erhalten wir $x\in S$ und schließlich $x=px(1-p)\in pS(1-p)$.
\end{proof}
\begin{satz}\label{satz:mult-wstar-stetig}
  Für ein $a\in A$ sind die Translationen $x\mapsto ax$ und
  $x\mapsto xa$ schwach-*-stetig.
\end{satz}
\begin{proof}
  Im gesamten Beweis bezeichne $p\in A\setminus\{0,1\}$ eine
  Projektion.
  \begin{enumerate}[label=(\roman*),ref=\roman*]
  \item\label{item:bew-mult-wstar-stetig-i} \emph{Die Mengen $pSp$ und
      $(1-p)S(1-p)$ sind schwach-*-kompakt und die Räume $pAp$ und
      $(1-p)A(1-p)$ sind schwach-*-abgeschlossen}:

    Zunächst gilt
    $p(A^{+}\cap S)p=\set{x\in A_{sa}\cap S}{0\leq x\leq p}$: Für
    $y\in A^{+}\cap S$ ist $x=pyp\in p(A^{+}\cap S)p$ selbstadjungiert
    und liegt wegen
    \begin{equation}\label{eq:bew-mult-wstar-stetig-i}
      \norm{x}=\norm{pyp}\leq\norm{y}\leq 1
    \end{equation}
    in $S$. Außerdem gelten wegen Lemma~\ref{lem:eig-ho} die beiden
    Ungleichungsketten $0\leq y\leq\norm{y}\leq 1$ und
    $0=p0p\leq pyp\leq p1p=p$. Aus $0\leq x\leq p$ für ein
    $x\in A_{sa}\cap S$ folgt umgekehrt
    $0\leq (1-p)x(1-p)\leq (1-p)p(1-p)=0$, also $(1-p)x(1-p)=0$
    bzw. $\norm{x^{1/2}(1-p)}^{2}=\norm{(1-p)x(1-p)}=0$. Wir erhalten
    $x(1-p)=x^{1/2}\left(x^{1/2}(1-p)\right)=0$ bzw. $x=xp$. Durch
    Adjungieren folgt $px=x$ und insgesamt $x=pxp\in p(A^{+}\cap S)p$.
    
    Die gerade nachgewiesene Gleichheit lässt sich auch durch
    \begin{displaymath}
      p(A^{+}\cap S)p=A^{+}\cap(p-A^{+})\cap A_{sa}\cap S
    \end{displaymath}
    ausdrücken, womit diese Menge abgeschlossen und, da in $S$
    enthalten, sogar kompakt ist. Daraus folgt, dass
    \begin{displaymath}
      M:=\left(p(A^{+}\cap S)p-p(A^{+}\cap S)p\right)+i\left(p(A^{+}\cap S)p-p(A^{+}\cap S)p\right)
    \end{displaymath}
    als Bild des Kompaktums
    $(p(A^{+}\cap S)p)\times (p(A^{+}\cap S)p)\times (p(A^{+}\cap
    S)p)\times (p(A^{+}\cap S)p)$ unter der stetigen Funktion
    \begin{displaymath}
      (x_{1},x_{2},x_{3},x_{4})^{T}\mapsto
      (x_{1}-x_{2})+i(x_{3}-x_{4})
    \end{displaymath}
    ebenfalls kompakt ist. Aus der Inklusion
    \eqref{eq:pos-negteil-span-ii} in
    Bemerkung~\ref{bem:pos-negteil-span}(\ref{item:pos-negteil-span-ii})
    folgt $pSp\subseteq M$. Umgekehrt gilt für $y\in M$ sicher $pyp=y$
    und folglich $M\subseteq pAp$. Außerdem gilt mit demselben
    Argument wie in \eqref{eq:bew-mult-wstar-stetig-i} die Inklusion
    $pSp\subseteq S$. Zusammen erhalten wir
    \begin{equation}\label{eq:bew-mult-wstar-stetig-ii}
      pSp=pSp\cap S\subseteq M\cap S\subseteq pAp\cap
      S\stackrel{(*)}{=}pSp;
    \end{equation}
    die Gleichheit $(*)$ folgt dabei aus
    \eqref{eq:bew-mult-wstar-stetig-i} und aus der Tatsache, dass man
    ein $pxp\in pAp\cap S$ auch als $p(pxp)p\in pSp$ schreiben kann.
    Damit ist $pSp=M\cap S$ kompakt. Folglich ist $pAp\cap S=pSp$
    insbesondere abgeschlossen, sodass sich $pAp$ nach dem Satz von
    Banach-Dieudonné ebenfalls als abgeschlossen herausstellt. Durch
    Betrachten der Projektion $\tilde{p}:=1-p$ ergibt sich der Rest
    aus dem schon Bewiesenen.
  \item\label{item:bew-mult-wstar-stetig-ii} \emph{Die Mengen
      $pS(1-p)$ und $(1-p)Sp$ sind schwach-*-kompakt und die Räume
      $pA(1-p)$ und $(1-p)Ap$ sind schwach-*-abgeschlossen}:

    Wir betrachten nur $pS(1-p)$ bzw. $pA(1-p)$, da die übrige Aussage
    daraus wieder durch Betrachten von $\tilde{p}:=1-p$ folgt.

    Genauso wie die Gleichheit $(*)$ in
    \eqref{eq:bew-mult-wstar-stetig-ii} zeigt man
    \begin{equation}\label{eq:bew-mult-wstar-stetig-iii}
      pA(1-p)\cap S=pS(1-p).
    \end{equation}
    Aufgrund des Satzes von Banach-Dieudonné genügt es für den zweiten
    Teil der Aussage daher, die Abgeschlossenheit von $pS(1-p)$
    nachzuweisen. Diese Abgeschlossenheit ist auch für die
    Kompaktheitsaussage ausreichend, da $pS(1-p)$ dann eine
    abgeschlossene Teilmenge des Kompaktums $S$ und somit selbst
    kompakt ist.

    Definiert man die Menge
    \begin{displaymath}
      N:=p(A_{sa}\cap S)(1-p)+ip(A_{sa}\cap S)(1-p),
    \end{displaymath}
    so zeigt man auf ähnliche Weise wie in
    \eqref{eq:bew-mult-wstar-stetig-ii} die Gleichung
    \begin{equation}\label{eq:bew-mult-wstar-stetig-iv}
      pS(1-p)=N\cap S.
    \end{equation}
    Daraus folgt
    \begin{align}
      \begin{split}\label{eq:bew-mult-wstar-stetig-v}
        \cl{pS(1-p)}\subseteq\cl{N}&=\cl{p(A_{sa}\cap
          S)(1-p)+ip(A_{sa}\cap
          S)(1-p)}\\
        &\subseteq\cl{(\cl{p(A_{sa}\cap S)(1-p)})+i(\cl{p(A_{sa}\cap
            S)(1-p)})}.
      \end{split}
    \end{align}
    
    Nach Lemma~\ref{lem:mult-wstar-stetig} und
    \eqref{eq:bew-mult-wstar-stetig-i} gilt
    \begin{displaymath}
      \cl{p(A_{sa}\cap S)(1-p)}\subseteq pS(1-p)\subseteq S,
    \end{displaymath}
    sodass die beiden Summanden in der zweiten Zeile von
    \eqref{eq:bew-mult-wstar-stetig-v} abgeschlossene Teilmengen von
    $S$ und daher kompakt sind. Somit ist auch ihre Summe kompakt und
    folglich abgeschlossen. Wir schließen auf
    \begin{displaymath}
      \cl{pS(1-p)}\subseteq(\cl{p(A_{sa}\cap
        S)(1-p)})+i(\cl{p(A_{sa}\cap S)(1-p)})\subseteq (pS(1-p))+i(pS(1-p)).
    \end{displaymath}
    Die Tatsache~\eqref{eq:bew-mult-wstar-stetig-iv} liefert weiters
    $\cl{pS(1-p)}\subseteq\cl{S}=S$, sodass wir
    \begin{displaymath}
      \cl{pS(1-p)}\subseteq\big(pS(1-p)+ipS(1-p)\big)\cap
      S=\big(p(S+iS)(1-p)\big)\cap S.
    \end{displaymath}
    erhalten. Eine analoge Überlegung zu $(*)$ in
    \eqref{eq:bew-mult-wstar-stetig-ii} liefert, dass diese Menge in
    $pS(1-p)$ enthalten ist. Also ist $pS(1-p)$ tatsächlich
    abgeschlossen.
  \item\label{item:bew-mult-wstar-stetig-iii} \emph{Die Mengen $pS$,
      $(1-p)S$, $Sp$ und $S(1-p)$ sind schwach-*-kompakt und die Räume
      $pA$, $(1-p)A$, $Ap$ und $A(1-p)$ sind schwach-*-abgeschlossen}:

    Wie in Beweisschritt~(\ref{item:bew-mult-wstar-stetig-ii}) haben
    wir nur die Abgeschlossenheit der Mengen $pS$, $(1-p)S$, $Sp$ und
    $S(1-p)$ zu zeigen. Zunächst betrachten wir $pS$.

    Wegen $px=pxp+px(1-p)$ für jedes $x\in S$ sowie der trivialen
    Inklusion\footnote{Siehe auch \eqref{eq:bew-mult-wstar-stetig-i}.}
    $pS\subseteq S$ gilt $pS\subseteq\big(pSp+pS(1-p)\big)\cap S$. Ist
    umgekehrt $x:=pyp+pz(1-p)\in S$ mit $y,z\in S$ gegeben, so gilt
    $x=px\in pS$ und daher $\big(pSp+pS(1-p)\big)\cap S\subseteq
    pS$. Wir erhalten insgesamt
    \begin{equation}\label{eq:bew-mult-wstar-stetig-vi}
      pS=\big(pSp+pS(1-p)\big)\cap S.
    \end{equation}
    Die beiden Summanden sind nach
    (\ref{item:bew-mult-wstar-stetig-i}) und
    (\ref{item:bew-mult-wstar-stetig-ii}) kompakt, sodass es auch ihre
    Summe ist. Aus \eqref{eq:bew-mult-wstar-stetig-vi} folgt die
    Abgeschlossenheit von $pS$.

    Die Operation $.^{*}$ ist als stetige Involution ein
    Homöomorphismus, sodass mit $pS$ auch $Sp=(pS)^{*}$ abgeschlossen
    ist. Die Abgeschlossenheit von $(1-p)S$ bzw. $S(1-p)$ folgt
    schließlich aus jener von $pS$ bzw. $Sp$ durch Betrachten von
    $\tilde{p}:=1-p$.
  \item\label{item:bew-mult-wstar-stetig-iv} \emph{Die Räume
      $(1-p)Ap+A(1-p)$ und $(1-p)A+pAp$ sind schwach-*-abgeschlossen}:

    Wie bisher reicht es, die Schnitte mit $S$ zu betrachten, wobei
    wir mit dem ersten Raum beginnen. Dazu behaupten wir zunächst
    \begin{equation}\label{eq:bew-mult-wstar-stetig-vii}
      \big((1-p)Ap+A(1-p)\big)\cap S=\big((1-p)Sp+S(1-p)\big)\cap S.
    \end{equation}
    Für die nichttriviale Inklusion sei $x:=(1-p)yp+z(1-p)\in S$ mit
    $y,z\in A$ gegeben. Es gilt $(1-p)xp=(1-p)yp$ sowie
    $x(1-p)=z(1-p)$ und daher
    \begin{displaymath}
      x=(1-p)xp+x(1-p)\in (1-p)Sp+S(1-p).
    \end{displaymath}
    Nach den Beweisschritten~(\ref{item:bew-mult-wstar-stetig-ii}) und
    (\ref{item:bew-mult-wstar-stetig-iii}) ist die rechte Seite von
    \eqref{eq:bew-mult-wstar-stetig-vii} als Kompaktum abgeschlossen,
    sodass eine erneute Anwendung des Satzes von Banach-Dieudonné die
    erste Aussage zeigt.

    Für die zweite Aussage beweist man ganz analog
    \begin{displaymath}
      \big((1-p)A+pAp\big)\cap S=\big((1-p)S+pSp\big)\cap S
    \end{displaymath}
    und argumentiert auf dieselbe Weise.
  \item\label{item:bew-mult-wstar-stetig-v} \emph{Die Funktionen
      $x\mapsto pxp$, $x\mapsto px(1-p)$ und $x\mapsto px$ sind
      schwach-*-stetig}:

    Wir beweisen zunächst die Zerlegung
    $A=pAp\dotplus\big((1-p)Ap+A(1-p)\big)$. Die Summe dieser beiden
    Unterräume ist sicher ganz $A$, denn es gilt
    $x=pxp+\big((1-p)xp+x(1-p)\big)$ für beliebiges $x\in A$. Für
    $x\in pAp\cap\big((1-p)Ap+A(1-p)\big)$ gilt
    \begin{displaymath}
      pyp=x=(1-p)z_{1}p+z_{2}(1-p)
    \end{displaymath}
    mit gewissen $y,z_{1},z_{2}\in A$. Daraus folgt
    \begin{displaymath}
      x=pyp=p(pyp)p=pxp=p\big((1-p)z_{1}p+z_{2}(1-p)\big)p=0.
    \end{displaymath}
    Nach den Beweisschritten (\ref{item:bew-mult-wstar-stetig-i}) und
    (\ref{item:bew-mult-wstar-stetig-iv}) können wir
    Lemma~\ref{lem:proj-wstar-stetig} anwenden, um die Stetigkeit der
    Projektion auf $pAp$ zu erhalten. Aus der expliziten Darstellung
    \begin{displaymath}
      x=pxp+\big((1-p)xp+x(1-p)\big)
    \end{displaymath}
    sehen wir, dass diese Projektion genau $x\mapsto pxp$ ist, womit
    die erste Aussage gezeigt ist.

    Weiters gilt $A=pA(1-p)\dotplus\big((1-p)A+pAp\big)$: Einerseits
    kann man jedes $x\in A$ in der Form
    \begin{displaymath}
      x=px(1-p)+\big((1-p)x+pxp\big)
    \end{displaymath}
    schreiben, andererseits folgt aus
    \begin{displaymath}
      py(1-p)=x=(1-p)z_{1}+pz_{2}p
    \end{displaymath}
    mit $y,z_{1},z_{2}\in A$ die Gleichung
    \begin{displaymath}
      x=py(1-p)=p\big(py(1-p)\big)(1-p)=px(1-p)=p\big((1-p)z_{1}+pz_{2}p\big)(1-p)=0.
    \end{displaymath}
    Wieder wegen Lemma~\ref{lem:proj-wstar-stetig} ist die Funktion
    $x\mapsto px(1-p)$ stetig, da sie die Projektion auf $pA(1-p)$
    bezüglich dieser Zerlegung ist.

    Schließlich folgt die Stetigkeit der Summe
    $x\mapsto pxp+px(1-p)=px$ dieser beiden Funktionen.
  \item\label{item:bew-mult-wstar-stetig-vi} \emph{Die Translationen
      $x\mapsto ax$ und $x\mapsto xa$ sind schwach-*-stetig}:

    Wir zeigen als Erstes, dass für jedes schwach-*-stetige Funktional
    $\varphi$ das ebenfalls lineare Funktional $\psi(x):=\varphi(ax)$
    schwach-*-stetig auf $S$ ist.

    Dazu sei $\epsilon>0$ gegeben. Nach
    Korollar~\ref{kor:wstar-proj-dicht} gibt es eine natürliche Zahl
    $n$, Projektionen $p_{j}$ und Skalare $\lambda_{j}\in\C$,
    $j=1,\dots,n$, mit
    $\norm{a-\sum_{j=1}^{n}\lambda_{j}p_{j}}\leq\epsilon$. Sei
    $(x_{i})_{i\in I}$ ein Netz aus $S$, das gegen $x$
    konvergiert. Wegen der Abgeschlossenheit von $S$ gilt $x\in S$.
    Nach Bemerkung~\ref{bem:wstar-top}(\ref{item:wstar-top-iii}) ist
    $\varphi$ ein beschränktes Funktional und es gilt
    \begin{align*}
      \abs{\varphi(a(x_{i}-x))}&\leq\abs{\varphi\left(\left(a-\sum_{j=1}^{n}\lambda_{j}p_{j}\right)(x_{i}-x)\right)}+\abs{\varphi\left(\left(\sum_{j=1}^{n}\lambda_{j}p_{j}\right)(x_{i}-x)\right)}\\
                               &\leq\norm{\varphi}\norm{a-\sum_{j=1}^{n}\lambda_{j}p_{j}}\underbrace{\norm{x_{i}-x}}_{\leq
                                 2}+\sum_{j=1}^{n}\abs{\lambda_{j}}\abs{\varphi(p_{j}(x_{i}-x))}\\
                               &\leq
                                 2\norm{\varphi}\epsilon+\sum_{j=1}^{n}\abs{\lambda_{j}}\abs{\varphi(p_{j}(x_{i}-x))}\longrightarrow
                                 2\norm{\varphi}\epsilon,
    \end{align*}
    da die Abbildungen $y\mapsto\varphi(p_{j}y)$ wegen
    (\ref{item:bew-mult-wstar-stetig-v}) stetig sind. Hier ist zu
    beachten, dass die Fälle $p_{j}=0,1$ durch den bisherigen Beweis
    nicht abgedeckt werden; diese sind aber trivial. Daraus folgt
    $\limsup_{i\in I}\abs{\varphi(a(x_{i}-x))}\leq
    2\norm{\varphi}\epsilon$.  Da $\epsilon$ beliebig war, gilt
    $\lim_{i\in I}\abs{\varphi(a(x_{i}-x))}=0$ bzw.
    \begin{displaymath}
      \lim_{i\in I}\psi(x_{i})=\lim_{i\in I}\varphi(ax_{i})=\varphi(ax)=\psi(x).
    \end{displaymath}
    Somit ist $\psi$ tatsächlich auf $S$ stetig. Nach
    Korollar~\ref{kor:banach-dieud}
    bzw. Bemerkung~\ref{bem:wstar-top}(\ref{item:wstar-top-ii}) ist
    $\psi$ auf ganz $A$ stetig. Da der Zielraum $A$ der Abbildung
    $x\mapsto ax$ die initiale Topologie bezüglich aller
    schwach-*-stetigen linearen Funktionale $\varphi$ trägt, erhalten
    wir die Stetigkeit von $x\mapsto ax$.  Wegen $xa=(a^{*}x^{*})^{*}$
    und der Stetigkeit von $.^{*}$ folgt daraus auch die zweite
    Aussage.
  \end{enumerate}
\end{proof}

\clearpage
\chapter{Der Satz von Sakai}
\label{cha:satz_sakai}
Dieses Kapitel enthält trotz seiner Kürze die Hauptaussage der
gesamten Arbeit, nämlich den Satz von Sakai. Dieser besagt, dass die
in Kapitel~\ref{cha:w-algebren} eingeführten $W^{*}$-Algebren
tatsächlich die in Definition~\ref{def:vn-alg} eingeführten
Von-Neumann-Algebren hilbertraumfrei axiomatisieren. Mit anderen
Worten ist jede $W^{*}$-Algebra isometrisch isomorph zu einer
Von-Neumann-Algebra. In \cite{sakai:cstar-wstar} wird gezeigt, dass
dieser isometrische Isomorphismus stetig bezüglich der
schwach-*-Topologien ist. Wir werden zusätzlich beweisen, dass auch
die Umkehrabbildung schwach-*-stetig ist, sodass der isometrische
Isomorphismus ein Homöomorphismus bezüglich der schwach-*-Topologien
ist. Somit können wir wie beim Satz von Gelfand-Naimark für
$C^{*}$-Algebren von einer \emph{vollständigen} Axiomatisierung
sprechen, da der Isomorphismus die komplette Struktur der
$W^{*}$-Algebra, nämlich die algebraische, die vom Prädualraum
induzierte topologische und die Normstruktur, auf die
Von-Neumann-Algebra überträgt. Wir werden ganz ähnlich wie in
Kapitel~\ref{cha:abstrakte-c-algebren} beim Satz von Gelfand-Naimark
vorgehen, indem wir eine sogenannte treue $W^{*}$-Darstellung
konstruieren. Es ist jedoch zu beachten, dass Stetigkeitsüberlegungen
bezüglich der schwach-*-Topologien eine wesentlich größere Rolle
einnehmen werden als in Kapitel~\ref{cha:abstrakte-c-algebren}, da wir
die Isometrie und damit die $\norm{\cdot}$-Stetigkeit aus der
algebraischen Strukturverträglichkeit ableiten konnten; siehe
Satz~\ref{satz:inj-homo-iso}.

\section{Darstellungen von $W^{*}$-Algebren}
\label{sec:darst-von-wstar}
\begin{notation}
  Bezeichne $A$ wie im letzten Kapitel eine $W^{*}$-Algebra und
  $\TT_{w^{*}}=\TT_{w^{*},A}$ die schwach-*-Topologie auf $A$.
\end{notation}

\begin{definition}\label{def:wstar-unteralg}
  Eine schwach-*-abgeschlossene $*$-Unteralgebra $B\leq A$ nennen wir
  $W^{*}$-Unter\-algebra.
\end{definition}

In Bemerkung~\ref{bem:motiv-wstar} kann man $L^{1}(H)$ durch
$\pred{A}$, den Raum $L_{b}(H)$ durch $A$ und $A$ durch $B$ sowie den
isometrischen Isomorphismus $\theta:L_{b}(H)\to L^{1}(H)'$ durch
$j_{A}:A\to\pred{A}'$ ersetzen, um zu zeigen, dass eine
$W^{*}$-Unteralgebra $B\leq A$ für sich genommen ebenfalls eine
$W^{*}$-Algebra ist\footnote{In dieser Situation folgt die
  $\sigma\left(\pred{A}',\pred{A}\right)$-Abgeschlossenheit von
  $j_{A}(B)$ nicht mithilfe eines Analogons von
  Satz~\ref{satz:uw-optop-wstar} aus der
  $\TT_{w^{*},A}$-Abgeschlossenheit von $B$, sondern schlicht aus der
  Definition der schwach-*-Topologie $\TT_{w^{*},A}$. Diese Topologie
  ist nämlich so konstruiert, dass $j_{A}$ ein
  schwach-*-Homöomorphismus wird.}. Dabei ist $\pred{B}$ gegeben durch
$\pred{A}/M$ mit dem Linksannihilator $M:=\lanh{(j_{A}(B))}$. Als
isometrischen Isomorphismus $j_{B}:B\to\pred{B}'$ kann man
$\tau^{-1}\circ j_{A}|_{B}$ wählen, mit der analog zu
\eqref{eq:motiv-wstar} definierten Abbildung $\tau$. Daher existieren
auf $B$ zwei kanonische $W^{*}$-Topologien, nämlich die
schwach-*-Topologie $\TT_{w^{*},B}$ und die Spurtopologie
$(\TT_{w^{*},A})|_{B}$. Es ist eine wesentliche Tatsache, dass diese
beiden Topologien gleich sind.

\begin{lemma}\label{lem:top-wstar-unteralg}
  Ist $B$ eine $W^{*}$-Unteralgebra von $A$, so gilt
  $\TT_{w^{*},B}=(\TT_{w^{*},A})|_{B}$.
\end{lemma}
\begin{proof}
  Die Topologie $\TT_{w^{*},A}$ ist die initiale Topologie bezüglich
  $j_{A}:A\to\pred{A}'$, wenn der Zielraum $\pred{A}'$ die
  schwach-*-Topologie trägt. Somit ist die Spurtopologie
  $(\TT_{w^{*},A})|_{B}$ die initiale Topologie bezüglich
  $j_{A}\circ\iota_{B\to A}:B\to\pred{A}'$, wobei $\iota_{B\to A}$ die
  Inklusionsabbildung $B\to A$, $b\mapsto b$ bezeichnet. Mit den
  Bezeichnungen nach Definition~\ref{def:wstar-unteralg} ist weiters
  $\TT_{w^{*},B}$ die initiale Topologie bezüglich
  \begin{displaymath}
    j_{B}=\tau^{-1}\circ j_{A}|_{B}=\tau^{-1}\circ j_{A}\circ\iota_{B\to A}
  \end{displaymath}
  als Abbildung $B\to\pred{B}'=(\pred{A}/M)'$. Dabei ist
  $(\pred{A}/M)'$ ebenfalls mit der schwach-*-Topologie versehen. Die
  schwach-*-Topologie auf dem Dualraum eines Banachraums $X$ ist
  wieder als initiale Topologie definiert, nämlich bezüglich der Menge
  $\iota_{X}(X)$ von Funktionalen, wobei $\iota_{X}:X\to X''$ die
  kanonische Einbettung $x\mapsto (x'\mapsto\chevr{x}{x'})$
  bezeichnet. Wenden wir diese Tatsache auf $X=\pred{A}$ bzw.
  $X=\pred{B}$ an, so erhalten wir aufgrund der Assoziativität des
  Bildens initialer Topologien
  \begin{equation}\label{eq:bew-top-wstar-unteralg-i}
    (\TT_{w^{*},A})|_{B}=\TT_{\text{init}}\left(\set{\iota_{\pred{A}}(\rho)\circ
        j_{A}\circ\iota_{B\to A}}{\rho\in\pred{A}}\right)
  \end{equation}
  sowie
  \begin{align}
    \begin{split}\label{eq:bew-top-wstar-unteralg-ii}
      \TT_{w^{*},B}&=\TT_{\text{init}}\left(\set{\iota_{\pred{B}}(\rho_{B})\circ\tau^{-1}\circ
          j_{A}\circ\iota_{B\to A}}{\rho_{B}\in\pred{B}}\right) \\
      &=\TT_{\text{init}}\left(\set{\iota_{\pred{B}}(\pi(\rho))\circ\tau^{-1}\circ
          j_{A}\circ\iota_{B\to A}}{\rho\in\pred{A}}\right).
    \end{split}
  \end{align}
  Dabei liegen jeweils Mengen von Abbildungen von $B$ in die mit der
  euklidischen Topologie versehenen komplexen Zahlen $\C$
  vor. Außerdem bezeichnet $\pi$ die kanonische
  Faktorisierungsabbildung $\pred{A}\to\pred{A}/M=\pred{B}$. Um die
  Aussage des Lemmas zu zeigen, genügt es folglich, die Gleichheit der
  beiden Mengen skalarwertiger Abbildungen in
  \eqref{eq:bew-top-wstar-unteralg-i} und
  \eqref{eq:bew-top-wstar-unteralg-ii} nachzuweisen.

  Dazu sei zunächst an die explizite Definition von $\tau$ erinnert:
  \begin{displaymath}
    \tau:
    \begin{cases}
      (\pred{A}/M)'&\to j_{A}(B)\\
      \hfill f&\mapsto f\circ\pi
    \end{cases}
  \end{displaymath}
  Mit diesen Informationen und Definitionen berechnen wir für
  beliebige Elemente $\rho\in\pred{A}$ und $b\in B$
  \begin{align*}
    \chevr{b}{\iota_{\pred{B}}(\pi(\rho))\circ\tau^{-1}\circ
    j_{A}}&=\chevr{\tau^{-1}\circ j_{A}(b)}{\iota_{\pred{B}}(\pi(\rho))}=\chevr{\pi(\rho)}{\tau^{-1}(j_{A}(b))}\\
          &=\chevr{\rho}{\tau^{-1}(j_{A}(b))\circ\pi}=\chevr{\rho}{\tau(\tau^{-1}(j_{A}(b)))}\\
          &=\chevr{\rho}{j_{A}(b)}\\
          &=\chevr{j_{A}(b)}{\iota_{\pred{A}}(\rho)}=\chevr{b}{\iota_{\pred{A}}(\rho)\circ
            j_{A}}\\
          &=\chevr{b}{\iota_{\pred{A}}(\rho)\circ
            j_{A}\circ\iota_{B\to A}}.
  \end{align*}
  Daraus ergibt sich
  \begin{displaymath}
    \iota_{\pred{B}}(\pi(\rho))\circ\tau^{-1}\circ
    j_{A}=\iota_{\pred{A}}(\rho)\circ j_{A}\circ\iota_{B\to A}
  \end{displaymath}
  und infolge die Gleichheit der Mengen in
  \eqref{eq:bew-top-wstar-unteralg-i} und
  \eqref{eq:bew-top-wstar-unteralg-ii}.
\end{proof}

Die Begriffe eines $*$-Algebrenhomomorphismus und einer Darstellung
werden als Nächstes um eine Stetigkeitsbedingung erweitert. Wegen der
Unterscheidung zwischen der schwachen Operatortopologie $\TT_{w}$ und
der ultraschwachen Operatortopologie $\TT_{uw}=\TT_{w^{*},L_{b}(H)}$
geschieht dies auf zwei verschiedene Weisen.
\begin{definition}\label{def:wstar-homo-darst}
  \hspace{0mm}
  \begin{enumerate}[label=(\roman*),ref=\roman*]
  \item Sei $B$ eine weitere $W^{*}$-Algebra. Ein
    $*$-Algebrenhomomorphismus $\Phi:A\to B$ heißt
    \emph{$W^{*}$-Homomorphismus}, wenn er
    $\TT_{w^{*},A}|\TT_{w^{*},B}$-stetig ist. Ein bijektiver
    $W^{*}$-Homomorphismus, der zusätzlich ein Homöomorphismus
    bezüglich der schwach-*-Topologien ist, heißt
    \emph{$W^{*}$-Isomorphismus}.
  \item Einen $W^{*}$-Homomorphismus $\Phi:A\to L_{b}(H)$, also einen
    $\TT_{w^{*},A}|\TT_{uw}$-stetigen $*$-Algebren\-homomorphismus,
    nennt man \emph{$W^{*}$-Darstellung}. Dabei heißt $\Phi$
    \emph{treu}, wenn $\Phi$ injektiv ist.
  \item Ist $\Phi:A\to L_{b}(H)$ ein $*$-Homomorphismus, der
    $\TT_{w^{*},A}|\TT_{w}$-stetig ist, so heißt $\Phi$
    \emph{Von-Neumann-Darstellung}. Im Falle der Injektivität heißt
    $\Phi$ dabei \emph{treu}.
  \end{enumerate}
\end{definition}

Man beachte, dass der Begriff der Von-Neumann-Darstellung schwächer
als der einer $W^{*}$-Darstellung $A\to L_{b}(H)$ ist, da $\TT_{w}$
gröber als $\TT_{uw}$ ist.

Zunächst beweisen wir ein Analogon von Satz~\ref{satz:inj-homo-iso}.
\begin{satz}\label{satz:inj-wstar-homo-inv-wstar}
  Sei $B$
  \begin{enumerate}[label=(\roman*),ref=\roman*]
  \item\label{item:inj-wstar-homo-inv-wstar-i} eine $W^{*}$-Algebra
    \quad oder
  \item\label{item:inj-wstar-homo-inv-wstar-ii} $B=L_{b}(H)$
  \end{enumerate}
  und sei $\Phi:A\to B$ ein injektiver $*$-Homomorphismus, der
  \begin{enumerate}[label=(\roman*),ref=\roman*]
  \item $\TT_{w^{*},A}|\TT_{w^{*},B}$-stetig \quad oder
  \item $\TT_{w^{*},A}|\TT_{w}$-stetig (d.~h. eine treue
    Von-Neumann-Darstellung)
  \end{enumerate}
  ist.  Dann ist $\ran\Phi$ eine $W^{*}$-Unteralgebra von $B$, also im
  Fall (\ref{item:inj-wstar-homo-inv-wstar-ii}) eine
  Von-Neumann-Algebra. Außerdem ist $\Phi^{-1}:\ran\Phi\to A$
  jedenfalls ein $W^{*}$-Homomorphismus.
\end{satz}
\begin{proof}
  Wir setzen $C:=\ran\Phi$. Nach dem Satz von Banach-Dieudonné genügt
  es für die erste Aussage zu zeigen, dass $C\cap S_{B}$ abgeschlossen
  bezüglich $\TT_{w^{*},B}$ ist. Gemäß Satz~\ref{satz:inj-homo-iso}
  ist $\Phi$ isometrisch, woraus $C\cap S_{B}=\Phi(S_{A})$ folgt. In
  Fall (\ref{item:inj-wstar-homo-inv-wstar-i}) erhalten wir, dass
  $C\cap S_{B}$ als stetiges Bild eines Kompaktums
  $\TT_{w^{*},B}$-kompakt und daher $\TT_{w^{*},B}$-abgeschlossen
  ist. In Fall (\ref{item:inj-wstar-homo-inv-wstar-ii}) erhalten wir
  analog die $\TT_{w}$-Abgeschlossenheit von $\Phi(S_{A})$. Daraus
  folgt die Abgeschlossenheit bezüglich der Spurtopologie
  $(\TT_{w})|_{S_{B}}=(\TT_{uw})|_{S_{B}}$;
  vgl. Korollar~\ref{kor:uw-optop-wstar}(\ref{item:uw-optop-wstar-ii}). Aus
  der $\TT_{uw}$-Abgeschlossenheit von $S_{B}$ erhalten wir
  $\cl{\Phi(S_{A})}^{\TT_{uw}}\subseteq S_{B}$. Daraus ergibt sich die
  Beziehung
  \begin{equation}\label{eq:inj-wstar-homo-inv-wstar}
    \Phi(S_{A})=\cl{\Phi(S_{A})}^{(\TT_{uw})|_{S_{B}}}=\cl{\Phi(S_{A})}^{\TT_{uw}}\cap S_{B}=\cl{\Phi(S_{A})}^{\TT_{uw}},
  \end{equation}
  und somit die gewünschte Abgeschlossenheit bezüglich
  $\TT_{uw}=\TT_{w^{*},B}$.
  
  Für die zweite Aussage gilt zunächst einmal, dass $\Phi^{-1}:C\to A$
  genau dann $\TT_{w^{*},C}|\TT_{w^{*},A}$-stetig ist, wenn
  $j_{A}\circ\Phi^{-1}:C\to\pred{A}'$ stetig bezüglich
  $\TT_{w^{*},C}|\sigma(\pred{A}',\pred{A})$ ist. Dies wiederum ist --
  unter Verwendung der Notation aus Lemma~\ref{lem:top-wstar-unteralg}
  -- genau dann der Fall, wenn die Verkettungen
  $\iota_{A}(\rho)\circ j_{A}\circ\Phi^{-1}:C\to\C$ für alle
  $\rho\in\pred{A}$ stetig bezüglich $\TT_{w^{*},C}|\TT_{\C}$ sind. Da
  es sich dabei um lineare Funktionale handelt, haben wir nach einem
  bekannten Satz der Funktionalanalysis die
  $\TT_{w^{*},C}$-Abgeschlossenheit der Kerne
  \begin{displaymath}
    \ker\left(\iota_{A}(\rho)\circ
      j_{A}\circ\Phi^{-1}\right)=(j_{A}\circ\Phi^{-1})^{-1}(\ker\iota_{A}(\rho))=\Phi\left(j_{A}^{-1}(\ker\iota_{A}(\rho))\right)
  \end{displaymath}
  zu untersuchen. Dazu betrachten wir den Schnitt
  \begin{displaymath}
    \Phi\left(j_{A}^{-1}(\ker\iota_{A}(\rho))\right)\cap
    S_{B}=\Phi\left(j_{A}^{-1}(\ker\iota_{A}(\rho))\cap S_{A}\right).
  \end{displaymath}
  Der Kern $\ker\iota_{A}(\rho)$ ist trivialerweise
  $\sigma(\pred{A}',\pred{A})$-abgeschlossen, womit das
  $j_{A}^{-1}$-Bild davon wegen der Homöomorphie
  $\TT_{w^{*},A}$-abgeschlossen ist. Die Einheitskugel $S_{A}$ ist
  $\TT_{w^{*},A}$-kompakt, sodass
  $j_{A}^{-1}(\ker\iota_{A}(\rho))\cap S_{A}$ ebenfalls
  $\TT_{w^{*},A}$-kompakt ist. In Fall
  (\ref{item:inj-wstar-homo-inv-wstar-i}) folgt daraus die
  $\TT_{w^{*},B}$-Kompaktheit, insbesondere Abgeschlossenheit, von
  $\Phi\left(j_{A}^{-1}(\ker\iota_{A}(\rho))\cap S_{A}\right)$.  Eine
  Anwendung des Satzes von Banach-Dieudonné liefert die
  Abgeschlossenheit von
  $\Phi\left(j_{A}^{-1}(\ker\iota_{A}(\rho))\right)$ bezüglich
  $\TT_{w^{*},B}$. Da $C$ bezüglich $\TT_{w^{*},B}$ abgeschlossen ist,
  ist dies analog zu \eqref{eq:inj-wstar-homo-inv-wstar}
  gleichbedeutend zur Abgeschlossenheit bezüglich
  $(\TT_{w^{*},B})|_{C}$. Nach Lemma~\ref{lem:top-wstar-unteralg}
  stimmt jene Topologie mit $\TT_{w^{*},C}$ überein, sodass die
  behauptete Aussage unmittelbar folgt. In Fall
  (\ref{item:inj-wstar-homo-inv-wstar-ii}) erhalten wir auf analoge
  Weise die $\TT_{w}$-Abgeschlossenheit von
  $\Phi\left(j_{A}^{-1}(\ker\iota_{A}(\rho))\cap S_{A}\right)$. Aus
  der Inklusion $\TT_{uw}\supseteq\TT_{w}$ ergibt sich
  $\TT_{uw}$-Abgeschlossenheit. Da die ultraschwache Operatortopologie
  mit $\TT_{w^{*},B}$ übereinstimmt, erhalten wir wie in Fall
  (\ref{item:inj-wstar-homo-inv-wstar-i}) die Abgeschlossenheit von
  $\Phi\left(j_{A}^{-1}(\ker\iota_{A}(\rho))\right)$ bezüglich
  $(\TT_{w^{*},B})|_{C}=\TT_{w^{*},C}$.
\end{proof}
Formuliert man diesen Satz etwas um, so ergeben sich sehr prägnante
Aussagen über injektive $W^{*}$-Homomorphismen. Ein weiteres Argument
liefert einen interessanten Zusammenhang zwischen treuen
$W^{*}$-Darstellungen und Von-Neumann-Darstellungen.
\begin{korollar}\label{kor:inj-wstar-homo-inv-wstar}
  \hspace{0mm}
  \begin{enumerate}[label=(\roman*),ref=\roman*]
  \item\label{item:kor-inj-wstar-homo-inv-wstar-i} Ist $B$ eine
    $W^{*}$-Algebra und $\Phi:A\to B$ ein bijektiver
    $W^{*}$-Homomorphismus, so ist $\Phi$ auch ein Homöomorphismus,
    wenn man $A$ und $B$ mit den jeweiligen schwach-*-Topologien
    versieht. Mit anderen Worten ist $\Phi$ ein $W^{*}$-Isomorphismus.
  \item\label{item:kor-inj-wstar-homo-inv-wstar-ii} Ist $B$ eine
    $W^{*}$-Algebra und $\Phi:A\to B$ ein injektiver
    $W^{*}$-Homomorphismus, so ist $\ran\Phi$ eine
    $W^{*}$-Unteralgebra von $B$. Außerdem sind $A$ und $\ran\Phi$
    isomorph als $W^{*}$-Algebren (d.~h. es existiert ein
    $W^{*}$-Isomorphismus\footnote{Wegen Satz~\ref{satz:inj-homo-iso}
      ist dieser $*$-Algebrenisomorphismus automatisch isometrisch,
      sodass $A$ und $\ran\Phi$ auch isometrisch isomorph sind.} von
    $A$ auf $\ran\Phi$).
  \item\label{item:kor-inj-wstar-homo-inv-wstar-iii} Eine treue
    Von-Neumann-Darstellung $\Phi:A\to L_{b}(H)$ ist auch eine treue
    $W^{*}$-Darstellung. In diesem Fall ist $\ran\Phi$ eine zu $A$ als
    $W^{*}$-Algebra isomorphe Von-Neumann-Algebra.
  \end{enumerate}
\end{korollar}
\begin{proof}
  Die Aussage (\ref{item:kor-inj-wstar-homo-inv-wstar-i}) folgt
  unmittelbar aus
  Satz~\ref{satz:inj-wstar-homo-inv-wstar}(\ref{item:inj-wstar-homo-inv-wstar-i}),
  zum Beweis von (\ref{item:kor-inj-wstar-homo-inv-wstar-ii})
  kombinieren wir (\ref{item:kor-inj-wstar-homo-inv-wstar-i}) mit
  Satz~\ref{satz:inj-wstar-homo-inv-wstar}(\ref{item:inj-wstar-homo-inv-wstar-i}).

  Für Aussage (\ref{item:kor-inj-wstar-homo-inv-wstar-iii}) gilt nach
  Satz~\ref{satz:inj-wstar-homo-inv-wstar}(\ref{item:inj-wstar-homo-inv-wstar-ii}),
  dass $\ran\Phi$ eine Von-Neumann-Algebra ist, also für sich eine
  $W^{*}$-Algebra, wobei noch $\Phi^{-1}:\ran\Phi\to A$ ein
  $W^{*}$-Homomorphismus ist. Klarerweise ist $\Phi^{-1}$ injektiv,
  sodass
  Satz~\ref{satz:inj-wstar-homo-inv-wstar}(\ref{item:inj-wstar-homo-inv-wstar-i}),
  angewandt auf $\Phi^{-1}$, die
  $\TT_{w^{*},A}|\TT_{w^{*},\ran\Phi}$-Stetigkeit von
  $\left(\Phi^{-1}\right)^{-1}=\Phi$ liefert. Aus
  Lemma~\ref{lem:top-wstar-unteralg} folgt
  $\TT_{w^{*},\ran\Phi}=(\TT_{w^{*},L_{b}(H)})|_{\ran\Phi}$, womit
  wegen elementarer topologischer Sachverhalte die Abbildung
  $\Phi:A\to L_{b}(H)$ auch
  $\TT_{w^{*},A}|\TT_{w^{*},L_{b}(H)}$-stetig ist. Mit anderen Worten
  ist $\Phi:A\to L_{b}(H)$ eine treue $W^{*}$-Darstellung. Die
  Isomorphieaussage folgt aus
  (\ref{item:kor-inj-wstar-homo-inv-wstar-ii}).
\end{proof}

Damit können wir -- erneut in Analogie zu
Kapitel~\ref{cha:abstrakte-c-algebren} -- nach einer treuen
Von-Neumann-Darstellung fragen, um die in unserer Verschärfung des
Satzes von Sakai behauptete $W^{*}$-Isomorphie von $A$ und einer
Von-Neumann-Algebra zu beweisen. Es wäre sicherlich natürlicher, eine
treue $W^{*}$-Darstellung zu suchen; die Stetigkeit bezüglich der
schwachen Operatortopologie wird sich aber als wesentlich leichter
handhabbar herausstellen. Daher kann unser Vorgehen als weiteres
Beispiel für den freien Wechsel zwischen den drei Operatortopologien
auf $L_{b}(H)$ gesehen werden.

\section{Beweis des Satzes}
\label{sec:beweis-des-satzes}
Wir wollen erneut die GNS-Konstruktion für positive Funktionale
benützen. Wegen der zusätz\-lichen topologischen Struktur auf einer
$W^{*}$-Algebra ist relativ klar, dass wir uns dabei auf die
schwach-*-stetigen Funktionale beschränken müssen. Im Folgenden
verwenden wir die Notation aus Kapitel~\ref{cha:abstrakte-c-algebren}.

\begin{lemma}\label{lem:gns-darst-wstar}
  Die von einem schwach-*-stetigen, positiven Funktional $\varphi$
  induzierte GNS-Dar\-stellung $\Phi_{\varphi}:A\to L_{b}(H_{\varphi})$
  ist eine Von-Neumann-Darstellung.
\end{lemma}
\begin{proof}
  Die zu zeigende $\TT_{w^{*},A}|\TT_{w}$-Stetigkeitsbedingung ist zur
  Stetigkeit aller Funktionale
  \begin{displaymath}
    f_{x,y}:
    \begin{cases}
      \hfill A&\to\C\\
      a&\mapsto\left(\Phi_{\varphi}(a)x,y\right)_{H_{\varphi}}
    \end{cases}
  \end{displaymath}
  äquivalent. Wegen Korollar~\ref{kor:banach-dieud} und
  Bemerkung~\ref{bem:wstar-top}(\ref{item:wstar-top-ii}) genügt der
  Beweis der Stetigkeit der Einschränkung auf die Einheitskugel
  $S$. Da nach Konstruktion der Prähilbertraum $A/N_{\varphi}$ dicht
  in $H_{\varphi}$ ist, gibt es Folgen\footnote{Man beachte, dass sich
    diese Aussage auf die Normtopologie von $H_{\varphi}$ bezieht,
    weshalb Folgen ausreichend sind.}
  $(x_{n})_{n\in\N}=(\tilde{x}_{n}+N_{\varphi})_{n\in\N}$
  bzw. $(y_{n})_{n\in\N}=(\tilde{y}_{n}+N_{\varphi})_{n\in\N}$ aus
  $A/N_{\varphi}$, die gegen $x$ bzw. $y$ konvergieren. Insbesondere
  ist $(x_{n})_{n\in\N}$ als konvergente Folge auch beschränkt,
  d.~h. $\norm[H_{\varphi}]{x_{n}}\leq C$ für alle $n\in\N$. Die
  Funktionale $f_{x_{n},y_{n}}$ erfüllen für beliebiges $a\in S$ wegen
  $\norm{\Phi_{\varphi}}\leq 1$ die Abschätzung
  \begin{align}
    \begin{split}\label{eq:gns-darst-wstar-i}
      \abs{f_{x,y}(a)-f_{x_{n},y_{n}}(a)}&=\abs{\left(\Phi_{\varphi}(a)(x-x_{n}),y\right)_{H_{\varphi}}+\left(\Phi_{\varphi}(a)x_{n},y-y_{n}\right)_{H_{\varphi}}}\\
      &\leq\abs{\left(\Phi_{\varphi}(a)(x-x_{n}),y\right)_{H_{\varphi}}}+\abs{\left(\Phi_{\varphi}(a)x_{n},y-y_{n}\right)_{H_{\varphi}}}\\
      &\leq\norm{\Phi_{\varphi}}\norm{a}\norm[H_{\varphi}]{x-x_{n}}\norm[H_{\varphi}]{y}+\norm{\Phi_{\varphi}}\norm{a}\norm[H_{\varphi}]{x_{n}}\norm[H_{\varphi}]{y-y_{n}}\\
      &\leq\norm[H_{\varphi}]{x-x_{n}}\norm[H_{\varphi}]{y}+C\norm[H_{\varphi}]{y-y_{n}}.
    \end{split}
  \end{align}
  Dieser Ausdruck konvergiert für $n\to\infty$ gegen $0$ und ist von
  $a$ unabhängig, sodass $f_{x_{n},y_{n}}$ auf $S$ gleichmäßig gegen
  $f_{x,y}$ konvergiert. Wegen
  \begin{displaymath}
    f_{x_{n},y_{n}}(a)=\left(\Phi_{\varphi}(a)(\tilde{x}_{n}+N_{\varphi}),\tilde{y}_{n}+N_{\varphi}\right)_{H_{\varphi}}=(a\tilde{x}_{n}+N_{\varphi},\tilde{y}_{n}+N_{\varphi})_{H_{\varphi}}=\varphi((a\tilde{x}_{n})^{*}\tilde{y}_{n})=\varphi(\tilde{x}_{n}^{*}a^{*}\tilde{y}_{n})
  \end{displaymath}
  ist die Funktion $f_{x_{n},y_{n}}$ als Verkettung der Operation
  $.^{*}$, der multiplikativen Translationen
  $b\mapsto\tilde{x}_{n}^{*}b$ und $b\mapsto b\tilde{y}_{n}$ sowie des
  Funktionals $\varphi$ schwach-*-stetig; siehe
  Korollar~\ref{kor:wstar-adj-stetig} und
  Satz~\ref{satz:mult-wstar-stetig}. Damit ist $f_{x,y}$ als
  gleichmäßiger Grenzwert stetiger Abbildungen selbst stetig auf $S$.
\end{proof}

Nun sind wir in der Lage, die angestrebte treue
Von-Neumann-Darstellung zu definieren. Dazu betrachten wir die Menge
$\mathcal{S}_{w^{*}}(A)$ aller schwach-*-stetigen Zustände auf $A$ und die von
ihnen induzierten GNS-Darstellungen
$\Phi_{\varphi}:A\to L_{b}(H_{\varphi})$. Analog zum Beweis des Satzes
von Gelfand-Naimark definieren wir die direkte Summe
$H:=\bigoplus_{\varphi\in\mathcal{S}_{w^{*}}(A)}H_{\varphi}$ und die Abbildung
$\Phi:A\to L_{b}(H)$ unter Verwendung von
\eqref{eq:univ-darst}. Klarerweise handelt es sich dabei um eine
Darstellung (im Sinne von Definition~\ref{def:darstellung}), die wir
die \emph{universelle $W^{*}$-Darstellung} nennen\footnote{Dass es
  sich dabei um eine $W^{*}$-Darstellung handelt, wird aber erst der
  Satz von Sakai zeigen.}. Der Satz von Sakai bekommt die folgende
Form:
\begin{satz}[Sakai]\label{satz:sakai}
  Die universelle $W^{*}$-Darstellung $\Phi$ von $A$ ist eine treue
  Von-Neumann- und $W^{*}$-Darstellung. Insbesondere ist $A$ isomorph
  als $W^{*}$-Algebra zu einer Von-Neumann-Algebra.
\end{satz}
\begin{proof}
  Der Beweis gliedert sich in zwei Teile. Wir müssen zeigen, dass
  $\Phi$ eine Von-Neumann-Darstellung ist und dass $\Phi$ treu
  ist. Diese Tatsachen gemeinsam mit
  Korollar~\ref{kor:inj-wstar-homo-inv-wstar} liefern dann die
  Aussage.

  Mit vollkommen analoger Begründung zum Beweis von
  Lemma~\ref{lem:gns-darst-wstar} genügt es für den ersten Teil zu
  zeigen, dass die Funktionale
  $f_{x,y}(a):=(\Phi(a)x,y)_{H}=\sum_{\varphi\in\mathcal{S}_{w^{*}}(A)}(\Phi_{\varphi}(a)x_{\varphi},y_{\varphi})_{H_{\varphi}}$
  auf $S$ stetig sind, wobei $x,y\in H$ beliebig zu wählen sind. Die
  Menge
  \begin{displaymath}
    Y:=\set{(x_{\varphi})_{\varphi\in\mathcal{S}_{w^{*}}(A)}\in
      H}{x_{\varphi}\neq 0\,\,\text{nur für endlich
        viele}\,\,\varphi\in\mathcal{S}_{w^{*}}(A)}
  \end{displaymath}
  ist in $H$ dicht, sodass Folgen $(x_{n})_{n\in\N}, (y_{n})_{n\in\N}$
  aus $Y$ existieren mit $x_{n}\to x$ und $y_{n}\to y$. Die
  Abschätzung~\eqref{eq:gns-darst-wstar-i} zeigt, dass die
  Funktionenfolge $f_{x_{n},y_{n}}|_{S}$ gleichmäßig gegen
  $f_{x,y}|_{S}$ konvergiert. Somit genügt es, die Stetigkeit der
  Funktionale $f_{x_{n},y_{n}}$ zu zeigen. Anders formuliert können
  wir uns auf den Fall
  $x=(x_{\varphi})_{\varphi\in\mathcal{S}_{w^{*}}(A)},y=(y_{\varphi})_{\varphi\in\mathcal{S}_{w^{*}}(A)}\in
  Y$ beschränken. Nach Definition von $Y$ gibt es endliche Mengen
  $F_{x},F_{y}\in\EE(S_{w^{*}}(A))$ mit $x_{\varphi}=0$ für alle
  $\varphi\notin F_{x}$ und $y_{\varphi}=0$ für alle
  $\varphi\notin F_{y}$. Die Menge $F:=F_{x}\cup F_{y}$ ist
  klarerweise ebenfalls endlich und es gilt
  $x_{\varphi}=0=y_{\varphi}$ für jedes $\varphi\notin F$. Wir
  erhalten
  \begin{displaymath}
    f_{x,y}(a)=\sum_{\varphi\in F}(\Phi_{\varphi}(a)x_{\varphi},y_{\varphi})_{H_{\varphi}}.
  \end{displaymath}
  Die einzelnen Summanden sind $\TT_{w^{*},A}$-stetig, da die
  $\Phi_{\varphi}$ Von-Neumann-Darstellungen sind, womit auch deren
  endliche Summe stetig ist.

  Für den zweiten Teil sei $\Phi(a)=0$, also $\Phi_{\varphi}(a)=0$ für
  alle $\varphi\in\mathcal{S}_{w^{*}}(A)$. Setzen wir Elemente der
  Form $b+N_{\varphi}$ in $\Phi_{\varphi}(a)=0$ ein, so erhalten wir
  $ab+N_{\varphi}=0$, also $ab\in N_{\varphi}$, für alle $b\in A$. Das
  bedeutet $\varphi((ab)^{*}ab)=0$ für beliebiges
  $\varphi\in\mathcal{S}_{w^{*}}(A)$. Durch nötigenfalls
  erforderliches Skalieren gilt dies sogar für alle
  schwach-*-stetigen, positiven Funktionale. Aus
  Korollar~\ref{kor:wstar-pos-fkt-trennend} folgt $(ab)^{*}(ab)=0$ und
  daher $ab=0$. Setzen wir $b=a^{*}$, so ergibt sich
  \begin{displaymath}
    \norm{a}=\norm{a^{*}}=\norm{aa^{*}}^{1/2}=0.
  \end{displaymath}
\end{proof}

\begin{bemerkung}
  Nach Lemma~\ref{lem:top-wstar-unteralg} ist die schwach-*-Topologie
  auf einer Von-Neumann-Algebra genau die Spurtopologie der
  ultraschwachen Operatortopologie. Bei unserem abstrakten und
  axiomatischen Zugang zu Von-Neumann-Algebren kann daher die
  ultraschwache Topologie am ehesten als kanonische Wahl unter den
  Operatortopologien angesehen werden. Es sei daran erinnert, dass in
  Definition~\ref{def:vn-alg} alle drei Operatortopologien
  $\TT_{s},\TT_{w},\TT_{uw}$ gleichberechtigt sind.
\end{bemerkung}

\clearpage

\chapter{Eindeutigkeit}
\label{cha:eindeutigkeit}
So erfolgreich der Begriff der $W^{*}$-Algebren dabei ist,
Von-Neumann-Algebren zu axiomatisieren, so unbefriedigend erscheint
die Definition. Eine $W^{*}$-Algebra ist ja eine $C^{*}$-Algebra, die
(als Banachraum) isometrisch isomorph zum Dualraum eines Banachraums
ist. Davon ausgehend haben wir die schwach-*-Topologie auf $A$
definiert, sodass die gesamte Theorie an die Wahl des Prädualraums und
des isometrischen Isomorphismus anknüpft. Legt die Struktur von $A$
diese beiden Objekte nicht in ausreichender Weise fest, so wäre es
denkbar, dass ein weiterer Banachraum samt isometrischem Isomorphismus
existiert, der eine andere schwach-*-Topologie auf $A$ induziert. Dies
wäre nicht nur ästhetisch unvorteilhaft, sondern könnte die Analyse
von $W^{*}$-Algebren erschweren -- für die "`gleiche"' $W^{*}$-Algebra
könnte die Gültigkeit einer Aussage davon abhängen, welche Topologie
man betrachtet.

\section{Schwach-*-stetige und positive Funktionale}
\label{sec:wstar-stetige-pos-fkt}
Für Banachräume $\pred{A}$ und $\pred[2]{A}$ sowie isometrische
Isomorphismen $j:A\to\pred{A}'$ und $j_{2}:A\to\pred[2]{A}'$ ist
klarerweise $j_{2}\circ j^{-1}$ ein isometrischer Isomorphismus
zwischen $\pred{A}'$ und $\pred[2]{A}'$. Es ist zu zeigen, dass auch
$\pred{A}$ und $\pred[2]{A}$ isometrisch isomorph sind, wobei der
Isomorphismus $\pred{A}\to\pred[2]{A}$ mit den Abbildungen $j$ und
$j_{2}$ in gewisser Weise verträglich sein muss. Es wäre ja auch
möglich, dass es zu einem \emph{festen} Prädualraum zwei
\emph{verschiedene} isometrische Isomorphismen gibt, die verschiedene
schwach-*-Topologien induzieren. Dass bereits der erste Schritt alles
andere als trivial ist, zeigt das folgende Beispiel.
\begin{beispiel}\label{bsp:dualraum-iso}
  Der Raum $c(\N,\C)$ der konvergenten, komplexwertigen Folgen und der
  Raum $c_{0}(\N,\C)$ der komplexwertigen Nullfolgen, beide versehen
  mit der Supremumsnorm
  \begin{displaymath}
    \norm[\infty]{(x_{n})_{n\in\N}}:=\sup_{n\in\N}\abs{x_{n}},
  \end{displaymath}
  haben bis auf isometrische Isomorphie den gleichen Dualraum. Es gilt
  nämlich
  \begin{equation}\label{eq:bsp-eindeut-i}
    (c_{0}(\N,\C),\norm[\infty]{\cdot})'\cong (\ell^{1}(\N,\C),\norm[1]{\cdot}) \cong
    (c(\N,\C),\norm[\infty]{\cdot})'
  \end{equation}
  mit dem Raum $\ell^{1}(\N,\C)$ der absolut summierbaren Folgen.
  Dabei sind die Isomorphismen $\ell^{1}(\N,\C)\to c(\N,\C)'$
  bzw. $\ell^{1}(\N,\C)\to c_{0}(\N,\C)'$ durch
  \begin{align*}
    j_{0}&:\begin{cases}
      \ell^{1}(\N,\C)&\to c_{0}(\N,\C)' \\
      \hfill (y_{n})_{n\in\N}&\mapsto\Big((x_{n})_{n\in\N}\mapsto\sum_{n=1}^{\infty}x_{n}y_{n}\Big)
    \end{cases}\\
    \vspace{1cm}
         &\qquad\qquad\qquad\text{sowie}\\
    j&:\begin{cases}
      \ell^{1}(\N,\C)&\to c(\N,\C)' \\
      \hfill (y_{n})_{n\in\N}&\mapsto\Big((x_{n})_{n\in\N}\mapsto
      (\lim_{n\to\infty}x_{n})\cdot
      y_{1}+\sum_{n=1}^{\infty}x_{n}y_{n+1}\Big)
    \end{cases}
  \end{align*}
  gegeben. Wir zeigen zunächst, dass $j_{0}$ isometrisch ist. Ist
  $y=(y_{n})_{n\in\N}\in\ell^{1}(\N,\C)$, so zeigt direktes
  Nachrechnen $\norm{j_{0}(y)}\leq\norm[1]{y}$. Für die umgekehrte
  Ungleichung sei $\alpha_{n}\in\C$ mit $\abs{\alpha_{n}}=1$ und
  $\alpha_{n}y_{n}=\abs{y_{n}}$. Die Folge $x:=(x_{n})_{n\in\N}$
  definiert durch
  \begin{displaymath}
    x_{n}:=
    \begin{cases}
      \alpha_{n},& n\leq N\\
      \hfill 0,& \text{sonst}
    \end{cases}
  \end{displaymath}
  ist für $N\in\N$ klarerweise in $c_{0}(\N,\C)$ enthalten mit
  $\norm[\infty]{x}=1$. Daraus folgt
  \begin{displaymath}
    \norm{j_{0}(y)}\geq\chevr{x}{j_{0}(y)}=\sum_{n=1}^{N}\alpha_{n}y_{n}=\sum_{n=1}^{N}\abs{y_{n}}.
  \end{displaymath}
  Lassen wir in dieser Ungleichung $N$ gegen $\infty$ streben, so
  erhalten wir
  \begin{equation}\label{eq:bsp-eindeut-ii}
    \norm{j_{0}(y)}\geq\sum_{n=1}^{\infty}\abs{y_{n}}=\norm[1]{y},
  \end{equation}
  womit die Isometrie von $j_{0}$ gezeigt ist. Für den Nachweis der
  Surjektivität sei $\varphi\in c_{0}(\N,\C)'$ gegeben. Wir definieren
  $y_{n}:=\chevr{e_{n}}{\varphi}$, wobei
  $e_{n}:=(\delta_{nk})_{k\in\N}$ die $n$-te kanonische Folge
  ist. Setzen wir $y:=(y_{n})_{n\in\N}$ und davon ausgehend
  $\alpha_{n}$ sowie $x$ wie oben, so folgt ähnlich wie in
  \eqref{eq:bsp-eindeut-ii} die Abschätzung
  $\norm{\varphi}\geq\sum_{n=1}^{\infty}\abs{y_{n}}$. Insbesondere
  erhalten wir $y\in\ell^{1}(\N,\C)$. Aus der Linearität von $\varphi$
  und $j_{0}(y)$ ergibt sich, dass die beiden Funktionale auf
  $\spn\set{e_{n}}{n\in\N}$, also auf jenen Folgen mit nur endlich
  vielen Einträgen ungleich $0$, übereinstimmen. Dieser Raum ist dicht
  in $c_{0}(\N,\C)$, sodass wir wegen der Stetigkeit
  $\varphi=j_{0}(y)$ erhalten.

  Für $j$ argumentieren wir analog. Um die Isometrie nachzuweisen, ist
  wieder nur $\norm{j(y)}\geq\norm[1]{y}$ für alle
  $y=(y_{n})_{n\in\N}\in\ell^{1}(\N,\C)$ zu zeigen. Dazu sei
  $\alpha_{n}$ definiert wie oben. Wir betrachten
  \begin{displaymath}
    x_{n}:=
    \begin{cases}
      \alpha_{n+1},& n\leq N\\
      \alpha_{1},& \text{sonst}
    \end{cases}
  \end{displaymath}
  und bemerken $x=(x_{n})_{n\in\N}\in c(\N,\C)$ mit
  $\lim_{n\to\infty}x_{n}=\alpha_{1}$ und $\norm[\infty]{x}=1$. Es
  gilt also
  \begin{displaymath}
    \norm{j(y)}\geq\chevr{x}{j(y)}=\alpha_{1}y_{1}+\sum_{n=1}^{N}\alpha_{n+1}y_{n+1}+\alpha_{1}\sum_{n=N+1}^{\infty}y_{n+1}=\sum_{n=1}^{N+1}\abs{y_{n}}+\alpha_{1}\sum_{n=N+2}^{\infty}y_{n}.
  \end{displaymath}
  Der Grenzübergang $N\to\infty$ liefert auch hier die gewünschte
  Ungleichung $\norm{j(y)}\geq\norm[1]{y}$. Zum Beweis der
  Surjektivität sei $\varphi\in c(\N,\C)'$ beliebig. Klarerweise ist
  die Einschränkung von $\varphi$ auf $c_{0}(\N,\C)$ in
  $c_{0}(\N,\C)'$ enthalten, sodass wir nach dem Obigen
  $\varphi|_{c_{0}(\N,\C)}=j_{0}(z)$ für eine Folge
  $z=(z_{n})_{n\in\N}\in\ell^{1}(\N,\C)$ schreiben können. Wir
  definieren $e:=(1,1,\dots)\in c(\N,\C)$ und damit
  $\gamma:=\chevr{e}{\varphi}$. Für $x=(x_{n})_{n\in\N}\in c(\N,\C)$
  liegt die Folge $x-(\lim_{k\to\infty}x_{k})\cdot e$ in
  $c_{0}(\N,\C)$. Daraus folgt
  \begin{align*}
    \chevr{x}{\varphi}&=\chevr{x-\left(\lim_{k\to\infty}x_{k}\right)\cdot
                        e}{\varphi}+\left(\lim_{k\to\infty}x_{k}\right)\cdot\chevr{e}{\varphi}\\
                      &=\chevr{x-\left(\lim_{k\to\infty}x_{k}\right)\cdot
                        e}{j_{0}(z)}+\left(\lim_{k\to\infty}x_{k}\right)\cdot\chevr{e}{\varphi}\\
                      &=\sum_{n=1}^{\infty}\left(x_{n}-\lim_{k\to\infty}x_{k}\right)z_{n}+\left(\lim_{k\to\infty}x_{k}\right)\cdot\gamma=\\
                      &=\sum_{n=1}^{\infty}x_{n}z_{n}+\left(\lim_{k\to\infty}x_{k}\right)\cdot\left(\gamma-\sum_{n=1}^{\infty}z_{n}\right).
  \end{align*}
  Somit können wir $y_{1}:=\gamma-\sum_{n=1}^{\infty}z_{n}$ und
  $y_{n+1}:=z_{n}$ setzen, um $\varphi=j(y)$ zu erhalten. Wir haben
  also \eqref{eq:bsp-eindeut-i} gezeigt.

  Allerdings sind $c(\N,\C)$ und $c_{0}(\N,\C)$ nicht isometrisch
  isomorph: Die konstante Folge $e=(1,1,\dots)$ ist ein Extremalpunkt
  der Einheitskugel in $c(\N,\C)$, im Raum $c_{0}(\N,\C)$ hat die
  Einheitskugel hingegen keine Extremalpunkte.
\end{beispiel}
In unserer Argumentation werden wir ausgehend von $A$ anstelle des
schwer fassbaren Prädual\-raums den Dualraum untersuchen. Als
Motivation dafür kann die Tatsache dienen, dass sich für eine
herkömmliche schwach-*-Topologie $\sigma(X',X)$ in Banachräumen der
Prädualraum $X$ aus dem Dualraum zurückgewinnen lässt. Es gilt nämlich
$(X',\sigma(X',X))'=\iota(X)$ mit der kanonischen Einbettung
$\iota:X\to X''$. Dabei werden wir zeigen, dass man die
schwach-*-Stetigkeit eines Funktionals $\varphi:A\to\C$ auch ohne
direkte Bezugnahme auf die Topologie und damit den Prädualraum
charakterisieren kann, nämlich über die Ordnungsstruktur auf
$A$. Direkt gelingt dies allerdings nur für \emph{positive}
Funktionale, sodass wir anschließend noch ein Zerlegungsresultat für
Funktionale verwenden.

Bis auf Weiteres halten wir einen Prädualraum $\pred{A}$ und einen
isometrischen Isomorphismus $j:A\to\pred{A}'$ fest und betrachten die
davon ausgehend definierte schwach-*-Topologie $\TT_{w^{*}}$ auf
$A$.

Für das erste Ziel sind noch einige Vorarbeiten notwendig. Zunächst
definieren wir eine weitere Topologie auf $A$. Dazu verwenden wir
nochmals den allerersten Schritt der Konstruktion der universellen
($W^{*}$-)Darstellung, nämlich die Verbindung zwischen positiven
Funktionalen und Seminormen,
vgl. Lemma~\ref{lem:eig-pos-fkt}(\ref{item:eig-pos-fkt-iv}).
\begin{definition}\label{def:q-top}
  \hspace{0mm}
  \begin{enumerate}[label=(\roman*),ref=\roman*]
  \item Sei $\varphi$ ein schwach-*-stetiges, positives Funktional auf
    $A$. Die Seminorm $\alpha_{\varphi}$ auf $A$ sei definiert durch
    $\alpha_{\varphi}(x):=\varphi(x^{*}x)^{1/2}$.
  \item Die von der Familie
    $\set{\alpha_{\varphi}}{\varphi \ \text{ist schwach-*-stetig und
        positiv}}$ von Seminormen erzeugte lokalkonvexe Topologie auf
    $A$ heißt die \emph{q-Topologie}, in Zeichen $\TT_{q,A}$ oder
    kürzer $\TT_{q}$.
  \end{enumerate}
\end{definition}
Bevor wir fortfahren, ist zu klären, dass die Familie von Seminormen
aus dieser Definition separierend ist. Dies ergibt sich unmittelbar
aus Korollar~\ref{kor:wstar-pos-fkt-trennend}: Ist nämlich
$\varphi(x^{*}x)^{1/2}=0$ für alle schwach-*-stetigen und positiven
Funktionale $\varphi$, so folgt aus dem Korollar $x^{*}x=0$ und daraus
$\norm{x}=\norm{x^{*}x}^{1/2}=0$.

Es sei noch betont, dass auch bei der q-Topologie $\pred{A}$ und $j$
in die Definition eingehen; wir unterdrücken diese Abhängigkeit
allerdings zwecks übersichtlicherer Notation.
\begin{bemerkung}\label{bem:gns-sichtweise-2}
  Die Konstruktion der q-Topologie verwendet eine andere Sichtweise
  auf die GNS-Konstruktion als bei der universellen
  ($W^{*}$-)Darstellung; vgl. Bemerkung~\ref{bem:gns-sichtweise-1}. Es
  wirkt nämlich nicht $A$ auf einem Hilbertraum $H$, sondern wir
  betrachten die Konstruktion des zugrundeliegenden Hilbertraums
  $H$. Auch dabei waren die Seminormen $\alpha_{\varphi}$ der
  Ausgangspunkt -- um einen Hilbertraum zu erhalten, haben wir
  zusätzlich nach dem isotropen Anteil faktorisiert, was für die
  q-Topologie nicht notwendig war.
\end{bemerkung}

Ein wesentlicher Zwischenschritt ist es zu zeigen, dass der Dualraum
$(A,\TT_{q})'$ mit $(A,\TT_{w^{*}})'$ übereinstimmt; mit anderen
Worten sind genau die schwach-*-stetigen Funktionale auch be\-züglich
der q-Topologie stetig. Das entscheidende Hilfsmittel dazu ist die
Mackey-Topologie; siehe Definition~\ref{def:mackey-top}. Da wir die
Mackey-Topologie zur schwach-*-Topologie auf $A$ benötigen, gehen wir
analog zur Definition von $\TT_{w^{*}}$ vor und definieren die
Mackey-Topologie auf $A$ mithilfe von $j$ und der Mackey-Topologie
$\tau(\pred{A}',\pred{A})$.
\begin{definition}\label{def:mackey-top-wstar}
  Das Mengensystem
  $\TT_{\tau,A}:=\set{j^{-1}(O)}{O\in\tau(\pred{A}',\pred{A})}$ --
  oder kür\-zer $\TT_{\tau}$ -- bezeichnet man als
  \emph{Mackey-Topologie auf $A$}.
\end{definition}
\begin{bemerkung}\label{bem:mackey-top-wstar}
  \hspace{0mm}
  \begin{enumerate}[label=(\roman*),ref=\roman*]
  \item\label{item:mackey-top-wstar-i} Da $j$ sowohl als Abbildung
    \begin{displaymath}
      j:\left(A,\TT_{w^{*}}\right)\to\left(\pred{A}',\sigma(\pred{A}',\pred{A})\right)
    \end{displaymath}
    als auch als Abbildung
    \begin{displaymath}
      j:\left(A,\TT_{\tau}\right)\to\left(\pred{A}',\tau(\pred{A}',\pred{A})\right)
    \end{displaymath}
    ein Homöomorphismus ist, vererbt sich der Satz von Mackey-Arens,
    Satz~\ref{satz:mackey-arens}, von der gewöhnlichen
    Mackey-Topologie auf die Mackey-Topologie $\TT_{\tau,A}$;
    vgl. Bemerkung~\ref{bem:wstar-top}(\ref{item:wstar-top-ii}) für
    analoge Aussagen über die schwach-*-Topologie auf $A$.
  \item\label{item:mackey-top-wstar-ii} Ein Netz $(x_{i})_{i\in I}$
    aus $A$ konvergiert genau dann gegen $x\in A$ bezüglich der
    Mackey-Topologie $\TT_{\tau}$, wenn $(j(x_{i}))_{i\in I}$
    bezüglich $\tau(\pred{A}',\pred{A})$ gegen $j(x)$
    konvergiert. Explizit bedeutet das, dass für alle $\epsilon>0$ und
    alle kreisförmigen, konvexen und bezüglich
    $\sigma(\pred{A},\pred{A}')$ kompakten Mengen $C\subseteq\pred{A}$
    ein $i_{0}\in I$ existiert, sodass für jedes $\rho\in C$ und
    $i\succcurlyeq i_{0}$ die Ungleichung
    $\abs{\chevr{\rho}{j(x_{i})-j(x)}}<\epsilon$ gilt.
  \end{enumerate}
\end{bemerkung}
Wir wollen Konvergenz bezüglich $\TT_{\tau}$ ohne Verweis auf den
Prädualraum charakterisieren. Die konjugierte Abbildung
$j':\pred{A}''\to (A,\TT(\norm{\cdot}))'$ ist ebenfalls ein
isometrischer Isomorphismus. Mit dieser Notation gilt
\begin{equation}\label{eq:mackey-top-wstar-konv}
  \chevr{\rho}{j(x_{i})-j(x)}=\chevr{j(x_{i}-x)}{\iota(\rho)}=\chevr{x_{i}-x}{j'(\iota(\rho))}.
\end{equation}
Wir erhalten, dass Konvergenz bezüglich der Mackey-Topologie auf $A$
äquivalent ist zur gleichgradig schwachen Konvergenz
(vgl. Bemerkung~\ref{bem:mackey-top-konv}) auf den
$j'\circ\iota$-Bildern der bezüglich $\sigma(\pred{A},\pred{A}')$
kompakten Teilmengen von $\pred{A}$. Es ist wohlbekannt bzw. leicht zu
prüfen, dass die kanonische Einbettung $\iota:\pred{A}\to\pred{A}''$
betrachtet als Abbildung
\begin{equation}\label{eq:iota-homoeo}
  \iota:\left(\pred{A},\sigma(\pred{A},\pred{A}')\right)\to
\left(\iota(\pred{A}),\sigma(\pred{A}'',\pred{A}')|_{\iota(\pred{A})}\right)
\end{equation}
ein Homöomorphismus ist. Somit durchlaufen die Bilder $\iota(C)$ für
$\sigma(\pred{A},\pred{A}')$-kompakte Mengen $C$ genau die
$\sigma(\pred{A}'',\pred{A}')$-kompakten Teilmengen von
$\iota(\pred{A})$. Es geht also um die Frage, welche Mengen als
$j'$-Bilder davon auftreten.

\begin{notation}\label{not:dualraum-wstar}
  Im Folgenden bezeichnet $A^{\#}$ stets den Dualraum von $A$
  bezüglich der schwach-*-Topologie, d.~h. $A^{\#}:=(A,\TT_{w^{*}})'$.
\end{notation}
Da sich herausstellen wird, dass die schwach-*-Topologie eindeutig
bestimmt ist, wird auch $A^{\#}$ nicht vom gewählten Prädualraum
abhängen; siehe Korollar~\ref{kor:wstar-fkt-zerl}. Daher muss der
Prädual\-raum nicht durch die Notation reflektiert werden.
\begin{lemma}\label{lem:tech-mackey-wstar}
  \hspace{0mm}
  \begin{enumerate}[label=(\roman*),ref=\roman*]
  \item\label{item:tech-mackey-wstar-i} Die Einschränkung von $j'$ auf
    $\iota(\pred{A})$ ist ein Homöomorphismus
    \begin{displaymath}
      j'|_{\iota(\pred{A})}:\left(\iota(\pred{A}),\sigma(\pred{A}'',\pred{A}')|_{\iota(\pred{A})}\right)\to \left(A^{\#},\sigma(A^{\#},A)\right).
    \end{displaymath}
  \item\label{item:tech-mackey-wstar-ii} Ein Netz $(x_{i})_{i\in I}$
    aus $A$ konvergiert genau denn gegen $x\in A$ bezüglich der
    Mackey-Topologie auf $A$, wenn das Netz auf den kreisförmigen,
    konvexen und $\sigma(A^{\#},A)$-kompakten Teilmengen von $A^{\#}$
    gleichgradig schwach konvergiert.
  \end{enumerate}
\end{lemma}
\begin{proof}
  \hspace{0mm}
  \begin{enumerate}[label=(\roman*),ref=\roman*]
  \item Bisher wissen wir nur, dass $j'$ bijektiv ist als Abbildung
    vom gesamten Bidualraum $\pred{A}''$ in den Dualraum von $A$
    bezüglich der Normtopologie; für das Lemma betrachten wir die
    Einschränkung von $j'$ auf $\iota(\pred{A})$ und den Dualraum
    bezüglich der schwach-*-Topologie auf $A$.

    Daher zeigen wir zunächst, dass die Einschränkung
    $j'|_{\iota(\pred{A})}$ tatsächlich nach $A^{\#}$ abbildet. Nach
    Definition der schwach-*-Topologie auf $A$ ist ein Funktional
    $\varphi:A\to\C$ genau dann $\TT_{w^{*}}$-stetig, wenn
    $\varphi\circ j^{-1}$ stetig bezüglich
    $\sigma(\pred{A}',\pred{A})$ ist. Für $\rho\in\pred{A}$ gilt
    \begin{equation}\label{eq:bew-tech-mackey-wstar}
      j'(\iota(\rho))\circ j^{-1}=\iota(\rho)\circ j\circ j^{-1}=\iota(\rho),
    \end{equation}
    womit die Behauptung gezeigt ist.

    Aus \eqref{eq:bew-tech-mackey-wstar} folgt auch die Surjektivität
    von $j'|_{\iota(\pred{A})}$: Ist $\varphi$ ein schwach-*-stetiges
    Funktional auf $A$, so ist $\varphi\circ j^{-1}$ ein
    $\sigma(\pred{A}',\pred{A})$-stetiges Funktional. Somit gilt
    $\varphi\circ j^{-1}=\iota(\rho)$ für ein $\rho\in\pred{A}$ oder
    anders formuliert $\varphi=j'(\iota(\rho))$.

    Da $j'$ als Abbildung von $\pred{A}''$ in den Dualraum von $A$
    bezüglich der Normtopologie ein isometrischer Isomorphismus ist,
    ist $j'$ und insbesondere die Einschränkung
    $j'|_{\iota(\pred{A})}$ injektiv.
    
    Für die Homöomorphie-Eigenschaft sei $(\iota(\rho_{i}))_{i\in I}$
    ein Netz in $\iota(\pred{A})$ und $x\in A$ beliebig. Es gilt
    \begin{displaymath}
      \chevr{x}{j'(\iota(\rho_{i}))}=\chevr{x}{\iota(\rho_{i})\circ
        j}=\chevr{j(x)}{\iota(\rho_{i})}.
    \end{displaymath}
    Da $j$ surjektiv ist, folgt daraus, dass
    $(j'(\iota(\rho_{i})))_{i\in I}$ genau dann bezüglich
    $\sigma(A^{\#},A)$ gegen $0$ konvergiert, wenn
    $(\iota(\rho_{i}))_{i\in I}$ bezüglich
    $\sigma(\pred{A}'',\pred{A}')|_{\iota(\pred{A})}$ gegen $0$
    konvergiert. Dies ist äquivalent dazu, dass
    $j'|_{\iota(\pred{A})}$ ein Homöomorphismus bezüglich der oben
    angegebenen Topologien ist.
  \item Folgt sofort aus (\ref{item:tech-mackey-wstar-i}) und den
    Bemerkungen vor diesem Lemma.
  \end{enumerate}
\end{proof}
Mit der Topologie $\sigma(A^{\#},A)$ auf $A^{\#}$ wird es möglich,
Abbildungen, die Funktionale auf $A$ liefern, auf Stetigkeit zu
überprüfen. Ein Beispiel für eine derartige Operation ist die folgende
Definition, die wir zunächst für beliebige, nicht notwendigerweise
schwach-*-stetige Funktionale einführen.
\begin{definition}\label{def:l-op}
  Sei $\varphi$ ein Funktional auf $A$ und $a\in A$. Das Funktional
  $L_{a}\varphi$ ist definiert durch
  $(L_{a}\varphi)(x):=\varphi(ax)$.
\end{definition}
Vor dem Hintergrund von
\ref{lem:tech-mackey-wstar}(\ref{item:tech-mackey-wstar-ii}) wird das
folgende Stetigkeitsresultat sehr nützlich sein.
\begin{lemma}\label{lem:l-op-stetig}
  Sei $\varphi$ ein schwach-*-stetiges Funktional auf $A$. Die
  Funktion $a\mapsto L_{a}\varphi$ ist stetig als Abbildung
  $(A,\TT_{w^{*}})\to (A^{\#},\sigma(A^{\#},A))$. Insbesondere ist
  $L_{rS}\varphi:=\set{L_{a}\varphi}{a\in rS}$ für jedes $r>0$ kompakt
  bezüglich $\sigma(A^{\#},A)$.
\end{lemma}
\begin{proof}
  Da sowohl die Translation $x\mapsto ax$ als auch $\varphi$
  schwach-*-stetig sind, gilt $L_{a}\varphi\in A^{\#}$.

  Gelte nun $a_{i}\to a$ bezüglich $\TT_{w^{*}}$ für ein Netz
  $(a_{i})_{i\in I}$. Wieder wegen der Stetigkeit einer Translation,
  diesmal von $b\mapsto bx$ für festes $x\in A$, gilt $a_{i}x\to ax$,
  sodass aus der schwach-*-Stetigkeit von $\varphi$ die Konvergenz
  \begin{displaymath}
    \chevr{x}{L_{a_{i}}\varphi}=\varphi(a_{i}x)\to\varphi(ax)=\chevr{x}{L_{a}\varphi}
  \end{displaymath}
  folgt. Somit ist die behauptete Stetigkeit gezeigt.

  Die zweite Aussage ergibt sich sofort aus der
  $\TT_{w^{*}}$-Kompaktheit der Menge $rS$.
\end{proof}
Zur Bestimmung des Dualraums $(A,\TT_{q})'$ werden wir das folgende
Konzept verwenden.
\begin{definition}\label{def:selbstadj-fkt-re-im}
  Sei $\varphi$ ein lineares Funktional auf $A$.
  \begin{enumerate}[label=(\roman*),ref=\roman*]
  \item Das \emph{adjungierte} Funktional ist definiert durch
    $\varphi^{*}(x):=\overline{\varphi(x^{*})}$.
  \item Im Falle $\varphi^{*}=\varphi$ heißt $\varphi$
    \emph{selbstadjungiert}.
  \item Die linearen Funktionale
    \begin{displaymath}
      \re\varphi:=(\varphi+\varphi^{*})/2 \quad\text{und}\quad \im\varphi:=(\varphi-\varphi^{*})/2i
    \end{displaymath}
    heißen \emph{Real-} und \emph{Imaginärteil} von $\varphi$.
  \end{enumerate}
\end{definition}
Analog zu Lemma~\ref{lem:eig-re-im} gilt:
\begin{lemma}\label{lem:eig-re-im-fkt}
  $\varphi^{*}$, $\re\varphi$ und $\im\varphi$ sind lineare
  Funktionale auf $A$, wobei $\varphi=\re\varphi+i\im\varphi$. Ist
  $\varphi$ schwach-*-stetig, so auch $\varphi^{*}$, $\re\varphi$ und
  $\im\varphi$.
\end{lemma}
\begin{proof}
  Die erste Aussage ist klar, die zweite folgt aus der
  schwach-*-Stetigkeit von $.^{*}$; siehe
  Korollar~\ref{kor:wstar-adj-stetig}.
\end{proof}
\begin{lemma}\label{lem:dualraum-q-top}
  Es gilt $(A,\TT_{q})'=(A,\TT_{w^{*}})'$.
\end{lemma}
\begin{proof}
  Sei zunächst $\varphi$ ein $\TT_{q}$-stetiges Funktional auf $A$. Um
  die schwach-*-Stetigkeit von $\varphi$ zu zeigen, genügt es nach dem
  Satz von Mackey-Arens, Satz~\ref{satz:mackey-arens}
  bzw. Bemerkung~\ref{bem:mackey-top-wstar}(\ref{item:mackey-top-wstar-i}),
  die Stetigkeit bezüglich $\TT_{\tau}$ nachzuweisen. Nach
  Korollar~\ref{kor:mackey-arens} wiederum genügt es, die
  $\TT_{\tau}$-Stetigkeit auf der Einheitskugel $S$ zu zeigen. Sei
  dazu $(x_{i})_{i\in I}$ ein bezüglich der Mackey-Topologie
  $\TT_{\tau}$ gegen $x$ konvergentes Netz in $S$. Wenn wir zeigen,
  dass $(x_{i})_{i\in I}$ auch bezüglich der q-Topologie gegen $x$
  konvergiert, dann folgt $\varphi(x_{i})\to \varphi(x)$ und somit die
  behauptete Stetigkeit. Für jedes schwach-*-stetige und positive
  Funktional $\psi$ auf $A$ gilt
  \begin{displaymath}
    \alpha_{\psi}(x_{i}-x)=\psi\big((x_{i}-x)^{*}(x_{i}-x)\big)^{1/2}=\big((L_{(x_{i}-x)^{*}}\psi)(x_{i}-x)\big)^{1/2}.
  \end{displaymath}
  Aus der Dreiecksungleichung folgt
  $(x_{i}-x)^{*}=x_{i}^{*}-x^{*}\in 2S$, sodass wir
  $L_{(x_{i}-x)^{*}}\psi\in L_{2S}\psi$ und daher
  \begin{displaymath}
    0\leq\alpha_{\psi}(x_{i}-x)\leq\sup_{\kappa\in L_{2S}\psi}\big(\kappa(x_{i}-x)\big)^{1/2}=\left(\sup_{\kappa\in L_{2S}\psi}\kappa(x_{i}-x)\right)^{1/2}
  \end{displaymath}
  erhalten. Wegen Lemma~\ref{lem:l-op-stetig} ist $L_{2S}\psi$ kompakt
  bezüglich $\sigma(A^{\#},A)$, sodass die rechte Seite nach
  Lemma~\ref{lem:tech-mackey-wstar}(\ref{item:tech-mackey-wstar-ii})
  gegen $0$ konvergiert. Somit ist $x_{i}\xrightarrow{\TT_{q}} x$
  gezeigt.

  Für die Umkehrung sei $\varphi$ schwach-*-stetig. Wir zeigen die
  Stetigkeit von $\varphi$ bei $0$ bezüglich der q-Topologie. Im
  gesamten restlichen Beweis sei dazu $(x_{i})_{i\in I}$ ein Netz in
  $A$, das bezüglich $\TT_{q}$ gegen $0$ konvergiert. Im Spezialfall
  eines positiven $\varphi$ impliziert die Cauchy-Schwarz'sche
  Ungleichung,
  Lemma~\ref{lem:eig-pos-fkt}(\ref{item:eig-pos-fkt-iii}),
  \begin{displaymath}
    \abs{\varphi(x_{i})}=\abs{\varphi(1\cdot x_{i})}\leq\varphi(1)^{1/2}\varphi(x_{i}^{*}x_{i})^{1/2}=\varphi(1)^{1/2}\alpha_{\varphi}(x_{i}),
  \end{displaymath}
  sodass wir die Konvergenz $\varphi(x_{i})\to 0$ erhalten. Für den
  allgemeinen Fall sei bemerkt, dass es genügt, die
  $\TT_{q}$-Stetigkeit von $\re\varphi$ und $\im\varphi$ zu
  zeigen. Nach Lemma~\ref{lem:eig-re-im-fkt} sind Real- und
  Imaginärteil von $\varphi$ ebenfalls schwach-*-stetig, sodass wir
  ohne Beschränkung der Allgemeinheit annehmen können, dass $\varphi$
  selbstadjungiert ist. Außerdem können wir $\varphi\neq 0$
  annehmen. Wir betrachten die Menge
  \begin{displaymath}
    M:=\set{a\in A_{sa}\cap S}{\varphi(a)=\norm{\varphi}}
  \end{displaymath}
  und zeigen als Erstes $M\neq\emptyset$. Da $S$ bezüglich
  $\TT_{w^{*}}$ kompakt ist, existiert ein $b\in S$ mit
  $\abs{\varphi(b)}=\max_{x\in
    S}\abs{\varphi(x)}=\norm{\varphi}$. Durch Betrachten eines
  geeigneten Vielfachen $\lambda b$ für $\lambda$ in der komplexen
  Einheitskreislinie können wir $\varphi(b)>0$ annehmen. Wegen
  $\varphi^{*}=\varphi$ gilt\footnote{Man beachte, dass wir
    $\varphi(b^{*})=\overline{\varphi(b)}$ in
    Lemma~\ref{lem:pos-fkt-konj} nur für \emph{positive} Funktionale
    gezeigt haben.}
  $\varphi(b^{*})=\overline{\varphi^{*}(b)}=\overline{\varphi(b)}$. Daraus
  folgt
  \begin{equation}\label{eq:bew-dualraum-q-top}
    \varphi(\re b)=\frac{1}{2}(\varphi(b)+\varphi(b^{*}))=\frac{1}{2}(\varphi(b)+\overline{\varphi(b)})=\re\varphi(b)=\varphi(b)=\norm{\varphi}.
  \end{equation}
  Das Element $\re b\in A_{sa}\cap S$ liegt somit in $M$. Weiters ist
  $M$ als $\TT_{w^{*}}$-abgeschlossene Teilmenge von $S$ kompakt
  bezüglich $\TT_{w^{*}}$ und hat daher nach dem Satz von Krein-Milman
  einen Extremalpunkt, den wir mit $a_{0}$ bezeichnen. Wir behaupten,
  dass $a_{0}$ sogar ein Extremalpunkt von ganz $A_{sa}\cap S$
  ist. Ist nämlich $a_{0}=(c+d)/2$ mit $c,d\in A_{sa}\cap S$, so
  erhalten wir $\re\varphi(c)=\varphi(\re c)=\varphi(c)$ wie in
  \eqref{eq:bew-dualraum-q-top}, also $\varphi(c)\in\R$. Analog gilt
  auch $\varphi(d)\in\R$. Daraus folgt
  \begin{displaymath}
    \varphi(c)\leq\abs{\varphi(c)}\leq\norm{\varphi}\cdot\norm{c}\leq\norm{\varphi}
  \end{displaymath}
  und genauso $\varphi(d)\leq\norm{\varphi}$. Wir erhalten
  \begin{displaymath}
    \norm{\varphi}=\varphi(a_{0})=\frac{\varphi(c)+\varphi(d)}{2}\leq\norm{\varphi},
  \end{displaymath}
  was nur für $\varphi(c)=\norm{\varphi}=\varphi(d)$ möglich ist. Es
  folgt $c,d\in M$ und weiter $c=a_{0}=d$, da $a_{0}$ ein
  Extremalpunkt von $M$ ist. Nach
  Satz~\ref{satz:extremalpkt-selbstadj} ist $a_{0}$ unitär, sodass
  $a_{0}^{2}=a_{0}^{*}a_{0}=1$ gilt. Außerdem ist die Abbildung
  $x\mapsto a_{0}x$ ein isometrischer Isomorphismus; siehe
  \eqref{eq:bew-extremalpkt-selbstadj} im Beweis von
  Satz~\ref{satz:extremalpkt-selbstadj}. Setzen wir
  $\psi:=L_{a_{0}}\varphi$, so erhalten wir ein schwach-*-stetiges
  Funktional mit $\norm{\psi}=\norm{\varphi}$. Wegen $a_{0}\in M$ gilt
  \begin{displaymath}
    \psi(1)=\varphi(a_{0})=\norm{\varphi}=\norm{\psi},
  \end{displaymath}
  sodass $\psi$ nach Korollar~\ref{kor:char-pos-fkt} ein positives
  Funktional ist. Wir schätzen $\varphi(x_{i})$ ähnlich wie oben mit
  der Cauchy-Schwarz'schen Ungleichung und unter nochmaliger
  Verwendung von $a_{0}^{2}=1$ ab:
  \begin{displaymath}
    \varphi(x_{i})=\varphi(a_{0}^{2}x_{i})=\psi(a_{0}x_{i})\leq
    \psi(a_{0}^{2})^{1/2}\psi(x_{i}^{*}x_{i})^{1/2}=\psi(1)^{1/2}\alpha_{\psi}(x_{i})
  \end{displaymath}
  Daraus folgt $\varphi(x_{i})\to 0$, also die $\TT_{q}$-Stetigkeit
  von $\varphi$.
\end{proof}
Als weitere Vorarbeit müssen wir zeigen, dass das Supremum
bzw. äquivalent der Grenzwert -- vgl. Satz~\ref{satz:wstar-sup-stetig}
-- eines monoton wachsenden Netzes von Projektionen wieder eine
Projektion ist. Man beachte, dass dies nicht einfach durch
Grenzwertbildung in der Gleichung $p_{i}^{2}=p_{i}$ möglich ist, da
die Multiplikation nicht simultan stetig ist.
\begin{lemma}\label{lem:proj-wstar-abg}
  \hspace{0mm}
  \begin{enumerate}[label=(\roman*),ref=\roman*]
  \item\label{item:proj-wstar-abg-i} Ist $(x_{i})_{i\in I}$ ein
    gleichmäßig beschränktes und bezüglich $\TT_{w^{*}}$ gegen $0$
    konvergentes Netz positiver Elemente, so gilt auch
    $x_{i}^{2}\to 0$.
  \item\label{item:proj-wstar-abg-ii} Ist $(p_{i})_{i\in I}$ ein
    monoton wachsendes Netz von Projektionen, so ist
    $p:=\sup_{i\in I}p_{i}$ ebenfalls eine Projektion.
  \end{enumerate}
\end{lemma}
\begin{proof}
  \hspace{0mm}
  \begin{enumerate}[label=(\roman*),ref=\roman*]
  \item Sei $(x_{i})_{i\in I}$ ein Netz mit den vorausgesetzten
    Eigenschaften, wobei $C>0$ so gewählt ist, dass
    $\norm{x_{i}}\leq C$ für alle $i\in I$ gilt. Wir verwenden die
    Notation aus dem Beweis von Satz~\ref{satz:wstar-sup-stetig}, also
    bezeichnet $E$ die Menge der schwach-*-stetigen, positiven
    Funktionale auf $A$. Da alle $x_{i}^{2}$ in der kompakten Menge
    $C^{2}S$ enthalten sind, genügt es wie im Beweis von
    Satz~\ref{satz:wstar-sup-stetig} zu zeigen, dass das Netz
    $(x_{i}^{2})_{i\in I}$ bezüglich der schwachen Topologie
    $\sigma(A,E)$ gegen $0$ konvergiert. Mit anderen Worten ist die
    Konvergenz $\varphi(x_{i}^{2})\to 0$ für alle schwach-*-stetigen,
    \emph{positiven} Funktionale $\varphi$ zu zeigen.

    Aus Lemma~\ref{lem:eig-ho}(\ref{item:eig-ho-iii}) folgt
    $x_{i}\leq\norm{x_{i}}\leq C$, sodass wir mit
    Lemma~\ref{lem:eig-ho}(\ref{item:eig-ho-ii})
    \begin{displaymath}
      0\leq x_{i}^{2}=(x_{i}^{1/2})^{*}x_{i}x_{i}^{1/2}\leq (x_{i}^{1/2})^{*}Cx_{i}^{1/2}=Cx_{i}
    \end{displaymath}
    erhalten. Da $\varphi$ ein positives Funktional ist, folgt
    \begin{displaymath}
      0\leq \varphi(x_{i}^{2})\leq\varphi(Cx_{i})=C\varphi(x_{i})\to 0.
    \end{displaymath}
  \item Da $A_{sa}$ bezüglich $\TT_{w^{*}}$ abgeschlossen ist, haben
    wir nur $p^{2}=p$ zu zeigen. Die Elemente $x_{i}:=p-p_{i}$ sind
    positiv und konvergieren gegen $0$, da aus
    Satz~\ref{satz:wstar-sup-stetig} die Konvergenz $p_{i}\to p$
    folgt. Außerdem sind sie wegen $\norm{x_{i}}\leq 2$ gleichmäßig
    beschränkt. Mit (\ref{item:proj-wstar-abg-i}) folgt
    $x_{i}^{2}\to 0$. Wegen $p_{i}^{2}=p_{i}$ erhalten wir
    $x_{i}^{2}=p^{2}-pp_{i}-p_{i}p+p_{i}$, sodass die Stetigkeit der
    Translationen $x\mapsto px$ und $x\mapsto xp$ die Gleichung
    \begin{displaymath}
      p^{2}=\lim_{i\in I}pp_{i}+p_{i}p-p_{i}=p^{2}+p^{2}-p
    \end{displaymath}
    bzw. $p^{2}=p$ liefert.
  \end{enumerate}
\end{proof}
Nun können wir die bereits angekündigte Charakterisierung der
schwach-*-stetigen, positiven Funktionale durch die Verträglichkeit
mit der Ordnungsstruktur beweisen.

\begin{satz}\label{satz:char-wstar-fkt-sup}
  Ein positives Funktional $\varphi$ auf $A$ ist genau dann
  schwach-*-stetig, wenn für alle monoton wachsenden, gleichmäßig
  beschränkten Netze $(x_{i})_{i\in I}$ aus positiven Elementen
  \begin{equation}\label{eq:char-wstar-fkt-sup}
    \varphi(\sup_{i\in I}x_{i})=\sup_{i\in I}\varphi(x_{i})
  \end{equation}
  gilt.
\end{satz}
\begin{proof}
  Sei zunächst $\varphi$ schwach-*-stetig. Nach
  Satz~\ref{satz:wstar-sup-stetig} ist ein monoton wachsendes und
  gleichmäßig beschränktes Netz $(x_{i})_{i\in I}$ bezüglich
  $\TT_{w^{*}}$-konvergent mit
  $\lim_{i\in I}x_{i}=\sup_{i\in I}x_{i}$. Wegen der Positivität von
  $\varphi$ gilt für $i\preccurlyeq j$ die Ungleichung
  $\varphi(x_{j}-x_{i})\geq 0$, also
  $\varphi(x_{i})\leq\varphi(x_{j})$. Folglich ist
  $(\varphi(x_{i}))_{i\in I}$ ein monoton wachsendes Netz in
  $[0,+\infty)$. Dieses Netz ist außerdem aufgrund von
  $\abs{\varphi(x_{i})}\leq\norm{\varphi}\cdot\norm{x_{i}}$ ebenfalls
  gleichmäßig beschränkt. Somit ist es konvergent, wobei
  $\lim\varphi(x_{i})=\sup_{i\in I}\varphi(x_{i})$ ist. Aus der
  schwach-*-Stetigkeit folgt
  \begin{displaymath}
    \varphi(\sup_{i\in I}x_{i})=\varphi(\lim_{i\in I}x_{i})=\lim_{i\in
      I}\varphi(x_{i})=\sup_{i\in I}\varphi(x_{i}),
  \end{displaymath}
  also \eqref{eq:char-wstar-fkt-sup}.

  Gelte umgekehrt die Bedingung \eqref{eq:char-wstar-fkt-sup} für alle
  monoton wachsenden und gleichmäßig beschränkten Netze
  $(x_{i})_{i\in I}$. Zunächst zeigen wir, dass es eine bezüglich
  $\leq$ maximale Projektion $p_{0}$ gibt, für die die
  Abbildung\footnote{Diese Abbildung stimmt nicht mit
    $L_{p_{0}}\varphi$ überein!}  $x\mapsto\varphi(xp_{0})$
  schwach-*-stetig ist. Dazu betrachten wir die durch $\leq$
  halbgeordnete Menge
  \begin{displaymath}
    M:=\set{p\in A}{p\ \text{Projektion},\, x\mapsto\varphi(xp)\ \text{schwach-*-stetig}}
  \end{displaymath}
  und verwenden das Lemma von Zorn. Wegen $0\in M$ ist $M$ nicht
  leer. Ist $\KK$ eine $\leq$-Kette in $M$, so betrachten wir die
  Kette als Netz $(q)_{q\in\KK}$, wobei die Richtung auf $\KK$ durch
  die Ordnung $\leq$ gegeben ist. Dieses Netz besteht aus positiven
  Elementen und ist monoton wachsend sowie gleichmäßig beschränkt, da
  Projektionen stets $\norm{q}\leq 1$ erfüllen. Nach
  Satz~\ref{satz:wstar-sup-stetig} existiert das Element
  $p:=\sup_{q\in\KK}q=\lim_{q\in\KK}q$, das wegen
  Lemma~\ref{lem:proj-wstar-abg} eine Projektion ist. Sobald die
  schwach-*-Stetigkeit von $x\mapsto\varphi(xp)$ gezeigt ist, folgt
  $p\in M$, womit die Kette $\KK$ eine obere Schranke in $M$ hat und
  das Lemma von Zorn anwendbar ist. Dafür genügt es nach
  Korollar~\ref{kor:banach-dieud} bzw. Bemerkung~\ref{bem:wstar-top},
  die Stetigkeit auf $S$ zu zeigen. Für $x\in S$ und $q\in\KK$ gilt
  nach der Cauchy-Schwarz'schen Ungleichung,
  Lemma~\ref{lem:eig-pos-fkt}(\ref{item:eig-pos-fkt-iii})
  \begin{align*}
    \abs{\varphi(xp)-\varphi(xq)}&=\abs{\varphi(x(p-q))}=\abs{\varphi\left(x(p-q)^{1/2}(p-q)^{1/2}\right)}\\
                                 &=\abs{\varphi\left(\big((p-q)^{1/2}x^{*}\big)^{*}(p-q)^{1/2}\right)}\\
                                 &\leq\varphi\left(\big((p-q)^{1/2}x^{*}\big)^{*}\big((p-q)^{1/2}x^{*}\big)\right)^{1/2}\varphi(p-q)^{1/2}\\
                                 &=\varphi(x(p-q)x^{*})^{1/2}\varphi(p-q)^{1/2}.
  \end{align*}
  Wegen Lemma~\ref{lem:eig-ho}(\ref{item:eig-ho-v}) gilt
  $\norm{p-q}\leq 1$, sodass wir $\norm{x(p-q)x^{*}}\leq 1$ und weiter
  \begin{equation}\label{eq:bew-char-wstar-fkt-sup-i}
    \abs{\varphi(xp)-\varphi(xq)}\leq\norm{\varphi}^{1/2}\varphi(p-q)^{1/2}
  \end{equation}
  erhalten. Aus der Voraussetzung \eqref{eq:char-wstar-fkt-sup} folgt
  $\varphi(p)=\sup_{q\in\KK}\varphi(q)$. Da die nichtnegativen Zahlen
  $\varphi(q)$ für $q\in\KK$ monoton wachsend sind, stimmt dieses
  Supremum mit dem Grenzwert $\lim_{q\in\KK}\varphi(q)$ überein. Es
  gilt also
  \begin{displaymath}
    \lim_{q\in\KK}\varphi(p-q)=\varphi(p)-\lim_{q\in\KK}\varphi(q)=0,
  \end{displaymath}
  sodass die rechte Seite von \eqref{eq:bew-char-wstar-fkt-sup-i}
  gegen $0$ konvergiert. Folglich konvergieren die Funktionen
  $x\mapsto\varphi(xq)$ für $q\in\KK$ gleichmäßig in $S$ gegen
  $x\mapsto\varphi(xp)$. Dies zeigt die gesuchte Stetigkeit auf $S$
  und damit auf $A$, sodass die maximale Projektion $p_{0}$
  tatsächlich existiert.

  Wir zeigen $p_{0}=1$, indem wir das Gegenteil auf einen Widerspruch
  führen. Wegen $p_{0}\leq 1$ bedeutet dies $1-p_{0}>0$. Nach
  Korollar~\ref{kor:wstar-pos-fkt-trennend}(\ref{item:wstar-pos-fkt-trennend-ii})
  gibt es ein schwach-*-stetiges, positives Funktional $\psi$ mit
  $\psi(1-p_{0})\neq 0$. Da $\psi$ positiv ist, muss $\psi(1-p_{0})>0$
  gelten. Wegen der Positivität von $\varphi$ gilt auch
  $\varphi(1-p_{0})\geq 0$. Indem wir $\psi$ mit einer geeigneten
  positiven Zahl multiplizieren, können wir
  \begin{equation}\label{eq:bew-char-wstar-fkt-sup-ii}
    \varphi(1-p_{0})<\psi(1-p_{0})
  \end{equation}
  annehmen. Als Nächstes zeigen wir, dass es eine Projektion $p_{1}$
  gibt mit $0\neq p_{1}\leq 1-p_{0}$ und
  \begin{equation}\label{eq:bew-char-wstar-fkt-sup-iii}
    \varphi(p)<\psi(p)\quad\text{für alle Projektionen}\  0\neq p\leq p_{1}.
  \end{equation}
  Erneut nehmen wir das Gegenteil an, womit es für alle Kandidaten
  $0\neq p'\leq 1-p_{0}$ eine Projektion $p$ mit
  \begin{equation}\label{eq:bew-char-wstar-fkt-sup-iv}
    0\neq p\leq p'\quad\text{und}\quad \varphi(p)\geq\psi(p)
  \end{equation}
  gibt. Wir wollen nochmals das Lemma von Zorn anwenden, diesmal auf
  die Menge
  \begin{displaymath}
    N:=\set{p\in A}{p\ \text{Projektion}, \,p\leq 1-p_{0}, \,\varphi(p)\geq\psi(p)}.
  \end{displaymath}
  Aus \eqref{eq:bew-char-wstar-fkt-sup-iv} für $p'=1-p_{0}$ folgt
  $N\neq\emptyset$. Ist $\KK$ eine Kette in $N$, so betrachten wir das
  Supremum $p:=\sup_{q\in\KK}q$ und schließen wie oben, dass $p$ eine
  Projektion ist. Direkt nach Definition gilt außerdem
  $p\leq 1-p_{0}$, sodass wir nur noch $\varphi(p)\geq\psi(p)$ zu
  zeigen haben, um $p\in N$ nachzuweisen. Die Funktionale $\varphi$
  und $\psi$ erfüllen beide die Bedingung
  \eqref{eq:char-wstar-fkt-sup}: Für $\varphi$ gilt dies nach
  Voraussetzung, für das schwach-*-stetige $\psi$ nach dem allerersten
  Teil des Beweises. Daraus folgt
  \begin{displaymath}
    \varphi(p)=\varphi\big(\sup_{q\in\KK}q\big)=\sup_{q\in\KK}\varphi(q)\geq\sup_{q\in\KK}\psi(q)=\psi\big(\sup_{q\in\KK}q\big)=\psi(p),
  \end{displaymath}
  sodass wir $p\in N$ erhalten und daher das Lemma von Zorn anwendbar
  ist. Sei $q_{0}$ ein maximales Element von $N$. Da $1-p_{0}$ wegen
  \eqref{eq:bew-char-wstar-fkt-sup-ii} nicht in $N$ enthalten ist,
  gilt $q_{0}<1-p_{0}$ bzw. $1-p_{0}-q_{0}>0$. Wir behaupten, dass
  dieses Element eine Projektion ist, wobei es genügt, $p_{0}+q_{0}$
  als solche zu identifizieren. Die Selbstadjungiertheit ist klar, für
  die Idempotenz berechnen wir
  \begin{equation}\label{eq:bew-char-wstar-fkt-sup-v}
    (p_{0}+q_{0})^{2}=p_{0}+q_{0}+p_{0}q_{0}+q_{0}p_{0}.
  \end{equation}
  Aus $0\leq q_{0}\leq 1-p_{0}$ folgt mit
  Lemma~\ref{lem:eig-ho}(\ref{item:eig-ho-ii})
  \begin{equation}\label{eq:bew-char-wstar-fkt-sup-vi}
    0=p_{0}0p_{0}\leq p_{0}q_{0}p_{0}\leq p_{0}(1-p_{0})p_{0}=0,
  \end{equation}
  also $p_{0}q_{0}p_{0}=0$. Wir erhalten
  \begin{displaymath}
    \norm{q_{0}p_{0}}=\norm{(q_{0}p_{0})^{*}(q_{0}p_{0})}^{1/2}=\norm{p_{0}q_{0}p_{0}}^{1/2}=0.
  \end{displaymath}
  Damit folgt auch $p_{0}q_{0}=(q_{0}p_{0})^{*}=0$, sodass
  $p_{0}+q_{0}$ tatsächlich eine Projektion ist. Die Annahme
  \eqref{eq:bew-char-wstar-fkt-sup-iv} für $p'=1-p_{0}-q_{0}$ liefert
  eine Projektion $0\neq q\leq 1-p_{0}-q_{0}$ mit
  $\varphi(q)\geq\psi(q)$. Eine analoge Rechnung zu
  \eqref{eq:bew-char-wstar-fkt-sup-v} und
  \eqref{eq:bew-char-wstar-fkt-sup-vi} ausgehend von
  $0\leq q\leq 1-q_{0}$ zeigt, dass $q_{0}+q$ eine Projektion
  ist. Außerdem gilt $q_{0}+q\leq 1-p_{0}$ und
  $\varphi(q_{0}+q)\geq\psi(q_{0}+q)$, sodass wir $q_{0}+q\in N$
  erhalten. Dies widerspricht der Maximalität von $q_{0}$, sodass die
  Annahme \eqref{eq:bew-char-wstar-fkt-sup-iv} falsch gewesen sein
  muss. Somit gibt es tatsächlich eine Projektion
  $0\neq p_{1}\leq 1-p_{0}$ mit
  \eqref{eq:bew-char-wstar-fkt-sup-iii}. Lassen wir auch $p=0$ zu, so
  folgt also $\varphi(p)\leq\psi(p)$ für sämtliche Projektionen
  $p\leq p_{1}$.

  Als Zwischenschritt beweisen wir nun, dass sogar
  $\varphi(a)\leq\psi(a)$ für alle positiven Elemente $a$ von
  $B:=p_{1}Ap_{1}$ gilt. Da die Translationen $x\mapsto p_{1}x$ und
  $x\mapsto xp_{1}$ schwach-*-stetig sind (oder auch direkt nach dem
  ersten Beweisschritt von Satz~\ref{satz:mult-wstar-stetig}), ist $B$
  schwach-*-abgeschlossen in $A$ und somit selbst eine
  $W^{*}$-Algebra. Eine Projektion in $p\in B$ ist natürlich auch eine
  Projektion in $A$. Außerdem gilt $p\leq\norm{p}\leq 1$ nach
  Lemma~\ref{lem:eig-ho}(\ref{item:eig-ho-iii}). Wegen $p\in B$ gilt
  $p=p_{1}pp_{1}$, sodass wir mit
  Lemma~\ref{lem:eig-ho}(\ref{item:eig-ho-ii}) die schärfere
  Abschätzung
  \begin{displaymath}
    p=p_{1}pp_{1}\leq p_{1}1p_{1}=p_{1}
  \end{displaymath}
  erhalten. Alle Projektionen in $B$ erfüllen also $p\leq p_{1}$ und
  daher $\varphi(p)\leq\psi(p)$. Für Linearkombinationen
  $b=\sum_{k=1}^{n}\gamma_{k}p_{k}$ von Projektionen in $B$ mit
  nichtnegativen Koeffizienten $\gamma_{k}$ gilt dieselbe Ungleichung
  $\varphi(b)\leq\psi(b)$. Wegen der Zusatzaussage von
  Korollar~\ref{kor:wstar-proj-dicht}, angewandt in $B$, lassen sich
  alle positiven Elemente von $B$ als Grenzwerte bezüglich
  $\norm{\cdot}$ von derartigen Linearkombinationen schreiben. Da
  $\varphi$ und $\psi$ als positive Funktionale stetig bezüglich der
  Normtopologie sind (siehe
  Lemma~\ref{lem:eig-pos-fkt}(\ref{item:eig-pos-fkt-i})), folgt
  tatsächlich $\varphi(a)\leq\varphi(a)$ für alle
  $0\leq a\in B=p_{1}Ap_{1}$.
  
  Um unsere ursprüngliche Annahme $1-p_{0}>0$ zu widerlegen, zeigen
  wir $p_{0}+p_{1}\in M$, was der Maximalität von $p_{0}$
  widerspricht. Eine weitere Variante der Rechnungen
  \eqref{eq:bew-char-wstar-fkt-sup-v} und
  \eqref{eq:bew-char-wstar-fkt-sup-vi}, diesmal ausgehend von
  $0\leq p_{1}\leq 1-p_{0}$, zeigt, dass $p_{0}+p_{1}$ eine Projektion
  ist. Folglich ist der Beweis abgeschlossen, wenn wir die
  schwach-*-Stetigkeit der Abbildung
  \begin{displaymath}
    x\mapsto\varphi(x(p_{0}+p_{1}))=\varphi(xp_{0})+\varphi(xp_{1})
  \end{displaymath}
  gezeigt haben. Wegen $p_{0}\in M$ ist $x\mapsto\varphi(xp_{0})$
  stetig. Für $x\mapsto\varphi(xp_{1})$ zeigen wir die Stetigkeit
  bezüglich der q-Topologie; dies reicht nach
  Lemma~\ref{lem:dualraum-q-top} aus. Nach einem schon wiederholt
  verwendeten, einfachen Resultat der Funktionalanalysis müssen wir
  nur die $\TT_{q}$-Stetigkeit bei $0$ zeigen. Sei also
  $(x_{i})_{i\in I}$ ein Netz in $A$, das bezüglich der q-Topologie
  gegen $0$ konvergiert. Mit der Cauchy-Schwarz'schen Ungleichung,
  Lemma~\ref{lem:eig-pos-fkt}(\ref{item:eig-pos-fkt-iii}), erhalten
  wir
  \begin{displaymath}
    \abs{\varphi(x_{i}p_{1})}=\abs{\varphi(1\cdot
      x_{i}p_{1})}\leq\varphi(1)^{1/2}\varphi(p_{1}x_{i}^{*}x_{i}p_{1})^{1/2}.
  \end{displaymath}
  Das Element $p_{1}x_{i}^{*}x_{i}p_{1}$ liegt in $B$ und ist positiv,
  sodass der obige Zwischenschritt
  \begin{displaymath}
    \varphi(p_{1}x_{i}^{*}x_{i}p_{1})\leq\psi(p_{1}x_{i}^{*}x_{i}p_{1})
  \end{displaymath}
  liefert. Kombinieren wir diese beiden Abschätzungen, so folgt
  \begin{displaymath}
    \abs{\varphi(x_{i}p_{1})}\leq\varphi(1)^{1/2}\psi(p_{1}x_{i}^{*}x_{i}p_{1})^{1/2}.
  \end{displaymath}
  Das Funktional $\kappa(x):=\psi(p_{1}xp_{1})$ ist schwach-*-stetig
  und positiv, da wegen Lemma~\ref{lem:eig-ho}(\ref{item:eig-ho-ii})
  für $a\geq 0$ auch $p_{1}ap_{1}\geq 0$ gilt und $\psi$ ein positives
  Funktional ist. Daraus ergibt sich
  \begin{displaymath}
    \abs{\varphi(x_{i}p_{1})}\leq\varphi(1)^{1/2}\alpha_{\kappa}(x_{i})\to 0,
  \end{displaymath}
  also $\varphi(x_{i}p_{1})\to 0$ und infolge die behauptete
  $\TT_{q}$-Stetigkeit.

  Somit haben wir $p_{0}\neq 1$ auf einen Widerspruch geführt. Also
  gilt $1\in M$, woraus die schwach-*-Stetigkeit von $\varphi$ folgt.
\end{proof}
Da die Bedingung~\eqref{eq:char-wstar-fkt-sup} keinerlei Verweis auf
die schwach-*-Topologie oder den Prädual\-raum enthält, erhalten wir
als unmittelbares Korollar die folgende Aussage, die einen wichtigen
Schritt in Richtung Eindeutigkeit des Prädualraums darstellt.
\begin{korollar}\label{kor:char-wstar-fkt-sup}
  Sei $\pred[1]{A}$ ein weiterer Prädualraum von $A$ und
  $j_{1}:A\to\pred[1]{A}'$ ein isometrischer Isomorphismus. Bezeichnet
  $\TT_{w^{*}}^{(1)}$ die ausgehend von $\pred[1]{A}$ und $j_{1}$
  gebildete schwach-*-Topologie, so ist ein \emph{positives}
  Funktional $\varphi$ auf $A$ genau dann $\TT_{w^{*}}$-stetig, wenn
  es $\TT_{w^{*}}^{(1)}$-stetig ist.
\end{korollar}

Als Nächstes zeigen wir das bereits angekündigte Zerlegungsresultat,
präzise formuliert dass man ein schwach-*-stetiges Funktional
$\varphi$ in der Form
\begin{displaymath}
  \varphi=(\varphi_{1}-\varphi_{2})+i(\varphi_{3}-\varphi_{4})
\end{displaymath}
mit schwach-*-stetigen und \emph{positiven} Funktionalen $\varphi_{k}$
für $k=1,\dots,4$ darstellen kann. Die Ähnlichkeit zur analogen
Zerlegung in $C^{*}$-Algebren aus \eqref{eq:pos-negteil-span-i} in
Bemerkung~\ref{bem:pos-negteil-span} ist kein Zufall. Wir werden
nämlich in unserem Beweis, abweichend von
\cite[1.14.3~Theorem]{sakai:cstar-wstar}, diese Darstellung verwenden.

Da jede $W^{*}$-Algebra $A$ nach dem Satz von Sakai,
Satz~\ref{satz:sakai}, isomorph zu einer Von-Neumann-Algebra ist, wird
es genügen, den Fall einer Von-Neumann-Algebra zu betrachten. Wir
bestimmen zunächst die schwach-*-stetigen Funktionale.
\begin{lemma}\label{lem:vn-alg-dualraum}
  Sei $A\leq L_{b}(H)$ eine Von-Neumann-Algebra und
  \begin{displaymath}
    \theta:(L_{b}(H),\TT_{uw})\to\big(L^{1}(H)',\sigma(L^{1}(H)',L^{1}(H))\big)
  \end{displaymath}
  der kanonische lineare Homöomorphismus aus
  Satz~\ref{satz:uw-optop-wstar}, also $\theta(T)=\tr(.T)$. Dann sind
  die schwach-*-stetigen Funktionale auf $A$ gegeben durch alle
  Funktionale der Form $\varphi=\tr(S.)|_{A}$ für Spurklasseoperatoren
  $S\in L^{1}(H)$. Zu gegebenem $\varphi$ ist der Operator $S$ dabei
  bis auf Elemente des Linksannihilators
  \begin{displaymath}
    \lanh{(\theta(A))}=\set{S\in L^{1}(H)}{\tr(ST)=0\
      \text{für alle}\ T\in A}
  \end{displaymath}
  eindeutig bestimmt.
\end{lemma}
\begin{proof}
  Nach Lemma~\ref{lem:top-wstar-unteralg} ist die schwach-*-Topologie
  auf $A$ die Spurtopologie der schwach-*-Topologie auf
  $L_{b}(H)$. Diese ist genau die ultraschwache Operatortopologie
  $\TT_{uw}$; siehe
  Bemerkung~\ref{bem:wstar-top}(\ref{item:wstar-top-i}).

  Jedes Funktional der Form $\tr(S.)|_{A}$ ist als Einschränkung eines
  gemäß Lemma~\ref{lem:optop-dualraeume} bezüglich $\TT_{uw}$ stetigen
  Funktionals stetig bezüglich $(\TT_{uw})|_{A}$.

  Sei umgekehrt $\varphi$ schwach-*-stetig auf $A$, also stetig
  bezüglich $(\TT_{uw})|_{A}$. Nach dem Satz von Hahn-Banach gibt es
  eine Fortsetzung zu einem $\TT_{uw}$-stetigen Funktional $f$ auf
  $L_{b}(H)$. Lemma~\ref{lem:optop-dualraeume} liefert einen
  Spurklasseoperator $S\in L^{1}(H)$ mit $f=\tr(S.)$. Wir erhalten
  $\varphi=\tr(S.)|_{A}$. Die Eindeutigkeitsaussage folgt direkt aus
  der Definition von $\lanh{(\theta(A))}$.
\end{proof}
Das nächste Lemma verknüpft die Positivität von Spurklasseoperatoren
mit der Positivität des induzierten Funktionals.
\begin{lemma}\label{lem:pos-spklasseop-pos-fkt}
  Für einen positiven Spurklasseoperator $S$ ist auch das Funktional
  $\varphi=\tr(S.)$ positiv.
\end{lemma}
\begin{proof}
  Wir müssen $\tr(ST)\geq 0$ für alle Operatoren $T\in L_{b}(H)$ mit
  $T\geq 0$ nachweisen.

  Dazu beweisen wir zunächst $\tr(R)\geq 0$ für positive
  Spurklasseoperatoren $R$. Ist $E$ eine Orthonormalbasis von $H$, so
  gilt
  \begin{equation}\label{eq:bew-pos-spklasseop-pos-fkt}
    \tr(R)=\sum_{e\in E}(Re,e)=\sum_{e\in
      E}(R^{1/2}R^{1/2}e,e)=\sum_{e\in E}(R^{1/2}e,R^{1/2}e)\geq 0.
  \end{equation}
  
  Für einen positiven Operator $T\in A$ folgt
  $0=S^{1/2}0S^{1/2}\leq S^{1/2}TS^{1/2}$ aus
  Lemma~\ref{lem:eig-ho}(\ref{item:eig-ho-ii}). Wir behaupten, dass
  $S^{1/2}T$ und $S^{1/2}$ Hilbert-Schmidt-Operatoren sind; siehe
  Definition~\ref{def:hs-spklasse-op}. Dazu verwenden wir
  Lemma~\ref{lem:hs-equiv} und zeigen, dass das Quadrat des Betrags
  der beiden Operatoren ein Spurklasseoperator ist. Für $S^{1/2}$ gilt
  \begin{displaymath}
    \abs{S^{1/2}}^{2}=S\in L^{1}(H)
  \end{displaymath}
  und für $S^{1/2}T$ berechnen wir
  \begin{displaymath}
    \abs{S^{1/2}T}^{2}=(S^{1/2}T)^{*}(S^{1/2}T)=T^{*}ST,
  \end{displaymath}
  was nach Lemma~\ref{lem:eig-spklasse}(\ref{item:eig-spklasse-ii})
  ebenfalls ein Spurklasseoperator ist. Somit können wir
  Lemma~\ref{lem:tr-komm} mit $S^{1/2}$ anstelle von $S$ und
  $S^{1/2}T$ anstelle von $T$ anwenden und erhalten nach
  \eqref{eq:bew-pos-spklasseop-pos-fkt}
  \begin{displaymath}
    \tr(ST)=\tr\left(S^{1/2}\left(S^{1/2}T\right)\right)=\tr\left(\left(S^{1/2}T\right)S^{1/2}\right)\geq 0.
  \end{displaymath}
\end{proof}
Nach diesen Vorbereitungen können wir das Zerlegungsresultat
beweisen.
\begin{satz}\label{satz:wstar-fkt-zerl}
  Ist $\varphi$ ein schwach-*-stetiges Funktional auf der
  $W^{*}$-Algebra $A$, so existieren schwach-*-stetige und positive
  Funktionale $\varphi_{1},\dots,\varphi_{4}$ mit
  \begin{displaymath}
    \varphi=(\varphi_{1}-\varphi_{2})+i(\varphi_{3}-\varphi_{4}).
  \end{displaymath}
  Insbesondere ist der Raum $A^{\#}$ der schwach-*-stetigen Funktionale
  genau die lineare Hülle der schwach-*-stetigen und positiven
  Funktionale.
\end{satz}
\begin{proof}
  Wir nehmen zunächst an, dass $A\leq L_{b}(H)$ eine
  Von-Neumann-Algebra ist. Nach Lemma~\ref{lem:vn-alg-dualraum} gibt
  es einen Spurklasseoperator $S\in L^{1}(H)$ mit
  $\varphi=\tr(S.)|_{A}$. Nun betrachten wir $S$ als Element der
  $C^{*}$-Algebra $L_{b}(H)$ und schreiben $S$ in der Form
  \begin{equation}\label{eq:bew-wstar-fkt-zerl}
    S=\big((\re S)^{+}-(\re S)^{-}\big)+i\big((\im S)^{+}-(\im
    S)^{-}\big)
  \end{equation}
  für Operatoren $(\re S)^{\pm},(\im S)^{\pm}\in L_{b}(H)$,
  vgl. \eqref{eq:pos-negteil-span-i} in
  Bemerkung~\ref{bem:pos-negteil-span}(\ref{item:pos-negteil-span-ii}). Diese
  Operatoren sind sogar Spurklasseoperatoren, wobei wir nur
  $(\re S)^{\pm}$ behandeln: Es folgt $\re S\in L^{1}(H)$ direkt aus
  der Definition $\re S=(S+S^{*})/2$ und
  Lemma~\ref{lem:eig-spklasse}(\ref{item:eig-spklasse-ii}). Für
  Positiv- und Negativteil sei bemerkt, dass nach Definition von
  $L^{1}(H)$ auch $\abs{\re S}$ ein Spurklasseoperator ist. Daraus
  erhalten wir nach Gleichung \eqref{eq:absbetrag} in
  Bemerkung~\ref{bem:absbetrag}
  \begin{displaymath}
    (\re S)^{\pm}=\frac{1}{2}(\abs{\re S}\pm\re S)\in L^{1}(H).
  \end{displaymath}
  Mit den Funktionalen
  \begin{displaymath}
    \varphi_{1}:=\tr((\re S)^{+}.)|_{A}, \quad \varphi_{2}:=\tr((\re
    S)^{-}.)|_{A}, \quad \varphi_{3}:=\tr((\im S)^{+}.)|_{A}, \quad \varphi_{4}:=\tr((\im
    S)^{-}.)|_{A}
  \end{displaymath}
  folgt die Behauptung aus \eqref{eq:bew-wstar-fkt-zerl} sowie
  Lemma~\ref{lem:pos-spklasseop-pos-fkt} mit der Beobachtung, dass
  Einschrän\-kungen positiver Funktionale wieder positiv sind.

  Sei nun $A$ eine allgemeine $W^{*}$-Algebra. Nach dem Satz von
  Sakai, Satz~\ref{satz:sakai}, gibt es eine Von-Neumann-Algebra
  $B\leq L_{b}(H)$ und einen $W^{*}$-Isomorphismus $\Phi:A\to B$, also
  einen isometrischen Isomorphismus, der gleichzeitig ein
  Homöomorphismus bezüglich der schwach-*-Topologien ist. Ist
  $\varphi$ schwach-*-stetig auf $A$, so ist folglich
  $\psi:=\varphi\circ\Phi^{-1}$ schwach-*-stetig auf $B$. Nach dem
  oben Bewiesenen kann man $\psi$ darstellen in der Form
  \begin{displaymath}
    \psi=(\psi_{1}-\psi_{2})+i(\psi_{3}-\psi_{4})
  \end{displaymath}
  mit schwach-*-stetigen und positiven Funktionalen $\psi_{k}$ auf
  $B$. Als isometrischer Isomorphismus bildet $\Phi$ positive Elemente
  in $A$ auf positive Elemente in $B$ ab, sodass auch die Funktionale
  $\varphi_{k}:=\psi_{k}\circ\Phi$ positiv sind. Da $\Phi$ ein
  schwach-*-Homöomorphismus ist, sind die $\varphi_{k}$ auch
  schwach-*-stetig und es gilt
  \begin{displaymath}
    \varphi=(\varphi_{1}-\varphi_{2})+i(\varphi_{3}-\varphi_{4}).
  \end{displaymath}

  Die Zusatzaussage folgt unmittelbar daraus, dass eine
  Linearkombination schwach-*-stetiger Funktionale selbst
  schwach-*-stetig ist.
\end{proof}
\begin{bemerkung}
  Es sei nicht verschwiegen, dass man mit einer anderen Konstruktion
  eine Zerlegungsaussage erhalten kann, für die sogar ein
  Eindeutigkeitsresultat gilt. Für ein selbstadjungiertes Funktional
  $\varphi$ kann man die positiven Funktionale $\varphi_{1}$ und
  $\varphi_{2}$ nämlich auf eindeutige Art so wählen, dass
  $\norm{\varphi}=\norm{\varphi_{1}}+\norm{\varphi_{2}}$ gilt;
  siehe~\cite[Theorem~1.14.3]{sakai:cstar-wstar}.

  Der Schritt von der Von-Neumann-Algebra hin zu einer allgemeinen
  $W^{*}$-Algebra fußt entscheidend auf der Erweiterung des Satzes von
  Sakai, dass nicht nur $\Phi$, sondern auch $\Phi^{-1}$
  schwach-*-stetig ist. Da wir keine Eindeutigkeitsaussage benötigen,
  haben wir den Beweis in der obigen, transparenteren Form geführt.
\end{bemerkung}
Damit erhalten wir die folgende Verallgemeinerung von
Korollar~\ref{kor:char-wstar-fkt-sup}:
\begin{korollar}\label{kor:wstar-fkt-zerl}
  Sei $\pred[1]{A}$ ein weiterer Prädualraum von $A$ und
  $j_{1}:A\to\pred[1]{A}'$ ein isometrischer Isomorphismus. Bezeichnet
  $\TT_{w^{*}}^{(1)}$ die ausgehend von $\pred[1]{A}$ und $j_{1}$
  gebildete schwach-*-Topologie, so ist ein \emph{beliebiges}
  Funktional $\varphi$ auf $A$ genau dann $\TT_{w^{*}}$-stetig, wenn
  es $\TT_{w^{*}}^{(1)}$-stetig ist.
\end{korollar}
\begin{proof}
  Wenden wir Satz~\ref{satz:wstar-fkt-zerl} sowohl in
  $(A,\TT_{w^{*}})$ als auch in $\left(A,\TT_{w^{*}}^{(1)}\right)$ an, so folgt
  mit Korollar~\ref{kor:char-wstar-fkt-sup}
  \begin{align*}
    \big(A,\TT_{w^{*}}\big)'&=\spn\set{\varphi}{\varphi \ \text{ist} \,\TT_{w^{*}}\text{-stetig
                      und positiv}} \\
                    &=\spn\set{\varphi}{\varphi \ \text{ist} \,\TT_{w^{*}}^{(1)}\text{-stetig
        und positiv}}=\left(A,\TT_{w^{*}}^{(1)}\right)'.
  \end{align*}
\end{proof}
Dieses Korollar zeigt, dass die schwach-*-stetigen Funktionale vom
gewählten Prädualraum unabhängig sind, was die etwas ungenaue
Schreibweise aus Notation~\ref{not:dualraum-wstar} rechtfertigt.

\section{Das Eindeutigkeitsresultat}
\label{sec:eind-resultat}
Zum Abschluss kommen wir zum zentralen Ergebnis dieses Kapitels,
nämlich der Eindeutigkeit des Prädualraums und der
schwach-*-Topologie.
\begin{satz}\label{satz:pd-wstar-top-eindeut}
  Sei $\pred[1]{A}$ ein weiterer Prädualraum von $A$ und
  $j_{1}:A\to\pred[1]{A}'$ ein isometrischer Isomorphismus. Bezeichne
  außerdem $\TT_{w^{*}}^{(1)}$ die ausgehend von $\pred[1]{A}$ und
  $j_{1}$ gebildete schwach-*-Topologie. Dann sind die beiden
  Prädualräume isometrisch isomorph, also existiert eine lineare und
  isometrische Bijektion $\chi:\pred{A}\to\pred[1]{A}$. Außerdem gilt
  $\TT_{w^{*}}=\TT_{w^{*}}^{(1)}$.
\end{satz}
\begin{proof}
  Wir betrachten nochmals den Homöomorphismus
  \begin{displaymath}
    j'|_{\iota(\pred{A})}:\left(\iota(\pred{A}),\sigma(\pred{A}'',\pred{A}')|_{\iota(\pred{A})}\right)\to
    \left(A^{\#},\sigma(A^{\#},A)\right)
  \end{displaymath}
  aus Lemma~\ref{lem:tech-mackey-wstar}. Insbesondere ist
  $j'|_{\iota(\pred{A})}:\iota(\pred{A})\to A^{\#}$ eine lineare
  Bijektion. Da die kanonische Abbildung $\iota:\pred{A}\to\pred{A}''$
  injektiv ist, ist die Komposition
  \begin{equation}\label{eq:bew-pd-wstar-top-eindeut-i}
    j'\circ\iota:\pred{A}\to A^{\#}
  \end{equation}
  ebenfalls eine lineare Bijektion, genauso wie die analog gebildete
  Funktion
  \begin{equation}\label{eq:bew-pd-wstar-top-eindeut-ii}
    j_{1}'\circ\iota_{1}:\pred[1]{A}\to A^{\#}.
  \end{equation}
  Es sei nochmals explizit darauf hingewiesen, dass der Dualraum
  $A^{\#}$ in \eqref{eq:bew-pd-wstar-top-eindeut-i} bezüglich
  $\TT_{w^{*}}$ gebildet wird, wohingegen in
  \eqref{eq:bew-pd-wstar-top-eindeut-ii} mit $\TT_{w^{*}}^{(1)}$
  gearbeitet wird. Nach Korollar~\ref{kor:wstar-fkt-zerl} stimmen die
  beiden Dualräume aber überein. Somit können wir die lineare
  Bijektion
  \begin{displaymath}
    \chi:=(j_{1}'\circ\iota_{1})^{-1}\circ (j'\circ\iota):\pred{A}\to\pred[1]{A}
  \end{displaymath}
  betrachten. Da die konjugierten Abbildungen $j'$ und $j_{1}'$ wegen
  der Isometrie von $j$ und $j_{1}$ genauso wie die kanonischen
  Abbildungen $\iota$ und $\iota_{1}$ isometrisch sind, ist $\chi$ der
  gesuchte isometrische Isomorphismus.

  Für die Gleichheit der schwach-*-Topologien $\TT_{w^{*}}$ und
  $\TT_{w^{*}}^{(1)}$ sei daran erinnert, dass diese Topologien als
  initiale Topologie bezüglich
  \begin{displaymath}
    j:A\to\big(\pred{A}',\sigma(\pred{A}',\pred{A})\big)
  \end{displaymath}
  bzw. analog für $j_{1}$ und $\pred[1]{A}$ definiert sind. Die
  schwach-*-Topologie auf $\pred{A}'$ bzw. $\pred[1]{A}'$ ist
  ebenfalls eine initiale Topologie, nämlich bezüglich aller
  Funktionale der Form $\iota(\rho):\pred{A}'\to\C$
  bzw. $\iota_{1}(\rho_{1}):\pred[1]{A}'\to\C$ mit $\rho\in\pred{A}$
  bzw. $\rho_{1}\in\pred[1]{A}$. Da das Bilden einer initialen
  Topologie bekanntermaßen assoziativ ist, ist die schwach-*-Topologie
  $\TT_{w^{*}}$ bzw. $\TT_{w^{*}}^{(1)}$ die initiale Topologie
  bezüglich sämtlicher Verkettungen
  $\iota(\rho)\circ j=j'(\iota(\rho))$ bzw.
  $\iota_{1}(\rho_{1})\circ j_{1}=j_{1}'(\iota_{1}(\rho_{1}))$. Diese
  Kompositionen durchlaufen nach den obigen Überlegungen beide die
  gleiche Menge, nämlich den Dualraum $A^{\#}$. Somit stimmen beide
  schwach-*-Topologien mit der initialen Topologie bezüglich aller
  schwach-*-stetigen\footnote{Man beachte, dass man diese Funktionale
    nach den Sätzen~\ref{satz:wstar-fkt-zerl} und
    \ref{satz:char-wstar-fkt-sup} ohne Bezugnahme auf die Topologie
    beschreiben kann.} Funktionale überein und sind folglich gleich.
\end{proof}


\clearpage
\addcontentsline{toc}{chapter}{Literaturverzeichnis}
\bibliographystyle{gerabbrv}
\bibliography{/home/clemens/Dokumente/std_bibl}

\end{document}